\documentclass[a4paper,10pt]{amsart}
\usepackage[utf8]{inputenc} 
\usepackage[T1]{fontenc}
\usepackage{mathtools,amssymb,amsthm} 
\usepackage{dsfont}
\usepackage[all]{xy}
\usepackage{xparse} 
\usepackage{microtype}
\usepackage{mathrsfs}
\usepackage{varioref}
\usepackage{pdfpages}
\usepackage{enumitem}
\usepackage{pgfplots}
\pgfplotsset{compat=1.16}

\usetikzlibrary{arrows.meta, positioning, fit}

\newcommand{\Z}{\mathbb{Z}}
\newcommand{\R}{\mathbb{R}}
\newcommand{\C}{\mathbb{C}}

\renewcommand{\AA}{\mathcal{A}}
\newcommand{\Sc}{\mathcal{S}}

\newcommand{\GL}{\mathrm{GL}}
\newcommand{\SL}{\mathrm{SL}}
\newcommand{\SO}{\mathrm{SO}}
\newcommand{\Sp}{\mathrm{Sp}}

\newcommand{\Jz}{\mathrm{Jz}}

\newcommand{\Rep}{\mathrm{Rep}}
\newcommand{\Data}{\mathrm{Data}}

\newcommand{\gp}{\mathrm{gp}}
\newcommand{\temp}{\mathrm{temp}}
\newcommand{\Temp}{\mathrm{Temp}}
\newcommand{\Ind}{\mathrm{Ind}}
\newcommand{\Jac}{\mathrm{Jac}}

\newcommand{\scusp}{\mathrm{scusp}}

\newcommand{\soc}{\mathrm{soc}}

\newcommand{\good}{\mathrm{good}}
\newcommand{\bad}{\mathrm{bad}}
\newcommand{\ugly}{\mathrm{ugly}}

\newcommand{\Seg}{\mathrm{Seg}}

\def\m{\mathfrak{m}}
\def\n{\mathfrak{n}}
\def\soc{{\rm soc}}
\def\X{{\rm X}}
\def\MS{{\rm Mult}}

\newcommand{\pair}[1]{\left\langle #1 \right\rangle}

\newcommand{\Cusp}{\mathscr{C}}
\newcommand{\Rr}{\mathscr{R}}
\newcommand{\resp}{resp.~}

\DeclareMathOperator{\trans}{trans}

\usepackage{hyperref}

\swapnumbers 
\theoremstyle{definition} 
\newtheorem{defi}{Definition}[subsection] 
\theoremstyle{plain} 
\newtheorem{prop}[defi]{Proposition} 
\newtheorem{lem}[defi]{Lemma} 
\newtheorem{thm}[defi]{Theorem} 
\newtheorem{cor}[defi]{Corollary} 
\newtheorem{hypo}[defi]{Hypothesis} 
\newtheorem{ques}[defi]{Question} 

\theoremstyle{remark}
\newtheorem{rem}[defi]{Remark} 
\newtheorem{ex}[defi]{Example} 

\newcommand{\NN}{\mathbb{N}}

\DeclareMathOperator{\Irr}{Irr}
\DeclareMathOperator{\card}{card}
\DeclareMathOperator{\Mult}{Mult}
\DeclareMathOperator{\Symm}{Symm}
\DeclareMathOperator{\USymm}{\mathfrak{S}ymm}
\DeclareMathOperator{\USeg}{\mathfrak{S}eg}
\DeclareMathOperator{\AD}{AD}
\DeclareMathOperator{\ADd}{\mathbf{AD}}

\title[An algorithm for Aubert--Zelevinsky duality]{An algorithm for Aubert--Zelevinsky duality à la M{\oe}glin--Waldspurger}
\author{Thomas Lanard}
\address{
Laboratoire de Math\'ematiques de Versailles, 
UVSQ, 
CNRS, 
Universit\'e Paris-Saclay,
78035 Versailles, France}
\email{thomas.lanard@uvsq.fr}

\author{Alberto Mínguez}
\address{Departamento de \'Algebra and Instituto de Matem\'aticas (IMUS), Universidad de Sevilla, C/ Tarfia s/n, 41012 Sevilla \newline
University of Vienna, 
Fakult{\"a}t f{\"u}r Mathematik,
Oskar-Morgenstern-Platz 1,
1090 Wien
}
\email{alberto.minguez@univie.ac.at}

\begin{document}

\begin{abstract}
  Let $F$ be a locally compact non-Archimedean field of characteristic $0$, and let $G$ be either the split special orthogonal group $\SO_{2n+1}(F)$ or the symplectic group $\Sp_{2n}(F)$. The goal of this paper is to give an explicit description of the Aubert--Zelevinsky duality for $G$ in terms of Langlands parameters. We present a new algorithm, inspired by the M{\oe}glin--Waldspurger algorithm for $\GL_n(F)$, which computes the dual Langlands data in a recursive and combinatorial way. Our method is simple enough to be carried out by hand and provides a practical tool for explicit computations. Interestingly, the algorithm was discovered with the help of machine learning tools, guiding us toward patterns that led to its formulation.
\end{abstract}

\maketitle

\tableofcontents

\section{Introduction}\label{sec:intro}
\subsection{} 
Let $F$ be a locally compact non-Archimedean field. In 1980, A.~Zelevinsky \cite{Z} introduced an involution of the Grothendieck group of (smooth complex) finite-length representations of $\GL_n(F)$, for $n \geq 1$. Notably, this involution maps the trivial representation to the Steinberg representation and fixes every supercuspidal representation. He conjectured that this involution preserves irreducibility.

Inspired by the Alvis--Curtis duality \cite{A1, A2, Curtis}, S.-I. Kato \cite{Kato} introduced an involution on the Grothendieck group of finite-length Iwahori-fixed representations of a split connected reductive group defined over $F$. Using properties of the functor of invariants under an Iwahori subgroup \cite{Borel}, Kato was able to prove that his involution preserves irreducibility, up to a sign.

Some years later, in 1995, A.-M. Aubert \cite{Au} showed that Kato's construction can be generalized to define an involution on the Grothendieck group of finite-length representations of any connected reductive group defined over $F$ (see Paragraph \ref{sub:Aubert}), and proved that her involution preserves irreducibility, up to a sign. In the case of $\GL_n(F)$, it coincides with Zelevinsky's involution up to a sign, thereby confirming Zelevinsky's conjecture. Using the theory of coefficient systems on the Bruhat--Tits building, P. Schneider and U.~Stuhler \cite{SS} similarly defined a duality and proved that it preserves irreducibility and, at the level of the Grothendieck groups, coincides with Aubert's involution up to the contragredient. Another approach to this duality can be found in the work of J.~Bernstein, R.~Bezrukavnikov, and D.~Kazhdan \cite{BBK}.

For simplicity, when restricted to the set of irreducible representations of a connected reductive group $G$ defined over $F$, in this article, this involution will be referred to as  \emph{the Aubert--Zelevinsky duality} and will be denoted $\tau \mapsto \widehat{\tau}$.

\subsection{} 
One of the most exciting accomplishments in recent years in Number Theory has been the proof of the Local Langlands Correspondence for quasi-split classical groups, namely unitary, symplectic, and orthogonal groups. In the case when the characteristic of $F$ is $0$, using the twisted trace formula and an inductive process known as endoscopy, J. Arthur \cite{Art} established a natural bijection between the set of isomorphism classes of irreducible representations of quasi-split orthogonal and symplectic groups and the set of so-called Langlands data, thereby proving the Local Langlands Correspondence for these groups. It should be noted that these results are conditional (only) on the validity of the twisted weighted fundamental lemma; see \cite{AGIKMS} and the discussion in \textsection 0.4 therein for details.

A natural question then arises: can the Aubert--Zelevinsky duality be explicitly described in terms of  Langlands data? In other words, given the Langlands data of an irreducible representation $\tau$ of a classical group $G$, what are the Langlands data of $\widehat{\tau}$?

\subsection{} \label{sub:goodalgo}
The answer to this question is known for $\GL_n(F)$. In 1986, C.~M{\oe}glin and J.-L.~Waldspurger \cite{MW2} studied the Zelevinsky involution and developed an algorithm for computing the Langlands data of $\widehat{\tau}$ for every irreducible representation $\tau$ of $\GL_n(F)$. For $\GL_n(F)$, Zelevinsky established that the set $\Irr^\GL$ of isomorphism classes of irreducible representations of $\GL_n(F)$, for $n \geq 0$, is parametrized by $\Mult$, a set consisting of certain combinatorial objects known as multisegments (see Section \ref{sec:GL} for a precise definition). We denote this correspondence by $\m \mapsto L(\m)$. Given an irreducible representation $\tau$ of $\GL_n(F)$ with corresponding multisegment $\m$, M{\oe}glin and Waldspurger provided an algorithm to compute the multisegment $\m^t \in \Mult$ such that $\widehat{\tau} = L(\m^t)$.

The algorithm goes as follows. Given a multisegment $\m = \Delta_1 + \dots + \Delta_N$, one constructs a segment $\Delta'_1$ using certain ends of the segments $\Delta_i$ (see Paragraph \ref{sub:MW} for more precisions), and then forms a multisegment $\m_{(1)}$ by removing these ends from $\m$. They then proved that $\m^t = \Delta'_1 + (\m_{(1)})^t$, which allows  to compute $\m^t$ inductively.

\subsection{}\label{sub:MWidea}
As this will be relevant in the discussion that follows, let us briefly outline how M{\oe}glin--Waldspurger proved that their algorithm determines the Zelevinsky dual on the representation side. The category of finite-length representations of $\GL_n(F)$, for $n \geq 0$, has a monoidal structure given by (normalized) parabolic induction, denoted by $\pi_1 \times \pi_2$. 
A key consequence of the commutativity of the Zelevinsky duality with the parabolic induction functor is that if $\tau$ is a subrepresentation of $\pi_1 \times \pi_2$, then $\widehat{\tau}$ is a subrepresentation of $\widehat{\pi_2} \times \widehat{\pi_1}$.

Given $\m \in \Mult$ and $\tau = L(\m)$, where $\tau$ is non-supercuspidal, M{\oe}glin and Waldspurger showed that there exist a supercuspidal representation $\rho_1$ and a representation $\tau_1 \in \Irr^\GL$ such that:
\begin{enumerate}
\item $\tau$ is the unique subrepresentation of $\rho_1 \times \tau_1$.
\item One can compute the multisegment $\m_1$ corresponding to $\tau_1$.
\item Conversely, given $\m_1 \in \Mult$ and a supercuspidal representation $\rho_1$, one can compute the unique multisegment $\m \in \Mult$ such that $L(\m) \hookrightarrow L(\m_1) \times \rho_1$.
\end{enumerate}
By induction, their algorithm computes the multisegment $\m_1^t$ corresponding to $\widehat{\tau_1}$, and the final step is to verify that the candidate $\m^t$ satisfies $L(\m^t) \hookrightarrow L(\m_1^t) \times \rho_1$, which they confirmed.

\subsection{}\label{sub:badalgo}
It is important to note that Properties (1), (2) and (3) above allow the computation of $\m \mapsto \m^t$ even without knowing the M{\oe}glin--Waldspurger's algorithm. Indeed, given $\m \in \Mult$, we construct $\m_1$ as described above. For $\m_1$, there also exists a supercuspidal representation $\rho_2$ and a multisegment $\m_2$ such that $L(\m_1)$ is the unique subrepresentation of $\rho_2 \times L(\m_2)$, and this process continues iteratively. The first step of induction occurs when $L(\m_i)$ is supercuspidal. For supercuspidal representations, the involution acts as the identity, so $\m_i^t = \m_i$. Going backwards, from $\m_i^t$, one can reconstruct $\m_{i-1}^t$ as the unique multisegment satisfying $L(\m_{i-1}^t) \hookrightarrow L(\m_i^t) \times \rho_i$. By continuing this process, one can recover $\m^t$.

\subsection{} Let us return to the case of classical groups. Let $G = G_n$ be either the split special orthogonal group $\SO_{2n+1}(F)$ or the symplectic group $\Sp_{2n}(F)$ of rank $n$, where $F$ is a non-Archimedean local field of characteristic $0$. The Langlands classification asserts that any irreducible representation $\pi$ of $G_n$ is the unique irreducible subrepresentation of a standard module. We denote this by $\pi = L(\m, \pi_{\mathrm{temp}})$, where $\m$ is a negative multisegment and $\pi_{\mathrm{temp}}$ is an irreducible tempered representation (see Section \ref{sub:lang}). Tempered irreducible representations are themselves classified by local Langlands parameters $\phi : W_F \times \SL_2(\C) \rightarrow \GL_n(\C)$ together with some character $\eta$ of the component group of $\phi$ (see Section \ref{sub:endo}). We write $\pi_{\mathrm{temp}} = \pi(\phi, \eta)$. Altogether, the irreducible representation $\pi$ is uniquely determined by the triple $(\m, \phi, \eta)$, which we refer to as the ``Langlands data'' of $\pi$; and we write $\pi \simeq L(\m;\pi(\phi,\eta))$.

\subsection{} 
 H.~Atobe and the second-named author described an algorithm to compute the Aubert--Zelevinsky duality, similar to the one in Paragraph \ref{sub:badalgo}, building on previous work by Jantzen \cite{J-dual} and Atobe \cite{At2}. Let us provide more details.

For $G = G_n$, if $\pi$ (\resp $\tau_i$) is a smooth representation of $G_{n_0}$ (\resp $\GL_{d_i}(F)$), with $d_1 + \cdots + d_r + n_0 = n$, it is customary to denote
\[
\tau_1 \times \dots \times \tau_r \rtimes \pi 
\]
 the normalized parabolically induced representation of $\tau_1 \boxtimes \dots \boxtimes \tau_r \boxtimes \pi$ from the standard parabolic subgroup $P$ of $G_n$ with Levi subgroup isomorphic to $\GL_{d_1}(F) \times \dots \times \GL_{d_r}(F) \times G_{n_0}$.

Given an irreducible representation $\pi$ of $G_n$ and a supercuspidal \emph{non-self-dual} representation $\rho$ of $\GL_d(F)$, there exist a unique $k \geq 0$ and a unique irreducible representation $\pi_0$ of $G_{n_0}$, with $n = dk + n_0$, such that:
\begin{itemize}
\item $\pi$ is the unique irreducible subrepresentation of 
\begin{equation}\label{eq:subr}
\underbrace{\rho \times \dots \times \rho}_{k \text{ times}} \rtimes \pi_0.
\end{equation}
\item $k$ is maximal, in the sense that for every irreducible representation $\pi'_0$ of $G_{n_0 - d}$, $\pi_0$ is not a subrepresentation of $\rho \rtimes \pi'_0$.
\end{itemize}
We call $\pi_0$ the \emph{highest $\rho$-derivative} of $\pi$ and denote it by $D_\rho^{\max}(\pi)$. Note that, when $\rho$ is self-dual, a representation of the form \eqref{eq:subr} may have several irreducible subrepresentations, and there is no simple way to distinguish them.

Furthermore, Atobe and the second-named author provided an explicit formula for the Langlands data of $D_\rho^{\max}(\pi)$ in terms of those of $\pi$. Conversely, they found a formula to explicitly determine the Langlands data of $\pi$ in terms of those of $D_\rho^{\max}(\pi)$, just as M{\oe}glin--Waldspurger did for linear groups in Paragraph \ref{sub:MWidea}.

Since the Aubert--Zelevinsky duality commutes with Jacquet functors, we have:
\begin{equation*}
\widehat{D_{\rho}^{\max}(\pi)} = D_{\rho^\vee}^{\max}(\widehat{\pi}),
\end{equation*}
where $\rho^\vee$ denotes the contragredient of $\rho$. The algorithm in \cite{AM} follows the same lines as the one in Paragraph \ref{sub:badalgo}: given an irreducible representation $\pi$ of $G_n$, if there exists a supercuspidal non-self-dual representation $\rho$ of $\GL_d(F)$ such that $D_\rho^{\max}(\pi) \neq \pi$, then, by induction, we can compute the Langlands data of $\widehat{D_{\rho}^{\max}(\pi)} = D_{\rho^\vee}^{\max}(\widehat{\pi})$, which allows us to compute the Langlands data of $\widehat{\pi}$.

The issue, however, lies in the fact that one must assume $\rho$ is not self-dual, making the first step of the induction too complicated to handle, and the algorithm becomes more intricate (see Section \ref{sec:Gn} for further details). While the algorithm can be implemented on a computer, unlike the M{\oe}glin--Waldspurger algorithm, it is not practical for computing the Aubert--Zelevinsky dual of a representation by hand, as the formulas for the derivatives are rather involved (they are recalled in Section \ref{sec:expder}).

\subsection{} The goal of this article is to provide a new algorithm for the case of split special odd orthogonal groups or symplectic groups, in the spirit of the M{\oe}glin--Waldspurger algorithm for $\GL_n(F)$. Given the Langlands data $(\m,\phi,\eta)$ attached to a representation $\pi$ of $G_n$, our algorithm produces $\AD(\m,\phi,\eta)$, the Langlands data associated with the representation $\widehat{\pi}$. The algorithm is described in Section \ref{sec:defalgo}. Here, we will only mention three key points:

\begin{itemize}
  \item First, given Langlands data $(\m,\phi,\eta)$ attached to a representation $\pi$ of $G_n$, the algorithm constructs new Langlands data $(\m_1,\phi_1,\eta_1)$ and $(\m^{\#},\phi^{\#},\eta^{\#})$. We then prove that the operator $\AD$ is given recursively by $\AD(\m,\phi,\eta) = (\m_1,\phi_1,\eta_1) + \AD(\m^{\#},\phi^{\#},\eta^{\#})$.
  \item As in the case of $\GL_n$, Langlands data for classical groups admit a decomposition along supercuspidal lines. This is known as the Jantzen decomposition (see Section \ref{sub:Jantzen}). The Aubert--Zelevinsky dual $\AD(\m, \phi, \eta)$ can then be computed line by line, by computing $\AD(\m_\rho, \phi_\rho, \eta_\rho)$ for each supercuspidal representation $\rho$. However, in the case of classical groups, we must distinguish between three different cases. Let $\rho \in \Cusp^{\GL}$ be a supercuspidal representation of a general linear group, and let $\mathbb{Z}_\rho :=\{\rho |\cdot|^n : n \in \Z\}$ denote its \textit{supercuspidal line} (see Paragraph \ref{sub:lines}). We fix a supercuspidal representation $\sigma$ of $G_n$. Then

\begin{enumerate}
  \item We say that $\rho$ is \emph{ugly} if $\mathbb{Z}_\rho \neq \mathbb{Z}_{\rho^\vee}$.
  \item  We say that $\rho$ is \emph{good} if $\mathbb{Z}_\rho = \mathbb{Z}_{\rho^\vee}$ and $\rho' \rtimes \sigma$ is reducible for some $\rho' \in \Z_\rho$.
  \item  We say that $\rho$ is \emph{bad} if $\mathbb{Z}_\rho = \mathbb{Z}_{\rho^\vee}$ and $\rho' \rtimes \sigma$ is irreducible for all $\rho' \in \Z_\rho$.
\end{enumerate}
This classification is independent of the choice of $\sigma$. Accordingly, we define three distinct versions of the duality operator $\AD$, corresponding to the good, the bad and the ugly cases respectively.

\item Finally, regarding the proof: once we have a tentative algorithm, in order to prove that it corresponds to the Aubert--Zelevinsky duality, one \emph{only} needs to check, as in the case of general linear groups, that the algorithm commutes with derivatives. However, this is not as straightforward as one might expect, due to the complexity of the derivative formulas.
In order to carry out these computations, it was useful to introduce a simplified notation for derivatives; this streamlined notation turned out to be interesting in its own right, as it also provides a unified treatment of the positive and negative cases considered in \cite{AM} (see Section \ref{sec:expder}).

The proof is provided in great detail in Sections \ref{sec:proofgood}, \ref{sec:proofbad}, and \ref{sec:proofugly}.
\end{itemize}

\subsection{} Let us give an example to illustrate how the algorithm works in practice. Here, we treat the case of a cuspidal representation $\rho$ of good parity. We will not provide all the details, but just a glimpse into how the process unfolds. Let $\rho \in \Cusp^{\GL}$ be of good parity and take $\pi \in \Irr^{G}$ with Langlands data $(\m,\phi,\eta)$ where $\m=[-3, -1]_\rho + [-2, 0]_\rho + [-2, -2]_\rho + [-1, 0]_\rho$, $\phi = \rho \boxtimes S_3$, where $S_a$ is the unique irreducible algebraic representation of $\SL_2(\mathbb{C})$ of dimension $a$, and $\eta(\rho \boxtimes S_3)=1$.

The first step of the algorithm is to construct a symmetric multisegment with signs. This is done by sending each segment $\Delta \in \m$ to $\Delta + \Delta^{\vee}$, and each $\rho \boxtimes S_a$ to the segment $[-(a-1)/2, (a-1)/2]_\rho$ (see Section \ref{sec:symmetrization}). This is almost the transfer to $\GL_n$, but the centered segments carry signs, so that we can distinguish representations in the same $L$-packet (recall that the Aubert--Zelevinsky duality does not preserve $L$-packets). In this example, the resulting symmetric multisegment is:
 \[
  \mathfrak{s}=[-3, -1]_\rho + [1,3]_\rho +[-2, 0]_\rho +[0,2]_\rho + [-2, -2]_\rho + [2, 2]_\rho + [-1, 0]_\rho + [0,1]_\rho + [-1,1]_\rho
 \]
 with $[-1, 1]_\rho$ having sign $+1$. We then order these segments according to an order $\preceq$ defined in Section \ref{sec:labelseg}. This yields the following diagram:

      
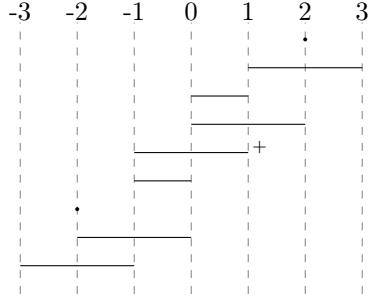
\begin{figure}[h!]
        \centering
        \begin{tikzpicture}[scale=0.75]
                  \foreach \x in {-3,-2,-1,0,1,2,3} {
                \draw[dashed, gray, very thin] (\x,-5) -- (\x,-0.2);
                \node[black] at (\x,0) {\x};
            }
            \fill[black] (2,-0.5) circle (1pt);
            \draw[black] (1,-1) -- (3,-1);
            \draw[black] (0,-1.5) -- (1,-1.5);
            \draw[black] (0,-2) -- (2,-2);
            \draw[black] (-1,-2.5) -- (1,-2.5);
            \node at (1.2,-2.4) {\tiny $+$};
            \draw[black] (-1,-3) -- (0,-3);
            \fill[black] (-2,-3.5) circle (1pt);
            \draw[black] (-2,-4) -- (0,-4);
            \draw[black] (-3,-4.5) -- (-1,-4.5);
          \end{tikzpicture}
      \end{figure}

We now define a sequence of segments $\Delta_1 \succeq \Delta_2 \succeq \cdots \succeq \Delta_l$ as follows (see Section \ref{sec:algogood} for full details). The segment $\Delta_1$ is the largest segment with maximal end. Then recursively, $\Delta_j$ is the largest segment such that $\Delta_j \prec \Delta_{j-1}$, $e(\Delta_j) = e(\Delta_{j-1}) - 1$, and if the segments have signs, those signs are opposite. In our example, we have $l=5$, $\Delta_1 = [1,3]_\rho$, $\Delta_2 = [0,2]_\rho$, $\Delta_3 = [-1,1]_\rho$, $\Delta_4 = [-1,0]_\rho$ and $\Delta_5 = [-3,-1]_\rho$. The ends of these segments are removed, as well as the beginnings of their duals. These removed ends form the first segment of the dual; the removed beginnings form the symmetric counterpart. The process is then repeated on the remaining symmetric multisegment. In our example, this gives:

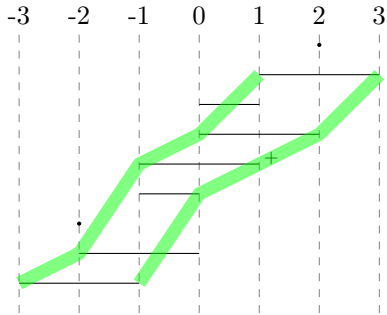
\begin{figure}[h!]
        \centering
        \begin{tikzpicture}[scale=0.79]
                  \foreach \x in {-3,-2,-1,0,1,2,3} {
                \draw[dashed, gray, very thin] (\x,-5) -- (\x,-0.2);
                \node[black] at (\x,0) {\x};
            }
            \fill[black] (2,-0.5) circle (1pt);
            \draw[black] (1,-1) -- (3,-1);
            \draw[black] (0,-1.5) -- (1,-1.5);
            \draw[black] (0,-2) -- (2,-2);
            \draw[black] (-1,-2.5) -- (1,-2.5);
            \node at (1.2,-2.4) {\tiny $+$};
            \draw[black] (-1,-3) -- (0,-3);
            \fill[black] (-2,-3.5) circle (1pt);
            \draw[black] (-2,-4) -- (0,-4);
            \draw[black] (-3,-4.5) -- (-1,-4.5);

             \draw[green, line width=5pt, opacity=0.5] (3,-1) -- (2,-2) -- (1,-2.5) -- (0,-3) -- (-1,-4.5);
            \draw[green, line width=5pt, opacity=0.5] (1,-1) -- (0,-2) -- (-1,-2.5)-- (-2,-4)-- (-3,-4.5);
          \end{tikzpicture}
      \end{figure}

So the first part of the predicted dual is $[-3,1]_\rho + [-1,3]_\rho$. We then continue with the remaining multisegment:
 \[
  \mathfrak{s}=[2, 2]_\rho + [2, 2]_\rho + [0, 1]_\rho + [1, 1]_\rho  + [0, 0]_\rho + [-1, -1]_\rho + [-2, -2]_\rho + [-1, 0]_\rho + [-2, -2]_\rho
 \]
 with $[0,0]_\rho$ having sign $+1$. 
 It may happen that the two green paths above ``meet in the middle''. In such cases, the result is a centered segment, which corresponds to an element in the tempered part. This occurs in the current multisegment, where we obtain:

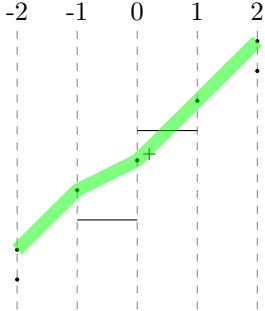
\begin{figure}[h!]
        \centering
        \begin{tikzpicture}[scale=0.79]
                  \foreach \x in {-2,-1,0,1,2} {
                \draw[dashed, gray, very thin] (\x,-5) -- (\x,-0.2);
                \node[black] at (\x,0) {\x};
            }
            \fill[black] (2,-0.5) circle (1pt);
            \fill[black] (2,-1) circle (1pt);
            \fill[black] (1,-1.5) circle (1pt);
            \draw[black] (0,-2) -- (1,-2);
            \fill[black] (0,-2.5) circle (1pt);
            \node at (0.2,-2.4) {\tiny $+$};
            \fill[black] (-1,-3) circle (1pt);
            \draw[black] (-1,-3.5) -- (0,-3.5);
            \fill[black] (-2,-4) circle (1pt);
            \fill[black] (-2,-4.5) circle (1pt);

            \draw[green, line width=5pt, opacity=0.5] (2,-0.5) -- (1,-1.5) -- (0,-2.5) -- (-1,-3) -- (-2,-4) ;
          \end{tikzpicture}
      \end{figure}

The resulting segment is $[-2,2]_\rho$ with sign $+1$, and the remaining multisegment becomes
\[
[2, 2]_\rho + [0, 1]_\rho + [-1, 0]_\rho + [-2, -2]_\rho.
\]
We repeat the process until nothing remains. Here, one more step is required where $[-2,0]_\rho + [0,2]_\rho$ is produced. At the end, the symmetric multisegment is:
\[
[-3,1]_\rho + [-1,3]_\rho + [-2,2]_\rho + [-2,0]_\rho + [0,2]_\rho.
\]
Unsymmetrizing this gives the Langlands data of $\hat{\pi}$ as $L(\m; \pi(\phi, \eta))$ with $\m=[-3,1]_\rho + [-2,0]_\rho$, $\phi = \rho \boxtimes S_5$ and $\eta(\rho \boxtimes S_5)=1$.

\subsection{} Our algorithm generalizes Atobe's algorithm for the so-called ladder representations \cite{AtobeLadder}. For representations of Arthur type, Atobe also provides an algorithm in terms of his parametrization \cite{AtobePackets}. Ruben La describes the Iwahori--Matsumoto dual of tempered representations with real infinitesimal character in \cite{La}.

\subsection{} All algorithms described in this paper are implemented in Python and SageMath, and are publicly available on GitHub\footnote{\url{https://github.com/ThomasLanard/aubert-zelevinsky-duality}}.

\subsection{} 
This duality has numerous interesting applications to the Langlands program. We will briefly mention two of them.

One notable feature of the Aubert--Zelevinsky duality is that it does not preserve temperedness. In Arthur's local classification, the first step beyond tempered representations involves considering the Aubert--Zelevinsky dual of tempered representations. One first constructs $A$-packets and proves the local theorems when the $A$-parameters are co-tempered, that is, when they are Aubert-dual to tempered $L$-parameters. This is possible provided we understand how Aubert duality interacts with the local classification and the local intertwining relation. See \cite[\S 7]{Art}, \cite{AGIKMS}.

The Aubert--Zelevinsky duality is also a very useful tool in studying the wavefront set, a fundamental invariant of admissible representations, arising from the Harish-Chandra--Howe local character expansion. Significant progress has been made in recent years \cite{CMBO,CMBO2,CMBO3,HLLS,La,W3}, and we expect that our algorithm will be valuable in solving some of the conjectures in the mentioned papers.

However, to be transparent, we did not begin this research with the primary goal of exploring applications of the Aubert--Zelevinsky duality. Our initial aim was to investigate the potential uses of deep learning in the Langlands program. Inspired by the work of \cite{Wil}, we sought an excuse to determine whether deep learning could help reveal patterns in problems within our field. The difficult-to-use algorithm in \cite{AM} for computing the Aubert--Zelevinsky duality gave us the perfect opportunity. A discussion of how we applied deep learning and developed our new algorithm can be found in Appendix \ref{ann:com}. For us, it is remarkable and somewhat paradoxical that we used an AI to create an algorithm that feels more intuitive and human.

\subsection{}  
We now briefly discuss some issues that are not addressed in this article. The first concerns the use of symplectic and split special odd orthogonal groups. What about other (quasi-split) classical groups? With the endoscopic classification being (almost) complete \cite{Art, AGIKMS, Mok}, it is natural to ask whether similar results can be expected for even orthogonal or unitary groups. We believe so, and the proof should be similar. We focused on symplectic and split odd orthogonal groups because the work in \cite{At2, AM} deals exclusively with these groups. Now that the results in \cite{At} have been extended to all quasi-split classical groups (see \cite[Appendix C]{AGIKMS}), and given that \cite{J-dual} is written for all classical groups, the extension to quasi-split classical groups should not pose significant challenges. (A much more challenging task would be to extend our results to the case where $F$ has positive characteristic.)

Secondly, one of the goals of M{\oe}glin--Waldspurger's algorithm was to demonstrate that Zelevinsky's involution has a geometric interpretation, as he had conjectured \cite{Z2}. We aim to explore this problem in future work.

\subsection{} The contents of this paper are as follows. In Section~\ref{sec:Preliminaries}, we recall some general results on the representation theory of $p$-adic groups and the Aubert--Zelevinsky involution. Section~\ref{sec:GL} reviews the representation theory of $\GL_n(F)$, including the Zelevinsky classification and the M{\oe}glin--Waldspurger algorithm. In Section~\ref{sec:Gn}, we present the representation theory of classical groups and the associated local Langlands data. Section~\ref{sec:defalgo} introduces the definition of our algorithm $\AD$. In Section~\ref{sec:ADdef}, we verify that this definition is well-posed, particularly in the good parity case. Section~\ref{sec:impprop} establishes several important properties of $\AD$. Section~\ref{sec:AM} recalls the theory of derivatives, while Section~\ref{sec:expder} provides explicit formulas for these derivatives. Section~\ref{sec:proofgood} concludes the proof that the algorithm $\AD$ indeed realizes the Aubert--Zelevinsky involution in the good parity case. The bad and ugly cases are treated in Sections~\ref{sec:proofbad} and~\ref{sec:proofugly}, respectively. Finally, Appendix~\ref{ann:com} explains how machine learning was used to discover this algorithm.

\subsection*{Acknowledgment}
We would like to thank Hiraku Atobe for stimulating discussions about this work. In particular, this article would not have been possible without his Sage implementation of the algorithm from \cite{AM}. We gratefully acknowledge the hospitality of Kyoto University, where part of this work was developed in Spring 2023. In particular, we would like to thank Atsushi Ichino for kindly inviting both of us. 

A. Mínguez was partially funded by the Principal Investigator project PAT-4832423 of the Austrian Science Fund (FWF) and Proyecto PID2024-156912NB-I00 financiado por MICIU/ AEI / 10.13039/501100011033 / FEDER, UE. T. Lanard was partially funded by the PEPS JCJC 2023 of INSMI (CNRS).

\section{Preliminaries}\label{sec:Preliminaries}

Throughout this article, we fix a non-Archimedean, locally compact field $F$ of characteristic zero, with normalized absolute value $|\cdot|$. 

\subsection{} 
Let $G$ be the group of $F$-points of a connected reductive group defined over $F$, equipped with its usual topology. We will only consider smooth complex representations of $G$, meaning representations on $\C$-vector spaces where the stabilizer of each vector is an open subgroup of $G$. Henceforth, by ``representation'' we will always mean ``smooth complex representation''. We write $\Rep(G)$ for the category of finite-length representations of $G$ and denote by $\Irr(G)$ the set of equivalence classes of irreducible objects in $\Rep(G)$. For $\pi, \pi' \in \Rep(G)$, we write $\pi \hookrightarrow \pi'$ (\resp $\pi \twoheadrightarrow \pi'$) to indicate the existence of an injective (\resp surjective) morphism from $\pi$ to $\pi'$. For any $\pi \in \Rep(G)$, we denote by $\pi^\vee$ the contragredient of $\pi$. 

Let $ \pi \in \Rep(G) $. The \textit{socle} of $ \pi $ is the largest semisimple subrepresentation of $ \pi $ and is denoted $\soc(\pi)$. We say that $\pi$ is socle irreducible (SI) if $\soc(\pi)$ is irreducible and occurs with multiplicity one in $\pi$.

Fix a minimal $F$-parabolic subgroup $P_0$ of $G$. A parabolic subgroup $P$ of $G$ is called \emph{standard} if it contains $P_0$. From now on, $P$ will always denote a standard parabolic subgroup of $G$, with an implicit standard Levi decomposition $P = MU$. Let $\Sigma$ denote the set of roots of $G$ with respect to $P_0$, and let $\Delta$ be a basis of $\Sigma$. For any subset $\Theta \subset \Delta$, let $P_\Theta$ denote the standard parabolic subgroup of $G$ corresponding to $\Theta$, and let $M_\Theta$ be the corresponding standard Levi subgroup. Finally, let $W$ denote the Weyl group of $G$.

Let $\tau$ be a representation of $M$, viewed as a representation of $P$ where $U$ acts trivially. We denote by $\Ind_P^G \tau$ the representation of $G$ parabolically induced from $\tau$, using normalized induction. We treat $\Ind_P^G$ as a functor, whose left adjoint is the \emph{Jacquet functor} with respect to $P$, denoted by $\Jac_P^G$.

An irreducible representation $\pi$ of $G$ is called \emph{supercuspidal} if it is not a composition factor of any representation of the form $\Ind_P^G(\tau)$, where $P$ is a proper parabolic subgroup of $G$ and $\tau$ is a representation of $M$. We denote by $\Cusp(G)$ the subset of $\Irr(G)$ consisting of supercuspidal representations. We also let $\Irr_\temp(G)$ be the set of equivalence classes of irreducible tempered representations of $G$.

Let $\Rr(G)$ denote the Grothendieck group of $\Rep(G)$. The canonical map from the objects of $\Rep(G)$ to $\Rr(G)$ is denoted by $\pi \mapsto [\pi]$. Both the induction and Jacquet functors, $\Ind_P^G$ and $\Jac_P^G$, are exact and preserve finite-length representations. Therefore, they induce morphisms of $\Z$-modules:
\begin{align*}
\Ind_P^G & : \Rr(M) \longrightarrow \Rr(G), \\
\Jac_P^G & : \Rr(G) \longrightarrow \Rr(M).
\end{align*}

\subsection{The Aubert--Zelevinsky involution}\label{sub:Aubert} 
Define:
\begin{align*}
D_G \colon \Rr(G) & \longrightarrow \Rr(G) \\
\pi &\mapsto \sum_{P} (-1)^{\dim A_M} \Ind_P^G(\Jac_P^G(\pi)),
\end{align*}
where $P $ runs over all standard parabolic subgroups of $G$ and $A_M$ is the maximal split torus of the center of $M$. A.-M.~Aubert \cite{Au} showed that if $\pi$ is irreducible, there exists a sign $\epsilon \in \{\pm 1\}$ such that $\widehat{\pi} := \epsilon \cdot D_G(\pi)$ is also an irreducible representation. We refer to the map:
\begin{align*}
\Irr(G) & \longrightarrow \Irr(G) \\
\pi & \mapsto \widehat{\pi}
\end{align*}
as the \emph{Aubert--Zelevinsky duality}.

This duality satisfies the following important properties:
\begin{enumerate}
\item The map $\pi \mapsto \widehat{\pi}$ is an involution of $\Irr(G)$ \cite[Th{\'e}or{\`e}me 1.7 (3)]{Au}.
\item If $\pi \in \Cusp(G)$, then $\widehat{\pi} = \pi$ \cite[Th{\'e}or{\`e}me 1.7 (4)]{Au}.
\item Let $\Theta \subset \Delta$, and consider the standard parabolic subgroup $P = P_\Theta$ with Levi decomposition $P = MN$. Let $w_0$ be the longest element in the set $\{w \in W \mid w^{-1}(\Theta) > 0\}$, and let $P'$ be the standard parabolic with Levi subgroup $M' = w_0^{-1}(M)$. Then, we have (\textit{cf.} \cite[Th{\'e}or{\`e}me 1.7 (2)]{Au}):
\begin{equation}\label{eq:Jacquet}
\Jac_P^G \circ D_G = {\rm Ad}(w_0) \circ D_{M'} \circ \Jac_{P'}^G.
\end{equation}
\end{enumerate}

\section{Representation theory of $\GL_n(F)$}\label{sec:GL}
\subsection{} 
The representation theory of the groups $\GL_n(F)$, for $n \geq 0$, plays a particularly important role
in the general theory of representations of $p$-adic groups.
Bernstein and Zelevinsky studied this extensively in their fundamental work in the 1970s \cite{BZ,Z}, where they emphasized the benefits of considering all $n$'s together.

We define $\Irr^\GL \coloneqq \bigcup_{n \geq 0} \Irr(\GL_n(F))$, and let $\Cusp^\GL \subset \Irr^\GL$ denote the subset of supercuspidal representations. Furthermore, we set $\Rr^\GL \coloneqq \bigoplus_{n \geq 0} \Rr(\GL_n(F))$.

Let $d_1, \dots, d_r$ be positive integers, and for each $1 \leq i \leq r$, let $\tau_i \in \Rep(\GL_{d_i}(F))$. The normalized parabolically induced representation is customarily denoted by
\[
\tau_1 \times \dots \times \tau_r \coloneqq \Ind_{P}^{\GL_{d_1 + \dots + d_r}(F)} (\tau_1 \boxtimes \dots \boxtimes \tau_r).
\]
This operation induces a $\Z$-graded commutative algebra structure on $\Rr^\GL$. If $\tau_1 = \dots = \tau_r = \tau$, we simplify the notation to $\tau^r = \tau \times \dots \times \tau$ ($r$ times).

The Jacquet functor for $\GL_m(F)$, along the maximal standard parabolic subgroup $P_{(d, m - d)}$, with Levi decomposition $\GL_d(F) \times \GL_{m - d}(F)$, is simply denoted by $\Jac_{(d, m - d)} = \Jac_{P_{(d, m - d)}}^{\GL_m(F)}$.

For any $\pi \in \Rep(\GL_n(F))$, we call $n$ its degree and denote it $\deg(\pi)$. Moreover, if $\chi$ is  a character of $F^\times$, we denote the representation obtained by twisting $\pi$ by $\chi \circ \det$ as $\pi \chi$. 

A supercuspidal representation is unitary if and only if its central character is unitary. Therefore, given any supercuspidal representation $\rho$, there exists a unique $x \geq 0$ such that $\rho_u:=\rho | \cdot|^x$ is unitary. We call $\rho_u$ the unitarization of $\rho$.
\subsection{} 
A \emph{segment} $\Delta$ is a finite, non-empty subset of $\Cusp^\GL$ of the form
\begin{equation}
\Delta = \{\rho |\cdot|^x, \rho |\cdot|^{x-1}, \dots, \rho |\cdot|^y\},
\end{equation}
where $\rho \in \Cusp^\GL$, and $x, y \in \R$ with $x - y \in \Z$ and $x \leq y$. We denote such a segment by $\left[x,y\right]_\rho$. Thus, $\left[x,y\right]_\rho = \left[x',y'\right]_{\rho'}$ if and only if $\rho|\cdot|^x = \rho'|\cdot|^{x'}$ and $\rho|\cdot|^y = \rho'|\cdot|^{y'}$. Hence, one can assume, when needed, that $\rho$ is unitary. Notice that our notation differs from \cite{AM}:  our segments are increasing, while in \cite{AM} they are decreasing.

We denote by $\Seg$ the set of all segments. Let $\rho \in \Cusp^\GL$ be unitary. We denote by $b(\left[x,y\right]_\rho)=x$ the beginning of the segment $\left[x,y\right]_\rho$ and by $e(\left[x,y\right]_\rho)=y$ its end. The size of a segment $\Delta$, also called its length, is denoted by $l(\Delta)$. We denote by $\deg(\Delta)$ its degree, that is $\deg(\rho)l(\Delta)$. The determinant of $\Delta$ is defined by $\det(\Delta)=\prod_{i=x}^{y} \det(\rho|\cdot|^i)$.
If $\Delta =\left[x,y\right]_\rho$, we call $\frac{x+y}{2}\in \R$ the \emph{center} of $\Delta$ and denote it $c(\Delta)$.
We denote by $\Seg^{>0}$ (\resp $\Seg^{0}$ or $\Seg^{<0}$) the subset of $\Seg$ composed of segments $\Delta$ such that $c(\Delta)>0$ (\resp $c(\Delta)=0$ or $c(\Delta)<0$). We say that $\Delta$ is a \emph{centered} segment if $c(\Delta)=0$.

We define the following operations on a segment $\Delta = [x, y]_\rho$:
\begin{align*}
\Delta^-    &= [x, y - 1]_\rho,\quad {}^-\Delta = [x + 1, y]_\rho, \\
\Delta^+    &= [x, y + 1]_\rho,\quad {}^+\Delta = [x - 1, y]_\rho,\\
\Delta^\vee &= [-y, -x]_{\rho^\vee}.
\end{align*}

Let $\Delta = [x, y]_\rho$ and $\Delta' = [x', y']_{\rho'}$ be two segments. We say that $\Delta$ and $\Delta'$ are \emph{linked} if $\Delta \cup \Delta'$ forms a segment, but neither $\Delta \subset \Delta'$ nor $\Delta' \subset \Delta$.

If $\Delta$ and $\Delta'$ are linked and $\rho'|\cdot|^{y'} = \rho|\cdot|^{y+j}$ with $j > 0$, we say that $\Delta$ \emph{precedes} $\Delta'$.
Thus, if $\Delta$ and $\Delta'$ are linked, then either $\Delta$ precedes $ \Delta'$ or $\Delta'$ precedes $\Delta$, but not both.

\subsection{}\label{sub:Jacseg} For any segment $\Delta = \left[x,y\right]_\rho$, we define:
\begin{eqnarray*}
Z(\Delta) & = & \soc(\rho|\cdot|^x \times \rho|\cdot|^{x+1} \times \dots \times \rho|\cdot|^y), \\
L(\Delta) & = &  \soc(\rho|\cdot|^y \times \rho|\cdot|^{y-1} \times \dots \times \rho|\cdot|^x).
\end{eqnarray*}

Zelevinsky proved that $Z(\Delta)$ and $L(\Delta)$ are irreducible representations. Moreover, $L(\Delta)$ is an essentially discrete series representation, and all essentially discrete series representations are of this form \cite[Theorem 9.3]{Z}.

If $\Delta = [x, y]_\rho$ with $x = y + 1$, we set by convention $Z(\Delta) = L(\Delta)$ to be the trivial representation of the trivial group $\GL_0(F)$.

\subsection{Classification} \label{sub:multi}
Given a set $\X$, write $\NN(\X)$ for the commutative semigroup of maps from
$\X$ to $\NN$ with finite support.

A \textit{multisegment} is a multiset of segments, that is, an element in $\Mult := \NN(\Seg)$. We will see a multisegment $\m$ as a finite sum $\m=\Delta_1+ \dots + \Delta_N$, with $\Delta_i \in \Seg$. We denote the multiplicity of a segment $\Delta$ in $\m$ by $m_{\m}(\Delta)$. By linearity one extends the definition of contragredient, length, degree and determinant from segments to multisegments.
 If $\m=\Delta_1+ \dots + \Delta_N \in \Mult$, we define the support of $\m$ to be $\cup\Delta_i\subset\Cusp^\GL$. We say $\m$  is positive (\resp negative) if $c(\Delta_i)>0$ (\resp $c(\Delta_i)<0$), for all $1\leq i \leq N$.
 We denote by $\Mult^{\clubsuit}$ the subset of $\Mult$ made of segments in $\Seg^{\clubsuit}$ where $\clubsuit$ is one of the following symbols: ${>0}$, $0$ or $<0$. The natural map $\Mult \longrightarrow \Mult^{>0} \times \Mult^{0}\times \Mult^{<0}$ will be denoted $\m \mapsto (\m^{>0}, \m^{0}, \m^{<0})$.

A sequence of segments $(\Delta_1, \dots, \Delta_N)$ is said to be \textit{arranged} if, for every $1 \leq i < j \leq N$, $\Delta_i$ does not precede $\Delta_j$. This is the case in particular if $c(\Delta_1) \geq c(\Delta_2) \geq \dots \geq  c(\Delta_N)$. If $\m \in \MS$ and $(\Delta_1, \dots, \Delta_N)$ is an arranged sequence of segments such that $\m = \Delta_1 + \dots + \Delta_N$, we say that $(\Delta_1, \dots, \Delta_N)$ is an \emph{arranged form} of $\m$.

For any arranged form $(\Delta_1, \Delta_2, \ldots, \Delta_N)$ of a multisegment $\m$, we define
\[
\tilde{\lambda}(\m) = L(\Delta_N) \times \cdots \times L(\Delta_2) \times L(\Delta_1).
\]

Zelevinsky showed that $\tilde{\lambda}(\m)$ is SI and that, up to isomorphism, it does not depend on the choice of the arranged form of $\m$. Furthermore, if one defines
\[
L(\m) : = \soc(\tilde{\lambda}(\m)) \in \Irr^\GL,
\]
the map $\Mult \to \Irr^\GL$ given by 
\begin{equation} \label{eq:class}
\m \mapsto L(\m)
\end{equation}
is a bijection, providing a classification of the set $\Irr^\GL$ in terms of multisegments.

\subsection{Reduction to supercuspidal lines} \label{sub:lines}

The equivalence relation on \( \Cusp^\GL \) generated by \( \rho \sim \rho |\cdot| \) partitions \( \Cusp^\GL \) into equivalence classes, each of which is called a \emph{supercuspidal line}. The supercuspidal line containing \( \rho \in \Cusp^\GL \) is thus given by  
\[
\Z_\rho := \{\rho |\cdot|^n : n \in \Z\}.
\]  
We denote the set of all such equivalence classes by \( \Cusp^\GL/\sim \).

Let $\rho \in \Cusp^\GL$, and denote by $\Mult_\rho$ the submonoid of multisegments supported in $\Z_\rho$, and by $\Irr^\GL_\rho$ the image of $\Mult_\rho$ under the map \eqref{eq:class}, so that $\Irr^\GL_\rho$ consists of irreducible representations with supercuspidal support in $\Z_\rho$. Again, we denote by $\Mult_\rho^{\clubsuit}=\Mult_\rho \cap \Mult^{\clubsuit}$ where $\clubsuit$ is one of the following symbols: ${>0}$, $0$ or $<0$.

There is a natural map 
\begin{eqnarray}\label{eq:reduction}
\Mult&\longrightarrow& \Mult_\rho \\
\m &\mapsto &\m_\rho, \notag
\end{eqnarray}
where $\m_\rho$ is the sum of all segments in $\m$ with support in $\Z_\rho$. It induces a natural decomposition:
\[L(\m) =\bigtimes_{\rho \in \Cusp^\GL/\sim} L(\m_\rho).\]

Thus, any $\pi \in \Irr^\GL$ can be written uniquely (up to permutation) as $\pi = \pi_1 \times \dots \times \pi_r$, where $\pi_i \in \Irr^\GL_{\rho_i}$ with $\Z_{\rho_i} \neq \Z_{\rho_j}$ for $i \neq j$. In practice, this allows us to reduce questions about $\Irr^\GL$ to $\Irr^\GL_\rho$.

\subsection{The M{\oe}glin--Waldspurger algorithm} \label{sub:MW}
Given $\m \in \Mult$, there exists a unique $\m^t \in \Mult$ such that $L(\m^t) \simeq \widehat{L(\m)}$. The map $\m \mapsto \m^t$ factors through \eqref{eq:reduction}, that is,
\[
(\m^t)_\rho = (\m_\rho)^t.
\]

This allows us to fix $\rho \in \Cusp^\GL$ for the remainder of this subsection. The combinatorial description of the map
\begin{align*}
    \Mult_\rho &\longrightarrow \Mult_\rho \\
    \m &\mapsto \m^t
\end{align*}
was provided by M{\oe}glin and Waldspurger \cite{MW2}, which we recall here.

First, we need to fix an order on the set of segments supported in $\Z_\rho$. 
We define that $[x_1, y_1]_\rho < [x_2, y_2]_\rho$ if either $x_1 < x_2$, or $x_1 = x_2$ and $y_1 > y_2$.

\begin{rem}
  This is not the classical lexicographical order on segments. This order has the advantage that if $[x_1, y_1]_\rho$ and $[x_2, y_2]_\rho$ are two segments with $y_2 = y_1 - 1$, then $[x_2, y_2]_\rho \le [x_1, y_1]_\rho$ if and only if $[x_2, y_2]_\rho$ precedes $[x_1, y_1]_\rho$.
\end{rem}

Let $0 \neq \m \in \Mult_\rho$. Let $\rho_u$ be the unitarization of $\rho$. Let $y_{\max}$ be the maximum of the ends $e(\Delta)$ of the segments $\Delta \in \m$. Let $\Delta_1 \in \m$ be the largest segment (with respect to the above order) such that $e(\Delta_1) = y_{\max}$. We then define recursively a sequence $\Delta_1 \geq \dots \geq \Delta_l$, where $\Delta_j$ is the largest segment in $\m$ such that $\Delta_j \le \Delta_{j-1}$ and $e(\Delta_j) = e(\Delta_{j-1}) - 1$.

Define $\Delta'_1(\m) = [e(\Delta_l),e(\Delta_1)]_{\rho_u}$, and
\[
\m_{(1)} = \m + \sum_{j=1}^l (\Delta_{j}^{-} - \Delta_{j}).
\]

The map $\m \mapsto \m^t$ is defined recursively by $0^t = 0$ and
\[
\m^t = (\m_{(1)})^t + \Delta'_1, \ \ \text{if } \m \neq 0.
\]

\section{Representation theory of classical groups}\label{sec:Gn}
Throughout this paper, we denote by \( G_n \) either the split special orthogonal group \( \SO_{2n+1}(F) \) or the symplectic group \( \Sp_{2n}(F) \), both of rank \( n \), and maintain this choice throughout.  
\subsection{} 
We define  
\[
\Irr^G \coloneqq \bigcup_{n \geq 0} \Irr(G_n), \quad \text{and} \quad \Rr^G \coloneqq \bigoplus_{n \geq 0} \Rr(G_n),
\]  
where the union and direct sum are taken over groups of the fixed type. Similarly, let \( \Cusp^G \subset \Irr^G \) denote the subset of supercuspidal representations of \( G_n \) for all \( n \geq 0 \).

\par

Let $P$ be a standard parabolic subgroup of $G_n$ with a Levi subgroup isomorphic to
$\GL_{d_1}(F) \times \dots \times \GL_{d_r}(F) \times G_{n_0}$. 
Let $\pi \in \Rep(G_{n_0})$ and $\tau_i \in \Rep(\GL_{d_i}(F))$ for $1 \leq i \leq r$. 
As in the introduction, we denote the normalized parabolically induced representation by
\[
\tau_1 \times \dots \times \tau_r \rtimes \pi 
\coloneqq 
\Ind_{P}^{G_n}(\tau_1 \boxtimes \dots \boxtimes \tau_r \boxtimes \pi).
\]

\par

If $\pi \in \Irr^G$, there exist $\rho_1, \dots, \rho_r \in \Cusp^\GL$ and $\sigma \in \Cusp^G$ such that $\pi$ is a subrepresentation of $\rho_1 \times \dots \times \rho_r \rtimes \sigma$. The set 
\[
\scusp(\pi) \coloneqq \{\rho_1, \dots, \rho_r, \rho_1^\vee, \dots, \rho_r^\vee, \sigma\}
\]
is uniquely determined by $\pi$ and is called the \emph{supercuspidal support} of $\pi$. For $\sigma \in \Cusp^G$, we set $\Irr_\sigma \coloneqq \{\pi\in\Irr^G \;|\; \sigma \in \scusp(\pi)\}$.

\subsection{The Langlands Subrepresentation Theorem}\label{sub:lang}
Let $\m = \Delta_1 + \dots + \Delta_r$ be a negative multisegment, and let $\pi_\temp \in \Irr^G$ be tempered. 
A parabolically induced representation of the form
\[
\tilde{\lambda}(\m) \rtimes \pi_\temp 
\]
is called a \emph{standard module}.

The Langlands classification asserts that any standard module has an irreducible socle, 
and that any irreducible representation $\pi  \in \Irr^G$ is the unique irreducible subrepresentation (Langlands subrepresentation) of a standard module $\tilde{\lambda}(\m) \rtimes \pi_\temp$, which is unique up to isomorphism. In this case, we write $\pi = L(\Delta_1 + \dots + \Delta_r; \pi_\temp)$.
For more details, see \cite{K}.

\subsection{The endoscopic classification}\label{sub:endo}
The Langlands subrepresentation theorem reduces the classification of irreducible representations of $G_n$ to the case of tempered irreducible representations. These representations, in turn, are classified by Local Langlands parameters, which we briefly recall in this section \cite{Art}.

Let $W_F$ denote the Weil group of $F$. A homomorphism 
\[
\phi \colon W_F \times \SL_2(\C) \rightarrow \GL_n(\C)
\]
is called an \emph{$L$-parameter for $\GL_n(F)$} if it satisfies the following conditions:
\begin{itemize}
    \item $\phi(W_F)$ consists of semisimple elements;
    \item $\phi|_{W_F}$ is smooth, meaning it has an open kernel;
    \item $\phi|_{\SL_2(\C)}$ is algebraic.
\end{itemize}

Any irreducible representation of $W_F \times \SL_2(\C)$ has the form $\rho \boxtimes S_a$, where  $\rho$ is an irreducible representation of $W_F$ and $S_a$ is the unique irreducible algebraic representation of $\SL_2(\C)$ of dimension $a$. For simplicity, we often write $\rho = \rho \boxtimes S_1$. For a given $L$-parameter $\phi$, the multiplicity of $\rho \boxtimes S_a$ in $\phi$ is denoted by $m_\phi(\rho \boxtimes S_a)$. 

The local Langlands correspondence for $\GL_d(F)$ establishes a canonical bijection  between the set of irreducible unitary supercuspidal representations of $\GL_d(F)$ and the set of irreducible $d$-dimensional representations of $W_F$ with bounded image. These sets are identified, and we use the symbol $\rho$ to refer to their elements.

We say that $\phi$ is an \emph{$L$-parameter for $\SO_{2n+1}(F)$} if it is an $L$-parameter for $\GL_{2n}(F)$ of symplectic type, i.e., 
\[
\phi \colon W_F \times \SL_2(\C) \rightarrow \Sp_{2n}(\C).
\]
Similarly, $\phi$ is called an \emph{$L$-parameter for $\Sp_{2n}(F)$} if it is an $L$-parameter for $\GL_{2n+1}(F)$ of orthogonal type with trivial determinant, i.e.,
\[
\phi \colon W_F \times \SL_2(\C) \rightarrow \SO_{2n+1}(\C).
\]

For $G_n = \SO_{2n+1}(F)$ (respectively, $G_n = \Sp_{2n}(F)$), let $\Phi(G_n)$ denote the set of $\widehat{G_n}$-conjugacy classes of $L$-parameters for $G_n$ with bounded image, where $\widehat{G_n} = \Sp_{2n}(\C)$ (respectively, $\widehat{G_n} = \SO_{2n+1}(\C)$). We say that $\phi \in \Phi(G_n)$ is \emph{of good parity} if $\phi$ is a sum of irreducible self-dual representations of the same type as $\phi$.

Let $\Phi_\gp(G_n)$ denote the subset of $\Phi(G_n)$ consisting of $L$-parameters of good parity. We define the sets 
\[
\Phi(G) \coloneqq \bigcup_{n \geq 0} \Phi(G_n) \quad \text{and} \quad \Phi_\gp(G) \coloneqq \bigcup_{n \geq 0} \Phi_\gp(G_n).
\]

For each $\phi \in \Phi(G)$, a \emph{component group} $\Sc_\phi$ is attached, which is defined as follows. We express $\phi$ as a direct sum:
\begin{equation}\label{eq:deco}
\phi = \bigoplus_{i=1}^{r} \phi_i \oplus \phi' \oplus \phi^{\prime \vee},
\end{equation}
where $\phi_1, \dots, \phi_r$ are irreducible self-dual representations of the same type as $\phi$, and $\phi'$ is a sum of irreducible representations that are not of the same type as $\phi$. We then denote 
\[
\phi_\gp = \bigoplus_{i=1}^{r} \phi_i.
\] 
and define the \emph{enhanced component group} $\AA_\phi$ as
\[
\AA_\phi \coloneqq \bigoplus_{i=1}^{r} (\Z/2\Z) \alpha_{\phi_i}.
\]
Thus, $\AA_\phi$ is a free $\Z/2\Z$-module of rank $r$, with a basis $\{\alpha_{\phi_i}\}$ corresponding to the irreducible components $\{\phi_i\}$ of good parity. The element 
\[
z_\phi \coloneqq \sum_{i=1}^{r} \alpha_{\phi_i}
\]
is called the \emph{central element} of $\AA_\phi$. 

The \emph{component group} $\Sc_\phi$ is defined as the quotient of $\AA_\phi$ by the subgroup generated by $\alpha_{\phi_i} + \alpha_{\phi_{i'}}$ whenever $\phi_i \cong \phi_{i'}$.

Let $\widehat{\Sc_\phi}$ and $\widehat{\AA_\phi}$ denote the Pontryagin duals of $\Sc_\phi$ and $\AA_\phi$, respectively. Through the canonical surjection $\AA_\phi \twoheadrightarrow \Sc_\phi$, we may regard $\widehat{\Sc_\phi}$ as a subgroup of $\widehat{\AA_\phi}$. For any element $\eta \in \widehat{\AA_\phi}$, we write $\eta(\alpha_{\phi_i}) = \eta(\phi_i)$. The map $\phi \mapsto \phi_\gp$, from $\Phi(G)$ to $\Phi_\gp(G)$, induces canonical isomorphisms:
\begin{equation}\label{eq:canon}
\AA_\phi \simeq \AA_{\phi_\gp}, \qquad \Sc_\phi \simeq \Sc_{\phi_\gp}, \qquad \widehat{\Sc_\phi} \simeq \widehat{\Sc_{\phi_\gp}}.
\end{equation}

To each \(\phi \in \Phi(G_n)\), one can associate a subset \(\Pi_\phi \subseteq \Irr_\temp(G_n)\), called the \emph{\(L\)-packet} for \(G_n\) attached to \(\phi\), such that:  
\[
\Irr_\temp(G_n) = \bigsqcup_{\phi \in \Phi(G_n)} \Pi_\phi.
\]
Furthermore, there is a canonical injection  
\[
\Pi_\phi \to \widehat{\Sc_\phi}, \quad \pi \mapsto \pair{\cdot, \pi}_\phi,
\]
which satisfies certain endoscopic identities and has image  
\[
\widehat{\Sc_\phi}^+ := \{\eta \in \widehat{\Sc_\phi} \mid \eta(z_\phi) = 1\}.
\]
For further details, see \cite[Theorem 2.2.1]{Art} and \cite{Moe3}.  When $\pi \in \Pi_\phi$ corresponds to $\eta \in \widehat{\Sc_\phi}^+$, we write $\pi = \pi(\phi, \eta)$. 

A useful property of this classification is the following: if we have a decomposition as in \eqref{eq:deco}:
\begin{equation*}
\phi =\phi' \oplus \phi_\gp \oplus  \phi^{\prime \vee},
\end{equation*}
then, for all \(\eta \in \widehat{\Sc_\phi}^+\), we have  
\begin{equation}\label{eq:irre}
\pi(\phi, \eta) \simeq \pi_{\phi'} \rtimes \pi(\phi_\gp, \eta).
\end{equation}  
See  \cite[8.11]{X2} for a more general result.

\subsection{Local Langlands data} We denote by $\Temp(G_n)$  the set of pairs $(\phi, \eta)$ such that $\phi \in \Phi(G_n)$ and $\eta \in \widehat{\Sc_\phi}$. Similarly, we denote by $\Temp^+(G_n)$  the set of pairs $(\phi, \eta)$ such that $\phi \in \Phi(G_n)$ and $\eta \in \widehat{\Sc_\phi}^+$. 

Again, it will be convenient  to work with all $G_n$'s together so we let:
\[
\Temp(G)= \bigsqcup_{n\geq 0}\Temp(G_n), \qquad  \Temp^+(G)= \bigsqcup_{n\geq0}\Temp^+(G_n).
\]
The set of Langlands data is then defined by:
\[
\Data(G) =  \Mult^{<0} \times \Temp(G), \qquad  \Data^+(G) =  \Mult^{<0} \times \Temp^+(G)
\]
It follows from \ref{sub:lang} and \ref{sub:endo} that for every irreducible representation $\pi\in \Irr^G$, there exists a unique triple $(\m; \phi, \eta) \in \Data^+(G)$ such that:
\[
\pi \simeq L(\m; \pi(\phi, \eta)).
\]
We refer to $(\m; \phi, \eta)$ as the \emph{Langlands data} of $\pi$.

\subsection{The Jantzen Decomposition}\label{sub:Jantzen}

We now aim to establish a line decomposition for classical groups, analogous to the one for \(\GL_n\) discussed in Paragraph \ref{sub:lines}. This is achieved by the Jantzen decomposition, which we recall in this paragraph. However, in the case of classical groups, we must distinguish between three different cases.

We begin with the parameters. For $ \phi = \bigoplus_{i \in I} \rho_i \boxtimes S_{a_i} $ an  $L$-parameter, we define $ \pi_\phi =L(\m_\phi) \in \Irr^\GL$ by
\[
\m_\phi = \sum_{i \in I} [-\frac{a_i-1}{2},\frac{a_i-1}{2}]_{\rho_i}.
\]
We define $\scusp(\phi)$ the support of the parameter $\phi$ as the support of $\m_\phi$. The line decomposition in Paragraph \ref{sub:lines} induces, for every  $\rho \in \Cusp^\GL$, a natural map $\phi \mapsto \phi_\rho$ so that 
\[
\m_{\phi_\rho} = (\m_\phi)_\rho
\]

Let \(\rho \in \Cusp^\GL\). Assume that \(\Z_\rho = \mathbb{Z}_{\rho^\vee}\). We denote by \(\Temp_\rho(G)\) (\resp \(\Temp^+_\rho(G)\)) the subset of \(\Temp(G)\) (\resp \(\Temp^+(G)\)) consisting of parameters \((\phi, \eta)\) satisfying \(\scusp(\phi) \subset \Z_\rho\) and let:
\begin{align*}
\Data_\rho(G) =  \Mult_\rho^{<0}  \times \Temp_\rho(G), \qquad  \Data_\rho^+(G) =  \Mult_\rho^{<0} \times \Temp_\rho^+(G)
\end{align*}
the set of Langlands parameters in the $\rho$-line. 
\begin{rem}\label{rem:nolines} Consider the natural projection
\begin{eqnarray*}
\Data(G) &\longrightarrow &\Data_\rho(G) \\
y=(\m; \phi, \eta) &\mapsto &y_\rho:=(\m_\rho; \phi_\rho, \eta_{| \Sc_{\phi_\rho}}).
\end{eqnarray*}
Observe that,  if $\phi_\rho\in\Phi_\gp(G)$, the image of $\Data^+(G)$ is not necessarily contained in $\Data^+_\rho(G)$.
\end{rem}

When \(\Z_\rho \neq \mathbb{Z}_{\rho^\vee}\), the above definition of \(\Temp_\rho(G)\) is unsuitable, as it would not give a subset of \(\Temp(G)\). In this case, we define \(\Temp_\rho(G)\) (\resp \(\Temp^+_\rho(G)\)) as the subset of \(\Temp(G)\) (\resp \(\Temp^+(G)\)) consisting of parameters \((\phi, \eta)\) with \(\scusp(\phi) \subset \Z_\rho \cup \mathbb{Z}_{\rho'^\vee}\), and let again:
\begin{align*}
\Data_\rho(G) =  \Mult_\rho^{<0} \times  \Mult_{\rho^\vee}^{<0} \times \Temp_\rho(G), \\  \Data_\rho^+(G) =  \Mult_\rho^{<0} \times \Mult_{\rho^\vee}^{<0} \times \Temp_\rho^+(G)
\end{align*}
the set of Langlands parameters in the $\rho$-line and denote by $y \mapsto y_\rho$ the projection map from $\Data(G)$ to $\Data_\rho(G)$.

We say that two supercuspidals $\rho,\rho'\in \Cusp^\GL$ are \emph{line equivalent}, and we denote it by $\rho \sim' \rho'$, if $\rho \sim \rho'$ or $\rho \sim \rho'^{\vee}$; or equivalently $\rho \sim' \rho'$, if $\mathbb{Z}_\rho \cup \mathbb{Z}_{\rho^\vee} = \mathbb{Z}_{\rho'} \cup \mathbb{Z}_{\rho'^\vee}$.

We have a natural bijection
\begin{align*}
\Data(G) &\simeq  \oplus_{\rho \in \Cusp^\GL/\sim'} \Data_\rho(G) \\
y &\mapsto y_\rho
\end{align*}
Remark \ref{rem:nolines} implies there is no such decomposition for $\Data^+(G)$.

 The discussion above makes the following definitions natural.
 \begin{defi}\label{def:good/bad/ugly}
Let $\rho \in \Cusp^{\GL}$. We write $\rho = \rho_u |\cdot|^x$ with $\rho_u$ unitary and $x \in \R$.
\begin{enumerate}
  \item We say that $\rho$ is \emph{ugly} if  $\rho_u$ is not self-dual or $x \notin (1/2)\Z$ (that is $\mathbb{Z}_\rho \neq \mathbb{Z}_{\rho^\vee}$).
  \item  We say that $\rho$ is \emph{good} if  $\rho_u$ is self-dual and
  \begin{itemize}
    \item If $\rho_u$ is of the same type as $G$ then $x \in \Z$.
    \item If $\rho_u$ is of the opposite type as $G$ then $x \in (1/2)\Z \setminus \Z$.
  \end{itemize}
  \item  We say that $\rho$ is \emph{bad} if  $\rho_u$ is self-dual and
  \begin{itemize}
    \item If $\rho_u$ is of the same type as $G$ then $x \in (1/2)\Z \setminus \Z$.
    \item If $\rho_u$ is of the opposite type as $G$ then $x \in \Z$.
  \end{itemize}
\end{enumerate}

\begin{rem}
  Let $\rho \in \Cusp^{\GL}$. By \cite[Remark 5.1]{AM} we see that the following conditions are equivalent: 
  \begin{enumerate}
    \item $\rho$ is good (\resp bad)
\item
For every $\pi(\phi, \eta)$ with $\phi \in \Phi_\gp(G)$ and $\eta \in \widehat{\Sc_\phi}$, there exists $m \in \Z$ (\resp\ $m\in \frac{1}{2}\Z \setminus \Z$) such that $\rho|\cdot|^{m} \rtimes \pi(\phi, \eta)$ is reducible.
\item
For some $\pi(\phi, \eta)$ with $\phi \in \Phi_\gp(G)$ and $\eta \in \widehat{\Sc_\phi}$, there exists $m \in \Z$  (\resp\ $m\in \frac{1}{2}\Z \setminus \Z$) such that $\rho|\cdot|^{m} \rtimes \pi(\phi, \eta)$ is reducible.
\end{enumerate}
\end{rem}
\end{defi}

We denote by $\Cusp^{\good}$ (\resp $\Cusp^{\bad}$, \resp $\Cusp^{\ugly}$) a set of representatives of good (\resp bad, \resp ugly) representations under the line equivalence relation $\sim'$.

\begin{defi}\label{def:parity}
Let $\sigma \in \Cusp^G$ and let $\pi \in \Irr_\sigma$.
\begin{enumerate}
\item
If 
\[
\scusp(\pi) \subset \left( \bigcup_{\rho \in \Cusp^{\good}} \Z_\rho \right) \cup \{\sigma\}, 
\]
we say that $\pi$ is \emph{of good parity}. 
We write  $\Irr_\sigma^\good$ for the set of such representations.

For any  multisegment $\m$, we denote:
\[\m_\gp:=\sum_{\rho \in \Cusp^{\good}}\m_\rho\]

\item 
If $\scusp(\pi) \subset \Z_\rho \cup \{\sigma\}$ for some bad representation $\rho$, 
we say that $\pi$ is \emph{of bad parity} (or of $\rho$-bad parity if we want to specify $\rho$). 
We write  $\Irr_\sigma^{\rho-\bad}$ for the set of such representations.

\item 
If $\scusp(\pi) \subset (\Z_\rho \cup \Z_{\rho^\vee}) \cup \{\sigma\}$ for some ugly representation $\rho$, 
we say that $\pi$ is \emph{ugly} (or $\rho$-ugly if we want to specify $\rho$). 
We write  $\Irr_\sigma^{\rho-\ugly}$ for the set of such representations.
\end{enumerate}
\end{defi}

Let $\pi \in \Irr_\sigma$. Jantzen \cite{J-dec} defines the representations 
$\pi^\good \in \Irr_\sigma^\good$, $\pi^{\rho-\bad} \in \Irr_\sigma^{\rho-\bad}$, 
and $\pi^{\rho-\ugly} \in \Irr_\sigma^{\rho-\ugly}$ as follows: 
\begin{itemize}
\item 
$\pi^\good$ is the unique representation in $\Irr_\sigma^\good$ 
such that $\pi \hookrightarrow \tau \rtimes \pi^\good$, where no good supercuspidal representations appear in $\scusp(\tau)$.
\item 
If $\rho$ is a bad supercuspidal representation, 
then $\pi^{\rho-\bad}$ is the unique representation in $\Irr_\sigma^{\rho-\bad}$ 
such that $\pi \hookrightarrow \tau \rtimes \pi^{\rho-\bad}$, with $\scusp(\tau) \cap \Z_\rho = \emptyset$.
\item 
If $\rho$ is an ugly supercuspidal representation, 
then $\pi^{\rho-\ugly}$ is the unique representation in $\Irr_\sigma^{\rho-\ugly}$ 
such that $\pi \hookrightarrow \tau \rtimes \pi^{\rho-\ugly}$, with $\scusp(\tau) \cap (\Z_\rho \cup \Z_{\rho^\vee}) = \emptyset$.
\end{itemize}

\begin{thm}\label{thm:Jantzen}
The map
\begin{align*}
\Jz \colon \Irr_\sigma &\longrightarrow 
\Irr_\sigma^\good 
\sqcup \left( \bigsqcup_{\rho \in \Cusp^{\bad}} \Irr_\sigma^{\rho-\bad} \right)  
\sqcup \left( \bigsqcup_{\rho \in \Cusp^{\ugly}} \Irr_\sigma^{\rho-\ugly} \right), \\
\pi &\longmapsto \left( \pi^\good, \{\pi^{\rho-\bad}\}_\rho, \{\pi^{\rho-\ugly}\}_\rho \right)
\end{align*}
is a bijection. It commutes with the Aubert--Zelevinsky duality in the following sense:
$$\Jz(\widehat{\pi})=\left( \widehat{\pi^\good}, \{\widehat{\pi^{\rho-\bad}}\}_\rho, \{ \widehat{\pi^{\rho-\ugly}}\}_\rho \right).$$
Furthermore, let $y=(\phi_\sigma,\eta_\sigma)$ and $(\m; \phi,\eta)$ denote the Langlands data of $\sigma$ and $\pi$, respectively. Then, we have:\begin{itemize}
\item The Langlands data of $ \pi^\good$ are $y^\good:=(\m_\gp;\phi_\gp,\eta)$.
\item The Langlands data of $\pi^{\rho-\bad}$ are $y^{\rho-\bad}:=(\m_{\rho}; \phi_\rho+\phi_\sigma,\eta_\sigma)$.
\item The Langlands data of  $\pi^{\rho-\ugly}$ are $y^{\rho-\ugly}:=(\m_{\rho}+\m_{\rho^\vee}; \phi_\rho+\phi_{\rho^\vee}+\phi_\sigma,\eta_\sigma)$.
\end{itemize}
\end{thm}
\begin{proof}
The theorem is due to C. Jantzen; see \cite[Theorem 9.3]{J-dec}. The only point that needs some explanation is the description of the map $\Jz$ in terms of Langlands data. It is probably well-known, but for the convenience of the reader, we provide some details.
Let's start with the first bullet point. We write $\phi =\phi' \oplus \phi_\gp \oplus  \phi^{\prime \vee}$. By the Langlands classification, $\pi$ is the socle of $\tilde{\lambda}(\m) \rtimes \pi(\phi,\eta)$, which is isomorphic, by \cite[9.7]{Z} and \eqref{eq:irre}  to
\[ \tilde{\lambda}(\m -  \m_\gp) \times \tilde{\lambda}(\m_\gp)  \times \pi_{\phi'} \rtimes \pi(\phi_\gp,\eta)
\]
which is equivalent by \cite[9.7]{Z} to 
\[ \tilde{\lambda}(\m -  \m_\gp)   \times \pi_{\phi'} \times \tilde{\lambda}(\m_\gp) \rtimes \pi(\phi_\gp,\eta).
\]
We deduce that $\pi$ is the socle of
\[ \soc(\tilde{\lambda}(\m -  \m_\gp)   \times \pi_{\phi'}) \rtimes \soc( \tilde{\lambda}(\m_\gp) \rtimes \pi(\phi_\gp,\eta)),
\]
which, by the Langlands subrepresentation theorem, proves the first claim. The proofs of the second and third bullet points follow a similar approach. By \cite[Theorem 9.3.(8)]{J-dec}, we can assume that $\m= \emptyset$. Then, for the second bullet, $\pi(\phi,\eta)$ is isomorphic  by \eqref{eq:irre} to $ \pi_{\frac{1}{2}\phi_\rho}\rtimes\pi_{\phi -  \phi_\rho}$, which embeds into 
$ \pi_{\frac{1}{2}\phi_\rho}\times \tau \rtimes\pi(\phi_\sigma,\eta_\sigma)$, for some $\tau$  with $\scusp(\tau) \cap \Z_\rho = \emptyset$, so isomorphic by \cite[8.6]{Z} to $\tau \times \pi_{\frac{1}{2}\phi_\rho}\rtimes\pi(\phi_\sigma,\eta_\sigma)$. As  $ \pi_{\frac{1}{2}\phi_\rho}\rtimes\pi(\phi_\sigma,\eta_\sigma)$ is irreducible by \eqref{eq:irre} and isomorphic to $\pi(\phi_\rho+\phi_\sigma,\eta_\sigma)$ the result follows. The third bullet is proved in exactly the same manner.
\end{proof}

A consequence of the explicit description of the Jantzen decomposition is that the map $y \mapsto y_\rho$ factors through $\Jz$. In other words:
\begin{cor}\label{cor:lines}
Let $\pi \in \Irr_\sigma$, with Langlands data $y$.  Let $\rho\in \Cusp^\GL$. Then:
\begin{itemize}
\item If $\rho$ is good then, $y_\rho=(y^\good)_\rho$.
\item If $\rho$ is bad then, $y_\rho=(y^{\rho-\bad})_\rho$.
\item If $\rho$ is ugly then, $y_\rho=(y^{\rho-\ugly})_\rho$.
\end{itemize}
\end{cor}
\begin{rem}
One could define \emph{$\rho-\good$ representations} and obtain a decomposition similar to that in Theorem \ref{thm:Jantzen}, replacing $\good$ with $\rho-\good$. However, Remark \ref{rem:nolines} implies that such a definition would not allow for a natural  description —in terms of Langlands data— of the corresponding map. Furthermore, the first bullet point of Corollary \ref{cor:lines} would no longer hold if one replaces  $\good$ with $\rho-\good$. For this reason, following \cite{AM}, we have chosen to define only $\good$ representations and not $\rho-\good$ representations.
\end{rem}

\subsection{Symmetrization}
\label{sec:symmetrization}

Our algorithm to compute the Aubert--Zelevinsky involution is similar to the M{\oe}glin--Waldspurger algorithm and will use symmetrical multisegments with signs instead of Langlands data.

\bigskip

We will define a transfer map that sends the elements of $\Data(G)$ to symmetrical multisegments. The tempered representations will be sent to centered segments. As these representations come with a sign, we need to define the notion of a centered segment with a sign. A centered segment with a sign is a pair $(\Delta,\varepsilon)$ where $\Delta \in \Seg^{0}$ and $\varepsilon \in \{-1,1\}$. Let $\Mult^\varepsilon$ be the multiset composed of centered segments with signs and non centered segments, that is formally $\Mult^\varepsilon$ is the set of functions $\Seg^{< 0} \cup \Seg^{>0} \cup (\Seg^{0} \times \{-1,1\}) \to \NN$, with finite support. For $\mathfrak{s} \in \Mult^\varepsilon$, we call the underlying multisegment of $\mathfrak{s}$ the multisegment of $\Mult$ obtained by forgetting all the signs. We will usually write $\mathfrak{s} \in \Mult^\varepsilon$ as $\mathfrak{s}=(\m,\varepsilon)$, where $\m \in \Mult$ is the underlying multisegment of $\mathfrak{s}$ and, if $\Delta_i \in \m^0$, $\varepsilon(\Delta_i) \in \{-1,1\}$ is the sign of $\Delta_i$.

Let $\Symm \subseteq \Mult$ be the set of symmetrical multisegments, that is $\Symm = \{ \m \in \Mult, \m^{\vee}=\m\}$. We also define $\Symm^{\varepsilon} \subseteq \Mult^\varepsilon$ to be the subset of multisegments with signs of elements whose underlying multisegment is in $\Symm$. 

\bigskip

We have a natural transfer map

\[
  \trans : \Data(G) \to \Symm^{\varepsilon}
\]
defined as follows.
Let $(\n; \phi, \eta) \in \Data(G)$. Then $\trans(\n; \phi, \eta)=(\m,\varepsilon)$ where
\[
  \m := \sum_{\Delta \in \n} {(\Delta + \Delta^{\vee})} + \sum_{\rho \boxtimes S_a \in \phi} \left[\frac{-a+1}{2},\frac{a-1}{2}\right]_\rho
\]
and if $\rho \boxtimes S_a \in \phi$ then
\[
  \varepsilon (\left[\frac{-a+1}{2},\frac{a-1}{2}\right]_\rho):= \eta(\rho \boxtimes S_a).
\]

The map \(\trans\) is injective, and its image can be described as follows. Let \(\Symm^\varepsilon(G)\) be the subset  of \(\Symm^\varepsilon\) consisting of all elements \(\mathfrak{s} \in \Symm^\varepsilon\) satisfying the following conditions for every pair of signed centered segments $(\Delta,\varepsilon), (\Delta',\varepsilon') \in \mathfrak{s}$:
\begin{enumerate}
  \item If $\Delta=\Delta'$ then $\varepsilon=\varepsilon'$.
  \item If $\Delta$ is supported in $\Z_\rho$ with $\rho$ bad or ugly, then $\varepsilon=1$.
  \item If $\Delta$ is supported in $\Z_\rho$ with $\rho$ bad, then the multiplicity of $(\Delta,\varepsilon)$ in $\mathfrak{s}$ is even.
\end{enumerate}
We then define the following subsets of \(\Symm^\varepsilon(G)\):
\begin{itemize}
  \item \(\Symm^{\varepsilon,1}(G)\), the elements \((\m,\varepsilon)\) with
        \(\det(\m) = 1\);
  \item \(\Symm^{\varepsilon,+}(G)\), the elements for which the product of the
        signs of the centered segments equals \(1\);
  \item \(\Symm^{\varepsilon,+,1}(G) := \Symm^{\varepsilon,1}(G) \cap
        \Symm^{\varepsilon,+}(G)\).
\end{itemize}

The image of \(\Data(G)\) under \(\trans\) is \(\Symm^{\varepsilon,1}(G)\), so
\(\trans\) restricts to a bijection
\[
  \trans : \Data(G) \overset{\sim}{\to} \Symm^{\varepsilon,1}(G),
\]
which sends \(\Data^+(G)\) to \(\Symm^{\varepsilon,+,1}(G)\). Consequently, for every irreducible representation \(\pi \in \Irr^G\) there is a unique element
\((\m,\varepsilon) \in \Symm^{\varepsilon,+,1}(G)\) such that \(\pi \simeq L(\trans^{-1}(\m,\varepsilon))\). We call \((\m,\varepsilon)\) the \emph{symmetrical Langlands data} of \(\pi\) and write \(\pi = L(\m,\varepsilon)\).

\bigskip

Let $\rho \in \Cusp^{\GL}$. If $\rho$ is good or bad (\resp ugly), we define $\Symm_{\rho}^{\varepsilon}(G)$ to be the subset of $\Symm^{\varepsilon}(G)$ of elements with underlying multiset in $\Mult_\rho$ (\resp $\Mult_\rho \times \Mult_{\rho^{\vee}}$). This gives us a natural decomposition
\[
  \Symm^{\varepsilon}(G)\simeq  \oplus_{\rho \in \Cusp^\GL/\sim'} \Symm_{\rho}^{\varepsilon}(G).
\]

The restriction of $\trans$ to $\Data_\rho(G)$ gives us a map
\[
  \trans_\rho : \Data_\rho(G) \to \Symm_{\rho}^{\varepsilon}(G)
\]
making the following diagram commute:

\[\xymatrix@C=3cm{
  \Data(G) \ar[d]^{\trans} \ar[r]^-{y\mapsto y_\rho} & \Data_\rho(G) \ar[d]^{\trans_\rho} \\
  \Symm^{\varepsilon}(G) \ar[r]^-{(\tilde{\m},\tilde{\varepsilon}) \mapsto (\tilde{\m}_{\rho},\tilde{\varepsilon}_{|\tilde{\m}_{\rho}})} & \Symm_{\rho}^{\varepsilon}(G) }.
\]

The sets $\Symm_\rho$, $\Symm_\rho^{\varepsilon}$ and $\Symm_\rho^{\varepsilon}(G)$ are endowed with an order $\le$ coming from the order $\le$ on $\Mult_\rho$ defined in Section \ref{sub:MW}.

\section{Definition of the algorithm}
\label{sec:defalgo}

In this section we define a map $\ADd : \Data(G) \to \Data(G)$. We will prove later that $\ADd$ is in fact an involution and that for $y \in \Data^{+}(G)$ we have $\widehat{\pi(y)} = \pi(\ADd(y))$.

\bigskip

Using the bijection $\trans : \Data(G) \overset{\sim}{\to} \Symm^{\varepsilon,1}(G)$, it is enough to define $\AD : \Symm^\varepsilon(G) \to \Symm^\varepsilon(G)$, preserving the determinant, and set $\ADd = \trans^{-1} \circ \AD \circ \trans$. The definition of $\AD$ is the content of this section. Using $\Symm^\varepsilon(G) =\oplus_{\rho \in \Cusp^\GL/\sim'} \Symm^\varepsilon_\rho(G)$, we will define for each  $\rho \in \Cusp^\GL/\sim'$ a map
\[
  \AD_{\rho} : \Symm^\varepsilon_\rho(G) \to \Symm^\varepsilon_\rho(G),
\]
and $\AD = \oplus_{\rho \in \Cusp^\GL/\sim'} \AD_{\rho}$.

\begin{rem}
  We will show in Proposition \ref{prop:ADsigns} that for each $\rho \in \Cusp^\GL$, $\AD_\rho$ induces a map $\AD_{\rho} : \Symm_{\rho}^{\varepsilon,+}(G) \to \Symm_{\rho}^{\varepsilon,+}(G)$ and $\AD_{\rho} : \Symm_{\rho}^{\varepsilon}(G) \setminus \Symm_{\rho}^{\varepsilon,+}(G) \to \Symm_{\rho}^{\varepsilon}(G) \setminus \Symm_{\rho}^{\varepsilon,+}(G)$. This implies that $\ADd$ induces a map $\ADd : \Data^{+}(G) \to \Data^{+}(G)$.
\end{rem}

In the next subsections, we define $\AD_\rho$ when $\rho$ is ugly (Subsection \ref{sec:defugly}), $\rho$ is bad (Subsection \ref{sec:defbad}) and $\rho$ is good (Subsection \ref{sec:defgood}).

\subsection{The ugly case} \label{sec:defugly}
Let $\rho \in \Cusp^{\GL}$ be ugly and $\rho_u$ its unitarization. By definition of $\Symm^{\varepsilon}_\rho(G)$, all the signs are trivial. So we can identify an element of $\Symm^{\varepsilon}_\rho(G)$ with its underlying multisegment. Moreover, this multisegment is necessarily of determinant 1. Let $\m \in \Symm^{\varepsilon}_\rho(G)$ and we want to define $\AD_\rho(\m) \in \Symm^{\varepsilon}_\rho(G)$. The definition is essentially the M{\oe}glin--Waldspurger algorithm (see Remark \ref{rem:MWGL} below).

\bigskip

Let $e_{\max,\rho}$ be the maximum of the ends $e(\Delta)$ of the segments $\Delta \in \m$ supported in $\Z_\rho$. Let $\Delta_{1}$ be the biggest segment of $\m$ supported in $\Z_\rho$ such that $e(\Delta_1) = e_{\max,\rho}$. We then define recursively a sequence $\Delta_1 \geq \dots \geq \Delta_l$, where $\Delta_j$ is the biggest segment of $\m$ supported in $\Z_\rho$ such that $\Delta_j \le \Delta_{j-1}$ and $e(\Delta_j)=e(\Delta_{j-1})-1$. We call the sequence $\Delta_1,\cdots,\Delta_l$ the \emph{the initial sequence in the algorithm}.

From this sequence, we define
\[
  \m_1 := [e(\Delta_{l}),e(\Delta_{1})]_{\rho_u} + [-e(\Delta_{1}),-e(\Delta_{l})]_{\rho_u^{\vee}}
\]
and
\[
  \m^{\#} = \m + \sum_{i=1}^l (\Delta_{i}^{-} - \Delta_{i} + {}^{-}\Delta_{i}^{\vee} - \Delta_{i}^{\vee}).
\]

\begin{defi}
  When $\rho$ is ugly, we define $\AD_\rho : \Symm_\rho^\varepsilon(G) \to \Symm_\rho^\varepsilon(G)$ inductively by
  \[
  \AD_\rho(\m)=\m_1 + \AD_\rho (\m^{\#}).
\] 
\end{defi}

\begin{rem}
From the construction of $\AD_\rho$, it is clear that $\AD_\rho(\m)$ is a symmetrical multisegment, hence $\AD_\rho : \Symm_\rho^\varepsilon(G) \to \Symm_\rho^\varepsilon(G)$ is a well defined map.
\end{rem}

\begin{rem}\label{rem:MWGL}
  Let us write $\m = \m_\rho + \m_{\rho^{\vee}}$ with $\m_\rho \in \Mult_\rho$ and $\m_{\rho^{\vee}} \in \Mult_{\rho^{\vee}}$. It is clear from the definition that $\AD_\rho(\m)=\m_\rho^t+(\m_\rho^t)^{\vee}$, where $\m_\rho^t$ is the M{\oe}glin--Waldspurger dual of $\m_\rho$ (see Section \ref{sub:MW}).
\end{rem}

  \begin{ex}
  Let $\rho$ be ugly. Let $\pi:=L(\n;\pi(\phi,\eta))$ with $\n=[-3,-1]_{\rho} + [-2,-1]_{\rho} + [-2,0]_{\rho}$ and $\pi(\phi,\eta)$ trivial. We associate to $\pi$ the symmetric multisegment $\m \in \Symm$ defined by $\m=[-3,-1]_{\rho} + [-2,-1]_{\rho} + [-2,0]_{\rho} + [1,3]_{\rho^{\vee}} + [1,2]_{\rho^{\vee}} + [0,2]_{\rho^{\vee}}$. 

  In the diagram below, we represent the multisegment $\m$ ordered by $\le$. The solid lines are the segments supported on $\Z_{\rho^{\vee}}$ and the dotted lines are the segments on $\Z_{\rho}$. The thick black lines indicate the initial sequence of segments $\Delta_1, \dots, \Delta_l$. The portions highlighted in green mark the parts of $\m$ that are extracted to form the segment $\m_1$; the remaining parts will constitute the multisegment $\m^{\#}$ after this first step.

  
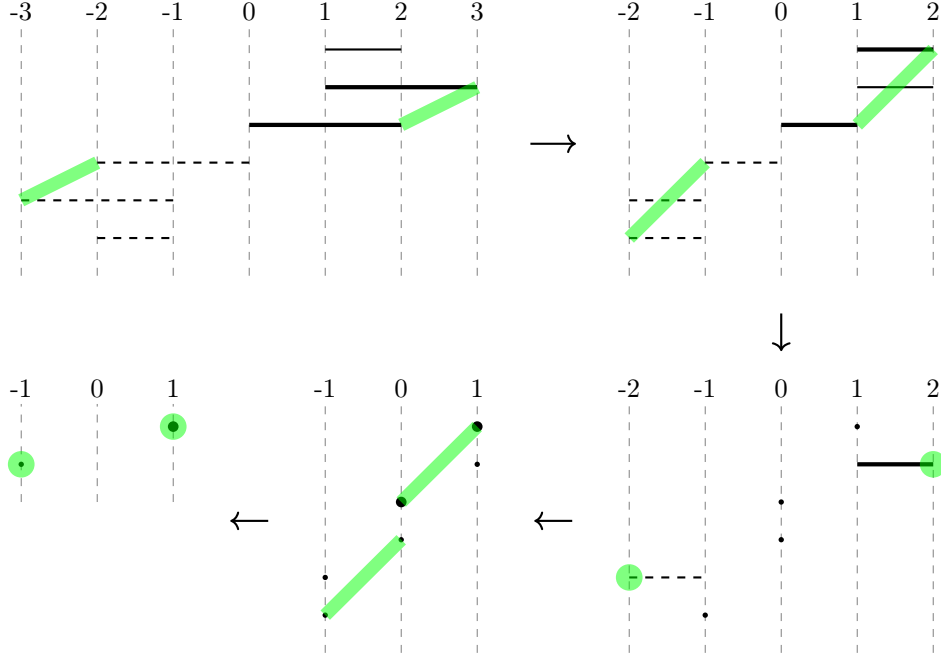
\begin{figure}[h!]
    \centering
    \begin{tikzpicture}[scale=0.65]
  
      \begin{scope}[shift={(0,0)}]
        \foreach \x in {-3,-2,-1,0,1,2,3} {
            \draw[dashed, gray, very thin] (\x,-3.5) -- (\x,-0.2);
            \node[black] at (\x,0) {\x};
        }
        \draw[black, thick] (1,-0.5) -- (2,-0.5);
        \draw[black, ultra thick] (1,-1) -- (3,-1);
        \draw[black, ultra thick] (0,-1.5) -- (2,-1.5);
        \draw[black, thick, dashed] (-2,-2) -- (0,-2);
        \draw[black, thick, dashed] (-3,-2.5) -- (-1,-2.5);
        \draw[black, thick, dashed] (-2,-3) -- (-1,-3);

        \draw[green, line width=5pt, opacity=0.5] (3,-1) -- (2,-1.5);
        \draw[green, line width=5pt, opacity=0.5] (-2,-2) -- (-3,-2.5);

      \end{scope}
  
      \begin{scope}[shift={(7,0)}]
        \foreach \x in {-2,-1,0,1,2} {
            \draw[dashed, gray, very thin] (\x,-3.5) -- (\x,-0.2);
            \node[black] at (\x,0) {\x};
        }
        \draw[black, ultra thick] (1,-0.5) -- (2,-0.5);
        \draw[black, thick] (1,-1) -- (2,-1);
        \draw[black, ultra thick] (0,-1.5) -- (1,-1.5);
        \draw[black, thick, dashed] (-1,-2) -- (0,-2);
        \draw[black, thick, dashed] (-2,-2.5) -- (-1,-2.5);
        \draw[black, thick, dashed] (-2,-3) -- (-1,-3);

        \draw[green, line width=5pt, opacity=0.5] (2,-0.5) -- (1,-1.5);
        \draw[green, line width=5pt, opacity=0.5] (-1,-2) -- (-2,-3);
      \end{scope}

      \begin{scope}[shift={(7,-5)}]
         \foreach \x in {-2,-1,0,1,2} {
            \draw[dashed, gray, very thin] (\x,-3.5) -- (\x,-0.2);
            \node[black] at (\x,0) {\x};
        }
        
        \fill[black] (1,-0.5) circle (1pt);
        \draw[black, ultra thick] (1,-1) -- (2,-1);
        \fill[black] (0,-1.5) circle (1pt);
        \fill[black] (0,-2) circle (1pt);
        \draw[black, thick, dashed] (-2,-2.5) -- (-1,-2.5);
        \fill[black] (-1,-3) circle (1pt);

        \fill[green, opacity=0.5] (2,-1) circle[radius=5pt];
        \fill[green, opacity=0.5] (-2,-2.5) circle[radius=5pt];

      \end{scope}

      \begin{scope}[shift={(2,-5)}]
        \foreach \x in {-1,0,1} {
            \draw[dashed, gray, very thin] (\x,-3.5) -- (\x,-0.2);
            \node[black] at (\x,0) {\x};
        }
        \fill[black] (1,-0.5) circle (2pt);
        \fill[black] (1,-1) circle (1pt);
        \fill[black] (0,-1.5) circle (2pt);
        \fill[black] (0,-2) circle (1pt);
        \fill[black] (-1,-2.5) circle (1pt);
        \fill[black] (-1,-3) circle (1pt);

        \draw[green, line width=5pt, opacity=0.5] (1,-0.5) -- (0,-1.5);
        \draw[green, line width=5pt, opacity=0.5] (0,-2) -- (-1,-3);
      \end{scope}

      \begin{scope}[shift={(-2,-5)}]
        \foreach \x in {-1,0,1} {
            \draw[dashed, gray, very thin] (\x,-1.5) -- (\x,-0.2);
            \node[black] at (\x,0) {\x};
        }
        \fill[black] (1,-0.5) circle (2pt);
        \fill[black] (-1,-1) circle (1pt);

        \fill[green, opacity=0.5] (1,-0.5) circle[radius=5pt];
        \fill[green, opacity=0.5] (-1,-1) circle[radius=5pt];
      \end{scope}
  
      \draw[->, thick] (3.7,-1.75) -- (4.3,-1.75);
      \draw[->, thick] (7,-4) -- (7,-4.5);
      \draw[->, thick] (4.25,-6.75) -- (3.75,-6.75);
      \draw[->, thick] (0.25,-6.75) -- (-0.25,-6.75);

    \end{tikzpicture}

  \end{figure}
  
  The algorithm gives (in five steps) that $\AD_\rho(\m)=([-3,-2]_\rho+[2,3]_{\rho^{\vee}})+([-2,-1]_\rho+[1,2]_{\rho^{\vee}}) + ([-2,-2]_\rho+[2,2]_{\rho^{\vee}})+([-1,0]_\rho+[0,1]_{\rho^{\vee}})+([-1,-1]_\rho+[1,1]_{\rho^{\vee}})$. Thus $\hat{\pi}=L([-3,-2]_\rho+[-2,-1]_\rho + [-2,-2]_\rho + [-1,0]_\rho + [-1,-1]_\rho; \pi(\phi,\eta))$ with $\pi(\phi,\eta)$ trivial.
  \end{ex}

  \begin{ex}
  With the notation and coloring as above. Let $\pi:=L(\n;\pi(\phi,\eta))$ with $\n=[-2,1]_{\rho}$ and $\pi(\phi,\eta)$ trivial. We associate to $\pi$ the symmetric multisegment $\m \in \Symm$ defined by $\m=[-2,1]_\rho+[-1,2]_{\rho^{\vee}}$. 
  
  \begin{figure}[h!]
    \centering
    \begin{tikzpicture}[scale=1]
  
      \begin{scope}[shift={(0,0)}]
        \foreach \x in {-2,-1,0,1,2} {
            \draw[dashed, gray, very thin] (\x,-1.5) -- (\x,-0.2);
            \node[black] at (\x,0) {\x};
        }
        \draw[black, ultra thick] (-1,-0.5) -- (2,-0.5);
        \draw[black, thick, dashed] (-2,-1) -- (1,-1);
        \fill[green, opacity=0.5] (2,-0.5) circle[radius=5pt];
        \fill[green, opacity=0.5] (-2,-1) circle[radius=5pt];
      \end{scope}
  
      \begin{scope}[shift={(5,0)}]
        \foreach \x in {-1,0,1} {
            \draw[dashed, gray, very thin] (\x,-1.5) -- (\x,-0.2);
            \node[black] at (\x,0) {\x};
        }
        \draw[black, ultra thick] (-1,-0.5) -- (1,-0.5);
        \draw[black, thick, dashed] (-1,-1) -- (1,-1);
        \fill[green, opacity=0.5] (1,-0.5) circle[radius=5pt];
        \fill[green, opacity=0.5] (-1,-1) circle[radius=5pt];
      \end{scope}

      \begin{scope}[shift={(0,-2.5)}]
        \foreach \x in {-1,0,1} {
            \draw[dashed, gray, very thin] (\x,-1.5) -- (\x,-0.2);
            \node[black] at (\x,0) {\x};
        }
        \fill[black] (-1,-0.5) circle (2pt);
        \fill[black] (1,-1) circle (1pt);
        \fill[green, opacity=0.5] (-1,-0.5) circle[radius=5pt];
        \fill[green, opacity=0.5] (1,-1) circle[radius=5pt];
      \end{scope}

      \begin{scope}[shift={(5,-2.5)}]
        \foreach \x in {-1,0,1} {
            \draw[dashed, gray, very thin] (\x,-1.5) -- (\x,-0.2);
            \node[black] at (\x,0) {\x};
        }
        \draw[black, ultra thick] (-1,-0.5) -- (0,-0.5);
        \draw[black, thick, dashed] (0,-1) -- (1,-1);
        \fill[green, opacity=0.5] (0,-0.5) circle[radius=5pt];
        \fill[green, opacity=0.5] (0,-1) circle[radius=5pt];
      \end{scope}
  
      \draw[->, thick] (2.5,-0.75) -- (3,-0.75);
      \draw[->, thick] (5,-1.7) -- (5,-2.1);
      \draw[->, thick] (3,-3.15) -- (2.5,-3.15);

    \end{tikzpicture}

  \end{figure}
  
  The algorithm gives (in four steps) that $\AD_\rho(\m)=([-2,-2]_\rho+[2,2]_{\rho^{\vee}})+([-1,-1]_\rho+[1,1]_{\rho^{\vee}}) + ([0,0]_\rho+[0,0]_{\rho^{\vee}})+([1,1]_\rho+[-1,-1]_{\rho^{\vee}})$. Thus $\hat{\pi}=L([-2,-2]_\rho+[-1,-1]_\rho+[-1,-1]_{\rho^{\vee}}; \pi(\rho \boxtimes S_1 + \rho^{\vee} \boxtimes S_1, 1))$.
  \end{ex}

\subsection{The bad case}\label{sec:defbad}

Let $\rho \in \Cusp^{\GL}$ be bad. Since $\rho$ is bad, all the signs of the elements in $\Symm^{\varepsilon}_\rho(G)$ are trivial. Thus, we can identify an element of $\Symm^{\varepsilon}_\rho(G)$ with its underlying multisegment, which is of determinant 1.

Let $\mathfrak{m} \in \Symm^{\varepsilon}_\rho(G)$. We want to define $\AD_\rho(\mathfrak{m}) \in \Symm^{\varepsilon}_\rho(G)$. The definition will be similar to the ugly case. The difference here is that all the segments in $\mathfrak{m}$ lie on the same line $\Z_\rho = \Z_{\rho^{\vee}}$. In particular, it may happen that for some $\Delta \in \mathfrak{m}$, we have $\Delta^{\vee} \le \Delta$.

We impose a condition (see (3) below) that must be satisfied so that both $\Delta$ and $\Delta^{\vee}$ appear in the initial sequence of the algorithm.

\bigskip

Let $e_{\max}$ be the biggest coefficient of the segments of $\m$ (hence the biggest end). Let $\Delta_{1}$ be the biggest segment of $\m$ such that $e(\Delta_1) = e_{\max}$. We then define recursively a sequence $\Delta_1 \geq \dots \geq \Delta_l$, where $\Delta_j$ is the biggest segment of $\m$, if it exists, satisfying
\begin{enumerate}
\item $\Delta_j \le \Delta_{j-1}$;
\item $e(\Delta_j)=e(\Delta_{j-1})-1$;
\item If there exists $i < j$ such that $\Delta_j^{\vee}=\Delta_i$ then $m_{\m}(\Delta_j) \ge 2$.
\end{enumerate} 
We call the sequence $\Delta_1,\cdots,\Delta_l$ the \emph{the initial sequence in the algorithm}.

From this sequence, we define
\[
  \m_1 := [e(\Delta_{l}),e(\Delta_{1})]_{\rho_u} + [-e(\Delta_{1}),-e(\Delta_{l})]_{\rho_u}
\]
and
\[
  \m^{\#} = \m + \sum_{i=1}^l (\Delta_{i}^{-} - \Delta_{i} + {}^{-}\Delta_{i}^{\vee} - \Delta_{i}^{\vee}).
\]

\begin{rem}
Condition (3) ensures that $\m^{\#}$ is a well-defined multisegment. Indeed, if $\Delta_i = \Delta_j^{\vee}$ for some $i \neq j$, then $m_{\m}(\Delta_i) \ge 2$, so we can suppress $\Delta_i$ twice from $\m$.
\end{rem}

\begin{defi}
  When $\rho$ is bad, we define $\AD_\rho : \Symm_\rho^\varepsilon(G) \to \Symm_\rho^\varepsilon(G)$ inductively by
  \[
    \AD_\rho(\m)=\m_1 + \AD_\rho (\m^{\#}).
  \]
\end{defi}

\begin{rem}
From the construction it is clear that $\AD_\rho(\m)$ is a symmetrical multisegment. It is also very easy to see that all the centered segments in $\m_1$ and $\m^{\#}$ have  even multiplicity. Therefore, the image of $\AD_\rho$ is indeed in $\Symm_\rho^\varepsilon(G)$ and $\AD_\rho$ is well defined.
\end{rem}

\begin{ex}\label{ex:badpar01}
Let $\rho$ be of bad parity and unitary. Let $\pi:=L(\n;\pi(\phi,\eta))$ with $\n=[-1,0]_\rho$ and $\pi(\phi,\eta)$ trivial. We associate to $\pi$ the symmetric multisegment $\m \in \Symm$ defined by $\m=[-1,0]_\rho+[0,1]_\rho$. 

In the diagram below, we represent the multisegment $\m$ ordered by $\le$. The thick black lines indicate the initial sequence of segments $\Delta_1, \dots, \Delta_l$. The portions highlighted in green mark the parts of $\m$ that are extracted to form the segment $\m_1$; the remaining parts will constitute the multisegment $\m^{\#}$ after this first step.

\begin{figure}[h!]
  \centering
  \begin{tikzpicture}[scale=1]

    \begin{scope}[shift={(0,0)}]
      \foreach \x in {-1,0,1} {
          \draw[dashed, gray, very thin] (\x,-1.5) -- (\x,-0.2);
          \node[black] at (\x,0) {\x};
      }
      \draw[black, ultra thick] (0,-0.5) -- (1,-0.5);
      \draw[black, thick] (-1,-1) -- (0,-1);
      \fill[green, opacity=0.5] (1,-0.5) circle[radius=5pt];
      \fill[green, opacity=0.5] (-1,-1) circle[radius=5pt];
    \end{scope}

    \begin{scope}[shift={(5,0)}]
      \foreach \x in {-1,0,1} {
          \draw[dashed, gray, very thin] (\x,-1.5) -- (\x,-0.2);
          \node[black] at (\x,0) {\x};
      }
      \fill[black] (0,-0.5) circle (2pt);
      \fill[black] (0,-1) circle (1pt);
      \fill[green, opacity=0.5] (0,-0.5) circle[radius=5pt];
      \fill[green, opacity=0.5] (0,-1) circle[radius=5pt];
    \end{scope}

    \draw[->, thick] (2,-0.75) -- (2.5,-0.75);

  \end{tikzpicture}
\end{figure}

The algorithm gives (in two steps) that $\AD_\rho(\m)=([-1,-1]_\rho+[1,1]_\rho)+([0,0]_\rho+[0,0]_\rho)$. Thus $\hat{\pi}=L([-1,-1]_\rho; \pi(\rho \boxtimes S_1 + \rho \boxtimes S_1, 1))$.
\end{ex}

\begin{ex} \label{ex:badpar0101}
  With the notation and coloring as above, we consider now the case $\pi:=L(\n; \pi(\phi,\eta))$ with $\n=[-1,0]_\rho + [-1,0]_\rho$ and $\pi(\phi,\eta)$ trivial. The symmetric multisegment $\m \in \Symm$ is now $\m=[-1,0]_\rho+[-1,0]_\rho+[0,1]_\rho+[0,1]_\rho$. 
  
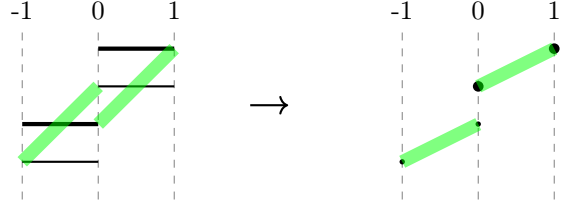
  \begin{figure}[h!]
    \centering
    \begin{tikzpicture}[scale=1]
  
      \begin{scope}[shift={(0,0)}]
        \foreach \x in {-1,0,1} {
            \draw[dashed, gray, very thin] (\x,-2.5) -- (\x,-0.2);
            \node[black] at (\x,0) {\x};
        }
        \draw[black, ultra thick] (0,-0.5) -- (1,-0.5);
        \draw[black, thick] (0,-1) -- (1,-1);
        \draw[black, ultra thick] (-1,-1.5) -- (0,-1.5);
        \draw[black, thick] (-1,-2) -- (0,-2);

        \draw[green, line width=5pt, opacity=0.5] (1,-0.5) -- (0,-1.5);
        \draw[green, line width=5pt, opacity=0.5] (0,-1) -- (-1,-2);
      \end{scope}
  
      \begin{scope}[shift={(5,0)}]
        \foreach \x in {-1,0,1} {
            \draw[dashed, gray, very thin] (\x,-2.5) -- (\x,-0.2);
            \node[black] at (\x,0) {\x};
        }
        \fill[black] (1,-0.5) circle (2pt);
        \fill[black] (0,-1) circle (2pt);
        \fill[black] (0,-1.5) circle (1pt);
        \fill[black] (-1,-2) circle (1pt);
        \draw[green, line width=5pt, opacity=0.5] (1,-0.5) -- (0,-1);
        \draw[green, line width=5pt, opacity=0.5] (0,-1.5) -- (-1,-2);
      \end{scope}
  
      \draw[->, thick] (2,-1.25) -- (2.5,-1.25);
  
    \end{tikzpicture}
  \end{figure}
  
  The algorithm gives (in two steps) that $\AD_\rho(\m)=([-1,0]_\rho+[0,1]_\rho)+([-1,0]_\rho+[0,1]_\rho)$. Thus, $\hat{\pi}=\pi$.
  \end{ex}

\subsection{The good case}\label{sec:defgood}

Let $\rho \in \Cusp^{\GL}$ be good and $(\m,\varepsilon) \in \Symm^{\varepsilon}_\rho$. Let $\rho_{u}$ be the unitarization of $\rho$. The definition of $\AD_\rho$ will also resemble the definition in the bad case; however, there are some differences:
\begin{itemize}
  \item The signs play a role here.
  \item A pair of segments $(\Delta,\Delta^{\vee})$ can appear in the initial sequence of the algorithm without any condition on the multiplicity. In this case, in $\m^{\#}$ the segment $\Delta$ will be replaced by ${}^{-}\Delta^{-}$.
  \item The parity of the multiplicity of the centered segments also plays a role.
\end{itemize}

To make these conditions more transparent, we introduce a new set $\USymm_\rho^{\varepsilon}$ equipped with an order $\preceq$, in which the definition of $\AD_\rho$ closely resembles the previous cases. We describe it in the following paragraph.

\subsubsection{The set $\USymm^{\varepsilon}$ and the order $\preceq$}
\label{sec:labelseg}

The idea of the set $\Symm^\epsilon(G)$ arises from the structure of maximal parabolic subgroups of $G_n$, which are of the form
\[
\left(
\begin{array}{c|c|c}
\mathrm{GL}_{n_1}(F) & * & * \\ \hline
0 & G_{n_0} & * \\ \hline
0 & 0 & \mathrm{GL}_{n_1}(F)
\end{array}
\right),
\]
where the bottom-right block $\mathrm{GL}_{n_1}(F)$ is a copy of the top-left one.

When dealing with a tempered representation that is not discrete, a centered segment may appear with multiplicity greater than one. Naturally, half of these will be assigned to the upper $\mathrm{GL}_{n_1}(F)$, the other half to the lower $\mathrm{GL}_{n_1}(F)$, and if there is one more left, it will remain in the middle block $G_{n_0}$. This needs to be taken into account when ordering the segments: some centered segments will be assigned to the upper $\mathrm{GL}_{n_1}(F)$, others to the lower $\mathrm{GL}_{n_1}(F)$, and some may remain in $G_{n_0}$. To formalize this, we introduce a set $\USymm_\rho^{\epsilon}(G)$, which distinguishes centered segments more carefully. To this end, we formally enrich segments with a label $\clubsuit \in \{\ge 0, =0, \le 0\}$ that encodes their position.

\bigskip

Let $\USeg$ be the set of labeled pairs $(\Delta, \clubsuit)$, where $\Delta \in \Seg$ and $\clubsuit$ satisfies the following conditions:
\begin{itemize}
\item If $c(\Delta) > 0$, then $\clubsuit$ is equal to $\ge 0$;
\item If $c(\Delta) < 0$, then $\clubsuit$ is equal to $\le 0$.
\end{itemize}

In the case $c(\Delta) = 0$, then $\clubsuit$ can be any of the three values. So only centered segments ($c(\Delta) = 0$) carry a nontrivial choice of label; in all other cases the label is determined uniquely and may be omitted. In those cases, we will simply write $\Delta$ instead of $(\Delta, \clubsuit)$. For centered segments, we usually indicate the label by a superscript $\Delta^{\clubsuit}$, \textit{e.g.}, $[-a, a]_{\rho_u}^{\ge 0}$, $[-a, a]_{\rho_u}^{=0}$, or $[-a, a]_{\rho_u}^{\le 0}$.

One can define the contragredient of an element of $\USeg$ in the following way. For $\Delta^{\clubsuit} \in \USeg$, we define $(\Delta^{\clubsuit})^{\vee} \in \USeg$, by
\begin{enumerate}
  \item If $c(\Delta) > 0$, then $(\Delta^{\ge 0})^{\vee} = (\Delta^{\vee})^{\le 0}$.
  \item If $c(\Delta) < 0$, then $(\Delta^{\le 0})^{\vee} = (\Delta^{\vee})^{\ge 0}$.
  \item If $c(\Delta) = 0$ and $\clubsuit$ is $\le 0$ then $(\Delta^{\le 0})^{\vee} = (\Delta^{\vee})^{\ge 0}$.
  \item If $c(\Delta) = 0$ and $\clubsuit$ is $=0$ or $\ge 0$, then $(\Delta^{\clubsuit})^{\vee} = (\Delta^{\vee})^{\clubsuit}$.
\end{enumerate}

There is also a natural involution $\iota$ on $\USeg$ defined by $\iota(\Delta, \clubsuit) = (\Delta^\vee, \iota(\clubsuit))$ where $\iota$ exchanges $\ge 0$ and $\le 0$ and fixes ${=0}$. The contragredient and the involution naturally extend to multisets. Let $\USymm$ denote the multisets in $\USeg$ that are symmetric under the involution, i.e., those satisfying $\iota(\m) = \m$.

\bigskip

There is a natural surjection
\[
  p : \USymm \twoheadrightarrow \Symm
\]
which forgets the labels. This projection has a section
\[
  s : \Symm \to \USymm
\]
where $s(\m)$ equals
\[
\sum_{\Delta \in \m, c(\Delta) \neq 0} \Delta + \sum_{\Delta \in \m, c(\Delta) = 0}  \left\lfloor \frac{m_{\m}(\Delta)}{2} \right\rfloor \Delta^{\le 0}+ \left\lfloor \frac{m_{\m}(\Delta)}{2} \right\rfloor \Delta^{\ge 0}+  \left(m_{\m}(\Delta) -2 \left\lfloor \frac{m_{\m}(\Delta)}{2} \right\rfloor\right) \Delta^{=0}.
\]

Let $\Delta_1,\Delta_2 \in \Symm$. We define an order relation $\prec$ on the segments of $\USeg$ supported in $\Z_\rho$ in the following way:
\begin{itemize}
  \item $(\Delta_1,\le 0) \prec (\Delta_2,= 0)$.
  \item $(\Delta_1,= 0) \prec (\Delta_2,\ge 0)$.
  \item If $\clubsuit$ is $\ge 0$ or $\le 0$, then $\Delta_1^{\clubsuit} \preceq \Delta_2^{\clubsuit}$ if and only if $\Delta_1 \le \Delta_2$.
  \item If $\clubsuit$ is $= 0$, then $\Delta_1^{\clubsuit} \preceq \Delta_2^{\clubsuit}$ if and only if $e(\Delta_1) \le e(\Delta_2)$.
\end{itemize}
The transitive closure of these relations defines an order on $\USeg$.

Finally, we add signs to centered segments. Let $\USymm^{\varepsilon}(G)$ be the set of pairs $(\m,\varepsilon)$
with $\m \in \USymm$ and $\varepsilon: \{ \Delta \in \m, c(\Delta)=0 \}  \to \{-1,1\}$. The order $\prec$ on $\USymm$ extends naturally to an order on $\USymm^{\varepsilon}$. The maps $p$ and $s$ give maps (fixing $\varepsilon$) $p : \USymm^{\varepsilon}(G) \to \Symm^{\varepsilon}(G)$ and $s : \Symm^{\varepsilon}(G) \to \USymm^{\varepsilon}(G)$.

\subsubsection{The algorithm}
\label{sec:algogood}
Now that we have defined $\USymm$ and $\preceq$ we can describe $\AD_\rho$.

\bigskip

Let $(\m,\varepsilon) \in \Symm_\rho^{\varepsilon}(G)$ and set $(y,\varepsilon)=s(\m,\varepsilon) \in \USymm^{\varepsilon}(G)$.

Let $e_{\max}$ be the biggest coefficient of the segments of $\m$ (hence the biggest end). The first step of the algorithm is to define a sequence $\Delta_{1} \succeq \cdots \succeq \Delta_{l}$ of segments in $y$. The segment $\Delta_{1}$ is the biggest (for $\preceq$) segment of $y$ such that $e(\Delta_1) = e_{\max}$. We define inductively the other segments. Let $j \ge 1$ and assume that $\Delta_j$ is defined. If $\rho$ is of the same type as $G$, and $\Delta_{j}=[0,0]_{\rho_u}^{\ge 0}$ or $\Delta_{j}=[0,0]_{\rho_u}^{=0}$, then $j = l$ (that is, we stop the process). If $\rho$ is not of the same type as $G$, and $\Delta_{j}=[1/2,1/2]_{\rho_u}$; or $\Delta_{j}=[-1/2,1/2]_{\rho_u}^{\ge 0}$ or $\Delta_{j}=[-1/2,1/2]_{\rho_u}^{=0}$, and $\varepsilon([-1/2,1/2]_{\rho_u})=-1$, then $j = l$. Otherwise, $\Delta_{j+1}$ is (if it exists) the biggest segment of $y$ such that:
\begin{itemize}
    \item $\Delta_{j+1} \preceq \Delta_j$;
    \item $e(\Delta_{j+1}) = e(\Delta_{j})-1$;
    \item if $c(\Delta_{j+1})= c(\Delta_{j})=0$ then $\varepsilon(\Delta_{j+1})=-\varepsilon(\Delta_{j})$.
\end{itemize}
If such a $\Delta_{j+1}$ does not exist then $j=l$. Again, we call the sequence $\Delta_1,\cdots,\Delta_l$ the \emph{the initial sequence in the algorithm}.

\bigskip

From this sequence, we will define $(\m_1,\varepsilon_1) \in \Symm_\rho^{\varepsilon}(G)$ and $(\m^{\#},\varepsilon^{\#}) \in \Symm_\rho^{\varepsilon}(G)$ so that we can set $\AD_\rho(\m,\varepsilon)=(\m_1,\varepsilon_1) + \AD_\rho (\m^{\#},\varepsilon^{\#})$.

We start by defining a sign $\varepsilon_0 \in \{-1,1\}$. This sign will determine when $\m_1$ is a centered segment.

\begin{defi}\label{def:epsil}
  We define $\varepsilon_0 \in \{-1,1\}$ by $\varepsilon_0 := -1$ if one of the following conditions is satisfied:
  \begin{itemize}
    \item $\rho_u$ is of the same type as $G$ and $\Delta_{l}=[0,0]_{\rho_u}^{\ge 0}$ or $\Delta_{l}=[0,0]_{\rho_u}^{=0}$;
    \item $\rho_u$ is not of the same type as $G$ and $\Delta_{l}=[1/2,1/2]_{\rho_u}$; or $\Delta_{l}=[-1/2,1/2]_{\rho_u}^{\ge 0}$ or $\Delta_{l}=[-1/2,1/2]_{\rho_u}^{=0}$ with $\varepsilon([-1/2,1/2]_{\rho_u})=-1$.
  \end{itemize}
  Otherwise, $\varepsilon_0 := 1$.
\end{defi}

\bigskip

\textbf{The pair $(\m_1,\varepsilon_1)$:}

\begin{enumerate}
  \item If $\varepsilon_0 = 1$. Then
          \[
              \m_1 := [e(\Delta_{l}),e(\Delta_{1})]_{\rho_u} + [-e(\Delta_{1}),-e(\Delta_{l})]_{\rho_u}
          \]
  \begin{rem}
  We will show in Lemma \ref{lem:ADnotcentered} that $e(\Delta_{1}) + e(\Delta_{l}) \neq 0$ and thus the segments in $\m_1$ are not centered segments.
\end{rem}
    \item If $\varepsilon_0 = -1$. Then
          \[
              \m_1 := [-e(\Delta_{1}),e(\Delta_{1})]_{\rho_u}.
          \]

          Since $\m_1$ is centered, we need to define its sign. Let $n_0$ be the number of centered segments in $\m$, that is $n_0 = \card \{\Delta \in \m, c(\Delta)=0\}$.

          \begin{itemize}
            \item If $\rho_u$ is of the same type as $G$ then
            \[
              \varepsilon_1(\m_1):=(-1)^{n_0 + 1} \varepsilon([0,0]_{\rho_u}).
          \]
          \item If $\rho_u$ is not of the same type as $G$ then
            \[
              \varepsilon_1(\m_1):=(-1)^{n_0}.
          \]
          \end{itemize}
\end{enumerate}

\textbf{The pair $(\m^{\#},\varepsilon^{\#})$:}

To construct $\m^{\#}$ we will remove the end of the segments $\Delta_1,\cdots,\Delta_l$ of $\m$ and the beginning of the segments $\Delta^{\vee}_1,\cdots,\Delta^{\vee}_l$. As there can be multiplicities in $y$, we need to be precise on which segments we modify. Let us write $y = \Lambda_1 + \cdots + \Lambda_k$, with $\Lambda_1 \succeq \cdots \succeq \Lambda_k$.

From $\Delta_1,\cdots,\Delta_l$ we construct two sequences $i_1,\cdots,i_l$ and $i'_1,\cdots,i'_l$. Let $1 \le j \le l$ and define 
\[
  i_j := \min \{ i \in \{1,…,k\}, \Lambda_i= \Delta_j \}
\]
and
\[
  i'_j := \min \{ i \in \{1,…,k\}, \Lambda_i= \Delta^{\vee}_j \}.
\]

We define $\m^{\#} = \Lambda^{\#}_1 + \Lambda^{\#}_2 + \cdots + \Lambda^{\#}_k$ (with the $\Lambda^{\#}_i$ possibly empty) by
\[
    \Lambda^{\#}_i = \left\{
    \begin{array}{ll}
        p(\Lambda_i)           & \text{if } i \notin \{ i_1, \cdots,i_l \} \text{ and } i \notin \{ i'_1,…,i'_l \} \\
        p(\Lambda_i)^{-}       & \text{if } i \in \{ i_1,…,i_l \} \text{ and } i \notin \{ i'_1,…,i'_l \}    \\
        {}^{-}p(\Lambda_i)     & \text{if } i \notin \{ i_1,…,i_l \} \text{ and } i \in \{ i'_1,…,i'_l \}    \\
        {}^{-}p(\Lambda_i)^{-} & \text{if } i \in \{ i_1,…,i_l \} \text{ and } i \in \{ i'_1,…,i'_l \}
    \end{array}
    \right.
\]

We are left to define the signs of the centered segments of $\m^{\#}$. Let $\Lambda^{\#}_i$ be a centered segment of $\m^{\#}$ supported in $\Z_\rho$.
\begin{itemize}
  \item If there exists $1 \le j \le l$ such that $\Lambda^{\#}_i=\Lambda^{\#}_{i_j}$ and $c(\Delta_j)=0$; then $\varepsilon^\#(\Lambda^{\#}_i)=\varepsilon_0 * \varepsilon(\Delta_j)$.
  \item If there exists $1 \le j \le l$ such that $\Lambda^{\#}_i=\Lambda^{\#}_{i_j}$, $c(\Delta_j)=1/2$ and $\Lambda^{\#}_i \notin \m$; then $\varepsilon^\#(\Lambda^{\#}_i)=\varepsilon_0$.
  \item If there exists $1 \le j \le l$ such that $\Lambda^{\#}_i=\Lambda^{\#}_{i_j}$, $c(\Delta_j)=1/2$ and $\Lambda^{\#}_i \in \m$; then $\varepsilon^\#(\Lambda^{\#}_i)=\varepsilon_0 * (-1) * \varepsilon(\Lambda^{\#}_i)$.
  \item Otherwise, $\varepsilon^\#(\Lambda^{\#}_i)=\varepsilon_0 * \varepsilon(\Lambda^{\#}_i)$.
\end{itemize}

It is often convenient to see $\m^{\#}$ as a modification of $\m$ where some segments $\Delta\in \m$ have been replaced by ${}^{-}\Delta$, $\Delta^{-}$ or ${}^{-}\Delta^{-}$. To describe these modifications, when $\Lambda_i^{\#} \neq p(\Lambda_i)$, we will say that the algorithm \emph{suppresses} $\Lambda_i$ and \emph{creates} $\Lambda_i^{\#}$.

\begin{defi}
  When $\rho$ is good, we define $\AD_\rho : \Symm_\rho^\varepsilon(G) \to \Symm_\rho^\varepsilon(G)$ inductively by
  \[
      \AD_\rho(\m,\varepsilon)=(\m_1,\varepsilon_1)+\AD_\rho(\m^{\#},\varepsilon^{\#})
  \]
\end{defi}

\begin{rem}\label{rem:defAD}
  Unlike the ugly and bad case, it is not clear here that $\AD_\rho$ is well defined. At this stage, it is a map $\Symm_\rho^\varepsilon(G) \to \Symm_\rho^\varepsilon$. To ensure that its image actually lies in $\Symm_\rho^\varepsilon(G)$, we need to verify that two centered segments which are equal have the same sign. This is established in the next section in Proposition~\ref{prop:ADdef}.
\end{rem}

\begin{rem}
By definition, $\AD_{\rho}$ preserves the support, and hence the determinant. We will also show that it preserves the product of the signs in Proposition \ref{prop:ADsigns}.
\end{rem}

\begin{ex}
  Let $\rho$ be of good parity and unitary. Let $\pi:=L(\n;\pi(\phi,\eta))$ with $\n=[-1,0]_{\rho}$ and $\pi(\phi,\eta)$ trivial. We associate to $\pi$ the symmetric multisegment $\m \in \Symm$ defined by $\m=[-1,0]_\rho+[0,1]_\rho$. 
  
  In the diagram below, we represent the multisegment $\m$ ordered by $\le$. 
  We colored in red the segments with label $\ge 0$ and in blue those with $\le 0$. The portions highlighted in green mark the parts of $\m$ that are extracted to form the segment $\m_1$; the remaining parts will constitute the multisegment $\m^{\#}$ after this first step.
  
  
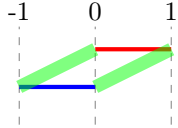
\begin{figure}[h!]
    \centering
    \begin{tikzpicture}[scale=1]
  
      \begin{scope}[shift={(0,0)}]
        \foreach \x in {-1,0,1} {
            \draw[dashed, gray, very thin] (\x,-1.5) -- (\x,-0.2);
            \node[black] at (\x,0) {\x};
        }
        \draw[red, ultra thick] (0,-0.5) -- (1,-0.5);
        \draw[blue, ultra thick] (-1,-1) -- (0,-1);
        \draw[green, line width=5pt, opacity=0.5] (1,-0.5) -- (0,-1);
        \draw[green, line width=5pt, opacity=0.5] (0,-0.5) -- (-1,-1);
      \end{scope}
    \end{tikzpicture}
  \end{figure}

  The algorithm gives (in one step) that $\AD_\rho(\m)=[-1,0]_\rho + [0,1]_\rho$. Thus $\hat{\pi}=\pi$ (remark the difference with the bad parity case in Example \ref{ex:badpar01}).
  \end{ex}

\begin{ex}
      With the notation and coloring as above, we consider now the case $\pi:=L(\n;\pi(\phi,\eta))$ with $\n=[-2,-2]_\rho$, $\phi=\rho \boxtimes S_1 + \rho \boxtimes S_1 + \rho \boxtimes S_3$, $\eta(\rho \boxtimes S_1)=-1$ and $\eta(\rho \boxtimes S_3)=1$. The symmetric multisegment $(\m,\varepsilon) \in \Symm^\varepsilon$ is $\m=[-2,-2]_\rho+[0,0]_\rho+[0,0]_\rho+[-1,1]_\rho+[2,2]_\rho$ with $\varepsilon([0,0]_\rho)=-1$ and $\varepsilon([-1,1]_\rho)=1$.

      
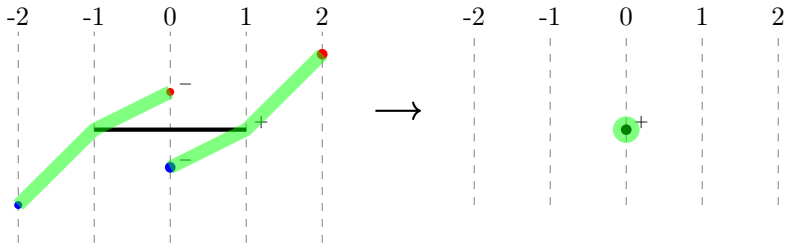
\begin{figure}[h!]
        \centering
        \begin{tikzpicture}[scale=1]
      
          \begin{scope}[shift={(0,0)}]
            \foreach \x in {-2,-1,0,1,2} {
                \draw[dashed, gray, very thin] (\x,-3) -- (\x,-0.2);
                \node[black] at (\x,0) {\x};
            }
            \fill[red, ultra thick] (2,-0.5) circle (2pt);
            \fill[red] (0,-1) circle (1.5pt);
            \node at (0.2,-0.9) {\tiny $-$};
            \draw[black, ultra thick] (-1,-1.5) -- (1,-1.5);
            \node at (1.2,-1.4) {\tiny $+$};

            \fill[blue] (0,-2) circle (2pt);
            \node at (0.2,-1.9) {\tiny $-$};
            \fill[blue] (-2,-2.5) circle (1.5pt);
            \draw[green, line width=5pt, opacity=0.5] (2,-0.5) -- (1,-1.5)-- (0,-2);
            \draw[green, line width=5pt, opacity=0.5] (0,-1) -- (-1,-1.5)-- (-2,-2.5);
          \end{scope}
  
          \begin{scope}[shift={(6,0)}]
            \foreach \x in {-2,-1,0,1,2} {
                \draw[dashed, gray, very thin] (\x,-2.5) -- (\x,-0.2);
                \node[black] at (\x,0) {\x};
            }
            \fill[black] (0,-1.5) circle (2pt);
            \node at (0.2,-1.4) {\tiny $+$};
            \fill[green, opacity=0.5] (0,-1.5) circle[radius=5pt];
  
          \end{scope}
  
        \draw[->, thick] (2.7,-1.25) -- (3.3,-1.25);
        \end{tikzpicture}
      \end{figure}

      The algorithm gives (in two steps) that $\AD_\rho(\m,\varepsilon)=(\m',\varepsilon')$ with $\m'=([-2,0]_\rho + [0,2]_\rho) +[0,0]_\rho$ and $\varepsilon'([0,0]_\rho)=1$. Thus $\hat{\pi}=L([-2,0]_\rho;\pi(\rho \boxtimes S_1,\eta'))$ with $\eta'(\rho \boxtimes S_1)=1$.
      \end{ex}

  \begin{ex}
    With the notation and coloring as above, we consider now the case $(\m,\varepsilon) \in \Symm^\varepsilon$ with $\m=[-2,-2]_\rho+[0,0]_\rho+[-1,1]_\rho+[2,2]_\rho$, $\varepsilon([0,0]_\rho)=-1$ and $\varepsilon([-1,1]_\rho)=1$.

    
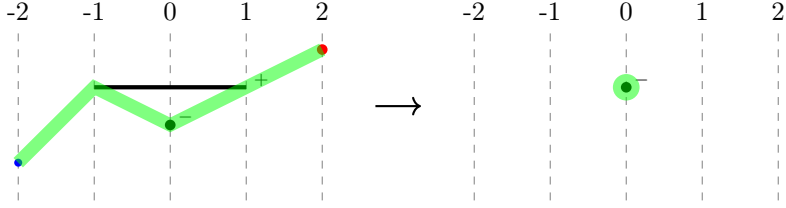
\begin{figure}[h!]
      \centering
      \begin{tikzpicture}[scale=1]
    
        \begin{scope}[shift={(0,0)}]
          \foreach \x in {-2,-1,0,1,2} {
              \draw[dashed, gray, very thin] (\x,-2.5) -- (\x,-0.2);
              \node[black] at (\x,0) {\x};
          }
          \fill[red, ultra thick] (2,-0.5) circle (2pt);
          \draw[black, ultra thick] (-1,-1) -- (1,-1);
          \node at (1.2,-0.9) {\tiny $+$};
          \fill[black, ultra thick] (0,-1.5) circle (2pt);
          \node at (0.2,-1.4) {\tiny $-$};
          \fill[blue] (-2,-2) circle (1.5pt);
          \draw[green, line width=5pt, opacity=0.5] (2,-0.5) -- (1,-1)-- (0,-1.5)-- (-1,-1)-- (-2,-2);
        \end{scope}

        \begin{scope}[shift={(6,0)}]
          \foreach \x in {-2,-1,0,1,2} {
              \draw[dashed, gray, very thin] (\x,-2.5) -- (\x,-0.2);
              \node[black] at (\x,0) {\x};
          }
          \fill[black] (0,-1) circle (2pt);
          \node at (0.2,-0.9) {\tiny $-$};
          \fill[green, opacity=0.5] (0,-1) circle[radius=5pt];

        \end{scope}

      \draw[->, thick] (2.7,-1.25) -- (3.3,-1.25);
      \end{tikzpicture}
    \end{figure}

    The algorithm gives (in two steps) that $\AD_\rho(\m,\varepsilon)=(\m',\varepsilon')$ with $\m'=[-2,2]_\rho + [0,0]_\rho$, $\varepsilon'([-2,2]_\rho)=1$ and $\varepsilon'([0,0]_\rho)=-1$. In this example, we see that the algorithm also works when the product of the signs is $-1$, and that it preserves this product (the general proof of this property is given in Section~\ref{sec:signpreserving}).
    \end{ex}

\begin{rem}
  For $\GL_n(F)$, Knight and Zelevinsky observed in \cite{KZ} that for a given multisegment $\m$, the number of segments in $\m^t$ that contain a given segment $[i,j]$ is equal to the capacity of the graph whose vertices are pairs $(\Delta, x)$ where $\Delta \in \m$ and $x \in \Delta \cap [i,j]$ and the edges connect $(\Delta, x)$ and $(\Delta', x+1)$ if $\Delta$ precedes $\Delta'$.

  A naive transposition of this result in our setting doesn't work as shown by the following example. Let $\rho$ be of good parity and consider $\pi:=L(\n;\pi(\phi,\eta))$ with $\n=[-3,-3]_\rho$, $\phi=\rho \boxtimes S_3 +\rho \boxtimes S_3 +\rho \boxtimes S_3 +\rho \boxtimes S_5 +\rho \boxtimes S_5 +\rho \boxtimes S_5 +\rho \boxtimes S_7 +\rho \boxtimes S_7 $, $\eta(\rho \boxtimes S_3)=1$, $\eta(\rho \boxtimes S_5)=-1$ and $\eta(\rho \boxtimes S_7)=1$. The corresponding labelled symmetric multisegment is $(y,\varepsilon) \in \USymm^\varepsilon(G)$ given by
  \begin{align*}
    y&=[3, 3]_\rho + [-1, 1]^{\ge 0}_\rho + [-2, 2]^{\ge 0}_\rho + [-3, 3]^{\ge 0}_\rho\\
&+ [-2, 2]^{= 0}_\rho + [-1, 1]^{= 0}_\rho\\
&+ [-1, 1]^{\le 0}_\rho + [-2, 2,]^{\le 0}_\rho + [-3, -3]_\rho + [-3, 3]^{\le 0}_\rho
  \end{align*}
with $\varepsilon([-1, 1]_\rho) = 1$, $\varepsilon([-2, 2]_\rho) = -1$ and $\varepsilon([-3, 3]_\rho) = 1$.

Applying $\AD$, we get that $\hat{\pi}=L(\n';\pi(\phi',\eta'))$ with $\n'=[-3, -1]_\rho + [-3, -2]_\rho + [-3, -3]_\rho + [-2, -2]_\rho + [-2, -2]_\rho + [-1, -1]_\rho + [-1, -1]_\rho + [-1, -1]_\rho + [-1, -1]_\rho + [-1, -1]_\rho$, $\phi'=\rho \boxtimes S_1 + \rho \boxtimes S_1 + \rho \boxtimes S_1+ \rho \boxtimes S_1+ \rho \boxtimes S_1+ \rho \boxtimes S_1+ \rho \boxtimes S_3+ \rho \boxtimes S_5$, $\eta'(\rho \boxtimes S_1) = -1$, $\eta'(\rho \boxtimes S_3) = 1$ and $\eta'(\rho \boxtimes S_5) = -1$.

Now let us consider the segment $[i,j]=[-3,-1]_\rho$. Its multiplicity in the dual is 1. However, the capacity of the graph is 2 as there are two paths $([-1,1]^{=0}_\rho,1) \to ([-2,2]_\rho^{\ge 0},2) \to ([3,3]_\rho,3)$ and $([-1,1]_\rho^{\le 0},1) \to ([-2,2]_\rho^{=0},2) \to ([-3,3]_\rho^{\ge 0},3)$.
\end{rem}

\subsection{Main Theorem}

The above constructions define for all $\rho \in \Cusp^{\GL}$ a map $\AD_\rho : \Symm_\rho^\varepsilon(G) \to \Symm_\rho^\varepsilon(G)$. As explained at the beginning of the section, we get a map $\AD : \Symm^\varepsilon(G) \to \Symm^\varepsilon(G)$ by $\AD = \oplus_{\rho \in \Cusp^\GL/\sim'} \AD_{\rho}$; and a map $\ADd : \Data(G) \to \Data(G)$ defined by $\ADd = \trans^{-1} \circ \AD \circ \trans$. The main theorem of this paper (proved in the following sections) is:

\begin{thm}\label{thm:MAIN}
  Let $\pi \in \Irr^G$ with Langlands data $(\m;\phi,\eta)$. Then
  \[
    \hat{\pi} \simeq L(\ADd(\m;\phi,\eta)).
  \]
\end{thm}

\section{Well-definedness of the algorithm in the good parity case}
\label{sec:ADdef}

In this section, we verify that, in the good parity case, the algorithm is well-defined. Let $\rho \in \Cusp^{\GL}$ be of good parity and $\rho_u$ its unitarization. As explained in Remark~\ref{rem:defAD}, at this stage $\AD_\rho$ is a map $\AD_\rho : \Symm_\rho^\varepsilon(G) \to \Symm_\rho^\varepsilon$. To ensure that the image lies in $\Symm_\rho^\varepsilon(G)$ we need to verify that any two equal centered segments are assigned the same sign. Establishing this compatibility is the main goal of this section.

\bigskip

Let $(\m,\varepsilon) \in \Symm^{\varepsilon}(G)$. Let $\Delta_1,\cdots,\Delta_l$ be the initial sequence in the algorithm for $(\m,\varepsilon)$ and $\Delta'_1,\cdots,\Delta'_{l'}$ be the initial sequence in the algorithm for $(\m^{\#},\varepsilon^{\#})$. Let $\varepsilon_0$ be the sign in $(\m,\varepsilon)$ and $\varepsilon'_0$ the sign in $(\m^{\#},\varepsilon^{\#})$ (see Definition \ref{def:epsil}).

\begin{lem}
  \label{lem:paritycent}
  Suppose that there exists $j \ge 1$ such that $j<l$, $c(\Delta_{j+1})=0$ and $c(\Delta_{j})=1/2$. Then $m_{\m^{\#}}(p(\Delta_{j+1}))$ is even.
\end{lem}

\begin{proof}
  Let $a \in \frac{1}{2}\Z$ such that $p(\Delta_{j+1})=[-a,a]_{\rho_u}$ and $\Delta_{j}=[-a,a+1]_{\rho_u}$. If $m_{\m}([-a,a]_{\rho_u})$ is odd, then $\Delta_{j+1} = [-a,a]_{\rho_u}^{=0}$. The algorithm suppresses one $[-a,a]_{\rho_u}$ (namely, $\Lambda_{i_{j+1}}$) and creates two new ones from $\Lambda_{i_j}$ and $\Lambda^{\vee}_{i_j}$. Thus $m_{\m^{\#}}([-a,a]_{\rho_u})$ is even. If $m_{\m}([-a,a]_{\rho_u})$ is even, then $\Delta_{j+1} = [-a,a]_{\rho_u}^{\le 0}$. The algorithm suppresses two $[-a,a]_{\rho_u}$ (namely, $\Lambda_{i_{j+1}}$ and $\Lambda_{i'_{j+1}}$) and creates two new ones from $\Lambda_{i_j}$ and $\Lambda^{\vee}_{i_j}$. Thus $m_{\m^{\#}}([-a,a]_{\rho_u})$ is also even. 
\end{proof}

\begin{lem}
  \label{lem:precalgoj}
  Let us assume that $e(\Delta_1)=e(\Delta'_1)$. Suppose that there exists $j \ge 1$ such that $j<l$, $j<l'$, $\Delta_{j+1}' \prec \Delta_j$, $c(\Delta_{j})\neq 0$ and $c(\Delta_{j+1})=0$. Then $\Delta'_{j+1} \preceq \Delta_{j+1}$.
\end{lem}

\begin{proof}
  If $b_{\rho_u}(\Delta'_{j+1}) < b_{\rho_u}(\Delta_{j+1})$ then $c(\Delta'_{j+1}) < 0$ and $\Delta'_{j+1} \preceq \Delta_{j+1}$. We assume now that $b_{\rho_u}(\Delta'_{j+1}) = b_{\rho_u}(\Delta_{j+1})$. Let $a \in \frac{1}{2}\Z$ such that $p(\Delta_{j+1})=p(\Delta'_{j+1})=[-a,a]_{\rho_u}$. Notice that, since $\Delta_{j+1}' \prec \Delta_j$, $c(\Delta_{j})\ge 0$ and thus $c(\Delta_{j})> 0$. If $b_{\rho_u}(\Delta_j) > -a$, then $\Delta_{i_j}^{\#} \neq [-a,a]_{\rho_u}$ and since $p(\Delta'_{j+1})=[-a,a]_{\rho_u}$ we get that $m_{\m}([-a,a]_{\rho_u})>1$. Thus $\Delta_{j+1}=[-a,a]_{\rho_u}^{\ge 0}$ and $\Delta'_{j+1} \preceq \Delta_{j+1}$.
  
  We can now assume that $b_{\rho_u}(\Delta_j) = -a$, that is $\Delta_j=[-a,a+1]_{\rho_u}$. By Lemma \ref{lem:paritycent}, $m_{\m^{\#}}([-a,a]_{\rho_u})$ is even. Then $\Delta'_{j+1} = [-a,a]_{\rho_u}^{\le 0}$ and $\Delta'_{j+1} \preceq \Delta_{j+1}$.
\end{proof}

\begin{lem}
  \label{lem:precalgocent}
  Let us assume that $e(\Delta_1)=e(\Delta'_1)$. Suppose that there exists $j \ge 1$ such that $j\le l$, $j<l'$ and $c(\Delta_{j})= 0$. We also assume that for all $i \le j$, $\Delta'_i \preceq \Delta_i$. Then $c(\Delta'_{j+1})\neq 0$.
\end{lem}

\begin{proof}
  We prove the result by contradiction. Assume that $c(\Delta'_{j+1}) = 0$. We may assume that $j$ is minimal among the indices such that $c(\Delta_{j})= 0$ and $c(\Delta'_{j+1}) = 0$.

  Since $\Delta_{j+1}' \prec \Delta'_{j} \preceq \Delta_j$ we get that $\Delta'_{j}$ is also centered. By minimality of $j$, $j=1$ or $c(\Delta_{j-1}) \neq 0$. Let $a \in \frac{1}{2}\Z$ such that $p(\Delta_j)=[-a-1,a+1]_{\rho_u}$. If $j=1$ or $p(\Delta'_{j}) \neq \Delta_{i_{j-1}}^{\#}$, then $\varepsilon^{\#}([-a-1,a+1]_{\rho_u})=\varepsilon_0 \varepsilon([-a-1,a+1]_{\rho_u})$ and $\varepsilon^{\#}([-a,a]_{\rho_u})=\varepsilon_0 \varepsilon([-a-1,a+1]_{\rho_u})$ contradicting the fact that $\Delta'_{j+1}$ follows $\Delta'_j$ in the algorithm and thus $\varepsilon^{\#}([-a-1,a+1]_{\rho_u})= - \varepsilon^{\#}([-a,a]_{\rho_u})$.

      Thus $j> 1$ and $p(\Delta'_{j}) = \Delta_{i_{j-1}}^{\#}$. Since $c(\Delta_{j-1}) \neq 0$, we get that $\Delta_{j-1} = [-a-1,a+2]_{\rho_u}$. Now $\Delta'_j \prec \Delta_{j-1}$, thus the label of $\Delta'_j$ is either $=0$ or $\le 0$. And $\Delta'_{j+1} \prec \Delta'_j$, so the label of $\Delta'_j$ is $=0$. But by Lemma \ref{lem:paritycent}, $m_{\m^{\#}}([-a-1,a+1]_{\rho_u})$ is even, which contradicts the fact that $\Delta'_j = [-a-1,a+1]_{\rho_u}^{=0}$.
\end{proof}

\begin{lem}
  \label{lem:sizem1}
  Suppose that $e(\Delta_1)=e(\Delta'_1)$. Then
  \begin{enumerate}
      \item If $\varepsilon_0 = 1$ then $l' \le l$;
      \item For all $j \le \min\{l,l'\}$, $\Delta'_j \preceq \Delta_j$.
  \end{enumerate}
\end{lem}

\begin{proof}
We prove by induction on $j \le l'$ the following result: if, for all $k \le \min\{j,l\}$, $\Delta_k$ is not equal to any of  $[0,0]_{\rho_u}^{=0}$, $[0,0]_{\rho_u}^{\ge 0}$, $[1/2,1/2]_{\rho_u}$, $[-1/2,1/2]_{\rho_u}^{\ge 0}$ with $\varepsilon([-1/2,1/2]_{\rho_u})=-1$ or $[-1/2,1/2]_{\rho_u}^{0}$ with $\varepsilon([-1/2,1/2]_{\rho_u})=-1$, then necessarily $l \ge j$ and $\Delta'_j \preceq \Delta_j$.

  For $j=1$, we have $l\ge 1$. Moreover, $e(\Delta'_1)=e_{\max}$, and the segments of $\m^{\#}$ ending in $e_{\max}$ are the $\Lambda_i^{\#}$ with $i \neq i_1$ and $e(\Lambda_i)=e_{\max}$. Thus, there exists $i \neq i_1$ such that $p(\Delta'_1)=\Lambda^{\#}_{i}$ (that is $p(\Lambda_{i})$ or ${}^{-}p(\Lambda_{i})$). Moreover, by definition of $i_1$, we have $i > i_1$ and thus $\Lambda_{i} \prec \Lambda_{i_1}$. If both $\Delta_1$ and $\Delta'_1$ are centered, then necessarily, $m_{\m}(p(\Delta_1))>1$ and therefore the label of $\Delta_1$ is $\ge 0$, which implies that $\Delta'_1 \preceq \Delta_1$. We may assume that one of $\Delta_1$ and $\Delta'_1$ is not centered. In this case, as $e(\Delta'_1)=e(\Delta_1)$, we have that $\Delta'_1 \preceq \Delta_1$ if and only if $b_{\rho_u}(\Delta'_1) \le b_{\rho_u}(\Delta_1)$. Since $\Lambda_{i} \prec \Lambda_{i_1}$, it follows  that $b_{\rho_u}(\Lambda_{i_1}) \le b_{\rho_u}(\Delta_1)$. And $p(\Delta'_1)=p(\Lambda_{i})$ or ${}^{-}p(\Lambda_{i})$ thus $b_{\rho_u}(\Delta'_1) = b_{\rho_u}(\Lambda_{i})$ or $b_{\rho_u}(\Lambda_{i}) + 1$. Hence, the only possible case where $\Delta'_1 \preceq \Delta_1$ might not hold would be when  $b_{\rho_u}(\Delta'_1) = b_{\rho_u}(\Lambda_{i}) + 1$ (that is $p(\Delta'_1)={}^{-}p(\Lambda_{i})$) and $p(\Lambda_{i})=p(\Lambda_{i_1})$. This implies that $i=i'_j$ for some $j$. But from the definition of $i'_j$ we should have $i'_j=i_1$ which contradicts $i\neq i_1$. 

  Now, let us assume the result for $j$ and prove it for $j+1$. We assume that $\Delta_j$ is not equal to any of $[0,0]_{\rho_u}^{=0}$, $[0,0]_{\rho_u}^{\ge 0}$, $[1/2,1/2]_{\rho_u}$, $[-1/2,1/2]_{\rho_u}^{\ge 0}$ with $\varepsilon([-1/2,1/2]_{\rho_u})=-1$ or $[-1/2,1/2]_{\rho_u}^{=0}$ with $\varepsilon([-1/2,1/2]_{\rho_u})=-1$. First, let us show that $l \ge j+1$. We have that $\Delta'_{j+1} \prec \Delta'_{j} \preceq \Delta_j$ (the last inequality follows from the induction hypothesis). There exists an index $i$ such that $p(\Delta_{j+1}') = \Lambda_i^{\#}$, and $p(\Delta_{j+1}')$ is one of the following segments: $p(\Lambda_i)$, $p(\Lambda_i)^{-}$, ${}^{-}p(\Lambda_i)$ or ${}^{-}p(\Lambda_i)^{-}$.
  \begin{itemize}
      \item If $p(\Delta_{j+1}') = p(\Lambda_i)$. Then $\Lambda_i$ is a segment ending in $e_{\max}-j$. First, suppose that $\Delta_{j+1}'$ is not a centered segment. Then $p(\Delta_{j+1}') = p(\Lambda_i)$ implies that $\Delta_{j+1}' = \Lambda_i$. Thus $\Lambda_{i} \prec \Lambda_j$ and $\Lambda_i$ satisfies the condition to be in the initial sequence (but may not be maximal) giving us $l \ge j+1$. Moreover, by maximality of $\Delta_{j+1}$ we get that $\Delta'_{j+1}=\Lambda_i \prec \Delta_{j+1}$. Now suppose that $\Delta_{j+1}'$ is centered. By Lemma \ref{lem:precalgocent} $\Delta_{j}$ is not centered. Since $\Delta'_{j+1} \preceq \Delta_j$ we get that $c(\Delta_j) > 0$. If the label of $\Lambda_i$ is $\le 0$ or $=0$ then $\Lambda_i \prec \Delta_j$. And if $\Lambda_i =[-a,a]_{\rho_u}^{\ge 0}$, then $\Lambda'_i:=[-a,a]_{\rho_u}^{\le 0} \in y$ and $\Lambda'_i \prec \Delta_j$. In both cases, $\Lambda_i$ or $\Lambda'_i$ satisfies the condition to be in the initial sequence and $l \ge j+1$. Let $\Lambda$ be either $\Lambda_i$ or $\Lambda'_i$ such that $\Lambda \prec \Delta_j$. By maximality $\Lambda \prec \Delta_{j+1}$. The segment $\Lambda$ is centered, so $c(\Delta_{j+1}) \ge 0$. If $c(\Delta_{j+1}) > 0$ then $\Delta'_{j+1} \preceq \Delta_{j+1}$. And if $c(\Delta_{j+1}) = 0$, Lemma \ref{lem:precalgoj} gives us that $\Delta'_{j+1} \preceq \Delta_{j+1}$.

      \item If $p(\Delta_{j+1}') = {}^{-}p(\Lambda_i)$. Then $\Lambda_i$ is a segment ending in $e_{\max}-j$. Moreover, as $p(\Delta_{j+1}') = {}^{-}p(\Lambda_i)$, we get that $\Lambda_{i} \prec \Delta_{j+1}'$. Thus $\Lambda_{i} \prec \Delta_{j}$. Here, it is impossible to have both $\Lambda_{i}$ and $\Delta_j$ centered. Indeed, if $p(\Delta_j)=[-a-1,a+1]_{\rho_u}$ and $p(\Lambda_i) = [-a,a]_{\rho_u}$, then $\Delta_{j+1}'=[-a + 1,a]_{\rho_u}$ and  in that case, $\Delta_{j+1}' \prec \Delta_j$ does not hold. Hence, $l \ge j+1$ and $\Lambda_i \preceq \Delta_{j+1}$. Since $\Lambda_i^{\#} = {}^{-}p(\Lambda_i)$, $\Lambda_i \neq \Lambda_{i_{j+1}}$. Thus $b_{\rho_u}(\Lambda_i) < b_{\rho_u}(\Delta_{j+1})$ and $b_{\rho_u}(\Delta_{j+1}') \le b_{\rho_u}(\Delta_{j+1})$. If $b_{\rho_u}(\Delta_{j+1}') < b_{\rho_u}(\Delta_{j+1})$ or $\Delta_{j+1}$ is not centered, then $\Delta_{j+1}' \preceq \Delta_{j+1}$. We assume now that $b_{\rho_u}(\Delta_{j+1}') = b_{\rho_u}(\Delta_{j+1})$ and $\Delta_{j+1}$ is centered. Let $a \in \frac{1}{2}\Z$ such that $p(\Delta_{j+1})=[-a,a]_{\rho_u}$. Hence, $\Lambda_i=[-a-1,a]_{\rho_u}$ and $\Delta_{j}=\Lambda^{\vee}_i=[-a,a+1]_{\rho_u}$. Since, $\Delta_{j+1} \preceq \Delta_j$, we get that $\Delta_{j+1}=[-a,a]^{=0}_{\rho_u}$ or $\Delta_{j+1}=[-a,a]^{\le 0}_{\rho_u}$; and similarly for $\Delta'_{j+1}$. But if  $\Delta_{j+1}=[-a,a]^{\le 0}_{\rho_u}$ then $m_{\m^{\#}}([-a,a]_{\rho_u})$ is even and $\Delta_{j+1}'=[-a,a]^{\le 0}_{\rho_u}$. In all cases, $\Delta_{j+1}' \preceq \Delta_{j+1}$.

      \item If $p(\Delta_{j+1}') = p(\Lambda_i)^{-}$. Hence, $i = i_j$ and $\Lambda_i = \Delta_j$. Moreover, as $\Lambda_i^{\#}=p(\Lambda_i)^{-}$ we cannot have that $\Lambda_i$ is a centered segment with label $\ge 0$ or $=0$.
      
      We first show that $\Delta_{j+1}'$ is a centered segment. From $p(\Delta_{j+1}') = p(\Delta_j)^{-}$ we get that $c(\Delta_{j+1}')=c(\Delta_j)+1$. Hence, if $c(\Delta_{j+1}')>0$, $\Delta_{j+1}'=\Delta_j^{-}$ and $\Delta_j \prec \Delta_{j+1}'$. Similarly, if $c(\Delta_{j}) < 0$ we get a contradiction. If $c(\Delta_{j})=0$, then $c(\Delta_{j+1}')=1/2$ and $\Delta_{j} \prec \Delta_{j+1}'$. Thus $c(\Delta_{j+1}')=0$. Let $a\in (1/2)\Z$ such that $p(\Delta_{j+1}')=[-a,a]_{\rho_u}$ and $\Lambda_i = \Delta_j = [-a,a+1]$. Since, $\Delta_j = [-a,a+1]_{\rho_u}$, we get that $l \ge j+1$ as $[-a-1,a]_{\rho_u} \prec [-a,a+1]_{\rho_u}$. But $\Delta_{i_j}^{\#}=p(\Delta_j)^{-}$ so $\Delta_{j+1} \neq [-a-1,a]_{\rho_u}$. Hence, $p(\Delta_{j+1})=[-a,a]_{\rho_u}$. As  $\Delta_{j+1}' \prec \Delta_j = [-a,a+1]_{\rho_u}$, we get that $\Delta_{j+1}' = [-a,a]^{=0}_{\rho_u}$ or $[-a,a]^{\le 0}_{\rho_u}$. But, $m_{\m^{\#}}([-a,a]_{\rho_u})$ is even by Lemma \ref{lem:paritycent}, so $\Delta_{j+1}' = [-a,a]_{\rho_u}^{\le 0}$ and we get that $\Delta_{j+1}' \prec \Delta_{j+1}$.

      \item If $p(\Delta_{j+1}') = {}^{-}p(\Lambda_i)^{-}$. Again, $i = i_j$ and $\Lambda_i = \Delta_j$, and if $\Delta \in \USeg$ is a segment such that $c(\Delta) \neq 0$, then $\Delta \prec {}^{-}\Delta^{-}$. Thus $c(\Delta_{j+1}')=c(\Delta_j)=0$. This contradicts Lemma \ref{lem:precalgocent}.
  \end{itemize}
\end{proof}

\begin{lem}
  \label{lem:starconditionAD}
  Suppose that $e(\Delta_1)=e(\Delta'_1)$. If $\varepsilon'_0 = -1$ then $\varepsilon_0 = -1$.
\end{lem}

\begin{proof}
  Suppose that $\varepsilon'_0 = -1$. Then one of the following conditions is satisfied.
\begin{itemize}
  \item If $\rho$ is of the same type as $G$ and $\Delta'_{l'}=[0,0]_{\rho_u}^{\ge 0}$ or $\Delta'_{l'}=[0,0]_{\rho_u}^{=0}$. Then, by Lemma \ref{lem:sizem1}, $l' \le l$ and $\Delta'_{l'} \preceq \Delta_{l'}$. Since $e(\Delta_{l'})=0$, we get that $\Delta_{l'}=[0,0]_{\rho_u}^{\ge 0}$ or $\Delta_{l'}=[0,0]_{\rho_u}^{=0}$. Hence $l=l'$ and $\varepsilon_0 = -1$.
  \item If $\rho$ is not of the same type as $G$ and $\Delta'_{l'}=[1/2,1/2]_{\rho_u}$. By Lemma \ref{lem:sizem1}, we get that $\Delta'_l \preceq \Delta_l$ and as $e(\Delta_l)=1/2$, we get that $\Delta_l=[1/2,1/2]_{\rho_u}$. Hence $\varepsilon_0 = -1$.
  \item If $\rho$ is not of the same type as $G$ and $\Delta'_{l'}=[-1/2,1/2]_{\rho_u}^{\ge 0}$ or $\Delta'_{l'}=[-1/2,1/2]_{\rho_u}^{=0}$ and $\varepsilon^{\#}([-1/2,1/2]_{\rho_u})=-1$. By Lemma \ref{lem:sizem1}, $\Delta'_{l'} \preceq \Delta_{l'}$, so $\Delta_{l'}=[1/2,1/2]_{\rho_u}$ or $\Delta_{l'}=[-1/2,1/2]_{\rho_u}^{\ge 0}$ or $\Delta_{l'}=[-1/2,1/2]_{\rho_u}^{=0}$. If $\Delta_{l'}=[1/2,1/2]_{\rho_u}$ then $\varepsilon_0 = -1$. We now show the other possibilities lead to contradictions. Suppose that $\Delta_{l'}=[-1/2,1/2]_{\rho_u}^{\ge 0}$ or $\Delta_{l'}=[-1/2,1/2]_{\rho_u}^{=0}$. If $l'=1$, then $-1=\varepsilon^{\#}([-1/2,1/2]_{\rho_u})=\varepsilon_0\varepsilon([-1/2,1/2]_{\rho_u})$ which is impossible. We assume now that $l' > 1$. By Lemma \ref{lem:precalgocent}, $c(\Delta_{l'-1}) \neq 0$. By the formula for $\varepsilon^{\#}$, if $c(\Delta_{l'-1}) \neq 1$, then $-1=\varepsilon^{\#}([-1/2,1/2]_{\rho_u})=\varepsilon_0\varepsilon([-1/2,1/2]_{\rho_u})$ which is impossible. Thus $c(\Delta_{l'-1}) = 1/2$ and $\Delta_{l'-1}=[-1/2,3/2]_{\rho_u}$. By Lemma \ref{lem:paritycent}, $m_{\m^{\#}}([-1/2,1/2]_{\rho_u})$ is even, and $\Delta'_{l'}=[-1/2,1/2]_{\rho_u}^{\ge 0}$. From Lemma \ref{lem:sizem1}, we get that $\Delta'_{l'-1}$ satisfies $\Delta'_{l'}=[-1/2,1/2]_{\rho_u}^{\ge 0} \prec \Delta'_{l'-1} \prec  \Delta_{l'-1}=[-1/2,3/2]_{\rho_u}$ which is impossible.
\end{itemize}
\end{proof}

The following proposition generalizes \cite[II.2.2.]{MW2} to our setting.
\begin{prop}
  \label{prop:lengthmax}
  Let $\eta_1 \in \m_1$ be the segment ending in $e_{\max}$. Then $\eta_1$ is the longest among the segments $\Delta$ of $\AD_{\rho}(\m,\varepsilon)$ such that $e(\Delta)=e_{\max}$.
\end{prop}

\begin{proof}
We prove the result by induction. By definition $\AD_{\rho}(\m,\varepsilon)=(\m_1,\varepsilon_1)+\AD_{\rho}(\m^{\#},\varepsilon^{\#})$. If $\m^{\#}=0$ then we are done. If not, let us write $\AD_{\rho}(\m^{\#},\varepsilon^{\#})=(\m'_1,\varepsilon'_1)+\AD_{\rho}(\m'^{\#},\varepsilon'^{\#})$. Let $e_{\max}'$ be the maximum of the coefficients of $\m^{\#}$. If $e_{\max}'<e_{\max}$ then no segments of $\AD(\m^{\#},\varepsilon^{\#})$ contain $e_{\max}$ and we are done. Hence we can assume that $e_{\max}'=e_{\max}$. Let $\eta'_1$ be the segment of $\m'_1$ ending in $e_{\max}$. By the induction hypothesis, every segment of $\AD_{\rho}(\m^{\#},\varepsilon^{\#})$ ending in $e_{\max}$ has length smaller than $l(\eta'_1)$. We are left to prove that $l(\eta'_1) \le l(\eta_1)$. Let $\Delta_1,\cdots,\Delta_l$ be the initial sequence in the algorithm for $(\m,\varepsilon)$ and $\Delta'_1,\cdots,\Delta'_{l'}$ be the initial sequence in the algorithm for $(\m^{\#},\varepsilon^{\#})$. If $\varepsilon_0 = -1$ then $l(\eta_1)=2l-1$, and if not $l(\eta_1)=l$. Similarly, $l(\eta'_1)=2l'-1$ if $\varepsilon'_0 = -1$, and $l(\eta'_1)=l'$ otherwise. Thus $l(\eta'_1) \le l(\eta_1)$ follows from Lemma \ref{lem:sizem1} and Lemma \ref{lem:starconditionAD}.
\end{proof}

\begin{prop}
  \label{prop:ADdef}
  The algorithm is well-defined. That is, all the centered segments with the same end have the same sign.
\end{prop}

\begin{proof}
  We prove the result by induction. We can assume that $\m_1$ is centered and that there is another $\m_1$ in $\AD(\m^{\#},\varepsilon^{\#})$. If we write $\AD(\m^{\#},\varepsilon^{\#})=(\m'_1,\varepsilon'_1)+\AD(\m'^{\#},\varepsilon'^{\#})$, then Proposition \ref{prop:lengthmax} tells us that $\m'_1 = \m_1$. We are left to prove that $\varepsilon_1(\m_1)=\varepsilon'_1(\m'_1)$. Let $n_0 = \card \{\Delta \in \m, c(\Delta)=0\}$ and $n^{\#}_0 = \card \{\Delta \in \m^{\#}, c(\Delta)=0\}$.

  \begin{itemize}
    \item If $\rho$ is of the same type as $G$. Then $\varepsilon_1(\m_1):=(-1)^{n_0 + 1} \varepsilon([0,0]_{\rho_u})$ and $\varepsilon'_1(\m'_1):=(-1)^{n^{\#}_0 + 1} \varepsilon^{\#}([0,0]_{\rho_u})$. After applying the algorithm, only one $[0,0]_{\rho_u}$ is suppressed and any centered segments created appear in pairs. Hence, $(-1)^{n_0 + 1}=(-1)^{n^{\#}_0}$. By Lemma \ref{lem:precalgocent}, $\Delta_{l-1}$ cannot be centered. Moreover, $c(\Delta_{l-1}) \neq 1/2$ because, by Lemma \ref{lem:paritycent}, this would imply that $\Delta'_{l}=[0,0]_{\rho_u}^{\ge 0}$, but this contradicts Lemma \ref{lem:sizem1} that says that $\Delta'_{l}=[0,0]_{\rho_u}^{\ge 0} \preceq \Delta'_{l-1} \prec \Delta_{l-1}=[0,1]_{\rho_u}$. Thus $c(\Delta_{l-1}) \neq 1/2$ and $\varepsilon^{\#}([0,0]_{\rho_u})=(-1) \varepsilon([0,0]_{\rho_u})$. Finally, we get that $(-1)^{n_0 + 1} \varepsilon([0,0]_{\rho_u}) = (-1)^{n^{\#}_0 + 1} \varepsilon^{\#}([0,0]_{\rho_u})$ and that $\varepsilon_1(\m_1)=\varepsilon'_1(\m'_1)$.
    \item If $\rho$ is not of the same type as $G$. By definition, $\varepsilon_1(\m_1):=(-1)^{n_0 }$ and $\varepsilon'_1(\m'_1):=(-1)^{n^{\#}_0}$. The proof of Lemma \ref{lem:starconditionAD}, tells us that $\Delta_{l}=[1/2,1/2]_{\rho_u}$. Now Lemma \ref{lem:sequencesingle} gives us that $\Delta_j = [e_{\max}-j+1,e_{\max}-j+1]_{\rho_u}$. Thus $n^{\#}_0=n_0$ and $\varepsilon_1(\m_1)=\varepsilon'_1(\m'_1)$.
  \end{itemize}
\end{proof}

\section{Important properties}
\label{sec:impprop}
In this section, we will study some properties of the recursive maps:
 \begin{align*}
      \AD_\rho: \Symm_\rho^\varepsilon(G) &\longrightarrow \Symm_\rho^\varepsilon(G)\\
      (\m,\varepsilon) &\mapsto (\m_1,\varepsilon_1)+\AD_\rho(\m^{\#},\varepsilon^{\#}).
  \end{align*}
\subsection{Maximality of the length of $\m_1$}

We fix $\rho \in \Cusp^{\GL}$. In this first paragraph we prove that $\m_1$ has the longest length among the segments ending in $e_{\max}$.

\begin{rem}
When $\rho$ is of good parity, this is proved in Proposition \ref{prop:lengthmax}.  In this case, the crucial lemma is Lemma \ref{lem:sizem1}. This lemma is false in the bad parity case (with respect to the order $\le$). Indeed,  we have seen in Example \ref{ex:badpar0101} that, if $\m=[-1,0]_\rho+[-1,0]_\rho+[0,1]_\rho+[0,1]_\rho$, then $\Delta_1 = [0,1]_\rho$ and $\Delta'_1=[1,1]_\rho$. Thus, $\Delta'_1 \nleq \Delta_1$.
\end{rem}

Let $\rho \in \Cusp^{\GL}$ be of bad parity. Let $\m \in \Symm_\rho^\varepsilon(G)$, and  denote by $\Delta_1,\cdots,\Delta_l$ the initial sequence in the algorithm. We also denote by $\Delta'_1,\cdots,\Delta'_{l'}$ the initial sequence in the algorithm for $\m^{\#}$.

Let $i_0 = \min \{i \in \{1,\cdots,l\}: \exists j \in \{1,\cdots,l\}, \Delta_i^{\vee}=\Delta_j \}$, and $j_0 = \max \{j \in \{1,\cdots,l\}: \exists i \in \{1,\cdots,l\}, \Delta_j^{\vee}=\Delta_i \}$

\begin{lem}
  \label{lem:duauxalgo}
  \begin{enumerate}
    \item $\Delta_{i_0}^{\vee}=\Delta_{j_0}$.
    \item Let $i,j \in \{1,\cdots,l\}$. Then, $\Delta_i^{\vee}=\Delta_j$ if and only if $i_0 \le i \le j_0$ and $i+j=i_0 + j_0$.
    \item For all $i \in \{i_0,\cdots,j_0\}$, $l(\Delta_i)=l(\Delta_{i_0})$.
  \end{enumerate}
\end{lem}

\begin{proof}
  If $\Delta_i^{\vee}=\Delta_j$, then $c(\Delta_i)=-c(\Delta_j)$. Moreover, for every $k$, $\Delta_{k+1} \le \Delta_k$ and $e(\Delta_{k+1}) = e(\Delta_k) - 1$ implies that $c(\Delta_{k+1}) < c(\Delta_k)$. This proves $(1)$. To get $(2)$ and $(3)$, note that for all $k$, $l(\Delta_{k+1}) \ge l(\Delta_k)$ and $l(\Delta_{j_0})=l(\Delta_{i_0}^{\vee})=l(\Delta_{i_0})$. Thus, for all $i \in \{i_0,\cdots,j_0\}$, $l(\Delta_i)=l(\Delta_{i_0})$. This also proves $(3)$ as $e(\Delta_i)=e(\Delta_{i_0}) - i + i_0$.
\end{proof}

\begin{lem}
  \label{lem:duauxmdieze}
Suppose that $e(\Delta_1)=e(\Delta'_1)$, $l' \ge i_0$ and ${}^{+}\Delta'_{i_0}=\Delta_{i_0}$. Then $l' \ge j_0$, and for all $i_0 \le i \le j_0$, ${}^{+}\Delta'_{i}=\Delta_{i}$.
\end{lem}

\begin{proof}
  This follows easily from a direct computation, since by Lemma \ref{lem:duauxalgo}, for $i \in \{i_0,\cdots,j_0\}$ we have $\Delta_i = [ b(\Delta_{i_0})- i + i_0, e(\Delta_{i_0})- i + i_0]_{\rho_u}$.
\end{proof}

\begin{prop}
  \label{prop:longueurbad}
  Suppose that $ e(\Delta_1) = e(\Delta'_1) $. Then the following holds:
  \begin{enumerate}
    \item $ l' \le l $.
    \item For all $ i \le l' $, we have $ {}^{+}\Delta'_{i} \le \Delta_{i} $.
    \item If there exists $ i \le l' $ such that $ {}^{+}\Delta'_{i} = \Delta_{i} $, then $ i_0 \le i \le j_0 $, $ l' \ge j_0 $, and for all $ i_0 \le j \le j_0 $, we have $ {}^{+}\Delta'_{j} = \Delta_{j} $.
  \end{enumerate}
\end{prop}

\begin{proof}
  We prove by induction on $ k \le l' $ the following properties:
  \begin{enumerate}
    \item $ k \le l $,
    \item $ {}^{+}\Delta'_{k} \le \Delta_{k} $,
    \item If $ {}^{+}\Delta'_{k} = \Delta_{k} $, then $ i_0 \le k \le j_0 $, $ l' \ge j_0 $, and for all $ i_0 \le i \le j_0 $, we have $ {}^{+}\Delta'_{i} = \Delta_{i} $.
  \end{enumerate}

  We start with the base case $ k = 1 $. Clearly, $ 1 \le l $. Suppose that $ i_0 = 1 $. If $ \Delta_1 = [e_{\max}, e_{\max}] $, then $ \Delta'_1 \le \Delta_1 $ holds trivially. Otherwise, since $ \Delta_1^{\vee} = \Delta_{j_0} $, the largest segment in $ \m^{\#} $ ending in $ e_{\max} $ is $ {}^{-}\Delta_1 $. Hence, $ \Delta'_1 = {}^{-}\Delta_1 $, which proves property (2). Property (3) then follows from Lemma \ref{lem:duauxmdieze}.

  Now assume $ i_0 \neq 1 $. By definition of the algorithm, $ \Delta'_1 = \Delta $ or $ {}^{-}\Delta $ for some $ \Delta \in \m $. Since $ \Delta_1 $ is the largest segment ending in $ e_{\max} $, we have $ \Delta \le \Delta_1 $. If $ \Delta'_1 \notin \m $, then $ \Delta \ne \Delta_1 $, as $ i_0 \ne 1 $, so $ \Delta < \Delta_1 $. Therefore, $ \Delta'_1 \le \Delta_1 $, as required.

  Now let $ k < l' $, and assume the result holds for all $ k' \le k $. We want to prove it for $ k+1 $.

  \begin{itemize}
    \item Suppose $ {}^{+}\Delta'_k = \Delta_k $ and $ k < j_0 $. Then $ k+1 \le j_0 $, and $ l \ge j_0 \ge k+1 $. Moreover, by the induction hypothesis (3), $ {}^{+}\Delta'_{k+1} = \Delta_{k+1} $, so (3) also holds for $ k+1 $. This completes the step.

    \item Suppose $ \Delta'_k \le \Delta_k $, or $ {}^{+}\Delta'_k = \Delta_k $ with $ k = j_0 $.

    We first show that $ \Delta'_{k+1} \le \Delta_k $. If $ \Delta'_k \le \Delta_k $, then $ \Delta'_{k+1} \le \Delta'_k \le \Delta_k $ immediately. Suppose instead $ {}^{+}\Delta'_k = \Delta_k $ and $ k = j_0 $. By the induction hypothesis, $ {}^{+}\Delta'_{i_0} = \Delta_{i_0} $. If $ i_0 = 1 $, then $ \Delta_{i_0} $ is the largest segment ending in $ e_{\max} $, so $ {}^{-}\Delta_{i_0} \notin \m $. If $ i_0 > 1 $, then $ \Delta'_{i_0 - 1} \le \Delta_{i_0 - 1} $, so $ {}^{-}\Delta_{i_0} = \Delta'_{i_0} \le \Delta_{i_0 - 1} $. Hence, $ {}^{-}\Delta_{i_0} \notin \m $, since otherwise it would contradict the definition of $ \Delta_{i_0} $. In all cases, $ {}^{-}\Delta_{i_0} \notin \m $.
Therefore, $ \Delta_k^{-} = \Delta_{j_0}^{-} = ({}^{-}\Delta_{i_0})^{\vee} \notin \m $. This implies $ m_{\m^{\#}}(\Delta_k^{-}) = 1 $. But $ (\Delta_k^{-})^{\vee} = {}^{-}\Delta_{i_0} = \Delta'_{i_0} $, so by definition of the algorithm, $ \Delta'_{k+1} \ne \Delta_k^{-} $. Hence, $ \Delta'_{k+1} \le \Delta_k $, as required.

    By construction of $ \m^{\#} $, $ \Delta'_{k+1} = \Delta $, $ {}^{-}\Delta $, or $ \Delta^{-} $ for some $ \Delta \in \m $. If $ \Delta'_{k+1}, {}^{+}\Delta'_{k+1} \notin \m $ and $ \Delta'_{k+1} = \Delta^{-} $, then $ \Delta = \Delta_k $, contradicting $ \Delta'_{k+1} \le \Delta_k $. Thus, $ \Delta'_{k+1} = \Delta $ or $ {}^{-}\Delta $, and in both cases $ \Delta \le \Delta'_{k+1} \le \Delta_k $.

    Suppose $ \Delta^{\vee} \notin \{\Delta_i \mid i \le k\} $ or $ m_{\m}(\Delta) \ge 2 $. Then $ l \ge k+1 $ and $ \Delta \le \Delta_{k+1} $. Consequently, $ {}^{+}\Delta'_{k+1} \le \Delta \le \Delta_{k+1} $. If $ {}^{+}\Delta'_{k+1} = \Delta_{k+1} $, and if there exists $ i $ such that $ \Delta_{k+1}^{\vee} = \Delta_i $, then by Lemma \ref{lem:duauxalgo}, $ k+1 \in \{i_0, \ldots, j_0\} $. Also, since $ {}^{-}\Delta_{k+1} \le \Delta_k $, we have $ l(\Delta_{k+1}) \ne l(\Delta_k) $, so $ k \notin \{i_0, \ldots, j_0\} $, and hence $ k+1 = i_0 $. Property (3) then follows from Lemma \ref{lem:duauxmdieze}.

    It remains to show that such $ i $ exists. If $ {}^{-}\Delta_{k+1} \notin \m $, then by the construction of $ \m^{\#} $, there exists $ i $ such that $ \Delta_{k+1}^{\vee} = \Delta_i $, and we are done. If instead $ {}^{-}\Delta_{k+1} \in \m $, then since $ {}^{-}\Delta_{k+1} = \Delta'_{k+1} \le \Delta_k $, there exists $ i \le k $ such that $ ({}^{-}\Delta_{k+1})^{\vee} = \Delta_i $, and $ m_{\m}({}^{-}\Delta_{k+1}) = 1 $.

    By the induction hypothesis, $ {}^{+}\Delta'_i \le \Delta_i $. If $ {}^{+}\Delta'_i = \Delta_i $, then $ i \in \{i_0, \ldots, j_0\} $, and there exists $ j $ such that $ \Delta_i^{\vee} = \Delta_j $, so $ \Delta_j = {}^{-}\Delta_{k+1} $, contradicting the uniqueness of $ \Delta_{k+1} $ as the segment ending at $ e(\Delta_{k+1}) $. Therefore, $ \Delta'_i \le \Delta_i $. From $ \Delta'_{k+1} \le \Delta'_i \le \Delta_i = (\Delta'_{k+1})^{\vee} $, it follows that $ \Delta'_i = \Delta_i $, and hence $ (\Delta'_{k+1})^{\vee} = \Delta'_i $. Thus, $ m_{\m^{\#}}(\Delta'_{k+1}) \ge 2 $, contradicting $ m_{\m}(\Delta'_{k+1}) = 1 $.

    Finally, suppose $ \Delta^{\vee} = \Delta_i $ for some $ i \le k $ and $ m_{\m}(\Delta) = 1 $. Then $ \Delta \le \Delta'_{k+1} \le \Delta_k \le \Delta_i = \Delta^{\vee} $, so $ l(\Delta) = l(\Delta'_{k+1}) $, hence $ \Delta = \Delta'_{k+1} $. By the induction hypothesis, $ {}^{+}\Delta'_i \le \Delta_i $. If $ \Delta'_i \le \Delta_i $, then from $ \Delta \le \Delta'_{k+1} \le \Delta'_i \le \Delta_i = \Delta^{\vee} $, we obtain $ \Delta'_i = \Delta_i $, so $ (\Delta'_{k+1})^{\vee} = \Delta'_i $, and $ m_{\m^{\#}}(\Delta'_{k+1}) \ge 2 $, again contradicting $ m_{\m}(\Delta) = 1 $. If $ {}^{+}\Delta'_i = \Delta_i $, then $ i \in \{i_0, \ldots, j_0\} $, and by Lemma \ref{lem:duauxalgo}, $ \Delta = \Delta_{i_0 + j_0 - i} $. Since $ \Delta = \Delta'_{k+1} $, we have $ k+1 = i_0 + j_0 - i \in \{i_0, \ldots, j_0\} $, so $ \Delta_{k+1} = \Delta'_{k+1} $, contradicting Lemma \ref{lem:duauxmdieze}.
  \end{itemize}
\end{proof}

We conclude that the same result holds for any $\rho \in \Cusp^{\GL} $.

\begin{prop}
  Let $ \rho \in \Cusp^{\GL} $ and $ (\m,\varepsilon) \in \Symm_\rho^\varepsilon $. Let $ (\m_1,\varepsilon_1) $ be the multisegment produced by $ \AD_\rho(\m,\varepsilon) $, and let $ \eta_1 \in \m_1 $ be the segment ending in $ e_{\max} $. Then $ \eta_1 $ is the longest among the segments $ \Delta $ of $ \AD_\rho(\m,\varepsilon) $ such that $ e(\Delta) = e_{\max} $.
\end{prop}

\begin{proof}
  When $ \rho $ is ugly, by Remark \ref{rem:MWGL}, $ \AD_\rho $ is the M{\oe}glin--Waldspurger algorithm, and this is proved in \cite[II.2.2.]{MW2}. When $ \rho $ is good, this is covered by Proposition \ref{prop:lengthmax}. The case where $ \rho $ is bad follows from Proposition \ref{prop:longueurbad}.
\end{proof}

\subsection{Centered segments in the good parity case}
\label{sec:e0cent}

In this section, we show that in the good parity case, $\varepsilon_0$ characterizes whether $\m_1$ contains a centered segment.

\bigskip

We assume that $\rho$ is good.

\begin{lem}
  \label{lem:ADnotcentered}
  If $\varepsilon_0 = 1$, then $e(\Delta_{1}) + e(\Delta_{l}) \neq 0$.
\end{lem}

\begin{proof}
  Suppose that $\varepsilon_0 = 1$. Necessarily, $e(\Delta_1) > 0$. Let $e_{\max} = e(\Delta_1)$ be the maximum of the coefficients of $\m$. Suppose that $e(\Delta_{1}) + e(\Delta_{l}) = 0$. Then $-e_{\max} \le b_{\rho_u}(\Delta_{l}) \le e(\Delta_{l}) = -e_{\max}$, so $\Delta_l = [-e_{\max}, -e_{\max}]_{\rho_u}$. 

  If $-e_{\max} + 1 < 0$, the only segment $\Delta$ such that $e(\Delta) = -e_{\max} + 1$ and $\Delta_l \prec \Delta$ is $[-e_{\max} + 1, -e_{\max} + 1]_{\rho_u}$. Let $i \ge 0$ be the largest integer such that $-e_{\max} + i < 0$. By induction, for $0 \le j \le i$, we have $\Delta_{l-j} = [-e_{\max} + j, -e_{\max} + j]_{\rho_u}$. For $0 \le j \le i$, the segment $\Delta^{\vee}_{l-j} = [e_{\max} - j, e_{\max} - j]_{\rho_u}$ is in  $\m$, and by maximality in the algorithm, $\Delta_1 = [e_{\max}, e_{\max}]_{\rho_u}, \ldots, \Delta_i = [e_{\max} - i, e_{\max} - i]_{\rho_u}$. 

  If $\rho$ is not of the same type as $G$, then $-e_{\max} \in \frac{1}{2}\Z \setminus \Z$ and $-e_{\max} + i = -\frac{1}{2}$. This is a contradiction since the algorithm stops at $[\frac{1}{2}, \frac{1}{2}]_{\rho_u}$. 

  If $\rho$ is of the same type as $G$, then $-e_{\max} \in \Z$ and $-e_{\max} + i = -1$. Hence $\Delta_{l-i} = [-1, -1]_{\rho_u}$ and $\Delta_{l-i-1} = [0, 0]_{\rho_u}^{\le 0}$ (had it been $[0, 0]_{\rho_u}^{=0}$ or $[0, 0]_{\rho_u}^{\ge 0}$, the algorithm would have stopped there).  However, we have shown that $\Delta_{l-i-2} = [1, 1]_{\rho_u}$. This contradicts the maximality of $\Delta_{l-i-1}$, since $[0, 0]_{\rho_u}^{\ge 0} \prec [1, 1]_{\rho_u}$.
\end{proof}

Also, when $\Delta_l = [0,0]_{\rho_u}^{\ge 0}$ or $[1/2,1/2]_{\rho_u}$, the sequence is very specific, as shown in the lemma below.

\begin{lem}
  \label{lem:sequencesingle}
  Suppose that $\Delta_l = [0,0]_{\rho_u}^{\ge 0}$ or $[1/2,1/2]_{\rho_u}$. Then, for all $j \in \{1, \ldots, l\}$, we have $\Delta_j = [e_{\max} - j + 1, e_{\max} - j + 1]_{\rho_u}$.
\end{lem}

\begin{proof}
  Let $x \ge 0$, and let $\Delta = [x, x]^{\ge 0}_{\rho_u}$. Then the only segment $\Delta' \in \USeg$ such that $e(\Delta') = x + 1$ and $\Delta \prec \Delta'$ is $\Delta' = [x + 1, x + 1]^{\ge 0}_{\rho_u}$. The result follows easily from this observation.
\end{proof}

\subsection{Sign preserving property in the good parity case}
\label{sec:signpreserving}

We assume that $\rho$ is good. In this section, we show that $\AD_\rho$ preserves the product of the signs.

\bigskip

Let $(\m, \varepsilon) \in \Symm_\rho^{\varepsilon}(G)$. By convention, if $\Delta \in \m$ and $c(\Delta) \neq 0$, we define $\varepsilon(\Delta) = 1$. We define the product of the signs in $(\m, \varepsilon)$ by 
\[
  S(\m, \varepsilon) := \prod_{\Delta \in \m} \varepsilon(\Delta).
\]
Our goal is to prove that 
\[
  S(\AD_\rho(\m, \varepsilon)) = S(\m, \varepsilon).
\]

\begin{lem}
  \label{lem:signm1md}
  We have 
  \[
    S(\m_1, \varepsilon_1) \cdot S(\m^{\#}, \varepsilon^{\#}) = S(\m, \varepsilon).
  \]
\end{lem}

\begin{proof}
  Let $\rho_u$ be the unitarization of $\rho$.
    \begin{itemize}
      \item Assume that $\varepsilon_0 = 1$ and in $\Delta_{1},\cdots,\Delta_{l}$ there is no centered segment with label $\ge 0$ or $=0$. By Lemma \ref{lem:ADnotcentered}, $\m_1$ is not a centered segment, so $S(\m_1,\varepsilon_1)=1$. So we need to show that $S(\m,\varepsilon)S(\m^{\#},\varepsilon^{\#})=1$. Since $\varepsilon_0=1$, in $\m^{\#}$ the segments that changed signs are the $\Lambda_{i_j}^{\#}$ which are centered. The only possible segment creating a centered segment would be a $\Lambda_{i_j}$ such that $c(\Lambda_{i_j})=1/2$. If there is no such $\Lambda_{i_j}$ then there is no change in the signs and thus $S(\m,\varepsilon)=S(\m^{\#},\varepsilon^{\#})$. If such a segment exists, then $\varepsilon^{\#}(\Lambda_{i_j}^{\#})=1$ if $\Lambda_{i_j}^{\#} \notin \m$ and in this case no sign change; or $\varepsilon^{\#}(\Lambda_{i_j}^{\#})=-\varepsilon(\Lambda_{i_j}^{\#})$ if $\Lambda_{i_j}^{\#} \in \m$. But if $\Lambda_{i_j}^{\#} \in \m$, necessarily $m_{\m}(\Lambda_{i_j}^{\#})$ is even (if not $[-e(\Lambda_{i_j}^{\#}),e(\Lambda_{i_j}^{\#})]_{\rho_u}^{=0}$ would follow $\Lambda_{i_j}$ in the algorithm). Two segments $\Lambda_{i_j}^{\#}$  are created in $\m^{\#}$ so the multiplicity stays even and the product of the signs remains $1$. Hence, $S(\m,\varepsilon)=S(\m^{\#},\varepsilon^{\#})$.

      \item Assume that  $\varepsilon_0 = 1$ and that in $\Delta_{1},\cdots,\Delta_{l}$ there are centered segments with label $\ge 0$ or $=0$. By definition of the order, all these segments are consecutive in $\Delta_{1},\cdots,\Delta_{l}$. There exist two integers $a,b$ with $1 \le a \le b \le l$ such that these segments are $\Delta_{a},\cdots,\Delta_{b}$. The only possible segment with label $\ge 0$ is $\Delta_{a}$. As before, $\m_1$ is not a centered segment, so $S(\m_1,\varepsilon_1)=1$ and $\varepsilon_0=1$. 
      The new centered segment in $\m^{\#}$ is created by $\Lambda_{i_a},\cdots,\Lambda_{i_b}$, and possibly $\Lambda_{i_{a-1}}$ if $c(\Lambda_{i_{a-1}}) = 1/2$. This gives the following change to $S(\m,\varepsilon)S(\m^{\#},\varepsilon^{\#})$. If $a=1$ or $c(\Lambda_{i_{a-1}}) \neq 1$, then the multiplicity of $\Lambda_{i_a}$ is decreased by one, hence multiplying $S(\m,\varepsilon)S(\m^{\#},\varepsilon^{\#})$ by $\varepsilon(\Lambda_{i_a})$. If $a > 1$ and $c(\Lambda_{i_{a-1}}) = 1/2$, then two segments $\Lambda_{i_a}$ are created (from $\Lambda_{i_a} + \Lambda^{\vee}_{i_a}$) and, then, one is suppressed, which also flips the sign of $S(\m,\varepsilon)S(\m^{\#},\varepsilon^{\#})$ by $\varepsilon(\Lambda_{i_a})$. With the hypotheses on $\Lambda_{i_{a-1}}$, necessarily $\Delta_{a}$ has label $\le 0$, thus $m_{\m}(p(\Delta_{a}))$ is odd. Hence, there is in $\m^{\#}$ an even number of $p(\Delta_{a})$ and the change of their signs does not affect $S(\m,\varepsilon)S(\m^{\#},\varepsilon^{\#})$. In both cases, $S(\m,\varepsilon)S(\m^{\#},\varepsilon^{\#})$ is multiplied by $\varepsilon(\Lambda_{i_a})$. For $ a + 1 \le m \le b$, the multiplicity of $\Lambda_{i_m}$ is unchanged, but $\varepsilon^{\#}(\Lambda_{i_m})=\varepsilon_0 * \varepsilon(\Lambda_{i_{m-1}})=-\varepsilon(\Lambda_{i_m})$, so $S(\m,\varepsilon)S(\m^{\#},\varepsilon^{\#})$ is multiplied by $(-1)^{m_{\m}(p(\Delta_{m}))}$. Finally, the multiplicity of $\Lambda_{i_b}^{\#}$ is increased by one, and the signs are changed to $\varepsilon(\Delta_{b})$, that is $S(\m,\varepsilon)S(\m^{\#},\varepsilon^{\#})$ is multiplied by $\varepsilon(\Lambda_{i_b}^{\#})^{m_{\m}(\Lambda_{i_b}^{\#})} * \varepsilon(\Delta_{b})^{m_{\m}(\Lambda_{i_b}^{\#}) + 1}$. At the end we get
      \[
        S(\m,\varepsilon)S(\m^{\#},\varepsilon^{\#}) = \varepsilon(\Delta_{a}) * \prod_{m=a+1}^{b} (-1)^{m_{\m}(p(\Delta_{m}))} * \varepsilon(\Lambda_{i_b}^{\#})^{m_{\m}(\Lambda_{i_b}^{\#})} * \varepsilon(\Delta_{b})^{m_{\m}(\Lambda_{i_b}^{\#}) + 1}
      \]
      Since, for $a+1 \le m \le b$, $\Delta_{m}$ has label $=0$, we have that $m_{\m}(p(\Delta_{m}))$ is odd and $(-1)^{m_{\m}(p(\Delta_{m}))} = -1$. Also by definition of the algorithm, $\varepsilon(\Delta_{m-1})=-\varepsilon(\Delta_{m})$. Hence, $\varepsilon(\Delta_{b}) = (-1)^{b-a} \varepsilon(\Delta_{a})$. We get
      \[
        S(\m,\varepsilon)S(\m^{\#},\varepsilon^{\#}) = \varepsilon(\Lambda_{i_b}^{\#})^{m_{\m}(\Lambda_{i_b}^{\#})} * \varepsilon(\Delta_{b})^{m_{\m}(\Lambda_{i_b}^{\#})}.
      \]
      If $m_{\m}(\Lambda_{i_b}^{\#})$ is even then $S(\m,\varepsilon)S(\m^{\#},\varepsilon^{\#})=1$ and we are done. And if $m_{\m}(\Lambda_{i_b}^{\#})$ is odd, as $\Lambda_{i_b}^{\#}$ is not in the sequence $\Delta_{1},\cdots,\Delta_{l}$ it means that $\varepsilon(\Lambda_{i_b}^{\#}) = \varepsilon(\Lambda_{i_b})$, and $S(\m,\varepsilon)S(\m^{\#},\varepsilon^{\#}) = 1$.

      \item Assume that  $\varepsilon_0 = -1$ and $\Delta_l=[1/2,1/2]_{\rho_u}$. By Lemma \ref{lem:sequencesingle}, for all $ j \in \{1,\cdots,l\}$, $\Delta_j = [e_{\max}-j+1,e_{\max}-j+1]_{\rho_u}$ and the algorithm just suppresses these segments and their symmetric counterparts. Since $\varepsilon_0 = -1$, all the centered segments change signs, so $S(\m,\varepsilon)S(\m^{\#},\varepsilon^{\#})=(-1)^{n_0}=\varepsilon_1(\m_1)$.

      \item Assume that  $\varepsilon_0 = -1$ and $\Delta_l \neq [1/2,1/2]_{\rho_u}$. Then either $\rho$ is of the same type as $G$ and $\Delta_{l} = [0,0]_{\rho_u}^{\ge 0}$ or $[0,0]_{\rho_u}^{=0}$; or $\rho$ is not of the same type as $G$ and $\Delta_{l} = [-1/2,1/2]_{\rho_u}^{\ge 0}$ or $[-1/2,1/2]_{\rho_u}^{=0}$ with $\varepsilon([-1/2,1/2]_{\rho_u})=-1$. Now, $\m_1$ is a centered segment, and $\varepsilon_0=-1$. Following the notation of the previous case, let $a$ be the smallest integer such that $\Delta_{a}$ is centered (and we would have $b=l$). Let $n_1$ be the number of centered segments $\Delta$ in $\m$ with $e_{\rho}(\Delta) \ge e_{\rho}(\Delta_{a})$. For $a \le j \le l-1$, $\varepsilon^{\#}(\Lambda_{i_a}^{\#})=\varepsilon(\Lambda_{i_a}^{\#})$. If $a > 1$ and $c(\Lambda_{i_{a-1}})= 1/2$, two segments $p(\Delta_{a})$ are created with the same sign. In all  cases, one $\Delta_{a}$ is suppressed, which flips the sign of $S(\m,\varepsilon)S(\m^{\#},\varepsilon^{\#})$ by $\varepsilon(\Delta_{a})$. Now all the other segments change sign. If $a = 1$ or $c(\Lambda_{i_{a-1}}) \neq 1$, there is $n_1 - 1$ such segments (as one $\Delta_{a}$ has been suppressed), thus the sign of $S(\m,\varepsilon)S(\m^{\#},\varepsilon^{\#})$ is changed by $(-1)^{n_1 - 1}$. If $a > 1$ and $c(\Lambda_{i_{a-1}}) = 1/2$, the number of segments that change signs are $n_1 - m_{\m}(p(\Delta_{a}))$. But $m_{\m}(p(\Delta_{a}))$ is odd. Thus $S(\m,\varepsilon)S(\m^{\#},\varepsilon^{\#})$ changes by $(-1)^{n_1 - 1}$. In all cases, we get
    \[
      S(\m,\varepsilon)S(\m^{\#},\varepsilon^{\#}) = (-1)^{n_1 + 1}\varepsilon(\Delta_{a}).
    \]
    The centered segments of $\m$ with $e(\Delta) < e(\Delta_{a})$ are the $\Lambda_{i_j}^{\#}$ with $a \le j \le l-1$. Thus $n_0 = n_1 + \sum_{j=a}^{l-1} m_{\m}(\Lambda_{i_j}^{\#})$. But $m_{\m}(\Lambda_{i_j}^{\#})$ is odd, thus $(-1)^{n_1 + 1} = (-1)^{n_0 + 1} * (-1)^{l-a}$. Also, $\varepsilon(\Delta_{a}) = -\varepsilon(\Delta_{a + 1})= \cdots = (-1)^{l-a}\varepsilon(\Delta_{l})$. Hence
  \[
      S(\m,\varepsilon)S(\m^{\#},\varepsilon^{\#}) = (-1)^{n_0 + 1}\varepsilon(p(\Delta_{l}))=\varepsilon_1(\m_1).
    \]
    \end{itemize}     
\end{proof}

\begin{prop}
  \label{prop:ADsigns}
  Let $(\m,\varepsilon) \in \Symm_\rho^{\varepsilon}(G)$. Then $S(\AD_\rho(\m,\varepsilon))=S(\m,\varepsilon)$.
\end{prop}

\begin{proof}
  We prove the result by induction. We have that
  \begin{align*}
    S(\AD_\rho(\m,\varepsilon)) &=  S(\m_1,
  \varepsilon_1)S(\AD_\rho (\m^{\#},\varepsilon^{\#}))  \\
  &= S(\m_1, \varepsilon_1)S(\m^{\#},\varepsilon^{\#}) \text{ (using the induction hypothesis).}
  \end{align*}
The result follows from Lemma \ref{lem:signm1md}.
\end{proof}

\section{The theory of derivatives and the Atobe--Mínguez algorithm}\label{sec:AM}
We now recall the theory of derivatives as presented in \cite{AM}. It will be the main tool to prove Theorem \ref{thm:MAIN}.
Let $d>0$ be an integer. Throughout this section, we fix $\rho \in \Cusp(\GL_d(F))$.

\subsection{}
We first treat the case of general linear groups. For $\tau \in \Rep(\GL_n(F))$, we define the semisimple representations $L_\rho^{(k)}(\tau)$ and $R_\rho^{(k)}(\tau)$ of $\GL_{n-dk}(F)$ by the equations:
\begin{align*}
\left[\Jac_{(dk,n-dk)}(\tau)\right] &= \rho^k \boxtimes L_{\rho}^{(k)}(\tau) + \sum_i \tau_i \boxtimes \sigma_i, \\
\left[\Jac_{(n-dk,dk)}(\tau)\right] &= R_{\rho}^{(k)}(\tau) \boxtimes \rho^k + \sum_i \sigma'_i \boxtimes \tau'_i,
\end{align*}
where $\tau_i$ and $\tau'_i$ are irreducible representations of $\GL_{dk}(F)$ that are not isomorphic to $\rho^k$. We refer to $L_\rho^{(k)}(\tau)$ (respectively, $R_\rho^{(k)}(\tau)$) as the \emph{$k$-th left $\rho$-derivative} (respectively, the \emph{$k$-th right $\rho$-derivative}) of $\tau$.

If $L_{\rho}^{(k)}(\tau) \neq 0$ but $L_{\rho}^{(k+1)}(\tau) = 0$, we say that $L_{\rho}^{(k)}(\tau)$ is the \emph{highest left $\rho$-derivative}. The \emph{highest right $\rho$-derivative} is defined in the same way using $R_{\rho}^{(k)}(\tau)$.

When $L_{\rho}^{(1)}(\tau) = 0$ (respectively, $R_{\rho}^{(1)}(\tau) = 0$), we say that $\tau$ is \emph{left $\rho$-reduced} (respectively, \emph{right $\rho$-reduced}).

\subsection{}
We now proceed to the case of $G_n$. Let $k \geq 0$, and let $P_{dk}$ denote the standard parabolic subgroup of $G_n$ with Levi subgroup isomorphic to $\GL_{dk}(F) \times G_{n-dk}$. For $\Pi \in \Rep(G_n)$, we define a semisimple representation $D_\rho^{(k)}(\Pi)$ of $G_{n-dk}$ by:
\[
\left[\Jac_{P_{dk}}^{G_n}(\Pi)\right] 
= \rho^k \boxtimes D_{\rho}^{(k)}(\Pi) + \sum_i \tau_i \boxtimes \Pi_i,
\]
where $\tau_i$ is an irreducible representation of $\GL_{dk}(F)$ which is not isomorphic to $\rho^k$. We call $D_\rho^{(k)}(\Pi)$ the \emph{$k$-th $\rho$-derivative} of $\Pi$.

If $D_{\rho}^{(k)}(\Pi) \neq 0$ but $D_{\rho}^{(k+1)}(\Pi) = 0$, we say that $D_{\rho}^{(k)}(\Pi)$ is the \emph{highest $\rho$-derivative}.
When $D_{\rho}^{(1)}(\Pi) = 0$, we say that $\Pi$ is \emph{$\rho$-reduced}.

\subsection{} 
Now assume that $\rho$ is not self-dual. Then, for all $\pi \in \Irr^G$, the highest $\rho$-derivative $D_\rho^{(k)}(\pi)$ is irreducible, and $\soc(\rho^r \rtimes \pi)$ is irreducible for all $r \geq 0$.  We define
\[
S_\rho^{(r)}(\pi) = \soc(\rho^r \rtimes \pi).
\]
One has that $D_\rho^{(r)}\circ S_\rho^{(r)}(\pi) = \pi$ and $S_\rho^{(r)} \circ D_\rho^{(r)}(\pi) = \pi$, if $D_\rho^{(r)}(\pi) \neq 0$. For more details, see \cite[\textsection 3]{AM}.

\subsection{}\label{other-derivatives}
In this paragraph, we assume that $\rho \in \Cusp(\GL_d(F))$ is self-dual. In this case, $\rho$-derivatives are not yet well understood. One of the ideas in \cite{AM} is to circumvent this issue by using alternative derivatives.

Let $\Pi \in \Rep(G_n)$. We define the \emph{$L([-1,0]_\rho)$-derivative} $D_{L([-1,0]_\rho)}^{(k)}(\Pi)$ and the \emph{$Z([0,1]_\rho)$-derivative} $D_{Z([0,1]_\rho)}^{(k)}(\Pi)$ as the semisimple representations of $G_{n-2dk}$ satisfying
\[
\left[\Jac_{P_{2dk}}^{G_n}(\Pi)\right] = 
L([-1,0]_\rho)^k \boxtimes D_{L([-1,0]_\rho)}^{(k)}(\Pi)
+ Z([0,1]_\rho)^k \boxtimes D_{Z([0,1]_\rho)}^{(k)}(\Pi)
+ \sum_i \tau_i \boxtimes \pi_i,
\]
where $\tau_i \in \Irr(\GL_{2dk}(F))$ such that 
$\tau_i \not\cong L([-1,0]_\rho)^k, Z([0,1]_\rho)^k$.

As before, we define the notions of \emph{highest $L([-1,0]_\rho)$-derivatives} (\resp \emph{highest $Z([0,1]_\rho)$-derivatives}) and the property of being \emph{$L([-1,0]_\rho)$-reduced} (\resp \emph{$Z([0,1]_\rho)$-reduced}).

If $\Pi \in \Irr(G_n)$ is $\rho|\cdot|^{-1}$-reduced (\resp $\rho|\cdot|^1$-reduced), then the highest $L([-1,0]_\rho)$-derivative $D_{L([-1,0]_\rho)}^{(k)}(\Pi)$ (\resp the highest $Z([0,1]_\rho)$-derivative $D_{Z([0,1]_\rho)}^{(k)}(\Pi)$) is irreducible. Similar definitions apply for $\GL_n(F)$.

For an irreducible representation $\pi$ of $G_n$ which is $\rho|\cdot|^{-1}$-reduced (\resp $\rho|\cdot|^1$-reduced), we also define:
\begin{align*}
S_{L([-1,0]_\rho)}^{(r)}(\pi) \coloneqq \soc(L([-1,0]_\rho)^r \rtimes \pi)\\
\text{(resp. } S_{Z([0,1]_\rho)}^{(r)}(\pi) \coloneqq \soc(Z([0,1]_\rho)^r \rtimes \pi) \text{ )}
\end{align*}
They are irreducible representations and we have:
$D_{L([-1,0]_\rho)}^{(r)}\circ S_{L([-1,0]_\rho)}^{(r)}(\pi) = \pi$ and $S_{L([-1,0]_\rho)}^{(r)} \circ D_{L([-1,0]_\rho)}^{(r)}(\pi) = \pi$, if $D_{L([-1,0]_\rho)}^{(r)}(\pi) \neq 0$, and similarly for $Z([0,1]_\rho)$.

\subsection{} The derivatives are compatible with the Aubert--Zelevinsky dual in the following sense.

\begin{prop}[{\cite[Prop. 3.9]{AM}}]
  \label{prop:derivativedual}
  Let $\pi \in \Irr^{G}$ and $\rho \in \Cusp^{\GL}$.
  \begin{enumerate}
    \item If $D_\rho^{(k)}$ is the highest $\rho$-derivative, then
\[
  \widehat{D_{\rho}^{(k)}(\pi)}=D_{\rho^{\vee}}^{(k)}(\hat{\pi}).
\]
\item If $\rho$ is self-dual, $\pi$ is $\rho|\cdot|^{-1}$-reduced and $D_{L([-1,0]_\rho)}^{(k)}(\pi)$ is the highest $L([-1,0]_\rho)$-derivative, then
\[
  \widehat{D_{L([-1,0]_\rho)}^{(k)}(\pi)}=D_{Z([0,1]_\rho)}^{(k)}(\hat{\pi}).
\]
  \end{enumerate}
\end{prop}

\subsection{} 
We now recall the Atobe--Mínguez algorithm for computing the Aubert--Zelevinsky dual of an irreducible representation $\pi$.

Assume that the dual $\hat\pi_0$ can be computed for all irreducible representations of $G_{n_0}$, where $n_0 < n$.
Let $\pi$ be an irreducible representation of $G_n$.

\begin{enumerate}
\item
If there exists a $\rho \in \Cusp^\GL$ such that $\rho$ is not self-dual, and $D_{\rho}^{(k)}(\pi)$ is the highest $\rho$-derivative with $k \geq 1$, 
then
\[
\hat\pi = S_{\rho^\vee}^{(k)}\left( \widehat{D_{\rho}^{(k)}(\pi)} \right).
\]

\item
Otherwise, if $\pi$ is not tempered, 
there exists a self-dual $\rho \in \Cusp^\GL$ such that 
$D_{L([-1,0]_\rho)}^{(k)}(\pi)$ is the highest $L([-1,0]_\rho)$-derivative with $k \geq 1$. 
In this case, we have
\[
\hat\pi = S_{Z([0,1]_\rho)}^{(k)}\left( \widehat{D_{L([-1,0]_\rho)}^{(k)}(\pi)} \right).
\]

\item
If neither of the above cases applies, then $\hat\pi$ is explicitly computed  (see \cite[Proposition 5.4]{AM}).
\end{enumerate}

\section{Explicit formulas for the derivatives}
\label{sec:expder}

In Section \ref{sec:AM}, we recalled the definition of derivatives. In this section, we provide explicit formulas for computing these derivatives. Such formulas are given in \cite{AM} in terms of Langlands data. Here, we instead work with our symmetric Langlands data (the space \( \Symm_{\rho}^{\varepsilon}(G) \)), which simplifies the formulas and unifies the treatment of the negative and positive cases, handled respectively in Sections 6 and 7 of \cite{AM}.

Accordingly, we define an operator \( D_{\rho} \) on \( \Symm_{\rho}^{\varepsilon}(G) \) such that, for \( \pi = L(\m, \varepsilon) \), the highest derivative satisfies \( D_{\rho}^{(k)}(\pi) = L(D_{\rho}(\m, \varepsilon)) \). This operator will play a crucial role in the proof of our main theorem. However, readers interested only in the overall strategy of the proof may skip this section and proceed directly to Sections \ref{sec:reducedseg} and \ref{sec:startegyproof}.

\subsection{Best matching functions}
\label{sec:bestmatching}
Following \cite[§5.3]{LapidMinguez}, we introduce best matching functions.

\bigskip

Let $X$ and $Y$ be finite sets, and let $\leadsto$ be a relation between elements of $Y$ and $X$.
We are interested in injective functions
$f: X \rightarrow Y$ such that $f(x) \leadsto x$ for all $x \in X$.
Any such function will be called a \emph{$\leadsto$-matching function} (or simply a \emph{matching function} when the relation $\leadsto$ is understood from context).

According to Hall's criterion, such a function $f$ exists if and only if, for every subset $A \subset X$, the following inequality holds:
\begin{equation} \label{eq: HallCriterion}
\#\{y \in Y : y \leadsto x \text{ for some } x \in A\} \ge \# A.
\end{equation}

In some cases, it is possible to construct such a function $f$ explicitly.
Assume $X$ and $Y$ are totally ordered by relations $\le_X$ and $\le_Y$, respectively.
One natural approach is to define $f$ recursively, starting from the largest element of $X$ and proceeding to the smallest, using the following rule:
\begin{equation} \label{eq: def f}
f(x) = \min\{y \in Y \setminus f(X_{>x}) : y \leadsto x\},
\end{equation}
where $X_{>x} := \{x' \in X : x' > x\}$. For this definition to be valid, we must ensure that, for each $x \in X$, there exists some $y \notin f(X_{>x})$ such that $y \leadsto x$.
Clearly, this requires additional assumptions about the relation $\leadsto$.

To this end, we introduce the following property.
We say that the relation $\leadsto$ is \emph{traversable} if for all $x_1, x_2 \in X$ and $y_1, y_2 \in Y$ with $x_1 \ge_X x_2$ and $y_1 \ge_Y y_2$, the following implication holds:
\begin{equation} \label{eq: specrelation}
y_1 \leadsto x_1,\ y_2 \leadsto x_1,\ \text{and } y_2 \leadsto x_2\ \Rightarrow\ y_1 \leadsto x_2.
\end{equation}

More generally, even when Hall's criterion is not satisfied, we can still speak of $\leadsto$-\emph{matchings} (or simply matchings, if $\leadsto$ is clear from context) between $X$ and $Y$. By this, we mean injective functions $f$ from a subset of $X$ to $Y$ such that $f(x) \leadsto x$ for all $x$ in the domain of $f$.
We view such a function as a relation between $X$ and $Y$.

Mimicking the earlier construction, if $\leadsto$ is \emph{traversable}, we define the \emph{best} $\leadsto$-matching between $X$ and $Y$: for this, we recursively define the domain $X^0 \subseteq X$ and the function $f$ on $X^0$ by
\begin{multline*}
x \in X^0 \iff \exists y \in Y \setminus f(X^0 \cap X_{>x}) \text{ such that } y \leadsto x, \\
\text{in which case we set } f(x) = \min\{y \in Y \setminus f(X^0 \cap X_{>x}) : y \leadsto x\}.
\end{multline*}

We set $Y^0:=f(X^0)$, $X^c := X \setminus X^0$ and $Y^c := Y \setminus Y^0$. Finally, for $x \in X^0$, we will say that $x$ \emph{protects} $f(x) \in Y^0$.

\subsection{The good parity case}
\label{sec:dergood}
Here, we give the formulas for computing the highest $\rho$-derivatives, for $\rho$ of good parity, in terms of symmetrical Langlands data.

\bigskip

Let $\rho \in \Cusp^{\GL}$ be of good parity. We write $\rho = \rho_u |\cdot|^{x}$ with $\rho_u$ unitary and $x \in (1/2)\Z$. Let $(\m,\varepsilon) \in \Symm_\rho^{\varepsilon}(G)$. We assume that $x \neq 0$.

By convention, when $x=1/2$, we set $[-x + 1, x - 1]_{\rho_u} = 0$, $m_{\m}([-x + 1, x - 1]_{\rho_u})=1$ and $\varepsilon([-x + 1, x - 1]_{\rho_u})=1$. Let $t = \card \{ [-x,x-1]_{\rho_u} \in \m \} = \card \{ [-x +1,x]_{\rho_u} \in \m\}$.

\begin{defi}
  We say that $(*)$ is satisfied if the four following conditions are satisfied: 
  \begin{enumerate}
    \item $x > 0$;
    \item $m_{\m}([-x,x]_{\rho_u})\neq 0$;
    \item $m_{\m}([-x + 1,x-1]_{\rho_u})\neq 0$;
    \item $\varepsilon([-x,x]_{\rho_u}) \varepsilon([-x + 1,x-1]_{\rho_u})=(-1)^{t+1}$.
  \end{enumerate}
\end{defi}

We write $\m = \Delta_1 + \cdots + \Delta_r$. If $(*)$ is satisfied, fix $i_0$ and $j_0$ such that $\Delta_{i_0}=[-x,x]_{\rho_u}$ and $\Delta_{j_0} = [-x +1,x - 1]_{\rho_u}$; otherwise let $i_0=j_0 = -1$. Let $A_{\rho_u|\cdot|^{x}} = \{ i \in \{1,\cdots,r\}, e(\Delta_i)=x\} \setminus \{i_0\}$ and $A_{\rho_u|\cdot|^{x-1}} = \{ i \in \{1,\cdots,r\}, e(\Delta_i)=x - 1\} \setminus \{j_0\}$. 

We define the traversable relation $\leadsto$ between $A_{\rho_u|\cdot|^{x}}$ and $A_{\rho_u|\cdot|^{x-1}}$ by 
\[
    i \in A_{\rho_u|\cdot|^{x}} \leadsto j \in A_{\rho_u|\cdot|^{x-1}} \Leftrightarrow \Delta_j \le \Delta_i.
\]

Let $A_{\rho_u|\cdot|^{x}}^c$ be given by the best matching function (see Section \ref{sec:bestmatching}). If $A_{\rho_u|\cdot|^{x}}^c = \{i_1,\cdots,i_k \}$ then we fix $j_1,\cdots, j_k \in \{1,\cdots,r\}$ such that $\Delta_{j_a} = \Delta_{i_a}^{\vee}$ and $j_a \neq j_b$ if $a\neq b$. We denote by $A_{\rho_u|\cdot|^{x}}^{c,\vee}$ the set $A_{\rho_u|\cdot|^{x}}^{c,\vee} = \{j_1,\cdots,j_k \}$. The ``naive'' highest $\rho$-derivative of $(\m,\varepsilon)$ would be $(\m',\varepsilon')$ defined as follows (the actual derivative will result from a slight adjustment of this pair). Let $\m'=\sum_{i=1}^r \Delta_i'$ where
\[
    \Delta_i' = \left\{
    \begin{array}{ll}
        \Delta_i^{-}           & \text{if } i \in A_{\rho_u|\cdot|^{x}}^c \text{ and } \Delta_i \neq [-x,x]_{\rho_u} \\
        {}^{-}\Delta_i      & \text{if } i \in A_{\rho_u|\cdot|^{x}}^{c,\vee} \text{ and } \Delta_i \neq [-x,x]_{\rho_u}   \\
        {}^{-}\Delta_i^{-}     & \text{if } i \in A_{\rho_u|\cdot|^{x}}^c \text{ and } \Delta_i = [-x,x]_{\rho_u}   \\
        \Delta_i & \text{otherwise }
    \end{array}
    \right.
\]

We also need to define the signs of the centered segments of $\m'$. Let $\Delta \in \m'$ such that $c(\Delta)=0$. If $\Delta \in \m$, we define $\varepsilon'(\Delta)=\varepsilon(\Delta)$ and if not, then necessarily $\Delta=[-x+1,x-1]_{\rho_u}$ and we define $\varepsilon'([-x+1,x-1]_{\rho_u})=(-1)^{t}\varepsilon([-x,x]_{\rho_u})$.

\bigskip

Finally we can define the operator $D_{\rho_u|\cdot|^{x}}$ on $\Symm_\rho^{\varepsilon}(G)$. Let $c = \card \{ i \in A_{\rho_u|\cdot|^{x}}^c, \Delta_i= [-x,x]_{\rho_u}\}$.

\begin{defi}
  \label{def:der}
  Let $(\m,\varepsilon) \in \Symm_\rho^{\varepsilon}(G)$. We define $D_{\rho_u|\cdot|^{x}}(\m,\varepsilon) \in \Symm_\rho^{\varepsilon}(G)$ by $D_{\rho_u|\cdot|^{x}}(\m,\varepsilon)=(\m_x,\varepsilon_x)$ where
  \begin{enumerate}
    \item If $(*)$ is not satisfied, $c$ is odd and $t \ge 1$. Then
    \[
      \m_x = \m' - [-x + 1,x]_{\rho_u} - [-x,x-1]_{\rho_u} + [-x,x]_{\rho_u} + [-x + 1, x - 1]_{\rho_u}.
    \] 
    and $\varepsilon_x([-x,x]_{\rho_u})=\varepsilon([-x,x]_{\rho_u})$, and for a centered segment $\Delta \in \m_x$ different from $[-x,x]_{\rho_u}$, $\varepsilon_x(\Delta)=\varepsilon'(\Delta)$.
    \item If $(*)$ is satisfied and $c$ is odd. Then 
    \[
      \m_x = \m' - [-x,x]_{\rho_u} - [-x+1,x-1]_{\rho_u} + [-x + 1,x]_{\rho_u} + [-x,x-1]_{\rho_u}.
    \] 
    and for a centered segment $\Delta \in \m_x$, $\varepsilon_x(\Delta)=\varepsilon'(\Delta)$.
    \item Otherwise, $\m_x = \m'$ and $\varepsilon_x=\varepsilon'$.
    \end{enumerate}
\end{defi}

\begin{rem}
  The cases $(1)$ and $(2)$ can be seen as a correction of some ``mistake'' made when transforming the segments. In $(1)$ a segment $[-x,x]_{\rho_u}$ has been changed into $[-x + 1, x - 1]_{\rho_u}$ but it should have been that $[-x + 1,x]_{\rho_u} + [-x,x-1]_{\rho_u}$ is changed into $[-x + 1,x-1]_{\rho_u} + [-x+1,x-1]_{\rho_u}$. In $(2)$ one segment $[-x,x]_{\rho_u}$ has been changed into $[-x + 1, x - 1]_{\rho_u}$ and another $[-x,x]_{\rho_u}$ was unchanged but the two $[-x,x]_{\rho_u}$ should have been changed into $[-x + 1,x]_{\rho_u} + [-x,x-1]_{\rho_u}$.
\end{rem}

Using $\Symm^\varepsilon(G) =\oplus_{\rho \in \Cusp^\GL/\sim'} \Symm^\varepsilon_\rho(G)$, we can extend $D_{\rho_u|\cdot|^{x}}: \Symm_\rho^\varepsilon(G) \to \Symm_\rho^\varepsilon(G)$ to an operator $D_{\rho_u|\cdot|^{x}}: \Symm^\varepsilon(G) \to \Symm^\varepsilon(G)$ by making it act as the identity on each $\Symm_{\rho'}^\varepsilon(G)$ with $\rho' \nsim' \rho$.

\begin{prop}
  \label{prop:derpi}
  Let $\rho=\rho_u|\cdot|^{x} \in \Cusp^{\GL}$ be of good parity with $\rho_u$ unitary and $x \neq 0$. Let $\pi = L(\m,\varepsilon) \in \Irr^{G}$ with symmetrical Langlands data $(\m,\varepsilon) \in \Symm^{\varepsilon}(G)$. Then $D_{\rho}^{(k)} (\pi) = L(D_{\rho}(\m,\varepsilon))$, where $D_{\rho}^{(k)}$ is the highest $\rho$-derivative.
\end{prop}

\begin{proof}
  When $x<0$ the formula is given by \cite[Prop. 6.1.]{AM} and when $x>0$ by \cite[Thm. 7.1.]{AM}.
\end{proof}

\subsection{The bad parity case}
\label{sec:derbad}
Similarly to the previous section, we give formulas for the highest $\rho$-derivative when $\rho$ is of bad parity.

\bigskip

Let $\rho  = \rho_u |\cdot|^{x} \in \Cusp^{\GL}$ be of bad parity with $\rho_u$ unitary and $x \neq 0$. Let $(\m,\varepsilon) \in \Symm_\rho^{\varepsilon}(G)$. Let $t = \card \{ [-x,x-1]_{\rho_u} \in \m \} = \card \{ [-x +1,x]_{\rho_u} \in \m\}$. We write $\m = \Delta_1 + \cdots + \Delta_r$. Let $A_{\rho_u|\cdot|^{x}} = \{ i \in \{1,\cdots,r\}, e(\Delta_i)=x\}$ and $A_{\rho_u|\cdot|^{x-1}} = \{ i \in \{1,\cdots,r\}, e(\Delta_i)=x - 1\}$.

The relation $\leadsto$ between $A_{\rho_u|\cdot|^{x}}$ and $A_{\rho_u|\cdot|^{x-1}}$ is similar to the one for the good parity with the exception that ``a segment cannot protect its own symmetric''. This can only happen for the segments $[-x,x-1]_{\rho_u}$ and $[-x+1,x]_{\rho_u}$ and we have an issue if $t$ is odd. So if $t$ is odd we fix two indices $i_0$ and $j_0$ such that $\Delta_{i_0} = [-x +1,x]_{\rho_u}$ and $\Delta_{j_0} =[-x,x-1]_{\rho_u}$; otherwise set $i_0=j_0=-1$. We define the traversable relation $\leadsto$ between $A_{\rho_u|\cdot|^{x}}$ and $A_{\rho_u|\cdot|^{x-1}}$ by 
\[
    i \in A_{\rho|\cdot|^{x}} \leadsto j \in A_{\rho|\cdot|^{x-1}} \Leftrightarrow \Delta_j \le \Delta_i \text{ and } (i,j) \neq (i_0,j_0).
\]

Let $A_{\rho_u|\cdot|^{x}}^c$ be given by the best matching function (see Section \ref{sec:bestmatching}). If $A_{\rho_u|\cdot|^{x}}^c = \{i_1,\cdots,i_k \}$ then we fix $j_1,\cdots, j_k \in \{1,\cdots,r\}$ such that $\Delta_{j_a} = \Delta_{i_a}^{\vee}$ and $j_a \neq j_b$ if $a\neq b$. We denote by $A_{\rho_u|\cdot|^{x}}^{c,\vee}$ the set $A_{\rho_u|\cdot|^{x}}^{c,\vee} = \{j_1,\cdots,j_k \}$.
Let $\m'=\sum_{i=1}^r \Delta_i'$ where
\[
    \Delta_i' = \left\{
    \begin{array}{ll}
        \Delta_i^{-}           & \text{if } i \in A_{\rho_u|\cdot|^{x}}^c \text{ and } \Delta_i \neq [-x,x]_{\rho_u} \\
        {}^{-}\Delta_i      & \text{if } i \in A_{\rho_u|\cdot|^{x}}^{c,\vee} \text{ and } \Delta_i \neq [-x,x]_{\rho_u}   \\
        {}^{-}\Delta_i^{-}     & \text{if } i \in A_{\rho_u|\cdot|^{x}}^c \text{ and } \Delta_i = [-x,x]_{\rho_u}   \\
        \Delta_i & \text{otherwise }
    \end{array}
    \right.
\]

Finally we can define the operator $D_{\rho_u|\cdot|^{x}}$ on $\Symm_\rho^{\varepsilon}(G)$. Let $c = \card \{ i \in A_{\rho_u|\cdot|^{x}}^c, \Delta_i= [-x,x]_{\rho_u}\}$.

\begin{defi}
  \label{def:derbap}
  Let $\m \in \Symm_\rho^{\varepsilon}(G)$. We define $D_{\rho_u|\cdot|^{x}}(\m) \in \Symm_\rho^{\varepsilon}(G)$ by $D_{\rho_u|\cdot|^{x}}(\m)=\m_x$ where
  \begin{enumerate}
    \item If $c$ is odd. Then
    \[
      \m_x = \m' - [-x + 1,x-1]_{\rho_u} - [-x,x]_{\rho_u} + [-x,x-1]_{\rho_u} + [-x + 1, x]_{\rho_u}.
    \] 
    \item Otherwise, $\m_x = \m'$.
    \end{enumerate}
\end{defi}

\begin{rem}
  Case $(1)$ can be seen as a correction of some ``mistake'' made when transforming the segments. The segments $[-x,x]_{\rho_u} + [-x,x]_{\rho_u}$ have been changed into $[-x + 1, x - 1]_{\rho_u}+ [-x,x]_{\rho_u}$ but they should have been changed into $[-x + 1,x]_{\rho_u} + [-x,x-1]_{\rho_u}$.
\end{rem}

We extend $D_{\rho_u|\cdot|^{x}}: \Symm_\rho^\varepsilon(G) \to \Symm_\rho^\varepsilon(G)$ to an operator $D_{\rho_u|\cdot|^{x}}: \Symm^\varepsilon(G) \to \Symm^\varepsilon(G)$ by making it act as the identity on each $\Symm_{\rho'}^\varepsilon(G)$ with $\rho' \nsim' \rho$.

\begin{prop}
  \label{prop:derpibad}
  Let $\rho=\rho_u|\cdot|^{x} \in \Cusp^{\GL}$ be of bad parity with $\rho_u$ unitary and $x \neq 0$. Let $\pi = L(\m,\varepsilon) \in \Irr^{G}$ with symmetrical Langlands data $(\m,\varepsilon) \in \Symm^{\varepsilon}(G)$. Then $D_{\rho}^{(k)} (\pi) = L(D_{\rho}(\m,\varepsilon))$, where $D_{\rho}^{(k)}$ is the highest $\rho$-derivative.
\end{prop}

\begin{proof}
  When $x<0$ the formula is given by \cite[Prop. 6.1.]{AM} and when $x>0$ by \cite[Thm. 7.4.]{AM}.
\end{proof}

\subsection{The derivative of a sum}
The proof of the main theorem will rely on computing $D_\rho(\AD(\m,\varepsilon))$. The definition of $\AD$ being recursive, $\AD(\m,\varepsilon)=(\m_1,\varepsilon_1) + \AD(\m^{\#},\varepsilon^{\#})$, we need to relate the derivative of a multisegment of the form $(\m_1,\varepsilon_1)+(\m',\varepsilon')$ to $D_\rho(\m_1,\varepsilon_1)$ and $D_\rho(\m',\varepsilon')$. In general, the derivative of a sum does not behave nicely. However, in our setting, $(\m_1,\varepsilon_1)$, the first multisegment produced by \( \AD \), satisfies certain favorable properties that make this comparison possible.

\bigskip

Let $\rho \in \Cusp^{\GL}$. We have recalled in Sections \ref{sec:dergood} and \ref{sec:derbad} explicit formulas for the highest derivative $D^{(k)}_\rho$ on $\Symm^\varepsilon(G)$. Using \cite[Prop. 6.1.]{AM}, \cite[Thm. 7.1.]{AM} and \cite[Thm. 7.4.]{AM} we also get an explicit formula for $S^{(1)}_\rho$ on $\Symm^\varepsilon(G)$.

Let $(\m,\varepsilon) \in \Symm^\varepsilon(G)$. We define $D_\rho^{\max - 1}(\m,\varepsilon)$ by $D_\rho^{\max - 1}(\m,\varepsilon)=(\m,\varepsilon)$ if $(\m,\varepsilon)$ is $\rho$-reduced; and $D_\rho^{\max - 1}(\m,\varepsilon) = S^{(1)}_\rho \circ D_\rho(\m,\varepsilon)$ otherwise.

\begin{lem}
  \label{lem:dersum}
  Let $\rho =\rho_u |\cdot|^{x}\in \Cusp^{\GL}$ with $\rho_u$ self-dual unitary and $x \neq 0$. Let $(\m,\varepsilon) \in \Symm^{\varepsilon}(G)$. Let $(a,b) \in (x\Z)^2$ such that $a \le b$. If $\rho$ is good and $a+b=0$, let $\m_1 = [a,b]_{\rho_u}$ and $\varepsilon_1 \in \{-1,1\}$. Otherwise, let $\m_1 = [a,b]_{\rho_u} + [-b,-a]_{\rho_u}$. We assume that $(\m_1,\varepsilon_1)+(\m,\varepsilon) \in \Symm^{\varepsilon}(G)$.
  \begin{enumerate}
    \item If $b\neq x,x-1$ and $a \neq -x,-x+1$, then $D_{\rho_u|\cdot|^{x}}((\m_1,\varepsilon_1)+(\m,\varepsilon))=(\m_1,\varepsilon_1)+D_{\rho_u|\cdot|^{x}}(\m,\varepsilon)$.
    \item If $b=x$, $a \neq -x,-x+1$ and there is no segment $\Delta \in \m$ supported in $\Z_\rho$ such that $e(\Delta)=x-1$ and $\Delta \le [a,b]_{\rho_u}$, then $D_{\rho_u|\cdot|^{x}}((\m_1,\varepsilon_1)+(\m,\varepsilon))=D_{\rho_u|\cdot|^{x}}(\m_1,\varepsilon_1)+D_{\rho_u|\cdot|^{x}}(\m,\varepsilon)$.
    \item If $b=x-1$, $a \neq -x,-x+1$ and for all $\Delta \in \m$ supported in $\Z_\rho$ such that $e(\Delta)=x$ then $[a,b]_{\rho_u} \le \Delta$, then $D_{\rho_u|\cdot|^{x}}((\m_1,\varepsilon_1)+(\m,\varepsilon))=(\m_1,\varepsilon_1)+D_{\rho_u|\cdot|^{x}}^{\max - 1}(\m,\varepsilon)$.
  \end{enumerate}
\end{lem}

\begin{proof}
It follows directly from the formula of $D_{\rho|\cdot|^{x}}$.
\end{proof}

We will also need an analogous formula with $D_{Z([0,1]_{\rho_u})}$. An explicit formula for the $Z([0,1]_{\rho_u})$-derivative can be found in \cite[Algorithm A.4.]{AtobeSocle}. As before, for $(\m,\varepsilon) \in \Symm^\varepsilon(G)$ we define $D_{Z([0,1]_{\rho_u})}^{\max - 1}(\m,\varepsilon)$ by $D_{Z([0,1]_{\rho_u})}^{\max - 1}(\m,\varepsilon)=(\m,\varepsilon)$ if $(\m,\varepsilon)$ is $Z([0,1]_{\rho_u})$-reduced; and $D_\rho^{\max - 1}(\m,\varepsilon) = S^{(1)}_{Z([0,1]_{\rho_u})} \circ D_{Z([0,1]_{\rho_u})}(\m,\varepsilon)$ otherwise.

\begin{lem}
\label{lem:dersum01}
Let $\rho = \rho_u |\cdot|^{x} \in \Cusp^{\GL}$ with $\rho_u$ self-dual unitary and $x \in \frac{1}{2}\Z$ be given. Let $(\m,\varepsilon) \in \Symm^{\varepsilon}(G)$. Let $(a,b) \in (x\Z)^2$ with $a \le b$ and $a+b \le 0$. If $\rho$ is good and $a+b=0$, let $\m_1 = [a,b]_{\rho_u}$ and choose $\varepsilon_1 \in \{-1,1\}$. Otherwise, let $\m_1 = [a,b]_{\rho_u} + [-b,-a]_{\rho_u}$. We assume that $(\m_1,\varepsilon_1)+(\m,\varepsilon) \in \Symm^{\varepsilon}(G)$ and that $(\m_1,\varepsilon_1)+(\m,\varepsilon)$ is $\rho_u |\cdot|$-reduced. We also assume that $a \le \min\{b(\Delta), \Delta \in \m\}$ and that for all $\Delta \in \m$ with $b(\Delta)=a$, $l([a,b]) \ge l(\Delta)$.
\begin{enumerate}
\item If $a, b \notin \{0,1,-1\}$, then
\[
D_{Z([0,1]_{\rho_u})}((\m_1,\varepsilon_1)+(\m,\varepsilon)) = (\m_1,\varepsilon_1)+D_{Z([0,1]_{\rho_u})}(\m,\varepsilon).
\]
\item If $b=0$, $a \neq -1,0$ and $(\m,\varepsilon)$ is $1$-reduced, then
\[
D_{Z([0,1]_{\rho_u})}((\m_1,\varepsilon_1)+(\m,\varepsilon)) = (\m_1,\varepsilon_1)+D_{Z([0,1]_{\rho_u})}(\m,\varepsilon).
\]
\item If $b=0$, $a \neq -1,0$ and $(\m,\varepsilon)$ is not $1$-reduced, then
\[
D_{Z([0,1]_{\rho_u})}((\m_1,\varepsilon_1)+(\m,\varepsilon))
= [a,-1]_{\rho_u} + [1,-a]_{\rho_u} + D_{Z([0,1]_{\rho_u})}(D_{\rho_u|\cdot|}(\m,\varepsilon)).
\]
\item If $b=-1$, then
\[
D_{Z([0,1]_{\rho_u})}((\m_1,\varepsilon_1)+(\m,\varepsilon))
= (\m_1,\varepsilon_1)+D_{Z([0,1]_{\rho_u})}^{\max - 1}(\m,\varepsilon).
\]
\end{enumerate}
\end{lem}

\begin{proof}
Part (1) follows immediately from \cite[Algorithm A.4]{AtobeSocle}. Let us focus on the other three parts. Although one could also use \cite[Algorithm A.4]{AtobeSocle}, we instead provide a simple proof using basic facts from representation theory. We will repeatedly use the following simple lemma:

\begin{lem}\label{lem:simple}
Let $\tau \in \Irr^\GL$, $\sigma \in \Irr^G$ such that:
\begin{enumerate}
\item The induced representation $\tau \rtimes \sigma$ is SI.
\item $\tau$ is left-$\rho_u |\cdot|$-reduced and $\sigma$ is $\rho_u |\cdot|^{x}$-reduced.
\item $\tau$ is left-$Z([0,1]_{\rho_u})$-reduced and $\sigma$ is $Z([0,1]_{\rho_u})$-reduced.
\end{enumerate}
Then,
\[
Z([0,1]_{\rho_u}) \times \tau \rtimes \sigma
\]
is SI.
\end{lem}

\begin{proof}
The conditions imply that $Z([0,1]_{\rho_u}) \otimes \soc(\tau \rtimes \sigma)$ appears with multiplicity $1$ in
\[
\Jac_P^G(Z([0,1]_{\rho_u}) \times \tau \rtimes \sigma)
\]
for some parabolic subgroup $P$ of a certain $G$. This implies the claim.
\end{proof}

Denote by $\pi$ (\resp $\pi'$) the representation associated to $(\m,\varepsilon)$ (\resp $(\m_1,\varepsilon_1)+(\m,\varepsilon)$). Then, by our hypotheses, $\pi'$ is the socle of $L([a,b]_{\rho_u})\rtimes \pi$.

Assume first that $b=0$ and $\pi$ is $\rho_u |\cdot|$-reduced. If we denote by $k$ the integer such that $D_{Z([0,1]_{\rho_u})}(\pi)=D^{(k)}_{Z([0,1]_{\rho_u})}(\pi)$, we have
\[
L([a,b]_{\rho_u}) \times Z([0,1]_{\rho_u})^k \rtimes D_{Z([0,1]_{\rho_u})} (\pi)
\]
is SI. Indeed, as $b=0$, by \cite[Thm. 0.1]{BLM}, it is isomorphic to
\[
Z([0,1]_{\rho_u})^k \times L([a,b]_{\rho_u}) \rtimes D_{Z([0,1]_{\rho_u})} (\pi),
\]
and the claim follows from Lemma \ref{lem:simple}. As $\pi'$ is clearly in the socle, we deduce
\[
D_{Z([0,1]_{\rho_u})} (\pi')
= \soc(L([a,b]_{\rho_u}) \rtimes D_{Z([0,1]_{\rho_u})} (\pi)).
\]
This proves (2).

Now assume that $b=0$ but $\pi$ is not $\rho_u |\cdot|$-reduced. Then, by \cite[Thm. 0.1]{BLM}, we have
\begin{align}
L([a,b]_{\rho_u}+[1,1]_{\rho_u}) \times Z([0,1]_{\rho_u})^k \rtimes D_{Z([0,1]_{\rho_u})} D_{\rho_u|\cdot|} (\pi)
& \simeq \label{eq:nuevo} \\
Z([0,1]_{\rho_u})^k \times L([a,b]_{\rho_u}+[1,1]_{\rho_u}) \rtimes D_{Z([0,1]_{\rho_u})} D_{\rho_u|\cdot|} (\pi). \notag
\end{align}
This representation embeds into
\[
Z([0,1]_{\rho_u})^{k+1} \times L([a,b-1]_{\rho_u}) \rtimes D_{Z([0,1]_{\rho_u})} D_{\rho_u|\cdot|} (\pi).
\]
Lemma \ref{lem:simple} implies that this induced is SI and so is \eqref{eq:nuevo}. But again this socle is isomorphic to $\pi'$. As proving it requires some more work, let's give some details. First see that, as $\pi'$ embeds in $ L([a,b]_{\rho_u}) \times \rho_u|\cdot| \rtimes D_{\rho_u|\cdot|} (\pi)$, it embeds in $\sigma_0 \rtimes D_{\rho_u|\cdot|} (\pi)$ for some subquotient $\sigma_0$ of $ L([a,b]_{\rho_u}) \times \rho_u|\cdot| $. But $\pi'$ is $\rho_u |\cdot|$-reduced, so we deduce that $\sigma_0 \simeq L([a,b]_{\rho_u}+[1,1]_{\rho_u})$. Therefore,
\begin{align*}
\soc(L([a,b]_{\rho_u}+[1,1]_{\rho_u}) \times Z([0,1]_{\rho_u})^k \rtimes D_{Z([0,1]_{\rho_u})} D_{\rho_u|\cdot|} (\pi))
& \simeq \\
\soc(L([a,b]_{\rho_u}+[1,1]_{\rho_u}) \times \soc(Z([0,1]_{\rho_u})^k \rtimes D_{Z([0,1]_{\rho_u})} D_{\rho_u|\cdot|} (\pi)))
& \simeq \\
\soc(L([a,b]_{\rho_u}+[1,1]_{\rho_u}) \times D_{\rho_u|\cdot|} (\pi)),
\end{align*}
which contains $\pi'$ and is thus isomorphic to it. This proves (3).

Finally, assume $b=-1$. Then $\pi'$ is a subrepresentation of
\[
L([a,b]_{\rho_u}) \times Z([0,1]_{\rho_u})^k \rtimes D_{Z([0,1]_{\rho_u})} (\pi).
\]
Here, $L([a,b]_{\rho_u}) \times Z([0,1]_{\rho_u})^k$ is not irreducible but has length $2$, with composition factors $\soc(Z([0,1]_{\rho_u}) \times L([a,b]_{\rho_u})) \times Z([0,1]_{\rho_u})^{k-1}$ and $\soc(L([a,b]_{\rho_u}) \times Z([0,1]_{\rho_u})) \times Z([0,1]_{\rho_u})^{k-1}$. By our assumption on $\m_1$, we must have that $\pi'$ is a subrepresentation of
\begin{align*}
\soc(L([a,b]_{\rho_u}) \times Z([0,1]_{\rho_u})) \times Z([0,1]_{\rho_u})^{k-1} \rtimes D_{Z([0,1]_{\rho_u})} (\pi)
& \simeq \\
Z([0,1]_{\rho_u})^{k-1} \times \soc(L([a,b]_{\rho_u}) \times Z([0,1]_{\rho_u})) \rtimes D_{Z([0,1]_{\rho_u})} (\pi).
\end{align*}
By Lemma \ref{lem:simple}, this induced representation is SI, and as before, this proves (4).
\end{proof}

\section{Proof in the good parity case}
\label{sec:proofgood}

The goal of this section is to prove Theorem~\ref{thm:MAIN} in the case of an irreducible representation of good parity. To do so, we will work with symmetric Langlands data throughout this section. More precisely, we will establish the following equivalent formulation of the theorem.

\begin{thm}
  \label{thm:ADAubertDual}
  Let $\pi \in \Irr^{G}$ be of good parity with symmetrical Langlands data $\pi \simeq L(\m,\varepsilon)$. Then we have
  \[
    \hat{\pi} \simeq L(\AD(\m,\varepsilon)).
  \]
\end{thm}

The strategy to prove Theorem~\ref{thm:ADAubertDual} is as follows. We have already seen in Proposition~\ref{prop:derivativedual} that the derivatives and the Aubert--Zelevinsky involution are compatible. This allows us to prove the theorem by induction on the length of \(\m\), using the fact that derivatives are injective and reduce the length.

The base case is when $\pi$ is $\rho$-reduced for all $\rho$, and also $L([-1,0]_\rho)$-reduced; this case is treated in Section~\ref{sec:reducedseg}. Now suppose that $\pi$ is not $\rho$-reduced for some $\rho$ (the case where it is not $L([-1,0]_\rho)$-reduced is similar). Let \((\m, \varepsilon) \in \Symm^{\varepsilon,+,1}(G)\) such that \(\pi = L(\m, \varepsilon)\). Then, if we can prove that
\[
D_{\rho}(\AD(\m, \varepsilon)) = \AD(D_{\rho^{\vee}}(\m, \varepsilon)),
\]
the result follows by induction, as shown in the following computation:
\begin{align*}
    D_{\rho}(L(\AD(\m,\varepsilon))) & = L(D_{\rho}(\AD(\m,\varepsilon))) \text{ by Lemma \ref{prop:derpi}}\\
 & = L(\AD(D_{\rho^{\vee}}(\m,\varepsilon))) \\
 & = L(D_{\rho^{\vee}}(\m,\varepsilon))^{\widehat{}} \text{ by the induction hypothesis} \\
 & = D_{\rho^{\vee}}(L(\m,\varepsilon))^{\widehat{}} \text{ by Lemma \ref{prop:derpi}} \\
 & = D_{\rho} (\hat{\pi}) \text{ by Proposition \ref{prop:derivativedual}} \\
  \end{align*}
and we conclude by the injectivity of $D_\rho$.

The key point in this argument is the compatibility between the derivative functor and \(\AD\): namely,
\[
D_{\rho}(\AD(\m, \varepsilon)) = \AD(D_{\rho^{\vee}}(\m, \varepsilon)).
\]
Establishing this identity is the main goal of this section.

However, a technical difficulty arises. We would like to prove this identity by induction using the recursive definition $ \AD(\m, \varepsilon) = (\m_1, \varepsilon_1) + \AD(\m^{\#}, \varepsilon^{\#})$, but this recursion is only valid when \((\m, \varepsilon) \in \Symm^{\varepsilon,+,1}(G)\), and unfortunately, \((\m^{\#}, \varepsilon^{\#})\) may not belong to \(\Symm^{\varepsilon,+,1}(G)\). As a result, we are led to prove a slightly modified version of the statement above, which we will explain in detail below.

\bigskip

If $\mathscr{C}=\{\rho_1,\cdots,\rho_r\} \subset \Cusp^{\GL}$ is a finite set composed of good supercuspidals, we denote by $\Symm_{\mathscr{C}}^{\varepsilon}(G) = \bigoplus_{\rho \in \mathscr{C}} \Symm_{\rho}^{\varepsilon}(G)$. 

We fix $\mathscr{C}=\{\rho_1,\cdots,\rho_r\} \subset \Cusp^{\GL}/{\sim'}$ a set composed of good supercuspidals, $\rho_0 \in \Cusp^{\GL} \setminus \mathscr{C}$ be a good self-dual supercuspidal and $\chi$ a character of $F^{\times}$ of order 2. For a multisegment $\m \in \Mult$, we denote by $l(\m)$ the length of $\m$. We will prove the following theorem.

\begin{thm}
  \label{thm:ADdualinduction}
  Let $N \in \NN$. Let $\pi \in \Irr^{G}$ be of good parity with symmetrical Langlands data $(\m',\varepsilon') \in \Symm^{\varepsilon,+,1}(G)$. We assume that there exist $(\m,\varepsilon) \in \Symm_{\mathscr{C}}^{\varepsilon}(G)$ and $(\m_0,\varepsilon_0) \in \Symm_{\{\rho_0,\chi \rho_0 \}}^{\varepsilon}(G)$ such that 
  \begin{enumerate}
    \item $(\m',\varepsilon')=(\m,\varepsilon) + (\m_0,\varepsilon_0)$;
    \item $l(\m) \le N$;
    \item $l(\m_0) \le 1$.
  \end{enumerate}
  Then $\hat{\pi} \simeq L(\AD(\m',\varepsilon'))$.
\end{thm}

If we prove \ref{thm:ADdualinduction} for all finite subsets $\mathscr{C} \in \Cusp^{\GL}$ and all $N \in \NN$, then Theorem \ref{thm:ADdualinduction} implies Theorem \ref{thm:ADAubertDual}.

\subsection{The case of a reduced multisegment}
\label{sec:reducedseg}

Let \((\m, \varepsilon) \in \Symm^{\varepsilon}\). Let  $\rho \in \Cusp^{\GL}$ be self-dual. For $x\neq 0$, we say that \((\m, \varepsilon)\) is $\rho|\cdot|^{x}$-reduced if \(D^{(1)}_{\rho|\cdot|^{x}}(\m, \varepsilon) \neq 0\). If \((\m, \varepsilon)\) is $\rho|\cdot|^{-1}$-reduced, we say that \((\m, \varepsilon)\) is $L([-1,0]_\rho)$-reduced if \(D^{(1)}_{L([-1,0]_\rho)}(\m, \varepsilon) \neq 0\). Finally, we say that \((\m, \varepsilon)\) is reduced if for every self-dual $\rho$ and $x \neq 0$, it is $\rho|\cdot|^{x}$-reduced and $L([-1,0]_\rho)$-reduced.

In this section we prove Theorem \ref{thm:ADAubertDual} for reduced multisegments.

\bigskip

Let $\rho \in \Cusp^{\GL}$ be self-dual and of the same type as $G$. Let $n_0, y_0 \in \NN$, with $n_0 \ge 1$. We define $(\m,\varepsilon) \in \Symm_\rho^{\varepsilon}(G)$ by $\m = n_0 [0,0]_\rho + \sum_{y=1}^{y_0}[-y,y]_\rho$ and for all $1 \le y \le y_0$, $\varepsilon([-y,y]_\rho) \varepsilon([-y+1,y-1]_\rho)=-1$ (we do not fix any condition on $\varepsilon([0,0]_\rho)$). 

\begin{lem}
  \label{lem:ADmetemp}
  We have that $\AD_\rho(\m,\varepsilon)=(\m',\varepsilon')$ where
  \begin{enumerate}
      \item If $y_0 = 0$, then $\m'=\m$ and $\varepsilon'([0,0]_\rho)= (-1)^{n_0 + 1}\varepsilon([0,0]_\rho)$.
      \item If $n_0$ is odd and $y_0 \neq 0$, then $(\m',\varepsilon')=(\m,\varepsilon)$. 
      \item If $n_0$ is even and $y_0 \neq 0$, then $\m'= \m - [0,0]_\rho - [-y_0,y_0]_\rho + [-y_0,0]_\rho + [0,y_0]_\rho$ and for $y<y_0$, $\varepsilon'([-y,y]_\rho)=-\varepsilon([-y,y]_\rho)$.
  \end{enumerate}
\end{lem}

\begin{proof}
\begin{enumerate}
  \item Suppose $y_0 = 0$. Then $\m_1=[0,0]_\rho$, $\varepsilon_1([0,0]_\rho)=(-1)^{n_0 + 1} \varepsilon([0,0]_\rho)$, $\m^{\#}=(n_0 - 1) [0,0]_\rho $ and $\varepsilon^{\#}([0,0]_\rho)=-\varepsilon([0,0]_\rho)$. We get the result by induction.
  \item Suppose $n_0$ is odd and $y_0 \neq 0$. Then $\m_1=[-y_0,y_0]_\rho$, $\varepsilon_1([-y_0,y_0]_\rho)=\varepsilon([-y_0,y_0]_\rho)$, $\m^{\#}=n_0 [0,0]_\rho + \sum_{y=1}^{y_0-1}[-y,y]_\rho$ and for $y<y_0$, $\varepsilon^{\#}([-y,y]_\rho)=\varepsilon([-y,y]_\rho)$. By induction, we see that $\AD(\m_0,\varepsilon_0)=(\m_0,\varepsilon_0)$. 
  \item Suppose $n_0$ is even and $y_0 \neq 0$. Then $\m_1=[-y_0,0]_\rho + [0,y_0]_\rho$, $\m^{\#}=(n_0 - 1) [0,0]_\rho + \sum_{y=1}^{y_0-1}[-y,y]_\rho$ and for $y<y_0$, $\varepsilon^{\#}([-y,y]_\rho)=-\varepsilon([-y,y]_\rho)$. We get the result using the previous case. 
\end{enumerate}
\end{proof}

Now, let $\rho \in \Cusp^{\GL}$ be self-dual of the opposite type as $G$. Let $y_0 \in (1/2)\NN \setminus \NN$. We define $(\m,\varepsilon) \in \Symm_\rho^{\varepsilon}$ by $\m = \sum_{y=1/2}^{y_0}[-y,y]_\rho$, $\varepsilon([-1/2,1/2]_\rho)=-1$ and for all $1/2 < y \le y_0$, $\varepsilon([-y,y]_\rho) \varepsilon([-y+1,y-1]_\rho)=-1$. 

\begin{lem}
  \label{lem:ADmetempnottype}
  We have that $\AD_{\rho|\cdot|^{1/2}}(\m,\varepsilon)=(\m,\varepsilon)$.
\end{lem}

\begin{proof}
Applying the algorithm, we get that $\m_1=[-y_0,y_0]_\rho$ and $\m^{\#}=\sum_{y=1/2}^{y_0-1}[-y,y]_\rho$. The result follows directly by induction.
\end{proof}

We can now prove Theorem \ref{thm:ADAubertDual} for reduced $(\m,\varepsilon)$.

\begin{lem}
  \label{lem:ADdualtemp}
  Let $\pi = L(\m,\varepsilon) \in \Irr^{G}$ be of good parity with $(\m,\varepsilon) \in \Symm^{\varepsilon,+,1}(G)$. If $(\m,\varepsilon)$ is reduced then $\hat{\pi} \simeq L(\AD(\m,\varepsilon))$.
\end{lem}

\begin{proof}
  We can write $(\m,\varepsilon)=\sum_{\rho_i} (\m_{\rho_i},\varepsilon_{\rho_i})$ such that the $\rho_i \in \Cusp^{\GL}$ are good, $\Z_{\rho_i} \neq \Z_{\rho_j}$ if $i \neq j$, and $(\m_{\rho_i},\varepsilon_{\rho_i}) \in \Symm_{\rho_i}^{\varepsilon}$. By definition, $\AD(\m,\varepsilon)=\sum_{\rho_i} \AD_{\rho_i}(\m_{\rho_i},\varepsilon_{\rho_i})$. The hypothesis made on $\pi$ imply that each $(\m_{\rho_i},\varepsilon_{\rho_i})$ satisfies the hypotheses of Lemma \ref{lem:ADmetemp} or Lemma \ref{lem:ADmetempnottype}. These lemmas explicitly compute $\AD(\m,\varepsilon)$. The representation $\hat\pi$ is explicitly computed in \cite[Prop. 5.4]{AM}. The two results are identical, proving the result.
\end{proof}

Because we want to prove Theorem \ref{thm:ADdualinduction}, we will also need the following lemma.

\begin{lem}
  \label{lem:ADtemprho}
  Let $\pi = L(\m',\varepsilon') \in \Irr^{G}$ be of good parity with $(\m',\varepsilon') \in \Symm^{\varepsilon,+,1}(G)$. We assume that there exist $(\m,\varepsilon) \in \Symm_{\mathscr{C}}^{\varepsilon}(G)$ and $(\m_0,\varepsilon_0) \in \Symm_{\{\rho_0,\chi\rho_0\}}^{\varepsilon}(G)$ such that 
  \begin{enumerate}
    \item $(\m',\varepsilon')=(\m,\varepsilon) + (\m_0,\varepsilon_0)$;
    \item $(\m,\varepsilon)$ is reduced;
    \item $l(\m_0) \le 1$.
  \end{enumerate}
  Then $\hat{\pi} \simeq L(\AD(\m',\varepsilon'))$.
\end{lem}

\begin{proof}
  The only elements $(\m_0,\varepsilon_0) \in \Symm_{\{\rho_0,\chi\rho_0\}}^{\varepsilon}$ of good parity with $l(\m_0,\varepsilon_0) \le 1$, are $\m_0 = 0$, $\m_0 = [0,0]_{\rho_0}$ and $\m_0 = [0,0]_{\chi\rho_0}$ with $\rho_0$ of the same type as $G$. They are all reduced, so the result follows from Lemma \ref{lem:ADdualtemp}.
\end{proof}

\subsection{The strategy of the proof}
\label{sec:startegyproof}

We explain here the strategy to prove Theorem \ref{thm:ADdualinduction}.

\bigskip

We prove Theorem \ref{thm:ADdualinduction} by induction on $N$. The case $N=0$ is handled by Lemma \ref{lem:ADtemprho}. Let $N \in \NN^{*}$. 

\begin{hypo}
  \label{hyp:recurrence}
We assume that Theorem \ref{thm:ADdualinduction} is true for all $N' < N$.
\end{hypo}
Until the end of the section, we will assume that Hypothesis \ref{hyp:recurrence} is true. We want to prove now that Theorem \ref{thm:ADdualinduction} is true for $N$. To do that, we will prove that the algorithm $\AD$ commutes with derivatives.

\begin{lem}
  \label{lem:commDerthm}
  We assume that for all non-reduced $(\m,\varepsilon) \in \Symm_{\mathscr{C}}^{\varepsilon}(G)$ with $l(\m) = N$, there exist $\rho_i \in \mathscr{C}$, $\rho \in \Cusp^{\GL}$ self-dual and $x \neq 0$ such that $\rho|\cdot|^{x} \in \Z_{\rho_i}$ and we have  either 
  \begin{enumerate}
    \item $(\m,\varepsilon)$ is not $\rho|\cdot|^{x}$-reduced and $\AD_{\rho_i}(D_{\rho|\cdot|^{x}}(\m,\varepsilon))=D_{\rho|\cdot|^{-x}}(\AD_{\rho_i}(\m,\varepsilon))$
    \item or, if it is defined, $(\m,\varepsilon)$ is not $L([-1,0]_\rho)$-reduced and $\AD_{\rho_i}(D_{L([-1,0]_\rho)}(\m,\varepsilon))=D_{Z([0,1]_\rho)}(\AD_{\rho_i}(\m,\varepsilon))$.
  \end{enumerate}
  Then Theorem \ref{thm:ADdualinduction} is true for $N$.
\end{lem}

\begin{proof}
  Let $\pi$, $(\m',\varepsilon')$, $(\m,\varepsilon)$ and $(\m_0,\varepsilon_0)$ as in Theorem \ref{thm:ADdualinduction}. If $(\m,\varepsilon)$ is reduced, Theorem \ref{thm:ADdualinduction} follows from Lemma \ref{lem:ADtemprho}. Hence, we can assume that $(\m,\varepsilon)$ is not reduced. By hypothesis, there exist $\rho_i \in \mathscr{C}$, $\rho \in \Cusp^{\GL}$ self-dual and $x \neq 0$ such that $(\m,\varepsilon)$ is not $\rho|\cdot|^{x}$-reduced and $\AD_{\rho_i}(D_{\rho|\cdot|^{x}}(\m,\varepsilon))=D_{\rho|\cdot|^{-x}}(\AD_{\rho_i}(\m,\varepsilon))$; or $(\m,\varepsilon)$ is not $L([-1,0]_\rho)$-reduced and $\AD_{\rho_i}(D_{L([-1,0]_\rho)}(\m,\varepsilon))=D_{Z([0,1]_\rho)}(\AD_{\rho_i}(\m,\varepsilon))$. We assume that $(\m,\varepsilon)$ is not $\rho|\cdot|^{x}$-reduced and $\AD_{\rho_i}(D_{\rho|\cdot|^{x}}(\m,\varepsilon))=D_{\rho|\cdot|^{-x}}(\AD_{\rho_i}(\m,\varepsilon))$, the other cases being treated similarly. Note that $\AD=\oplus_\rho \AD_\rho$ and similarly for the derivative. Thus we have $\AD(D_{\rho|\cdot|^{x}}(\m',\varepsilon'))=D_{\rho|\cdot|^{-x}}(\AD(\m',\varepsilon'))$. Then we get
  \begin{align*}
    D_{\rho|\cdot|^{-x}}(L(\AD(\m',\varepsilon'))) & = L(D_{\rho|\cdot|^{-x}}(\AD(\m',\varepsilon'))) \text{ by Lemma \ref{prop:derpi}}\\
 & = L(\AD(D_{\rho|\cdot|^{x}}(\m',\varepsilon'))) \\
 & = L(D_{\rho|\cdot|^{x}}(\m',\varepsilon'))^{\widehat{}} \text{ by Hypothesis \ref{hyp:recurrence}} \\
 & = D_{\rho|\cdot|^{x}}(L(\m',\varepsilon'))^{\widehat{}} \text{ by Lemma \ref{prop:derpi}} \\
 & = D_{\rho|\cdot|^{-x}} (\hat{\pi}) \text{ by \cite[Prop. 3.9.]{AM}} \\
  \end{align*}
  
By the injectivity of $D_{\rho|\cdot|^{-x}}$, we get that $\hat{\pi}=L(\AD(\m',\varepsilon'))$.
\end{proof}

Let $(\m,\varepsilon) \in \Symm_{\mathscr{C}}^{\varepsilon}(G)$ non-reduced with $l(\m) = N$. We will show that the hypotheses of Lemma \ref{lem:commDerthm} are satisfied. First note that we have the following result.

\begin{lem}
  \label{lem:ADcommDer}
  For all $(\m,\varepsilon) \in \Symm_{\mathscr{C}}^{\varepsilon}(G)$ such that $l(\m) < N$, for all $\rho_i \in \mathscr{C}$, $\rho \in \mathscr{C}$ self-dual and $x \neq 0$,
  \[
    \AD_{\rho_i}(D_{\rho|\cdot|^{x}}(\m,\varepsilon))=D_{\rho|\cdot|^{-x}}(\AD_{\rho_i}(\m,\varepsilon))
  \]
  and, if it is well-defined,
  \[
    \AD_{\rho_i}(D_{L([-1,0]_\rho)}(\m,\varepsilon))=D_{Z([0,1]_\rho)}(\AD_{\rho_i}(\m,\varepsilon)).
  \]
\end{lem}

\begin{proof}
  We can assume that $\rho|\cdot|^{x} \in \Z_{\rho_i}$. If $(\m,\varepsilon)\in \Symm^{\varepsilon,+,1}(G)$, by Hypothesis \ref{hyp:recurrence}, $\widehat{L(\m,\varepsilon)}=L(\AD(\m,\varepsilon))$. Thus the result follows from Proposition \ref{prop:derivativedual}. If not, we consider $\m_0=[0,0]_{\chi^{i}\rho_0}$, with $i\in \{0,1\}$ such that $\det(\chi^{i}\rho_0) = \det(\m)$, and $\varepsilon_0(\m_0)=S(\m,\varepsilon)$. Let $(\m',\varepsilon')=(\m,\varepsilon) + (\m_0,\varepsilon_0)$. Our choice of $(\m_0,\varepsilon_0)$ is made such that $(\m',\varepsilon') \in \Symm^{\varepsilon,+,1}(G)$. The element $(\m',\varepsilon')$ satisfies the conditions of Theorem \ref{thm:ADdualinduction}, thus $\widehat{L(\m',\varepsilon')}=L(\AD(\m',\varepsilon'))$. We also get the result from Proposition \ref{prop:derivativedual} and projecting on the line $\Z_{\rho_i}$.
\end{proof}

\begin{lem}
  \label{lem:ADcommSoc}
  For all $(\m,\varepsilon) \in \Symm_{\mathscr{C}}^{\varepsilon}(G)$ such that $l(\m) < N - 2$, for all $\rho_i \in \mathscr{C}$, $\rho \in \mathscr{C}$ self-dual and $x \neq 0$,
  \[
    \AD_{\rho_i}(S_{\rho|\cdot|^{x}}^{(1)}(\m,\varepsilon))=S_{\rho|\cdot|^{-x}}^{(1)}(\AD_{\rho_i}(\m,\varepsilon))
  \]
  and, if it is well-defined,
  \[
    \AD_{\rho_i}(S_{L([-1,0]_\rho)}^{(1)}(\m,\varepsilon))=S_{Z([0,1]_\rho)}^{(1)}(\AD_{\rho_i}(\m,\varepsilon)).
  \]
\end{lem}

\begin{proof}
  The proof is similar to the proof of Lemma \ref{lem:ADcommDer} using \cite[Thm. 31 (4)]{Bern}
\end{proof}

The multisegment $(\m,\varepsilon)$ can be written as $(\m,\varepsilon) = \sum_{i=1}^{r} (\m_i,\varepsilon_i)$, with $(\m_i,\varepsilon_i) \in \Symm_{\rho_i}^{\varepsilon}(G)$. If there exists $i \neq j$ such that $\m_i \neq 0$ and $\m_j \neq 0$, then, for all $1 \le i \le r$, $l(\m_i)<N$, and Lemma \ref{lem:ADcommDer} shows that the conditions of Lemma \ref{lem:commDerthm} are satisfied. Hence we can assume that $\mathscr{C}$ is composed of a single supercuspidal (that is $r=1$). We fix $\rho \in \Cusp^{\GL}$ of good parity and assume that $\mathscr{C} = \{\rho\}$. Let $\rho_u$ be the unitarization of $\rho$. Let $(\m,\varepsilon) \in \Symm_\rho^\varepsilon(G)$.

To simplify the notations, until the end of Section \ref{sec:proofgood}, we will write all the segments with respect to $\rho_u$ and we will omit $\rho$ and $\rho_u$ in the notations. That is $\AD := \AD_\rho$, $[x,y]:=[x,y]_{\rho_u}$, $D_x := D_{\rho_u|\cdot|^{x}}$, $D_{Z([0,1])}:=D_{Z([0,1]_{\rho_u})}$ and $D_{L([-1,0])}:=D_{L([-1,0]_{\rho_u})}$. We will also say that $(\m,\varepsilon)$ is $x$-reduced if it is $\rho_u|\cdot|^{x}$-reduced, and similarly for $L([-1,0])$-reduced and $Z([-1,0])$-reduced.

\bigskip

The goal is to find a suitable value of $y_0$ such that \( D_{y_0}(\m, \varepsilon) \) is easy to compute, and the initial sequence of \( D_{y_0}(\m, \varepsilon) \) remains relatively close to that of \( (\m, \varepsilon) \), in order to control the effect of \( \AD \). The simplest way to ensure that \( D_{y_0}(\m, \varepsilon) \) is easy to compute is to choose $y_0$ as the smallest half-integer $y \in (1/2) \Z^*$ such that \( (\m, \varepsilon) \) is not $y$-reduced. However, this choice can significantly alter the initial sequence of \( (\m, \varepsilon) \), especially when \( [-y_0, -y_0] \in \m \). In that case, this segment, which necessarily appears first in the initial sequence of \( (\m, \varepsilon) \), is removed by the derivative operator $D_{y_0}$. To avoid this issue, we instead choose $y_0$ as the smallest half-integer $y \in (1/2) \Z^*$ such that $y \neq -e_{\max}$ and \( (\m, \varepsilon) \) is not $y$-reduced. This provides a good compromise between the simplicity of the derivative's expression and control over the initial sequence.

The proof is divided into several subsections, as explained below:
\begin{enumerate} 
  \item By a direct computation, we deal with the case $e_{\max} \le 1$ in Section \ref{sec:emax1}.
  \item In Section \ref{sec:derneg}, we treat the case where $e_{\max} > 1$ and there exists $y<0$, $y \neq -e_{\max}$ such that $(\m,\varepsilon)$ is not $y$-reduced.
  \item In Section \ref{sec:derl01}, we prove the result when $e_{\max} > 1$, for all $-e_{\max} < y<0$, $(\m,\varepsilon)$ is $y$-reduced, and $(\m,\varepsilon)$ is not $L([-1,0])$-reduced.
  \item The Section \ref{sec:derpos} deals with the case $e_{\max} > 1$, $\rho$ of the same type as $G$, for all $-e_{\max} < y<0$, $(\m,\varepsilon)$ is $y$-reduced, $(\m,\varepsilon)$ is $L([-1,0])$-reduced, and there exists $y>0$ with $y \neq e_{\max}$ such that $(\m,\varepsilon)$ is not $y$-reduced.
  \item The Section \ref{sec:derposnottype} deals with the case $e_{\max} > 1$, $\rho$ not of the same type as $G$, for all $-e_{\max} < y<0$, $(\m,\varepsilon)$ is $y$-reduced, $(\m,\varepsilon)$ is $L([-1,0])$-reduced, and there exists $y>0$ with $y \neq e_{\max}$ such that $(\m,\varepsilon)$ is not $y$-reduced.
  \item In Section \ref{sec:deremax} we assume that $e_{\max} > 1$, for all $y \neq 0,e_{\max}, -e_{\max}$, $(\m,\varepsilon)$ is $y$-reduced, $(\m,\varepsilon)$ is $L([-1,0])$-reduced, and $(\m,\varepsilon)$ is not $e_{\max}$-reduced.
  \item Finally in Section \ref{sec:deremaxm} we assume that $e_{\max} > 1$, for all $y \neq 0, -e_{\max}$, $(\m,\varepsilon)$ is $y$-reduced and $(\m,\varepsilon)$ is $L([-1,0])$-reduced.
\end{enumerate}

\subsection{The case $e_{\max} \le 1$} \label{sec:emax1}

In this section we assume that $e_{\max} \le 1$. The goal is to compute $\AD(\m,\varepsilon)$ explicitly.

\bigskip

We start with the easiest case which is when $\rho$ is not of the same type as $G$. Then $\m$ is of the form
\[
  \m = c [-1/2,1/2] + n ([-1/2,-1/2] + [1/2,1/2])
\]
with $c,n \in \NN$. To simplify notation, set $\varepsilon=\varepsilon([-1/2,1/2])$.

Let $(\m',\varepsilon')=\AD(\m,\varepsilon)$. The maximum of the coefficients of $(\m',\varepsilon')$ is also smaller than $1$, so $(\m',\varepsilon')$ is of the form as above. We denote by $c',n' \in \NN$ and $\varepsilon'\in \{\pm 1\}$ the constants relative to $(\m',\varepsilon')$.

Let $(*)$ be the condition $c \neq 0$ and $\varepsilon= (-1)^{n+1}$.

\begin{prop}
  \label{prop:ADeunnt}  
  The dual $\AD(\m,\varepsilon)=(\m',\varepsilon')$ is given by the following formulas.
  \begin{enumerate}
      \item If $(*)$ is satisfied, then $c'=n+1$, $n'=c-1$ and $\varepsilon' =  (-1)^{c}$.
      \item If $(*)$ is not satisfied, then $c'=n$, $n'=c$ and $\varepsilon' =  (-1)^{c}$.
 \end{enumerate}
\end{prop}

\begin{proof}
  We write $\m_1$ and $\m^{\#}$ in the different cases and the result follows by an immediate induction.  
    \begin{itemize}
        \item Suppose $n \neq 0$. Then $\Delta_1= [1/2,1/2]$.  Thus $\m_1 = [-1/2,1/2]$, $\varepsilon_1([-1/2,1/2]) = (-1)^{c}$, $\m^{\#}= c [-1/2,1/2] + (n-1) ([-1/2,-1/2] + [1/2,1/2])$ and $\varepsilon^{\#}([-1/2,1/2]) = -\varepsilon([-1/2,1/2])$. 
        \item Suppose $n = 0$ and $\varepsilon=-1$. Then $\Delta_1= [-1/2,1/2]$.  Thus $\m_1 = [-1/2,1/2]$, $\varepsilon_1([-1/2,1/2]) = (-1)^{c}$, $\m^{\#}= (c-1) [-1/2,1/2]$ and $\varepsilon^{\#}([-1/2,1/2]) = -\varepsilon([-1/2,1/2])$.
        \item Suppose $n = 0$ and $\varepsilon=1$. Then $\Delta_1= [-1/2,1/2]$.  Thus $\m_1 = [-1/2,-1/2]+[1/2,1/2]$, $\m^{\#}= (c-1) [-1/2,1/2]$ and $\varepsilon^{\#}([-1/2,1/2]) = \varepsilon([-1/2,1/2])$.
    \end{itemize}
  \end{proof}

  Now, we check that $\AD$ commutes with the derivative.

  \begin{prop}
    \begin{enumerate}
      \item If $n \neq 0$, then $\AD(D_{-1/2}(\m,\varepsilon))=D_{1/2}(\AD(\m,\varepsilon))$.
      \item If $n = 0$, then $\AD(D_{1/2}(\m,\varepsilon))=D_{-1/2}(\AD(\m,\varepsilon))$.
    \end{enumerate}
    
  \end{prop}
  
  \begin{proof}
    This follows directly from the explicit formulas for the derivatives recalled in Section \ref{sec:expder} and Proposition \ref{prop:ADeunnt}.
  \end{proof}

Now, let us assume that $\rho$ is of the same type as $G$. Then $\m$ is of the form
\[
  \m = c_0 [0,0] + c_1 [-1,1] + t ([-1,0] + [0,1]) + n ([-1,-1] + [1,1])
\]
with $c_0,c_1,t,n \in \NN$. To simplify notation, set $\varepsilon(0)=\varepsilon([0,0])$ and $\varepsilon(1)=\varepsilon([-1,1])$.

Let $(\m',\varepsilon')=\AD(\m,\varepsilon)$. The maximum of the coefficients of $(\m',\varepsilon')$ is also smaller than $1$, so $(\m',\varepsilon')$ is of the form as above. We denote by $c'_0,c'_1,t',n' \in \NN$ and $\varepsilon'(0),\varepsilon'(1)\in \{\pm 1\}$ the constants relative to $(\m',\varepsilon')$.

Let $(*)$ be the condition $c_0 \neq 0$, $c_1 \neq 0$ and $\varepsilon(0)\varepsilon(1)= (-1)^{t+1}$.

\begin{prop}
  \label{prop:ADeun}
  The dual $\AD(\m,\varepsilon)=(\m',\varepsilon')$ is given by the following formulas.
  \begin{enumerate}
      \item If $n > c_0$, then $c'_0=c_1$, $c'_1=c_0$, $t'=t$, $n'=n-c_0 + c_1$, $\varepsilon'(0) = \varepsilon(1) * (-1)^{c_0 + c_1 + 1}$ and $\varepsilon'(1) = \varepsilon(0) * (-1)^{c_0 + c_1 + 1 }$.
      \item If $n = c_0$, then $c'_0=c_1$, $c'_1=c_0$, $t'=t$, $n'= c_1$, $\varepsilon'(0) = \varepsilon(1) * (-1)^{c_0 + c_1 + 1}$ and $\varepsilon'(1) = \varepsilon(0) * (-1)^{c_0 + c_1 + 1 }$.
    \item If $n < c_0$, $(*)$ is not satisfied and $c_0-n$ is even or $t=0$ . Then $t'=t$, $n'=c_1$, $c'_0=c_0 + c_1 - n $, $c'_1=n$, $\varepsilon'(0)=\varepsilon(0) * (-1)^{c_0 + c_1 + t + 1}$ and $\varepsilon'(1)=\varepsilon(0) * (-1)^{c_0 + c_1 + 1}$.
    \item If $n < c_0$, $(*)$ is not satisfied, $c_0-n$ is odd and $t \neq 0$. Then $t'=t-1$, $n'=c_1$, $c'_0=c_0 + c_1 -n + 1$, $c'_1 = n + 1$, $\varepsilon'(0)=\varepsilon(0) * (-1)^{c_0 + c_1 + t + 1}$ and $\varepsilon'(1)=\varepsilon(0)(-1)^{c_0 + c_1 + 1}$.
    \item If $n < c_0$, $(*)$ is satisfied and $c_0-n$ is even. Then $t'=t+1$, $n'=c_1-1$, $c'_0 = c_0 + c_1 -n - 2$, $c'_1 = n$, $\varepsilon'(0)=\varepsilon(0) * (-1)^{t + c_0 + c_1}$ and $\varepsilon'(1)=\varepsilon(0)(-1)^{c_0 + c_1 + 1}$.
    
    \item If $n < c_0$, $(*)$ is satisfied and $c_0-n$ is odd. Then $t'=t$, $n'=c_1 - 1$, $c'_0 = c_0 + c_1 - n - 1$, $c'_1 = n + 1$, $\varepsilon'(0)=\varepsilon(0) * (-1)^{t + c_0 + c_1}$ and $\varepsilon'(1)=\varepsilon(0)(-1)^{c_0 + c_1 + 1}$.
 \end{enumerate}
\end{prop}

\begin{proof}
  We write $\m_1$ and $\m^{\#}$ in the different cases and the result follows from an immediate induction.  
    \begin{itemize}
        \item Suppose $n,c_0 \neq 0$. Then $\Delta_1= [1,1]$ and $\Delta_2 = [0,0]^{=0}$ or $[0,0]^{\ge 0}$.  Thus $\m_1 = [-1,1]$, $\varepsilon_1(0) = \varepsilon(0) * (-1)^{c_0 + c_1 + 1}$, $\m^{\#}= (c_0 - 1) [0,0] + c_1 [-1,1] + t ([-1,0] + [0,1]) + (n - 1) ([-1,-1] + [1,1])$, $\varepsilon^{\#}(0) = -\varepsilon(0)$ and $\varepsilon^{\#}(1) = -\varepsilon(1)$. 
        \item Suppose $n,t \neq 0$ and $c_0 = 0$. Then $\Delta_1= [1,1]$ and $\Delta_2 = [-1,0]$.  Thus $\m_1 = [-1,0] + [0,1]$, $\m^{\#}= c_1 [-1,1] + (t-1) ([-1,0] + [0,1]) + n ([-1,-1] + [1,1])$ and $\varepsilon^{\#}(1) = \varepsilon(1)$.

        \item Suppose $n \neq 0$ and $c_0 ,t = 0$. Then $\Delta_1= [1,1]$.  Thus $\m_1 = [1,1] + [-1,-1]$, $\m^{\#}= (c_1 - 1) [-1,1] + n ([-1,-1] + [1,1])$ and $\varepsilon^{\#}(1) = \varepsilon(1)$.

        \item Suppose $n,c_0 = 0$ and $t\neq 0$. Then $\Delta_1= [0,1]$ and $\Delta_2 = [-1,0]$.  Thus $\m_1 = [-1,0] + [0,1]$, $\m^{\#}= c_1 [-1,1] + (t - 1)([-1,0] + [0,1])$ and $\varepsilon^{\#}(1)=\varepsilon(1)$.
      \item Suppose $n=0$, $t,c_0\neq 0$ and $c_0$ is even. Then $\Delta_1= [0,1]$ and $\Delta_2 = [0,0]^{\le 0}$.  Thus $\m_1 = [-1,0] + [0,1]$, $\m^{\#}= c_0 [0,0] + c_1 [-1,1] + (t - 1)([-1,0] + [0,1])$, $\varepsilon^{\#}(0)=-\varepsilon(0)$ and $\varepsilon^{\#}(1)=\varepsilon(1)$.
      \item Suppose $n=0$, $t\neq 0$ and $c_0$ is odd. Then $\Delta_1= [0,1]$ and $\Delta_2 = [0,0]^{=0}$.  Thus $\m_1 = [-1,1]$, $\varepsilon_1(1)=(-1)^{ c_1}\varepsilon(0)$, $\m^{\#}= (c_0 + 1) [0,0] + c_1 [-1,1] + (t - 1)([-1,0] + [0,1])$, $\varepsilon^{\#}(0)=\varepsilon(0)$ and $\varepsilon^{\#}(1)=-\varepsilon(1)$.
      \item Suppose $n,t,c_1 = 0$. Then $\Delta_1= [0,0]^{\ge 0}$ or $[0,0]^{=0}$. Thus $\m_1 = [0,0]$, $\varepsilon_1(0) = \varepsilon(0) * (-1)^{c_0 + 1}$, $\m^{\#}=(c_0 - 1)[0,0]$ and  $\varepsilon^{\#}(0) = -\varepsilon(0)$. 
      \item Suppose $n,t,c_0 = 0$. Then $\Delta_1= [-1,1]^{\ge 0}$ or $[-1,1]^{=0}$. Thus $\m_1 = [1,1] + [-1,-1]$, $\m^{\#}=[0,0] + (c_1 - 1) [-1,1]$, $\varepsilon^{\#}(0) = \varepsilon(1)$ and $\varepsilon^{\#}(1) = \varepsilon(1)$. 
      \item Suppose $n,t = 0$ $c_0,c_1 \neq 0$ and $\varepsilon(0)=\varepsilon(1)$. Then $\Delta_1= [0,0]^{\ge 0}$ or $[0,0]^{=0}$. Thus $\m_1 = [1,1] + [-1,-1]$, $\m^{\#}=(c_0 + 1)[0,0] + (c_1 - 1) [-1,1]$, $\varepsilon^{\#}(0) = \varepsilon(1)$ and $\varepsilon^{\#}(1) = \varepsilon(1)$. 
      \item Suppose $n,t = 0$, $c_0,c_1 \neq 0$, $\varepsilon(0)\varepsilon(1)=-1$ and $c_0$ is odd. Then $\Delta_1= [0,0]^{\ge 0}$ or $[0,0]^{=0}$ and $\Delta_2 = [0,0]^{=0}$. Thus $\m_1 = [-1,1]$, $\varepsilon_1(1) = \varepsilon(1) * (-1)^{c_1 + 1}$, $\m^{\#}=c_0[0,0] + (c_1 - 1)[-1,1]$, $\varepsilon^{\#}(0) = \varepsilon(0)$ and $\varepsilon^{\#}(1) = \varepsilon(0)$.
      \item Suppose $n,t = 0$, $c_0,c_1 \neq 0$, $\varepsilon(0)\varepsilon(1)=-1$ and $c_0$ is even. Then $\Delta_1= [0,0]^{\ge 0}$ or $[0,0]^{=0}$ and $\Delta_2 = [0,0]^{\le 0}$. Thus $\m_1 = [0,1] + [-1,0]$, $\m^{\#}=(c_0 - 1)[0,0] + (c_1 - 1)[-1,1]$, $\varepsilon^{\#}(0) = -\varepsilon(0)$ and $\varepsilon^{\#}(1) = -\varepsilon(0)$. 

    \end{itemize}
\end{proof}

We also check the commutativity with the derivative.
\begin{prop}
  \begin{enumerate}
    \item If $n \neq 0$, then $\AD(D_{-1}(\m,\varepsilon))=D_{1}(\AD(\m,\varepsilon))$.
    \item If $n = 0$ and $t \neq 0$, then $\AD(D_{L([-1,0])}(\m,\varepsilon))=D_{Z([0,1])}(\AD(\m,\varepsilon))$.
    \item If $n,t = 0$, then $\AD(D_{1}(\m,\varepsilon))=D_{-1}(\AD(\m,\varepsilon))$.
  \end{enumerate}
\end{prop}

\begin{proof}
  The cases $(1)$ and $(3)$ follow directly from the formulas in Section \ref{sec:expder} and Proposition \ref{prop:ADeun}. Let us do $(2)$. We assume that $n=0$ and $t \neq 0$. First $D_{L([-1,0])}(\m,\varepsilon) = (c_0 [0,0] + c_1 [-1,1],\varepsilon)$. Thus $\AD(D_{L([-1,0])}(\m,\varepsilon))$ is given by Proposition \ref{prop:ADeun}.

  Now for $D_{Z([0,1])}(\AD(\m,\varepsilon))$, we can compute $\AD(\m,\varepsilon)$ with Proposition \ref{prop:ADeun} and then $D_{Z([0,1])}$ thanks to the formula in \cite[Prop. A.2]{AtobeSocle}. There are four cases in Proposition \ref{prop:ADeun}. We give all the details for the first one. The other cases are treated similarly and the complete details are left to the reader.
  \begin{itemize}
      \item Suppose $c_0$ is even and $(*)$ is not satisfied. By Proposition \ref{prop:ADeun}, $\AD(\m,\varepsilon)=t([-1,0] + [0,1]) + c_1 ( [-1,-1] + [1,1]) + (c_0 + c_1) [0,0]$, with $\varepsilon'(0)=\varepsilon(1) * (-1)^{c_0 + c_1 + 1}$. With the notation of \cite[§A.3]{AtobeSocle}, $s=c_1$, $t=t$, $m=c_0 + c_1$ and $\delta = 0$. We have that $m \equiv s \pmod{2}$ thus $D_{Z([0,1])}(\AD(\m,\varepsilon))$ is given by \cite[Prop. A.2 (4)]{AtobeSocle}.
      
      \begin{itemize}
          \item If $t \equiv 1 \pmod{2}$ and $m > s=0$. Then by \cite[Prop. A.2]{AtobeSocle}, $D_{Z([0,1])}(\AD(\m,\varepsilon)) = (c_0 + c_1) [0,0]$. Since $s=c_1=0$, we are in the case $(3)$ in Proposition \ref{prop:ADeun} and the formulas coincide.
          \item If $t \equiv 1 \pmod{2}$ and $m > s >0$. Then by \cite[Prop. A.2]{AtobeSocle}, $D_{Z([0,1])}(\AD(\m,\varepsilon)) = ([-1,0] + [0,1]) + (c_1 - 1) ( [-1,-1] + [1,1]) + (c_0 + c_1 - 2) [0,0]$. Here $c_1 \neq 0$  (as $s \neq 0$) and $c_0 \neq 0$ (as $m > s$). Since $(*)$ is not satisfied, $\varepsilon(0) * (-1)^{t}=\varepsilon(1)$ and $t$ is odd (as $t \equiv 1 \pmod{2}$), thus $\varepsilon(0)\varepsilon(1)=-1$. We are in the case $(4)$ in Proposition \ref{prop:ADeun} and the formulas coincide.
          \item Otherwise, by \cite[Prop. A.2]{AtobeSocle}, $D_{Z([0,1])}(\AD(\m,\varepsilon))=  c_1 ( [-1,-1] + [1,1]) + (c_0 + c_1) [0,0]$. If $c_1 \neq 0$ and $c_0 \neq 0$ then $m > s > 0$, thus $t \equiv 0 \pmod{2}$. This means that $t$ is even and $\varepsilon(0)=\varepsilon(1)$. Therefore, we are in the case $(3)$ in Proposition \ref{prop:ADeun} and the formulas coincide.
      \end{itemize}
      
      \item The case $c_0$ even and $(*)$ satisfied of Proposition \ref{prop:ADeun} corresponds to the case $(3)$ of \cite[Prop. A.2]{AtobeSocle} ($s=c_1 - 1$, $t=t + 1$, $m=c_0 + c_1 - 2$ and $\delta = 0$). 
      
      \item The case $c_0$ odd and $(*)$ not satisfied of Proposition \ref{prop:ADeun} corresponds to the case $(2)$ of \cite[Prop. A.2]{AtobeSocle} ($s=n_1 $, $t=t - 1$, $m=c_0 + c_1 + 1$ and $\delta = 1$).
  
      \item The case $c_0$ odd and $(*)$ satisfied of Proposition \ref{prop:ADeun} corresponds to the case $(1)$ of \cite[Prop. A.2]{AtobeSocle} ($s=c_1 - 1$, $t=t$, $m=c_0 + c_1 - 1$ and $\delta = 1$).
  \end{itemize}
\end{proof}

\subsection{The negative derivative} \label{sec:derneg}

In this section, we assume that $e_{\max} > 1$ and there exists $y<0$, $y \neq -e_{\max}$ such that $(\m,\varepsilon)$ is not $y$-reduced. 

\bigskip

We define $y_0 \in (1/2)\Z$ to be the smallest $y \in (1/2)\Z$ such that $y \neq -e_{\max}$ and $(\m,\varepsilon)$ is not $y$-reduced. With our hypotheses on $(\m,\varepsilon)$ necessarily $y_0 < 0$. Let us give a more explicit description of $y_0$ using the formula of the derivative recalled in Section \ref{sec:expder}.

If $[e_{\max},e_{\max}] \notin \m$, let $y_1 = e_{\max} + 1$. Otherwise, let $y_1 \in (1/2)\NN^{*}$ be the smallest positive half-integer such that for all $y_1 \le y\le e_{\max}$ with $y-y_1 \in \NN$, $[y,y] \in \m$ and if $y \neq e_{\max }$, $m_{\m}([y,y]) \le m_{\m}([y+1,y+1])$. Then $y_0$ is the minimum of the $e(\Delta)$ for $\Delta \in \m$ such that $\Delta \neq [-y,-y]$ with $y_1 \le y$. Then $D_{y_0}$ removes all the ends of the segments ending in $y_0$ and all the beginnings of the segments starting in $-y_0$, apart possibly, when $y_1 \le -y_0 +1$, for $m_{\m}([y_0-1,y_0-1])$ segments $[-y_0,-y_0]$ and $[y_0,y_0]$.

\bigskip

Let $\Delta_1, \cdots, \Delta_l$ be the initial sequence in the algorithm for $(\m,\varepsilon)$. 

\begin{lem}
    \label{lem:endx0}
    We have that $e(\Delta_{l}) \ge y_0$.
\end{lem}

\begin{proof}
    Let us assume by contradiction that $e(\Delta_{l}) < y_0$. By definition of $y_0$, we must have that $\Delta_{l}=[-y,-y]$ with $y_1 \le y$. Now, with the definition of $y_1$, if $y \neq e_{\max }$ then there exists $[-y-1,-y-1] \in \m$ contradicting the fact that $\Delta_{l}$ is the last segment in the algorithm. Thus $\Delta_{l}=[-e_{\max },-e_{\max }]$ contradicting Lemma \ref{lem:ADnotcentered}.
\end{proof}

We fix an ordering \( y = \Lambda_1 + \cdots + \Lambda_k \), with \( \Lambda_1 \succeq \cdots \succeq \Lambda_k \). As in the formula for the derivative recalled in Section~\ref{sec:expder}, let \( A_{y_0} \), \( A_{y_0 - 1} \), and \( A_{y_0}^c \) be the sets defining \( D_{y_0} \) for \( \m \); and similarly, let \( A^{\#}_{y_0} \), \( A^{\#}_{y_0 - 1} \), and \( A^{\#,c}_{y_0} \) be those defining \( D_{y_0} \) for \( \m^{\#} \).

Note that the sets \( A_{y_0} \) and \( A_{y_0 - 1} \) (and likewise for \( \m^{\#} \)) are uniquely determined, whereas the set \( A_{y_0}^c \) is not. Indeed, the multiset of segments \( \Lambda_i \) for \( i \in A_{y_0}^c \) is uniquely determined (these are the segments modified by the derivative), but due to multiplicities, several different indices may correspond to the same segment. 

Given a subset \( A \subseteq \{1, \dots, k\} \), we write \( A^{\#,c}_{y_0} = A \) if the multiset of segments \( \Lambda_i^{\#} \) for \( i \in A \) matches the multiset of segments modified by \( D_{y_0} \) in \( \m^{\#} \).

We start by calculating $D_{y_0}(\m^{\#},\varepsilon^{\#})$.

\begin{lem}
  \label{lem:dermdieze}
  The set $A_{y_0}^{c}$ can be chosen such that if there is a $j \in \{i,\cdots,l\}$ such that $\Delta_j = [-y_0,-y_0]$ then $i'_j \notin A_{y_0}^{c}$. 
  We fix $A_{y_0}^{c}$ satisfying this condition.
  \begin{enumerate}
    \item If $e(\Delta_{l}) \ge y_0 + 2$, then $A_{y_0}^{\#,c} = A_{y_0}^{c}$.
    \item If $e(\Delta_{l}) = y_0$, then $A_{y_0}^{\#,c} = A_{y_0}^{c} \setminus \{i_l\}$.
    \item If $e(\Delta_{l}) = y_0 + 1$ and $\varepsilon_0=1$, then $A_{y_0}^{\#,c} = A_{y_0}^{c} \cup \{i_l\}$.
    \item If $e(\Delta_{l}) = y_0 + 1$ and $\varepsilon_0=-1$ (necessarily $y_0=-1$ or $-1/2$), then $A_{y_0}^{\#,c} = A_{y_0}^{c}$.
  \end{enumerate}
\end{lem}

\begin{proof}
  By definition $A_{y_0}^{\#} = \{i, e(\Lambda_i^{\#})=y_0 \text{ and } \Lambda_i^{\#} \neq 0\}$. If $i \notin \{i_1,\cdots,i_l\}$, then $e(\Lambda_i^{\#})=y_0$ if and only if $e(\Lambda_i)=y_0$. If $i \in \{i_1,\cdots,i_l\}$, then $e(\Lambda_i^{\#})=y_0$ if and only if $e(\Lambda_i)=y_0 + 1$. Thus $A_{y_0}^{\#} = \{i \in A_{y_0}, \Lambda_i^{\#} \neq 0\} \cup \{i \in \{i_1,\cdots,i_l\}, e(\Lambda_i)=y_0 + 1 \text{ and } \Lambda_i^{\#} \neq 0\} \setminus \{i \in \{i_1,\cdots,i_l\}, e(\Lambda_i)=y_0\}$. We have a similar result for $A_{y_0 - 1}^{\#}$. Let $i \notin \{i_1,\cdots,i_l\}$ such that $e(\Lambda_i)=y_0$ and $\Lambda_i^{\#} = 0$. Then $\Lambda_i=[y_0,y_0]$ and $i=i_j'$ for $j$ such that $\Delta_j = [-y_0,-y_0]$. Similarly, if $i \notin \{i_1,\cdots,i_l\}$ such that $e(\Lambda_i)=y_0 - 1$ and $\Lambda_i^{\#} = 0$, $i=i_n'$ for a $n$ such that $\Delta_n = [-y_0 +1,-y_0 + 1]$. Note that if there is a $j$ such that $\Delta_j = [-y_0,-y_0]$, then necessarily $n=j-1$, and $i_j \notin A_{y_0}^{c}$.
  In $(\m,\varepsilon)$ the only possible segments ending in $y_0 - 1$ are the $[y_0 - 1,y_0 - 1]$ and they protect possibly some segments $[y_0,y_0]$. Now, let us remark that if $i \in A_{y_0 - 1}^{\#}$ and $i \notin \{i_1,\cdots,i_l\}$ then $\Lambda_i=[y_0 - 1,y_0 - 1]$ thus $\Lambda_i^{\#}=[y_0 - 1,y_0 - 1]$. Also, if there is an $i \notin \{i_1,\cdots,i_l\}$ such that $\Lambda_i^{\#}=[y_0,y_0]$ and $\Lambda_i \neq [y_0,y_0]$, then $\Lambda_i = [y_0 - 1,y_0]$. Thus $[-y_0, -y_0 + 1]$ is in the sequence of the algorithm, which implies that $[-y_0 + 1,-y_0 + 1] \notin \m$. To deal with the $i \in \{i_1,\cdots,i_l\}$, we will analyse the different cases.
  \begin{itemize}
      \item If $e(\Delta_{l}) > y_0 + 1$. By the paragraph above, if there is a $j$ such that $\Delta_j = [-y_0,-y_0]$, then $A^{\#}_{y_0} = A_{y_0} \setminus \{i_j'\}$; otherwise $A^{\#}_{y_0} = A_{y_0}$. Similarly, if there is a $n$ such that $\Delta_n = [-y_0 +1,-y_0 + 1]$, $A^{\#}_{y_0 - 1} = A_{y_0 - 1} \setminus \{i_n'\}$; otherwise $A^{\#}_{y_0 - 1} = A_{y_0 - 1}$. We have seen that if such a $j$ exists then $i_j \notin A_{y_0}^{c}$. We have also studied the case of the segments $[y_0,y_0]$ and $[y_0-1,y_0-1]$. Hence, we see that $A_{y_0}^{\#,c} = A_{y_0}^{c}$.
      \item If $e(\Delta_{l}) = y_0$. We start with the case where $\Lambda_{i_l}^{\#} \neq 0$ and $\Lambda_{i_{l-1}}^{\#} \neq 0$. This time, if there is a $j$ such that $\Delta_j = [-y_0,-y_0]$, then $A^{\#}_{y_0} = A_{y_0} \cup\{i_{l-1}\} \setminus \{i_l,i_j'\}$; otherwise $A^{\#}_{y_0} = A_{y_0} \cup \{i_{l-1}\} \setminus \{i_l\}$. And if there is a $n$ such that $\Delta_n = [-y_0 +1,-y_0 + 1]$, $A^{\#}_{y_0 - 1} = A_{y_0 - 1} \cup \{i_l\} \setminus \{i_n'\}$; otherwise $A^{\#}_{y_0 - 1} = A_{y_0 - 1 }\cup \{i_l\}$. Now let us show that $\Lambda_{i_l}^{\#} \le \Lambda_{i_{l-1}}^{\#}$. We have $\Lambda_{i_l} \le \Lambda_{i_{l-1}}$, so the only issue would be if $\Lambda_{i_l}^{\#}={}^{-}\Lambda_{i_l}^{-}$, $\Lambda_{i_{l-1}}^{\#}=\Lambda_{i_{l-1}}^{-}$ and $b(\Lambda_{i_l})+1=b(\Lambda_{i_{l-1}})$. But if $\Lambda_{i_l}^{\#}={}^{-}\Lambda_{i_l}^{-}$, it means that there exists an $i$ such that $\Lambda_{i_l}^{\vee}=\Delta_i$. But as $b(\Lambda_{i_l})+1=b(\Lambda_{i_{l-1}})$ and $e(\Lambda_{i_l})+1=e(\Lambda_{i_{l-1}})$ we would have $\Delta_{i+1} = \Lambda_{i_{l-1}}^{\vee}$ contradicting $\Lambda_{i_{l-1}}^{\#}=\Lambda_{i_l}^{-}$. Moreover, $\Lambda_{i_l}^{\#}$ is the smallest segment such that $\Lambda_{i_l}^{\#} \le \Lambda_{i_{l-1}}^{\#}$ because another segment ending in $y_0-1$ smaller than $\Lambda_{i_{l-1}}^{\#}$ but bigger than $\Lambda_{i_l}^{\#}$ would contradict the definition of $\Delta_{l}$ in the algorithm. Thus $i_{l-1} \notin A_{y_0}^{\#,c}$. As before, we get $A_{y_0}^{\#,c} = A_{y_0}^{c} \setminus \{i_l\}$.
      
      If $\Lambda_{i_l}^{\#} = 0$. We show that $\Lambda_{i_{l-1}}^{\#} = 0$. We have two possibilities $\Lambda_{i_l} = [y_0,y_0]$ or $\Lambda_{i_l} = [y_0-1,y_0]$. If $\Lambda_{i_l} = [y_0,y_0]$ then $\Lambda_{i_{l-1}} = [y_0 + 1,y_0 + 1]$ and $\Lambda_{i_{l-1}}^{\#} = 0$. If $\Lambda_{i_l} = [y_0-1,y_0]$, then either $\Lambda_{i_{l-1}} = [y_0 + 1,y_0 + 1]$ (and $\Lambda_{i_{l-1}}^{\#} = 0$) or $\Lambda_{i_{l-1}} = [y_0,y_0 + 1]$, but in this case $\Lambda_{i_{l-1}}^{\vee}$ and $\Lambda_{i_{l}}^{\vee}$ are both in the sequence of the algorithm, thus $\Lambda_{i_{l-1}}^{\#} = 0$. Since $\Lambda_{i_l}^{\#} = 0$ and $\Lambda_{i_{l-1}}^{\#} = 0$, the result is similar to case $(1)$.

      If $\Lambda_{i_l}^{\#} \neq 0$ and $\Lambda_{i_{l-1}}^{\#} = 0$. Then by the definition of the sequence of the algorithm, we can see that $\Lambda_{i_l}$ is the biggest segment ending in $y_0$. Therefore, for all $i \in A_{y_0}^{\#}$, $\Lambda_{i_l}^{\#} \nleq \Lambda_{i}^{\#}$ and $A_{y_0}^{\#,c} = A_{y_0}^{c} \setminus \{i_l\}$.

      \item If $e(\Delta_l) = y_0 + 1$ and $\varepsilon_0 =1$. First, we show that $\Lambda_{i_l}^{\#} \neq 0$. Indeed, $\Lambda_{i_l}$ cannot be $[y_0 + 1,y_0 + 1]$ because it would be followed in the initial sequence of the algorithm by any segments ending in $y_0$ (and it is not $[0,0]^{=0}$, $[0,0]^{\ge 0}$ or $[1/2,1/2]$). If it is $[y_0,y_0 + 1]$ then $[-y_0-1,-y_0]$ is in the initial sequence of the algorithm. Hence $[y_0,y_0] \notin \m$ and $\Delta_l$ would be followed in the initial sequence by any segments ending in $y_0$. The last case to consider is $[-1/2,1/2]^{\ge 0}$ or $[-1/2,1/2]^{=0}$ (if $\Lambda_{i_l} = [-1/2,1/2]^{\le 0}$ then $\Lambda_{i_l}^{\#} \neq 0$). But, if $\varepsilon_0 =1$, these segments would be followed in the initial sequence by any segments ending in $-1/2$. We get that, if there is a $j$ such that $\Delta_j = [-y_0,-y_0]$, then $A^{\#}_{y_0} = A_{y_0} \cup\{i_{l}\} \setminus \{i_j'\}$; otherwise $A^{\#}_{y_0} = A_{y_0} \cup \{i_{l}\}$. And if there is a $n$ such that $\Delta_n = [-y_0 +1,-y_0 + 1]$, $A^{\#}_{y_0 - 1} = A_{y_0 - 1} \setminus \{i_n'\}$; otherwise $A^{\#}_{y_0 - 1} = A_{y_0 - 1 }$. As before, if such a $j$ exists, then also such a $n$ exists, and $i_j \notin A_{y_0}^{c}$.

      Let us show that $i_l \in A^{c,\#}_{y_0}$. If $\Lambda_{i_l}^{\#} \neq [y_0,y_0]$ then we have  $i_l \in A^{c,\#}_{y_0}$ since for all $i \in A^{\#}_{y_0 - 1}$, $\Lambda_i^{\#} = [y_0 - 1, y_0 - 1]$. If $\Lambda_{i_l}^{\#} = [y_0,y_0]$ and $\Lambda_{i_l} = [y_0,y_0 + 1]$, then the only segments of $\m$ ending in $y_0$ are $[y_0,y_0]$ and $i_l \in A^{c,\#}_{y_0}$. If $\Lambda_{i_l}^{\#} = [y_0,y_0]$ and $\Lambda_{i_l} = [y_0- 1,y_0 + 1]$, then $[-y_0- 1,-y_0 + 1]$ is in the sequence of the algorithm, so $[y_0 - 1,y_0 - 1] \notin \m$ and $i_l \in A^{c,\#}_{y_0}$. Finally, we get $A_{y_0}^{\#,c} = A_{y_0}^{c} \cup \{i_l\}$.

      \item If $e(\Delta_l) = y_0 + 1$ and $\varepsilon_0 =-1$. Then $\Lambda_{i_l}^{\#} = 0$. Similarly as before, if there is a $j$ such that $\Delta_j = [-y_0,-y_0]$, then $A^{\#}_{y_0} = A_{y_0} \setminus \{i_j'\}$; otherwise $A^{\#}_{y_0} = A_{y_0}$. If there is a $n$ such that $\Delta_n = [-y_0 +1,-y_0 + 1]$, $A^{\#}_{y_0 - 1} = A_{y_0 - 1} \setminus \{i_n'\}$; otherwise $A^{\#}_{y_0 - 1} = A_{y_0 - 1}$. Hence, we see that $A_{y_0}^{\#,c} = A_{y_0}^{c}$.
  \end{itemize}
\end{proof}

Now, we want to compute the effect of $\AD$ on $D_{y_0}(\m,\varepsilon)$. We denote by $(\tilde{\m},\tilde{\varepsilon}):=D_{y_0}(\m,\varepsilon)$ and by $\tilde{\Delta}_1, \cdots, \tilde{\Delta}_{\tilde{l}}$ the initial sequence in the algorithm for $(\tilde{\m},\tilde{\varepsilon})$.

\begin{lem}
    \label{lem:sequencederneg}
    \begin{enumerate}
        \item If $e(\Delta_{l})=y_0$, then $\tilde{l}=l-1$ and if not $\tilde{l}=l$.
        \item For $1 \le i \le \tilde{l}$, if $b(\Delta_i)=-y_0$ and $\Delta_i \neq [-y_0,-y_0]$ then $\tilde{\Delta}_i={}^{-}\Delta_i$, otherwise $\tilde{\Delta}_i=\Delta_i$.
    \end{enumerate}
\end{lem}

\begin{proof}
    First, notice that in the algorithm, when $\Delta_i$ is defined then $\Delta_{i+1}$ just depends on the segments smaller than $\Delta_{i}$ which terminates by $e(\Delta_{i})-1$ (together with a sign condition). In particular, if $\Delta_i$ appears in the initial sequence of the algorithm for $\m$ and $\tilde{\m}$ and $e(\Delta_i) \neq y_0 + 1$, then $D_{y_0}$ does not change the previous set of segments, and $\Delta_{i+1}$ is the next segment in the algorithm for both $\m$ and $\tilde{\m}$.

    Note that by definition $e(\Delta_{1})=e_{\max}$ is the maximum of the coefficients of the segments. Since by definition of $y_0$, $\Delta_{1} \neq [-y_0,-y_0]$, after taking the derivative $\Delta_{1}$ does not vanish. Hence $e_{\max}$ is still the maximum of the coefficients in $\tilde{\m}$.

    Let $i_0 = e_{\max } - y_1 + 2$. Then, for all $0<i<i_0$, $\Delta_{i}=[e_{\max }-i+1, e_{\max }-i+1]$. All of these segments belong to $\tilde{\m}$, so $\tilde{l} \ge i_0 - 1$ and for all $0<j<i_0$, $\tilde{\Delta}_{i}=\Delta_{i}=[e_{\max }-i+1, e_{\max }-i+1]$.

    If there is an integer $i$ such that $b(\Delta_i)=-y_0$ then $i \le i_0$. If $i<i_0$, we have treated the case before. Suppose that $b(\Delta_{i_0})=-y_0$. If $\Delta_{i_{0}}=[-y_0,-y_0]$. Then $\Delta_{i_{0} - 1}=[-y_0 + 1,-y_0 +1]$, thus one $\Delta_{i_{0}}$ is not modified by $D_{y_0}$. Hence $\tilde{l} \ge i_0$ and $\tilde{\Delta}_{i_{0}}=\Delta_{i_{0}}$. If $\Delta_{i_{0}}\neq [-y_0,-y_0]$, then $D_{y_0}$ transforms $\Delta_{i_{0}}$ into ${}^{-}\Delta_{i_{0}}$ and we get that $\tilde{\Delta}_{i_0}={}^{-}\Delta_{i_{0}}$.

    In the case that $\tilde{\Delta}_{i_{0}}={}^{-}\Delta_{i_{0}}$ and $l\ge i_0 + 1$, we have $\tilde{\Delta}_{i_{0} + 1}=\Delta_{i_{0} + 1}$. Indeed, $\Delta_{i_{0} + 1}$ is unchanged in $\tilde{\m}$, we still have $\Delta_{i_{0} + 1} \prec {}^{-}\Delta_{i_{0}}$, and the only segment smaller than ${}^{-}\Delta_{i_{0}}$ but not than $\Delta_{i_{0}}$ ending in $e(\Delta_{i_{0} + 1})$ is $[-y_0,e(\Delta_{i_{0} + 1})]$. Then only possibility to have such a segment in $\tilde{\m}$ is to have $e(\Delta_{i_{0} + 1})= -y_0$ and that there is a segment of the form $[-y_0 + 1,-y_0 + 1]$ in $\m$. But then $\Delta_{i_0}=[-y_0,-y_0+1]$ and this contradicts the maximality of $\Delta_{i_0}$ as $[-y_0 + 1,-y_0 + 1] \prec \Delta_{i_0 - 1}$.

    By Lemma \ref{lem:endx0}, $e(\Delta_{i_l}) \ge y_0$, thus $\tilde{l}\ge l-1$ and for all $j_0<i<l$, $\tilde{\Delta}_{i}=\Delta_{i}$. If $e(\Delta_{l}) \neq y_0$, then $\tilde{l}= l$ and $\tilde{\Delta}_{l}=\Delta_{l}$. If $e(\Delta_{l}) = y_0$, then $\Delta_{l} \neq [y_0,y_0]$ or if $\Delta_{l} = [y_0,y_0]$ there is no $[y_0 - 1,y_0 - 1]$. Thus $\Delta_{l}$ is changed by the derivative and $\tilde{l}=l-1$.
\end{proof}

\begin{lem}
  \label{lem:ADderneg}
  \begin{enumerate}
      \item If $e(\Delta_{l}) \ge y_0 + 2$, then $(\tilde{\m}_1,\tilde{\varepsilon}_1)=(\m_{1},\varepsilon_1)$ and $(\tilde{\m}^{\#},\tilde{\varepsilon}^{\#})=D_{y_0}(\m^{\#},\varepsilon^{\#})$.
      \item If $e(\Delta_{l}) = y_0$, then $(\tilde{\m}_1,\tilde{\varepsilon}_1)=D_{-y_0}(\m_{1},\varepsilon_1)$ and $(\tilde{\m}^{\#},\tilde{\varepsilon}^{\#})=D_{y_0}(\m^{\#},\varepsilon^{\#})$.
      \item If $e(\Delta_{l}) = y_0 + 1$ and $\varepsilon_0=1$, then $(\tilde{\m}_1,\tilde{\varepsilon}_1)=(\m_{1},\varepsilon_1)$ and $(\tilde{\m}^{\#},\tilde{\varepsilon}^{\#})=D^{\max - 1}_{y_0}(\m^{\#},\varepsilon^{\#})$.
      \item If $e(\Delta_{l}) = y_0 + 1$ and $\varepsilon_0=-1$, then $(\tilde{\m}_1,\tilde{\varepsilon}_1)=(\m_{1},\varepsilon_1)$ and $(\tilde{\m}^{\#},\tilde{\varepsilon}^{\#})=D_{y_0}(\m^{\#},\varepsilon^{\#})$.
  \end{enumerate}
\end{lem}

\begin{proof}
  Let us start with $\tilde{\m}_1$. We know that by definition $\m_1 = [-e(\Delta_{1}),e(\Delta_{1})]$ or $\m_1 =[e(\Delta_{l}),e(\Delta_{1})] + [-e(\Delta_{1}),-e(\Delta_{l})]=[e(\Delta_{1})-l+1,e(\Delta_{1})] + [-e(\Delta_{1}),-e(\Delta_{1})+l-1]$; and $\tilde{\m}_1 = [-e(\tilde{\Delta}_{1}),e(\tilde{\Delta}_{1})]$ or $\tilde{\m}_1=[e(\tilde{\Delta}_{1})-\tilde{l}+1,e(\tilde{\Delta}_{1})] + [-e(\tilde{\Delta}_{1}),-e(\tilde{\Delta}_{1})+\tilde{l}-1]$.  Lemma \ref{lem:sequencederneg} implies that $e(\Delta_{1})=e(\tilde{\Delta}_{1})$ and $l=\tilde{l}$ unless  $e(\Delta_{l}) = -y_0$ and in this case $\tilde{l}=l-1$. Thus $(\tilde{\m}_1,\tilde{\varepsilon}_1)=(\m_{1},\varepsilon_1)$ unless  $e(\Delta_{l}) = -y_0$ and in this case $\tilde{\m}_1=[y_0+1,e(\Delta_{1})] + [-e(\Delta_{1}),-y_0 - 1]=D_{-y_0}(\m_1)$.

  Now, let us calculate $\tilde{\m}^{\#}$. Let $A_{y_0}^c$ be the set of indices of segments in $\m$ ending in $y_0$ and modified by $D_{y_0}$; and $A_{y_0}^{\#,c}$ be the set of indices of segments in $\m^{\#}$ ending in $y_0$ and modified by $D_{y_0}$. Let $E^{\#}$ be the set of indices of the segments in $\m$ such that the end is modified by $\AD$; and $\tilde{E}^{\#}$ be the same for $\tilde{\m}$. In all these cases, all the symmetrics of these segments are also modified. From Lemma \ref{lem:sequencederneg}, we see that the $\tilde{\Delta}_{i}$ are the same as the $\Delta_{i}$ apart when $b(\Delta_{i})=-y_0$ or $e(\Delta_{l})=y_0$. Hence, if $e(\Delta_{l}) \neq y_0$, $\tilde{E}^{\#} = E^{\#}$, and if $e(\Delta_{l}) = y_0$, $\tilde{E}^{\#} = E^{\#} \setminus \{i_l\}$. If $e(\Delta_{l}) \neq y_0$, we define $\tilde{E}' = A_{y_0}^c$; otherwise $\tilde{E}' = A_{y_0}^c \setminus \{i_l\}$. We get $\tilde{\m}^{\#}$ is obtained from $\m^{\#}$ by suppressing the end of the segments indexed by $\tilde{E}'$ and the beginning of their symmetrics. To compare $\tilde{\m}^{\#}$ and $D_{y_0}(\m^{\#},\varepsilon^{\#})$, we need to compare $\tilde{E}'$ and $A_{y_0}^{\#,c}$.

  \begin{itemize}
      \item If $e(\Delta_{l}) > y_0 + 1$. By Lemma \ref{lem:dermdieze}, $A_{y_0}^{\#,c} = A_{y_0}^{c}$. Thus $A_{y_0}^{\#,c}=\tilde{E}'$, and $(\tilde{\m}^{\#},\tilde{\varepsilon}^{\#})=D_{y_0}(\m^{\#},\varepsilon^{\#})$.
      \item If $e(\Delta_{l}) = y_0$. By Lemma \ref{lem:dermdieze}, $A_{y_0}^{\#,c} = A_{y_0}^{c} \setminus \{i_l\}$. Thus $A_{y_0}^{\#,c}=\tilde{E}'$, and $(\tilde{\m}^{\#},\tilde{\varepsilon}^{\#})=D_{y_0}(\m^{\#},\varepsilon^{\#})$.
      \item If $e(\Delta_l) = y_0 + 1$ and $\varepsilon_0=1$. By Lemma \ref{lem:dermdieze}, $A_{y_0}^{\#,c} = A_{y_0}^{c} \cup \{i_l\}$. Let us show that in $\tilde{\m}^{\#}$, $\tilde{\Lambda}^{\#}_{i_l}$ is modified by $D_{y_0}$. The possible segments of $\tilde{\m}^{\#}$ ending in $y_0-1$ are possibly $[y_0-1,y_0-1]$ or coming from a segment of $\m$ ending in $y_0$ and modified by $D_{y_0}$. If they were a segment $\Lambda \in \m$ such that $e(\Lambda)=y_0$ and $\Lambda^{-} \le \Lambda_{i_l}^{-}$ or ${}^{-}\Lambda^{-} \le \Lambda_{i_l}^{-}$ this would imply that $\Lambda \prec \Lambda_{i_l}$ contradicting the fact that $\Delta_l$ is the last segment in the sequence of the algorithm. The segments $[y_0-1,y_0-1]$ only protect the segments $[y_0,y_0]$. If $\Lambda_{i_l}^{\#} = [y_0,y_0]$ and $\Lambda_{i_l} = [y_0,y_0 + 1]$, then the only segments of $\m$ ending in $y_0$ are $[y_0,y_0]$. If $\Lambda_{i_l}^{\#} = [y_0,y_0]$ and $\Lambda_{i_l} = [y_0- 1,y_0 + 1]$, then $[-y_0- 1,-y_0 + 1]$ is in the sequence of the algorithm, so $[y_0 - 1,y_0 - 1] \notin \m$. 
      
       Thus we get that $D_{y_0}(\m^{\#},\varepsilon^{\#}) = D^{1}_{y_0}(\tilde{\m}^{\#},\tilde{\varepsilon}^{\#})$, or that $(\tilde{\m}^{\#},\tilde{\varepsilon}^{\#})=D^{\max - 1}_{y_0}(\m^{\#},\varepsilon^{\#})$.
      \item If $e(\Delta_l) = y_0 + 1$ and $\varepsilon_0=-1$. Again by Lemma \ref{lem:dermdieze}, $A_{y_0}^{\#,c} = A_{y_0}^{c}$. Thus $(\tilde{\m}^{\#},\tilde{\varepsilon}^{\#})=D_{y_0}(\m^{\#},\varepsilon^{\#})$.
  \end{itemize}
\end{proof}

We can gather all the previous lemmas to obtain the desired equality.

\begin{prop}
    \label{prop:ADcommderneg}
     We have $\AD(D_{y_0}(\m,\varepsilon))=D_{-y_0}(\AD(\m,\varepsilon))$.
\end{prop}

\begin{proof}
    Recall that $(\tilde{\m},\tilde{\varepsilon})$ denotes $D_{y_0}(\m,\varepsilon)$. By definition of $\AD$ we have $\AD(\tilde{\m},\tilde{\varepsilon})=(\tilde{\m}_1,\tilde{\varepsilon}_1) + \AD(\tilde{\m}^{\#},\tilde{\varepsilon}^{\#})$. Now we have different cases by Lemma \ref{lem:ADderneg}.
    \begin{itemize}
        \item If $e(\Delta_{l}) > y_0 + 1$. Then $(\tilde{\m}_1,\tilde{\varepsilon}_1)=(\m_{1},\varepsilon_1)$ and $(\tilde{\m}^{\#},\tilde{\varepsilon}^{\#})=D_{y_0}(\m^{\#},\varepsilon^{\#})$. Thus $\AD(\tilde{\m},\tilde{\varepsilon})=(\m_{1},\varepsilon_1) + \AD(D_{y_0}(\m^{\#},\varepsilon^{\#}))$. By Lemma \ref{lem:ADcommDer}, $\AD(D_{y_0}(\m^{\#},\varepsilon^{\#}))=D_{-y_0}(\AD(\m^{\#},\varepsilon^{\#}))$. Since  $-e(\Delta_{l})$ is different from $-y_0$ and $-y_0 - 1$ and $-e(\Delta_{1}) < y_0$, by Lemma \ref{lem:dersum},  $(\m_{1},\varepsilon_1) + D_{-y_0}(\AD(\m^{\#},\varepsilon^{\#}))=D_{-y_0}((\m_{1},\varepsilon_1) + \AD(\m^{\#},\varepsilon^{\#}))=D_{-y_0}(\AD(\m,\varepsilon))$.
        \item If $e(\Delta_{l}) = y_0$. Then $(\tilde{\m}_1,\tilde{\varepsilon}_1)=D_{-y_0}(\m_{1},\varepsilon_1)$ and $(\tilde{\m}^{\#},\tilde{\varepsilon}^{\#})=D_{y_0}(\m^{\#},\varepsilon^{\#})$. This time $\AD(\tilde{\m},\tilde{\varepsilon})=D_{-y_0}(\m_{1},\varepsilon_1) + \AD(D_{y_0}(\m^{\#},\varepsilon^{\#})) = D_{-y_0}(\m_{1},\varepsilon_1) + D_{-y_0}(\AD(\m^{\#},\varepsilon^{\#}))$ (the last equality follows from Lemma \ref{lem:ADcommDer}). Since $-e(\Delta_{1})$ is the smallest coefficient and $-e(\Delta_{1})< y_0$, so Lemma \ref{lem:dersum} tells us that $D_{-y_0}(\m_{1},\varepsilon_1) + D_{-y_0}(\AD(\m^{\#},\varepsilon^{\#}))=D_{-y_0}((\m_{1},\varepsilon_1) + \AD(\m^{\#},\varepsilon^{\#}))=D_{-y_0}(\AD(\m,\varepsilon))$.
        \item If $e(\Delta_{l}) = y_0 + 1$. Then $(\tilde{\m}_1,\tilde{\varepsilon}_1)=(\m_{1},\varepsilon_1)$ and $(\tilde{\m}^{\#},\tilde{\varepsilon}^{\#})=D^{\max - 1}_{y_0}(\m^{\#},\varepsilon^{\#})$. Thus $\AD(\tilde{\m},\tilde{\varepsilon})=(\m_{1},\varepsilon_1) + \AD(D^{\max - 1}_{y_0}(\m^{\#},\varepsilon^{\#}))$. By Lemmas \ref{lem:ADcommDer} and \ref{lem:ADcommSoc}, $\AD(\tilde{\m},\tilde{\varepsilon})=(\m_{1},\varepsilon_1) + D^{\max - 1}_{-y_0}(\AD(\m^{\#},\varepsilon^{\#}))$. Now $-e(\Delta_{1})$ is the smallest coefficient and $-e(\Delta_{1})< y_0$. To show that the hypotheses of Lemma \ref{lem:dersum} are satisfied, we are left to prove that there are no segments of the form $[-e(\Delta_{1}),-y_0]$ in $\AD(\m^{\#},\varepsilon^{\#})$, but this follows from Proposition \ref{prop:lengthmax}. Thus Lemma \ref{lem:dersum} tells us that $(\m_{1},\varepsilon_1) + \AD(D^{\max - 1}_{y_0}(\m^{\#},\varepsilon^{\#}))=D_{-y_0}((\m_{1},\varepsilon_1) + \AD(\m^{\#},\varepsilon^{\#}))=D_{-y_0}(\AD(\m,\varepsilon))$.
    \end{itemize}
\end{proof}

\subsection{The $L([-1,0])$-derivative} \label{sec:derl01}

In this section, we assume that $e_{\max} > 1$ and that for all $-e_{\max} < y < 0$, $(\m,\varepsilon)$ is $y$-reduced. We also assume that $\rho$ is of the same type as $G$ and that $(\m,\varepsilon)$ is not $L([-1,0])$-reduced. We want to prove $\AD(D_{L([-1,0])}(\m,\varepsilon))=D_{Z([0,1])}(\AD(\m,\varepsilon))$.

\bigskip

The hypotheses made on $(\m,\varepsilon)$ imply that if $y < 0$, then the only possible segment $\Delta \in \m$ such that $e(\Delta)=y$ is $[y,y]$. Moreover, $m_{\m}([y,y]) \le m_{\m}([y-1,y-1])$.

The derivative $D_{L([-1,0])}$ performs the following transformations (see \cite[Prop. 3.8.]{AM}):
\begin{itemize}
  \item If $e(\Delta)=0$, and $\Delta \neq [0,0],[-1,0]$, $\Delta$ is transformed into $\Delta^{--}$.
  \item If $b(\Delta)=0$, and $\Delta \neq [0,0],[0,1]$, $\Delta$ is transformed into ${}^{--}\Delta$.
  \item $\max\{m_{\m}([-1,0]-m_{\m}([-2,-2]) + m_{\m}([-1,-1]),0\}$ segments $[-1,0]$ are suppressed.
  \item $\max\{m_{\m}([-1,0])-m_{\m}([-2,-2]) + m_{\m}([-1,-1]),0\}$ segments $[0,1]$ are suppressed.
\end{itemize}

\begin{lem}
  \label{lem:end01}
  We have that $e(\Delta_{l}) \ge 0$.
\end{lem}

\begin{proof}
  Let us assume by contradiction that $e(\Delta_{l}) < 0$. Thus $\Delta_{l}=[-y,-y]$ with $y < 0$. Since $\Delta_{l}$ is the last segment in the initial sequence of the algorithm, $\Delta_{l}=[-e_{\max },-e_{\max }]$ contradicting Lemma \ref{lem:ADnotcentered}.
\end{proof}

Let $A_{[-1,0]}$ be the set of indices $i$ of segments $\Lambda_i$ in $\m$ ending in $0$ modified by $D_{L([-1,0])}$. The multisegment $(\m^{\#},\varepsilon^{\#})$ may not be $-1$-reduced. Let $A^{\#}_{-1}$ be the set of indices of segments ending in $-1$ modified by $D_{-1}$ in $(\m^{\#},\varepsilon^{\#})$. And let $A^{\#}_{[-1,0]}$ be the set of segments ending in $0$ modified by $D_{L([-1,0])}$ in $D_{-1}(\m^{\#},\varepsilon^{\#})$.

\begin{lem}
  \label{lem:dermdieze01}
  \begin{enumerate}
    \item If $e(\Delta_{l}) \ge 2$, then $A^{\#}_{-1} = \emptyset$ and $A^{\#}_{[-1,0]} = A_{[-1,0]}$.
    \item If $e(\Delta_{l}) = 0$, $p(\Delta_{l}) \neq [0,0]$ and $\Delta_{l} \neq [-1,0]$ if $m_{\m}([-2,-2]) > m_{\m}([-1,-1])$; then we can assume that $i_l \in A_{[-1,0]}$, and we get that $A^{\#}_{-1} = \{i_l\}$ and $A^{\#}_{[-1,0]} = A_{[-1,0]} \setminus \{i_l\}$.
    \item If $e(\Delta_{l}) = 0$ and $p(\Delta_{l}) = [0,0]$ or $\Delta_{l} = [-1,0]$ with $m_{\m}([-2,-2]) > m_{\m}([-1,-1])$, then $A^{\#}_{-1} = \emptyset$ and $A^{\#}_{[-1,0]} = A_{[-1,0]}$.
    \item If $e(\Delta_{l}) = 1$, then $A^{\#}_{-1} = \emptyset$ and $A^{\#}_{[-1,0]} = A_{[-1,0]} \cup \{i_l\}$.
  \end{enumerate}
\end{lem}

\begin{proof}
  \begin{itemize}
      \item Suppose $e(\Delta_{l}) \ge 2$. Then all the segments of $\m^{\#}$ ending in $0$ (\resp $-1$, \resp $-2$) come from a segment of $\m$ ending in $0$ (\resp $-1$, \resp $-2$). First, let us check that $(\m^{\#},\varepsilon^{\#})$ is $-1$-reduced. We have seen that there are no ``new segments'' in $\m^{\#}$ ending in $-1$ or $-2$. The only segments of $\m$ ending in $-1$ (\resp $-2$) are $[-1,-1]$ (\resp $[-2,-2]$). If a $[-2,-2]$ is suppressed, it means that $\Delta_l=[2,2]$ and therefore $m_{\m}([1,1])=0$. Hence $(\m^{\#},\varepsilon^{\#})$ is $-1$-reduced, that is $A^{\#}_{-1} = \emptyset$. Now, let us compute $A^{\#}_{[-1,0]}$. As we have seen, if a segment $[-2,-2]$ is suppressed then there is no $[-1,-1]$ in $\m$ and no $[-1,0]$. Also, a segment $[2,2]$ or $[1,1]$ cannot be created (as $e(\Delta_{l}) \ge 2$ they cannot come from ${}^{-}\Lambda$ or ${}^{-}\Lambda^{-}$; and by hypothesis there are no segments $[2,3]$ or $[1,2]$ in $\m$). If a segment $[0,1]$ is created in $\m^{\#}$, then $\Delta_l = [0,2]$. Hence $m_{\m}([2,2])=0$ and this segment is modified by $D_{L([-1,0])}$. Thus $A^{\#}_{[-1,0]} = A_{[-1,0]}$.
      \item Suppose $e(\Delta_{l}) = 0$, $p(\Delta_{l}) \neq [0,0]$ and $\Delta_{l} \neq [-1,0]$ if $m_{\m}([-2,-2]) > m_{\m}([-1,-1])$. In $(\m^{\#},\varepsilon^{\#})$ there is a new segment ending in $-1$ which is $\Lambda_{i_l}^{\#}$. First, note that $\Lambda_{i_l}^{\#} \neq 0$. Indeed, $\Lambda_{i_l} \neq [0,0]$ and if $\Lambda_{i_l} = [-1,0]$ with $\Lambda_{i_l}^{\#} = {}^{-}\Lambda_{i_l}^{-}$, it means that $\Delta_{l-1}=[0,1]$. Thus $\Delta_{l-2}=[2,2]$ and therefore $m_{\m}([-2,-2]) \neq 0$ and $m_{\m}([-1,-1]) = 0$ contradicting the hypothesis. If $\Lambda_{i_l}^{\#} = [-1,-1]$, it means that $\Lambda_{i_l}=[-1,0]$. Thus $m_{\m}([-2,-2]) = m_{\m}([-1,-1])$ and $\Lambda_{i_l}^{\#}$ is modified by $D_{-1}$. In all cases, it is easy to see that $D_{-1}$ modifies $\Lambda_{i_l}^{\#}$, hence $A^{\#}_{-1} = \{i_l\}$. 
      
      In $D_{-1}(\m^{\#},\varepsilon^{\#})$ there is a new segment ending in $0$. It is $\Lambda_{i_{l-1}}^{\#}$. But, from the definition of the initial sequence in the algorithm, $\Lambda_{i_{l}}^{\#-}$ protects $\Lambda_{i_{l-1}}^{\#}$ and is the smallest to protect it, so $D_{L([-1,0])}$ does not modify it. The situation is similar to the previous case for the suppression and creation of $[-2,-2]$, $[-1,-1]$ and $[-1,0]$. At the end we get that $A^{\#}_{[-1,0]} = A_{[-1,0]}\setminus \{i_l\}$.
      \item Suppose $e(\Delta_{l}) = 0$ and $p(\Delta_{l}) = [0,0]$ or $\Delta_{l} = [-1,0]$ with $m_{\m}([-2,-2]) > m_{\m}([-1,-1])$. If $p(\Delta_{l}) = [0,0]$ there is no new segment ending in $-1$ in $\m^{\#}$. And if $\Delta_{l} = [-1,0]$ with $m_{\m}([-2,-2]) > m_{\m}([-1,-1])$, either $\Lambda_{i_{l}}^{\#} = 0$ or $\Lambda_{i_{l}}^{\#} = [-1,-1]$. If $\Lambda_{i_{l}}^{\#} = [-1,-1]$, then we have in $\m^{\#}$ that $m_{\m^{\#}}([-2,-2]) \ge m_{\m^{\#}}([-1,-1])$. In all cases, since the only possible segments ending in $-1$ in $\m^{\#}$ are $[-1,-1]$ with $m_{\m^{\#}}([-2,-2])\ge m_{\m^{\#}}([-1,-1])$, $(\m^{\#},\varepsilon^{\#})$ is $-1$-reduced and $A^{\#}_{-1} = \emptyset$.
      
      There could be one new segment ending in $0$, $\Lambda_{i_{l-1}}^{\#}$.  Moreover, $p(\Delta_{l}) = [0,0]$ or $\Delta_{l}=[-1,0]$, so $\Delta_{l-1}$ cannot be a negative segment and $\Delta_{l-1}=[-1,1]$ or $[0,1]$ or $[1,1]$. In all the cases, either $\Lambda_{i_{l-1}}^{\#}=0$ or $\Lambda_{i_{l-1}}^{\#}=[0,0]$, and it is not modified by $D_{L([-1,0])}$. We get that $A^{\#}_{[-1,0]} = A_{[-1,0]}$.
      \item Suppose $e(\Delta_{l}) = 1$. In $\m^{\#}$ there is a no new segment ending in $-1$ and $A^{\#}_{-1} = \emptyset$. First, let us show that $\Delta_{l} \neq [-1,1]^{=0}$, $[-1,1]^{\ge 0}$ and $[0,1]$. The case where $\Delta_{l} = [0,1]$ implies that $[-1,0] \in \m$ and contradicts the fact that $\Delta_{l}$ is the last segment in the algorithm. If $\Delta_{l} = [-1,1]^{=0}$ or $[-1,1]^{\ge 0}$, then any segment $\Delta$ such that $e(\Delta)=0$ and $c(\Delta)<0$ satisfies $\Delta \prec \Delta_l$, contradicts the fact that $\Delta_{l}$ is the last segment in the algorithm ($(\m,\varepsilon)$ is not $L([-1,0])$-reduced). Therefore, $\Lambda_{i_l}^{\#} \neq [0,0]$. The segment $\Lambda_{i_l}^{\#}$ could be $[-1,0]$ if $\Delta_l = [-1,1]^{\le 0}$ or $\Delta_l = [-2,1]$. If $\Delta_{l} = [-1,1]^{\le 0}$ or $\Delta_{l} = [-2,1]$ with $\Lambda_{i_l}^{\#}=[-1,0]$, then $\Delta_{l-1} \neq [2,2]$, and $m_{\m}([-2,-2])=0$. Hence, the new segment $[-1,0]$ is modified by $D_{L([-1,0])}$. In all cases, $\Lambda_{i_l}^{\#}$ is modified by $D_{L([-1,0])}$ and $A^{\#}_{[-1,0]} = A_{[-1,0]} \cup \{i_l\}$.
  \end{itemize}
\end{proof}

We denote by $(\tilde{\m},\tilde{\varepsilon}):=D_{L([-1,0])}(\m,\varepsilon)$ and by $\tilde{\Delta}_1, \cdots, \tilde{\Delta}_{\tilde{l}}$ the initial sequence in the algorithm for $(\tilde{\m},\tilde{\varepsilon})$.

\begin{lem}
    \label{lem:sequenceder01}
    \begin{enumerate}
        \item If $e(\Delta_{l})=0$, $p(\Delta_{l}) \neq [0,0]$ and $\Delta_{l} \neq [-1,0]$ if $m_{\m}([-2,-2]) > m_{\m}([-1,-1])$ then $\tilde{l}=l-1$;  otherwise $\tilde{l}=l$.
        \item For $1 \le i \le \tilde{l}$, if $b(\Delta_{i})=0$ and $\Delta_{i} \neq [0,1]$ and $p(\Delta_i) \neq [0,0]$ then $\tilde{\Delta}_{i}={}^{--}\Delta_{i}$; otherwise $\tilde{\Delta}_{i}=\Delta_{i}$.
        
    \end{enumerate}
\end{lem}

\begin{proof}
    Note that by definition $e(\Delta_{1})=e_{\max}$ is the maximum of the coefficients of the segments. Since by assumption $e_{\max} > 1$, after taking the $L([-1,0])$-derivative $\Delta_{1}$ does not vanish. Hence $e_{\max}$ is still the maximum of the coefficients in $\tilde{\m}$. Moreover, if $b(\Delta_{1})\neq 0$, $\Delta_{1}$ stays the biggest segment ending in $e_{\max}$ and if $b(\Delta_{1})= 0$, the biggest segment is now ${}^{--}\Delta_{1}$.

    As in Lemma \ref{lem:sequencederneg}, if $[e_{\max},e_{\max}] \notin \m$, let $y_1 = e_{\max} + 1$; otherwise, let $y_1 \in (1/2)\NN^{*}$ be the smallest positive half-integer such that $[y_1,y_1] \in \m$. Let $i_0 = e_{\max } - y_1 + 2$. Then, for all $0<i<i_0$, $\Delta_{i}=[e_{\max }-i+1, e_{\max }-i+1]$. We get that $\tilde{l} \ge i_0$, and for all $0<i<i_0$, $\tilde{\Delta}_{i}=\Delta_{i}=[e_{\max}-i+1, e_{\max}-i+1]$.

    The only possible integer $i$ such that $b(\Delta_{i})=0$ and $p(\Delta_{i}) \neq [0,0]$ is $i_0$. If $\Delta_{i_0}=[0,1]$, then $\Delta_{i_0 - 1}=[2,2]$ and $[1,1] \notin \m$, thus $\Delta_{i_0}$ is not modified by $D_{L([-1,0])}$. Hence $\tilde{l} \ge i_0$ and $\tilde{\Delta}_{i_0}=\Delta_{i_0}$. If $\Delta_{i_0}\neq [0,1]$, then $D_{L([-1,0])}$ transforms $\Delta_{i_0}$ into ${}^{- -}\Delta_{i_0}$ and we get that $\tilde{\Delta}_{i_0}={}^{- -}\Delta_{i_0}$.

    In the case that $\tilde{\Delta}_{i_0}={}^{- -}\Delta_{i_0}$, if $l \ge i_0 + 1$  then $\tilde{\Delta}_{i_0 + 1}=\Delta_{i_0 + 1}$. Indeed, $\Delta_{i_0 + 1}$ is unchanged in $\tilde{\m}$, we still have $\Delta_{i_0 + 1} < {}^{- -}\Delta_{i_0}$, and the only segments smaller than ${}^{- -}\Delta_{i_0}$ but not than $\Delta_{i_0}$ ending in $e(\Delta_{i_0 + 1})$ are $[0,e(\Delta_{i_0 + 1})]$ and $[1,e(\Delta_{i_0 + 1})]$. Then only possibility to have such segments in $\tilde{\m}$ is to have $e(\Delta_{i_0 + 1})= 0$ or $1$. First, $e(\Delta_{i_0 + 1})= 0$ is impossible since $e(\Delta_{i_0}) > 1$. If $e(\Delta_{i_0 + 1})= 1$, then $e(\Delta_{i_0})= 2$, and by definition of $i_0$ we see that there is no $[2,2]$ in $\m$ and thus no $[1,1]$ or $[0,1]$ in $\tilde{\m}$.

    By Lemma \ref{lem:end01}, $e(\Delta_{i}) \ge 0$, thus $\tilde{l}\ge l-1$ and for all $i_0<i<l$, $\tilde{\Delta}_{i}=\Delta_{i}$.

    Now, notice that if $x<0$ then $[x,0]<[0,0]^{\le 0}<[0,0]^{=0}<[0,0]^{\ge 0}$. Thus, if $e(\Delta_{l})=0$, $p(\Delta_{l}) \neq [0,0]$ and $\Delta_{l} \neq [-1,0]$ if $m_{\m}([-2,-2]) > m_{\m}([-1,-1])$, then $\tilde{l}=l-1$ and if not $l=\tilde{l}$ and $\tilde{\Delta}_{l}=\Delta_{l}$.
\end{proof}

\begin{lem}
    \label{lem:ADder01}
    \begin{enumerate}
        \item If $e(\Delta_{l}) \ge 2$ then $(\tilde{\m}_1,\tilde{\varepsilon}_1)=(\m_{1},\varepsilon_1)$, $(\m^{\#},\varepsilon^{\#})$ is $-1$-reduced and $(\tilde{\m}^{\#},\tilde{\varepsilon}^{\#})=D_{L([-1,0])}(\m^{\#},\varepsilon^{\#})$.
        \item If $e(\Delta_{l}) = 0$ and $p(\Delta_{l}) \neq [0,0]$ and $\Delta_{l} \neq [-1,0]$ if $m_{\m}([-2,-2]) > m_{\m}([-1,-1])$, then $\tilde{\m}_1=[1,e_{\max}] + [-e_{\max},-1]$, $(\m^{\#},\varepsilon^{\#})$ is not $-1$-reduced and $(\tilde{\m}^{\#},\tilde{\varepsilon}^{\#})=D_{L([-1,0])}(D_{-1}(\m^{\#},\varepsilon^{\#}))$.
        \item If $e(\Delta_{l}) = 0$ and $p(\Delta_{l}) = [0,0]$ or $\Delta_{l} = [-1,0]$ with $m_{\m}([-2,-2]) > m_{\m}([-1,-1])$, then $(\tilde{\m}_1,\tilde{\varepsilon}_1)=(\m_{1},\varepsilon_1)$, $(\m^{\#},\varepsilon^{\#})$ is $-1$-reduced and $(\tilde{\m}^{\#},\tilde{\varepsilon}^{\#})=D_{L([-1,0])}(\m^{\#},\varepsilon^{\#})$.
        \item If $e(\Delta_{l}) = 1$ then $(\tilde{\m}_1,\tilde{\varepsilon}_1)=(\m_{1},\varepsilon_1)$, $(\m^{\#},\varepsilon^{\#})$ is $-1$-reduced and $(\tilde{\m}^{\#},\tilde{\varepsilon}^{\#})=D^{\max - 1}_{L([-1,0])}(\m^{\#},\varepsilon^{\#})$.
    \end{enumerate}
\end{lem}

\begin{proof}
    Let us start with $(\tilde{\m}_1,\tilde{\varepsilon}_1)$. The multisegment $\m_1$ is determined uniquely by $\Delta_{1}$ and $\Delta_{l}$. Similarly for $\tilde{\m}_1$ with $\tilde{\Delta}_{1}$ and $\tilde{\Delta}_{\tilde{l}}$. We get the result for $\m_1$ and $\tilde{\m}_1$ by Lemma \ref{lem:sequenceder01}. If $\Delta_l \neq [0,0]^{=0}$ or $[0,0]^{\ge 0}$, then neither $\m_1$ nor $\tilde{\m}_1$ are centered and we are done. If $\Delta_l = [0,0]^{=0}$ or $[0,0]^{\ge 0}$, we also need to check that $\varepsilon_1 = \tilde{\varepsilon}_1$. By definition $\varepsilon_1(\m_1):=(-1)^{n_0 + 1} \varepsilon([0,0]_\rho)$ and $\tilde{\varepsilon}_1(\tilde{\m}_1):=(-1)^{\tilde{n}_0 + 1} \tilde{\varepsilon}([0,0]_\rho)$ where $n_0$ is the number of centered segments in $\m$ and $\tilde{n}_0$ in $\tilde{\m}$. The derivative $D_{L([-1,0])}$ does not create or suppress any centered segments in $\m$, so $n_0 = \tilde{n}_0$. Also, it does not change the sign of $[0,0]$, that is $\varepsilon([0,0]_\rho)=\tilde{\varepsilon}([0,0]_\rho)$. Hence $\varepsilon_1 = \tilde{\varepsilon}_1$ and we get the result.

    Now let us study the case of $(\tilde{\m}^{\#},\tilde{\varepsilon}^{\#})$. We do a similar proof as in Lemma \ref{lem:ADderneg}. Let $A_{[-1,0]}$ be the set of indices $i$ of segments of $\m$ ending in $0$ modified by $D_{L([-1,0])}$, $A^{\#}_{-1}$ be the set of indices of segments of $\m^{\#}$ ending in $-1$ modified by $D_{-1}$, and $A^{\#}_{[-1,0]}$ be the set of segments ending in $0$ modified by $D_{L([-1,0])}$ in $D_{-1}(\m^{\#},\varepsilon^{\#})$. Let $E$ be the set of indices of segments of $\m$ such that the end is modified by $\AD$ and $\tilde{E}$ be the set of indices of segments of $\tilde{\m}$ such that the end is modified by $\AD$. We denote by $(*)$ the condition $e(\Delta_{l})=0$, $p(\Delta_{l}) \neq [0,0]$ and $\Delta_{l} \neq [-1,0]$ if $m_{\m}([-2,-2]) > m_{\m}([-1,-1])$. From Lemma \ref{lem:sequenceder01}, we get that if $(*)$ is satisfied then $\tilde{E} = E \setminus \{i_l\}$; otherwise $\tilde{E} = E$. Moreover, $\tilde{E} \cap A_{[-1,0]} = \emptyset$. More precisely, if $(*)$ is satisfied then $E \cap A_{[-1,0]} = \{i_l\}$; otherwise $E \cap A_{[-1,0]} = \emptyset$.
    
    Let $\tilde{E}^{\#}_2 = A_{[-1,0]} \setminus \{i_l\}$ if $(*)$ is satisfied; and $\tilde{E}^{\#}_2 = A_{[-1,0]}$ otherwise. Let $\tilde{E}^{\#}_1 = \{i_l\}$ if $(*)$ is satisfied; and $\tilde{E}^{\#}_1 = \emptyset$ otherwise. Note that $\tilde{E}^{\#}_1 \cap \tilde{E}^{\#}_2 = \emptyset$. Then we see that $\tilde{\m}^{\#}$ is obtained from $\m^{\#}$ by transforming the segment $\Lambda_i^{\#}$ into $\Lambda_i^{\#-}$ for $i \in \tilde{E}^{\#}_1$ (and $\Lambda_i^{\# \vee}$ into ${}^{ -}(\Lambda_i^{\# \vee})$); and then transforming the segments $\Lambda_i^{\#}$ into $\Lambda_i^{\#--}$ for $i \in \tilde{E}^{\#}_2$ (and $\Lambda_i^{\# \vee}$ into ${}^{- -}(\Lambda_i^{\# \vee})$).

    \begin{itemize}
      \item Suppose $e(\Delta_{l}) \neq 1$. From Lemma \ref{lem:dermdieze01}, $\tilde{E}^{\#}_1 = A^{\#}_{-1}$ and $\tilde{E}^{\#}_2 = A^{\#}_{[-1,0]}$ giving us the result.
      \item Suppose $e(\Delta_{l}) = 1$. From Lemma \ref{lem:dermdieze01}, $\tilde{E}^{\#}_1 = A^{\#}_{-1}$ and $\tilde{E}^{\#}_2 \cup \{i_l\} = A^{\#}_{[-1,0]}$. Similarly as in Lemma \ref{lem:ADderneg}, we see that $D_{L([-1,0])}$ modifies $\Lambda_{i_l}^{\#}$ in $(\tilde{\m}^{\#},\tilde{\varepsilon}^{\#})$. Hence $D_{L([-1,0])}(\tilde{\m}^{\#},\tilde{\varepsilon}^{\#})=D^{1}_{L([-1,0])}(\tilde{\m}^{\#},\tilde{\varepsilon}^{\#})=D_{L([-1,0])}(\m^{\#},\varepsilon^{\#})$ and we get the desired result.
  \end{itemize}
\end{proof}

To make sense of the formula $\AD(D_{L([-1,0])}(\m,\varepsilon))=D_{Z([0,1])}(\AD(\m,\varepsilon))$, we need to check that $\AD(\m,\varepsilon)$ is $1$-reduced. 

\begin{lem}
  \label{lem:ad1reduced}
    The multisegment $\AD(\m,\varepsilon)$ is $1$-reduced.   
\end{lem}

\begin{proof}
   By Lemma \ref{lem:end01}, $e(\Delta_{l}) \ge 0$. 
    \begin{itemize}
      \item Suppose $e(\Delta_{l}) > 0$ or $p(\Delta_{l}) = [0,0]$ or $\Delta_{l} = [-1,0]$ with $m_{\m}([-2,-2]) > m_{\m}([-1,-1])$. From Lemma \ref{lem:ADder01}, $(\m^{\#},\varepsilon^{\#})$ is $-1$-reduced. And by Lemma \ref{lem:ADcommDer}, $D_1(\AD(\m^{\#},\varepsilon^{\#}))=\AD(D_{-1}(\m^{\#},\varepsilon^{\#}))$, hence $\AD(\m^{\#},\varepsilon^{\#})$ is $1$-reduced. By Lemma \ref{lem:dersum}, $D_1(\AD(\m,\varepsilon))=(\m_1,\varepsilon_1) + D_1(\AD(\m^{\#},\varepsilon^{\#}))$, and we get the desired result.
      \item Suppose $e(\Delta_{l}) = 0$, $p(\Delta_{l}) \neq [0,0]$ and $\Delta_{l} \neq [-1,0]$ if $m_{\m}([-2,-2]) > m_{\m}([-1,-1])$. This time from Lemma \ref{lem:ADder01} $(\m^{\#},\varepsilon^{\#})$ is not $-1$-reduced. Moreover, we have $D_{-1}(\m^{\#},\varepsilon^{\#})=D^{1}_{-1}(\m^{\#},\varepsilon^{\#})$. Since $\Delta_l \neq [0,0]^{=0}$ or $[0,0]^{\ge 0}$, from the definition $\m_1=[0,e_{\max}] + [-e_{\max},0]$. Then Lemma \ref{lem:dersum} and Proposition \ref{prop:lengthmax} tell us that $D_1(\AD(\m,\varepsilon))=(\m_1,\varepsilon_1) + D_1^{\max - 1}(\AD(\m^{\#},\varepsilon^{\#}))$. By Lemmas \ref{lem:ADcommDer} and \ref{lem:ADcommSoc}, $D_1^{\max - 1}(\AD(\m^{\#},\varepsilon^{\#}))=\AD(D_{-1}^{\max - 1}(\m^{\#},\varepsilon^{\#}))$. But, as $D_{-1}(\m^{\#},\varepsilon^{\#})=D^{1}_{-1}(\m^{\#},\varepsilon^{\#})$, we have $D_{-1}^{\max - 1}(\m^{\#},\varepsilon^{\#}) = (\m^{\#},\varepsilon^{\#})$ and we get the desired result.
    \end{itemize}  
\end{proof}

We can now prove the desired proposition.

\begin{prop}
  \label{prop:ADcommder01}
    We have $\AD(D_{L([-1,0])}(\m,\varepsilon))=D_{Z([0,1])}(\AD(\m,\varepsilon))$.
\end{prop}

\begin{proof}
    The proof is similar to Proposition \ref{prop:ADcommderneg} and follows from Lemmas \ref{lem:ADder01} and \ref{lem:dersum01}.
\end{proof}

\subsection{The positive derivative with $\rho$ of same type} \label{sec:derpos}

In this section, we assume that $e_{\max} > 1$, that for all $-e_{\max} < y<0$, $(\m,\varepsilon)$ is $y$-reduced, that $(\m,\varepsilon)$ is $L([-1,0])$-reduced, and that there exists $y>0$ with $y \neq e_{\max}$ such that $(\m,\varepsilon)$ is not $y$-reduced. We also assume that $\rho$ is of the same type as $G$.

We define $y_0 \in \Z$ to be the smallest $y \in \Z^{*}$ such that $y \neq -e_{\max}$, $y \neq e_{\max}$ and $(\m,\varepsilon)$ is not $y$-reduced. With our hypotheses on $(\m,\varepsilon)$ necessarily $y_0 > 0$. Let us describe more precisely the conditions satisfied by $(\m,\varepsilon)$ and $y_0$ using the explicit formula of the derivative recalled in Section \ref{sec:expder}.

\bigskip

We have assumed that for all $-e_{\max} < y <0$, $(\m,\varepsilon)$ is $y$-reduced. Hence, if $\Delta \in \m$ is a segment such that $e(\Delta) < 0$, then $\Delta = [y,y]$ (where $y=e(\Delta)$) and $m_{\m}([y,y]) \le m_{\m}([y-1,y-1])$. Let $y_1 \in \NN^{*}$ be the smallest positive half-integer such that $[y_1,y_1] \in \m$.

Let $t_0=m_{\m}([-1,0])$. We have also assumed that $(\m,\varepsilon)$ is $L([-1,0])$-reduced.  Hence, if $\Delta \in \m$ is a segment such that $e(\Delta) = 0$, then $\Delta = [0,0]$ or $\Delta = [-1,0]$ with $t_0 \le m_{\m}([-2,-2]) - m_{\m}([-1,-1])$.

\bigskip

Now, we recall the formula for $D_{y_0}$. First let us do $D_{1}$.

We have seen that the segments $\Delta \in \m$ with $e(\Delta)=0$ are of the form: $[0,0]$ or $[- 1,0]$. And the segments $\Delta \in \m$ with $e(\Delta)=1$ are of the form: $[a,1]$, with $a<-1$, $[-1,1]$, $[1,1]$ or $[0,1]$. Let $(*)$ be the condition: $m_{\m}([-1,1])\neq 0$, $m_{\m}([0,0])\neq 0$ and $\varepsilon([0,0]) \varepsilon([-1,1])=(-1)^{t_0 + 1}$.

Then $D_{1}$ does the following transformations. All the negative segments of the form $[a,1]$, with $a<-1$, are transformed into $[a,0]$ (and their symmetric $[-1,-a]$ are transformed into $[0,-a]$).
\begin{itemize}
  \item If $(*)$ is not satisfied,  $m_{\m}([-1,1])$ is odd and $t_0 \ge 1$; then all the segments $[-1,1]$ except one are transformed into $[0 ,0 ]$ and one $[-1,0] + [0,1]$ is transformed into $[0 ,0 ] + [0 ,0 ]$.
  \item If $(*)$ is not satisfied and, $m_{\m}([-1,1])$ is even or $t_0 = 0$; then all the segments $[-1,1]$ are transformed into $[0,0 ]$.
  \item If $(*)$ is satisfied and $m_{\m}([-1,1])$ is odd; then all the segments $[-1,1]$ except one are transformed into $[0 ,0]$.
  \item If $(*)$ is satisfied and $m_{\m}([-1,1])$ is even; then two segments $[-1,1]$ are transformed into $[-1,0]+[0,1]$. The other segments $[-1,1]$ are transformed into $[0,0]$.
\end{itemize}
And finally, it suppresses $n$ segments $[1,1]$ and $[-1,-1]$, where $n=\max\{m_{\m}([1,1])-m_{\m}([0,0]),0\}$ if $(*)$ is not satisfied; and $n=\max\{m_{\m}([1,1])-m_{\m}([0,0])+1,0\}$ otherwise.

In particular, we see that if $D_1(\m,\varepsilon)=0$, then 
\begin{itemize}
  \item There is no segment $[a,1]$ with $a<-1$;
  \item $m_{\m}([-1,1]) = 0$ or $1$;
  \item if $m_{\m}([-1,1]) = 1$, then $m_{\m}([0,0]) \neq 0$ and $\varepsilon([0,0]) \varepsilon([-1,1])=(-1)^{t_0 + 1}$;
  \item $m_{\m}([1,1]) \le m_{\m}([0,0])$ if $m_{\m}([-1,1]) = 0$; and $m_{\m}([1,1]) \le m_{\m}([0,0]) - 1$ if $m_{\m}([-1,1]) = 1$.
\end{itemize}

\bigskip

Let $y > 1$ be such that for all $1 \le x < y$, $D_x(\m,\varepsilon)=0$. We prove by induction on $y$ that if $\Delta \in \m$ is a segment with $0< e(\Delta) < y$, then $c(\Delta) \ge 0$. We make explicit the formula for $D_y$.

Using the induction hypothesis, if $1 < x < y$ and $\Delta \in \m$ with $e(\Delta) = x$, then $b(\Delta) \ge -x$. We also know that if $b(\Delta) \neq x$, then $b(\Delta)<0$. But $\Delta^{\vee} \in \m$, so if $b(\Delta) \neq x$ then $b(\Delta) = -x$. Hence, the only segments ending in $x$ are $[x,x]$ and $[-x,x]$. In particular, note that there are no segments $[-y+1,y]$ in $\m$. Also the segments $\Delta \in \m$ with $e(\Delta)=y-1$ are of the form: $[y-1,y-1]$, $[1-y,y-1]$ or $[0,1]$ if $y= 2$. And the segments $\Delta \in \m$ with $e(\Delta)=y$ are of the form: $[a,y]$, with $a<-y$, $[-y,y]$, $[y,y]$. Let $(*)$ be the condition: $m_{\m}([-y,y])\neq 0$, $m_{\m}([1-y,y-1])\neq 0$ and $\varepsilon([y,y]) \varepsilon([1-y,y-1])=(-1)$.

Then, $D_{y}$ does the following transformations. All the negative segments of the form $[a,y]$, with $a < -y$, are transformed into $[a,y - 1]$ (and their symmetric $[-y,-a]$ are transformed into $[-y + 1,-a]$). For the centered segments $[-y,y]$:
\begin{itemize}
  \item If $(*)$ is not satisfied; then all the segments $[-y,y]$ are transformed into $[-y + 1,y - 1]$.
  \item If $(*)$ is satisfied and $m_{\m}([-y,y])$ is odd; then all the segments $[-y,y]$ except one are transformed into $[-y + 1,y - 1]$.
  \item If $(*)$ is satisfied and $m_{\m}([-y,y])$ is even; then, two segments $[-y,y]$ are transformed into $[-y,y - 1]+[-y + 1,y]$. The other segments $[-y,y]$ are transformed into $[-y + 1,y - 1]$.
\end{itemize}
It also suppresses $n$ segments $[y,y]$ and $[-y,-y]$, where $n=\max\{m_{\m}([y,y])-m_{\m}([y-1,y-1]) - m_{\m}([1-y,y-1]),0\}$ if $y >2$ and $(*)$ is not satisfied; $n=\max\{m_{\m}([y,y])-m_{\m}([y-1,y-1])- m_{\m}([1-y,y-1])+1,0\}$ if $y >2$ and $(*)$ is satisfied; $n=\max\{m_{\m}([y,y])-m_{\m}([y-1,y-1]) - m_{\m}([1-y,y-1]) - t_0,0\}$ if $y =2$ and $(*)$ is not satisfied; and $n=\max\{m_{\m}([y,y])-m_{\m}([y-1,y-1])- m_{\m}([1-y,y-1]) - t_0 +1,0\}$ if $y =2$ and $(*)$ is satisfied. Note that we will show that $m_{\m}([1-y,y-1]) = 0$ or $1$.

In particular, we see that if $D_y(\m,\varepsilon)=0$, then 
\begin{itemize}
  \item There is no segment $[a,y]$ with $a<-y$;
  \item $m_{\m}([-y,y]) = 0$ or $1$;
  \item if $m_{\m}([-y,y]) = 1$, then $m_{\m}([1-y,y-1]) \neq 0$ and $\varepsilon([-y,y]) \varepsilon([1-y,y-1])=-1$;
  \item $m_{\m}([y,y]) \le m_{\m}([y-1,y-1]) + m_{\m}([1-y,y-1])$ if $m_{\m}([-y,y]) = 0$; and $m_{\m}([y,y]) \le m_{\m}([y-1,y-1])$ if $m_{\m}([-y,y]) = 1$.
\end{itemize}
This proves our induction hypothesis.

\bigskip

The formulas above gave us an explicit formula for $D_{y_0}$ (depending on whether $y_0 = 1$ or not). Moreover, we see that if $\Delta \in \m$ is such that $e(\Delta) < y_0$ then either $\Delta = [y,y]$ for some $y$, $\Delta = [-1,0]$, $\Delta = [0,1]$, or $\Delta = [-y,y]$ and in this case, $m_{\m}([-y,y]) = 1$ (if $y \neq 0$), moreover, if $y > 1$, $m_{\m}([-y+1,y-1]) = 1$ and $\varepsilon([-y,y]) \varepsilon([1-y,y-1])=-1$; and if $y = 1$,  $m_{\m}([0,0]) \neq 0$ and $\varepsilon([-1,1]) \varepsilon([0,0])=(-1)^{t_0 + 1}$.

\bigskip

For the rest of the section, let $(*)$ denote the condition:
\begin{enumerate}
    \item If $y_0 = 1$: $m_{\m}([-1,1])\neq 0$, $m_{\m}([0,0])\neq 0$ and $\varepsilon([0,0]) \varepsilon([-1,1])=(-1)^{t_0 + 1}$.
    \item If $y_0 > 1$: $m_{\m}([-y_0,y_0])\neq 0$, $m_{\m}([-y_0 + 1,y_0-1])\neq 0$ and $\varepsilon([-y_0,y_0]) \varepsilon([1-y_0,y_0-1])=-1$.
\end{enumerate}

\begin{lem}
  \label{lem:endpos}
  If $e(\Delta_l)< y_0$ then $e(\Delta_l) = 0$.
\end{lem}

\begin{proof}
  We assume that $e(\Delta_l)< y_0$. From Lemma \ref{lem:end01}, we have that $e(\Delta_l) \ge 0$. Thus $\Delta_l$ is one of the following segments : $[-1,0]$, $[0,1]$, $[-y,y]$, or $[y,y]$. Since $[-1,0] \prec [0,1]$, $\Delta_l \neq [0,1]$. If $\Delta_l = [-y,y]$ or $[y,y]$ with $y \neq 0$, the fact that $(\m,\varepsilon)$ is $y$-reduced implies that there exists a segment $\Delta$ ending in $y-1$ suitable for the algorithm such that $\Delta \prec \Delta_l$ which is a contradiction.
\end{proof}

We denote by $(\tilde{\m},\tilde{\varepsilon}):=D_{y_0}(\m,\varepsilon)$ and by $\tilde{\Delta}_1, \cdots, \tilde{\Delta}_{\tilde{l}}$ the initial sequence in the algorithm for $(\tilde{\m},\tilde{\varepsilon})$.

\begin{lem}
  \label{lem:derposx1inf}
  If $y_1 < y_0$, then 
  \begin{enumerate}
    \item $e(\Delta_l)=0$;
    \item $\tilde{l}=l$;
    \item for all $1 \le i \le l$, $\tilde{\Delta}_{i}=\Delta_{i}$.
  \end{enumerate}
\end{lem}

\begin{proof}
  We are in the case where $\Delta_1 = [e_{\max},e_{\max}]$. As $y_0 \neq e_{\max}$, $D_{y_0}$ does not suppress $\Delta_1$ and $\tilde{\Delta}_1 = \Delta_1$. Let $i_0 = e_{\max } - y_1 + 2$. Then, for all $0<i<i_0$, $\Delta_{i}=[e_{\max }-i+1, e_{\max }-i+1]$. Among all these segments, the only one that can be modified by $D_{y_0}$ is $[y_0,y_0]$. But as $y_1 < y_0$, $[y_0 - 1,y_0- 1] \in \m$, thus at least one segment of the form $[y_0,y_0]$ is not suppressed by $D_{y_0}$. We see that $\tilde{l} \ge i_0 - 1$ and for all $0<i<i_0$, $\tilde{\Delta}_{i}=\Delta_{i}=[e_{\max }-i+1, e_{\max }-i+1]$.
  
  We know that $(\m,\varepsilon)$ is $y_1$-reduced. If $y_1 = 1$, it means that $m_{\m}([0,0]) \neq 0$. In particular, $\Delta_{i_0} = [0,0]^{\ge 0}$ or $[0,0]^{=0}$ and $l=i_0$. The derivative $D_{y_0}$ does not create or suppress any new segment $[0,0]$, so $\tilde{\Delta}_{i_0} = \Delta_{i_0}$ and $\tilde{l}=l$. Now, we assume that $y_1 > 1$. By definition of $y_1$, we know that $[y_1 - 1,y_1 - 1] \notin \m$, and since $(\m,\varepsilon)$ is $y_1$-reduced, we get that $m_{\m}([1-y_1,y_1-1])=1$. In this case, we also know that for all $y_1 - 1 \ge y > 1$, $m_{\m}([-y+1,y-1]) = 1$ and $\varepsilon([-y,y]) \varepsilon([1-y,y-1])=-1$; and $m_{\m}([0,0]) \neq 0$ and $\varepsilon([-1,1]) \varepsilon([0,0])=(-1)^{t_0 + 1}$. Thus for all $ i_0 \le i < l$, $\Delta_i = [-e_{\max }+i-1,e_{\max }-i+1]^{=0}$; and $\Delta_l = [0,0]^{=0}$ or $[0,0]^{\le 0}$ if $t_0$ is even, and $\Delta_l = [-1,0]$ if $t_0$ is odd. None of these segments is changed by $D_{y_0}$ thus $\tilde{l}=l$ and for all $i_0 \le i \le l$, $\tilde{\Delta}_i = \Delta_i$.
\end{proof}

\begin{lem}
  \label{lem:ADderposx1inf}
    Suppose $y_1 < y_0$. Then $(\tilde{\m}_1,\tilde{\varepsilon}_1)=(\m_{1},\varepsilon_1)$ and $(\tilde{\m}^{\#},\tilde{\varepsilon}^{\#})=D_{y_0}(\m^{\#},\varepsilon^{\#})$.
\end{lem}

\begin{proof}
  By Lemma \ref{lem:derposx1inf}, we get that $(\tilde{\m}_1,\tilde{\varepsilon}_1)=(\m_{1},\varepsilon_1)$. Moreover, we know that for $1 \le i \le y_1$, $\Delta_{i}=[e_{\max }-i+1, e_{\max }-i+1]$. Hence, the segments of $\m^{\#}$ ending in $y_0$ or $y_0 - 1$ are the same as in $\m$ except that one segment $[y_0,y_0]$ and one $[y_0 - 1, y_0 - 1]$ have been suppressed. These segments were not modified by $D_{y_0}$. We deduce that $(\tilde{\m}^{\#},\tilde{\varepsilon}^{\#})=D_{y_0}(\m^{\#},\varepsilon^{\#})$.
\end{proof}

\begin{prop}
  If $y_1 < y_0$, then $\AD(D_{y_0}(\m,\varepsilon))=D_{-y_0}(\AD(\m,\varepsilon))$.
\end{prop}

\begin{proof}
  Remember that $(\tilde{\m},\tilde{\varepsilon})$ denotes $D_{y_0}(\m,\varepsilon)$. By definition of $\AD$ we have $\AD(\tilde{\m},\tilde{\varepsilon})=(\tilde{\m}_1,\tilde{\varepsilon}_1) + \AD(\tilde{\m}^{\#},\tilde{\varepsilon}^{\#})$. By Lemma \ref{lem:ADderposx1inf}, $(\tilde{\m}_1,\tilde{\varepsilon}_1)=(\m_{1},\varepsilon_1)$ and $(\tilde{\m}^{\#},\tilde{\varepsilon}^{\#})=D_{y_0}(\m^{\#},\varepsilon^{\#})$. Thus $\AD(\tilde{\m},\tilde{\varepsilon})=(\m_{1},\varepsilon_1) + \AD(D_{y_0}(\m^{\#},\varepsilon^{\#}))$. By Lemma \ref{lem:ADcommDer} $\AD(D_{y_0}(\m^{\#},\varepsilon^{\#}))=D_{-y_0}(\AD(\m^{\#},\varepsilon^{\#}))$. Note that $y_1 < y_0$ implies that $y_0 \neq 1$. Since $e(\Delta_l)=0$, we get by Lemma \ref{lem:dersum}, $(\m_{1},\varepsilon_1) + D_{-y_0}(\AD(\m^{\#},\varepsilon^{\#}))=D_{-y_0}((\m_{1},\varepsilon_1) +\AD(\m^{\#},\varepsilon^{\#}))=D_{-y_0}(\AD(\m,\varepsilon))$.
\end{proof}

\begin{lem}
  \label{lem:sequencederposx1eqx0}
  Suppose that $y_1 = y_0$. Let $j = e_{\max } - y_0 + 1$, such that $\Delta_j = [y_0,y_0]$.
  \begin{enumerate}
    \item Then, $e(\Delta_l)=y_0$ or $0$.
    \item If $e(\Delta_l)=y_0$, then $\tilde{l}=l-1$; otherwise $\tilde{l}=l$.
    \item For all $1 \le i \le \tilde{l}$, if $i \neq j,j+1$, then $\tilde{\Delta}_i=\Delta_i$.
    \item Assume $e(\Delta_l)=0$.
    \begin{enumerate}
      \item If $y_0 = 1$, $m_{\m}([0,0]) = 0$, $m_{\m}([-1,1])$ is odd and $t_0 \ge 2$, then $\Delta_{j+1} = [-1,0]$,  $\tilde{\Delta}_{j}=[0,1]$ and $\tilde{\Delta}_{j+1}=[0,0]^{\le 0}$. 
      \item If $y_0 = 1$, $m_{\m}([0,0]) = 0$, $m_{\m}([-1,1])$ is odd and $t_0 = 1$, then $\Delta_{j+1} = [-1,0]$, $\tilde{\Delta}_{j}=[-1,1]^{=0}$ and $\tilde{\Delta}_{j+1}=[0,0]^{\le 0}$. 
      \item If $y_0 = 1$, $m_{\m}([0,0]) = 0$ and $m_{\m}([-1,1])=0$, then $\Delta_{j+1} = [-1,0]$, $\tilde{\Delta}_{j}=[0,1]$ and $\tilde{\Delta}_{j+1}=[-1,0]$.
      \item If $y_0 = 1$, $m_{\m}([0,0]) = 0$, $m_{\m}([-1,1]) \neq 0$ and $m_{\m}([-1,1])$ is even, then $\Delta_{j+1} = [-1,0]$, $\tilde{\Delta}_{j}=[0,1]$ and $\tilde{\Delta}_{j+1}=[0,0]^{\le 0}$.
      \item If $y_0 = 1$, $m_{\m}([0,0]) \neq 0$ and $(*)$ is not satisfied, then $\Delta_{j+1} = [0,0]^{=0}$ or $[0,0]^{\ge 0}$, $\tilde{\Delta}_{j}=[1,1]$ and $\tilde{\Delta}_{j+1} = [0,0]^{=0}$ or $[0,0]^{\ge 0}$.
      \item If $y_0 = 1$,  $m_{\m}([0,0]) > 1$ and $(*)$ is satisfied, then $\Delta_{j+1} = [0,0]^{=0}$ or $[0,0]^{\ge 0}$,$\tilde{\Delta}_{j}=[1,1]$ and $\tilde{\Delta}_{j+1} = [0,0]^{=0}$ or $[0,0]^{\ge 0}$.
      \item If $y_0 = 1$, $m_{\m}([0,0]) = 1$, $t_0 \neq 0$ and $(*)$ is satisfied, then $\Delta_{j+1} = [0,0]^{=0}$ or $[0,0]^{\ge 0}$, $\tilde{\Delta}_{j}=[0,1]$ and $\tilde{\Delta}_{j+1} = [0,0]^{=0}$.
      \item If $y_0 = 1$, $m_{\m}([0,0]) = 1$, $t_0 = 0$, $m_{\m}([-1,1])$ is odd and $(*)$ is satisfied, then $\Delta_{j+1} = [0,0]^{=0}$ or $[0,0]^{\ge 0}$, $\tilde{\Delta}_{j}=[-1,1]$ and $\tilde{\Delta}_{j+1} = [0,0]^{=0}$.
      \item If $y_0 = 1$, $m_{\m}([0,0]) = 1$, $t_0 = 0$, $m_{\m}([-1,1])$ is even and $(*)$ is satisfied, then $\Delta_{j+1} = [0,0]^{=0}$ or $[0,0]^{\ge 0}$, $\tilde{\Delta}_{j}=[0,1]$ and $\tilde{\Delta}_{j+1} = [0,0]^{=0}$.
      \item If $y_0 = 2$, $t_0 \neq 0$, then $\Delta_{j+1} = [0,1]$, $\tilde{\Delta}_{j}=[y_0,y_0]$ and $\tilde{\Delta}_{j+1} = [0,1]$.
      \item If  $y_0 > 1$, $t_0 = 0$ if $y_0 = 2$ and $(*)$ is not satisfied, then $\Delta_{j+1} = [-y_0 + 1,y_0 - 1]^{=0}$, $\tilde{\Delta}_{j} = [y_0,y_0]$ and $\tilde{\Delta}_{j+1} = [1-y_0,y_0-1]^{\ge 0}$ or $[1-y_0,y_0-1]^{=0}$.
      \item If  $y_0 > 1$, $t_0 = 0$ if $y_0 = 2$, $(*)$ is satisfied and $m_{\m}([- y_0,y_0])$ is odd, then $\Delta_{j+1} = [-y_0 + 1,y_0 - 1]^{=0}$, $\tilde{\Delta}_{j} = [-y_0,y_0]^{=0}$ and $\tilde{\Delta}_{j + 1} = [1-y_0,y_0-1]^{=0}$.
      \item If $y_0 > 1$, $t_0 = 0$ if $y_0 = 2$, $(*)$ is satisfied and $m_{\m}([- y_0,y_0])$ is even, then $\Delta_{j+1} = [-y_0 + 1,y_0 - 1]^{=0}$, $\tilde{\Delta}_{j} = [1-y_0,y_0]$ and $\tilde{\Delta}_{j+1} = [1-y_0,y_0-1]^{=0}$ or $[1-y_0,y_0-1]^{\ge 0}$. 
    \end{enumerate}
  \end{enumerate}
\end{lem}

\begin{proof}
  As in Lemma \ref{lem:derposx1inf}, let $i_0 = e_{\max } - y_1 + 2$. Then, for all $0<i<i_0$, $\Delta_{i}=[e_{\max }-i+1, e_{\max }-i+1]$. This time, $\Delta_{i_0 - 1} = [y_0,y_0]$ could be modified by $D_{y_0}$. We get that $\tilde{l} \ge i_0 - 2$, and for all $i \le i_0 - 2$, $\tilde{\Delta}_i=\Delta_i$.

  \begin{itemize}
    \item Suppose $y_0 = 1$ and $e(\Delta_l)=1$. In particular $[0,0] \notin \m$, $t_0 = 0$ and $(*)$ is not satisfied. The derivative $D_1$ suppresses all the segments $[1,1]$ and all the $[-1,1]$. In $\tilde{\m}$ there are no other segments ending in $1$, so $\tilde{l}=l-1$.
    \item Suppose $y_0 = 1$ and $e(\Delta_l)<1$. If $m_{\m}([0,0]) \neq 0$ then $\Delta_l = [0,0]^{=0}$ or $[0,0]^{\ge 0}$. And if $m_{\m}([0,0]) = 0$ then $\Delta_l = [-1,0]$. 
    
    \begin{itemize}
      \item Suppose $m_{\m}([0,0]) = 0$, that is $\Delta_l = [-1,0]$. Then $(*)$ is not satisfied. And $D_1$ suppresses all the $[1,1]$. If $m_{\m}([-1,1])$ is odd and $t_0 \ge 2$; then $\tilde{l}=l$, $\tilde{\Delta}_{l-1}=[0,1]$ and $\tilde{\Delta}_{l}=[0,0]^{\le 0}$. If $m_{\m}([-1,1])$ is odd and $t_0 = 1$; then $\tilde{l}=l$, $\tilde{\Delta}_{l-1}=[-1,1]^{=0}$ and $\tilde{\Delta}_{l}=[0,0]^{\le 0}$. If $m_{\m}([-1,1])$ is even then $\tilde{l}=l$, $\tilde{\Delta}_{l-1}=[0,1]$, $\tilde{\Delta}_{l}=[-1,0]$ if $m_{\m}([-1,1]) = 0$ and $\tilde{\Delta}_{l}=[0,0]^{\le 0}$ if $m_{\m}([-1,1]) \neq 0$.
      \item Suppose $m_{\m}([0,0]) \neq 0$, that is $\Delta_l = [0,0]^{=0}$ or $[0,0]^{\ge 0}$. Suppose that $(*)$ is not satisfied. Then $\tilde{l}=l$, $\tilde{\Delta}_{l-1}=[1,1]$ and $\tilde{\Delta}_l = [0,0]^{=0}$ or $[0,0]^{\ge 0}$. Now suppose that $(*)$ is satisfied. If $m_{\m}([0,0]) > 1$, then $\tilde{l}=l$, $\tilde{\Delta}_{l-1}=[1,1]$ and $\tilde{\Delta}_l = [0,0]^{=0}$ or $[0,0]^{\ge 0}$. If $m_{\m}([0,0]) = 1$ and $t_0 \neq 0$, then $\tilde{l}=l$, $\tilde{\Delta}_{l-1}=[0,1]$ and $\tilde{\Delta}_l = [0,0]^{=0}$. If $m_{\m}([0,0]) = 1$, $t_0 = 0$ and $m_{\m}([-1,1])$ is odd, then $\tilde{l}=l$, $\tilde{\Delta}_{l-1}=[-1,1]$ and $\tilde{\Delta}_l = [0,0]^{=0}$.  If $m_{\m}([0,0]) = 1$, $t_0 = 0$ and $m_{\m}([-1,1])$ is even, then $\tilde{l}=l$, $\tilde{\Delta}_{l-1}=[0,1]$ and $\tilde{\Delta}_l = [0,0]^{=0}$.
    \end{itemize}
    \item Suppose $y_0 > 1$ and $e(\Delta_l)=y_0$. Then there are no segments ending in $y_0$ in $\m$. The derivative $D_{y_0}$ suppresses all the segments $[y_0,y_0]$ and all the $[-y_0,y_0]$. In $\tilde{m}$ there are no other segments ending in $y_0$, so $\tilde{l}=l-1$.
    \item Suppose $y_0 > 1$, $e(\Delta_l)<y_0$ and $t_0 = 0$ if $y_0 = 2$. Then $e(\Delta_l) = 0$ by Lemma \ref{lem:endpos}. Then $\Delta_{i_0} = [-y_0 + 1,y_0 - 1]^{=0}$. If $(*)$ is not satisfied, then $\tilde{l}=l$, $\tilde{\Delta}_{i_0 - 1} = [y_0,y_0]$ and $\tilde{\Delta}_{i_0} = [1+y_0,y_0-1]^{\ge 0}$ or $[1+y_0,y_0-1]^{=0}$. Suppose that $(*)$ is satisfied. Then $m_{\m}([1 - y_0,y_0 - 1]) = 1$. If $m_{\m}([- y_0,y_0])$ is odd, $\tilde{l}=l$, $\tilde{\Delta}_{i_0 - 1} = [-y_0,y_0]^{=0}$ and $\tilde{\Delta}_{i_0 } = [1-y_0,y_0-1]^{=0}$. If $m_{\m}([- y_0,y_0])$ is even, $\tilde{l}=l$, $\tilde{\Delta}_{i_0 - 1} = [1-y_0,y_0]$ and $\tilde{\Delta}_{i_0} = [1-y_0,y_0-1]^{=0}$ or $[1-y_0,y_0-1]^{\ge 0}$.
  \item Suppose $y_0 = 2$, $t_0 \neq 0$ and $e(\Delta_l)<y_0$. Then $e(\Delta_l) = 0$ by Lemma \ref{lem:endpos}. Then $\Delta_{i_0 } = [0,1]$. Then $\tilde{l}=l$, $\tilde{\Delta}_{i_0 - 1} = [2,2]$ and $\tilde{\Delta}_{i_0} = [0,1]$.
  \end{itemize}
\end{proof}

\begin{lem}
  \label{lem:ADderposx1eqx0}
    Suppose that $y_1 = y_0$.
    \begin{enumerate}
        \item If $e(\Delta_{l}) = y_0$, then $(\tilde{\m}_1,\tilde{\varepsilon}_1)=D_{-y_0}(\m_{1},\varepsilon_1)$ and $(\tilde{\m}^{\#},\tilde{\varepsilon}^{\#})=D_{y_0}(\m^{\#},\varepsilon^{\#})$.
        \item If $e(\Delta_{l}) =0$, then $(\tilde{\m}_1,\tilde{\varepsilon}_1)=(\m_{1},\varepsilon_1)$ and $(\tilde{\m}^{\#},\tilde{\varepsilon}^{\#})=D_{y_0}(\m^{\#},\varepsilon^{\#})$. 
    \end{enumerate}
\end{lem}

\begin{proof}
  The result for $(\tilde{\m}_1,\tilde{\varepsilon}_1)$ follows directly from Lemma \ref{lem:sequencederposx1eqx0}. If $e(\Delta_l)=y_0$, then, for all $1 \le i \le l$, $\Delta_{i}=[e_{\max }-i+1, e_{\max }-i+1]$, and these segments are just suppressed by $\AD$. It is easy to see that $(\tilde{\m}^{\#},\tilde{\varepsilon}^{\#})=D_{y_0}(\m^{\#},\varepsilon^{\#})$.

  If $e(\Delta_l)=0$, in each case of Lemma \ref{lem:sequencederposx1eqx0} we get that $(\tilde{\m}^{\#},\tilde{\varepsilon}^{\#})=D_{y_0}(\m^{\#},\varepsilon^{\#})$ by examining the formula of the derivative. Let us treat the first case; the others are handled similarly. Suppose that $y_0 = 1$, $m_{\m}([0,0]) = 0$, $m_{\m}([-1,1])$ is odd and $t_0 \ge 2$ then $\Delta_j = [1,1]$, $\Delta_{j+1} = [-1,0]$, $\tilde{\Delta}_{j}=[0,1]$ and $\tilde{\Delta}_{j+1}=[0,0]^{\le 0}$. Let us look at the effect of $\AD$ and $D_{1}$ on the segments ending in 1 and 0. The algorithm in $(\m,\varepsilon)$ suppresses one $[-1,-1] + [1,1]$ and transforms $[-1,0] + [0,1]$ into $[-1,-1] + [1,1]$. Thus it is similar to just suppressing $[-1,0] + [0,1]$. The algorithm $\AD$ in $D_1(\m,\varepsilon)$, transforms $[-1,0] + [0,1]$ into $[0,0] + [0,0]$ and suppresses $[0,0] + [0,0]$. Therefore, it is also similar to just suppressing $[-1,0] + [0,1]$. We can conclude  by examining the formula of $D_1$. In both $(\m,\varepsilon)$ and $(\m^{\#},\varepsilon^{\#})$, the derivative suppresses all the $[1,1] + [-1,1]$, transforms all the $[a,1]$ with $a < -1$ into $[a,0]$, transforms all the $[-1,1]$ except one into $[0,0]$ and transforms $[-1,0] + [0,1]$ into $[0,0] + [0,0]$. We get that $(\tilde{\m}^{\#},\tilde{\varepsilon}^{\#})=D_{y_0}(\m^{\#},\varepsilon^{\#})$.
\end{proof}

\begin{prop}
  If $y_1 = y_0$, then $\AD(D_{y_0}(\m,\varepsilon))=D_{-y_0}(\AD(\m,\varepsilon))$.
\end{prop}

\begin{proof}
  It follows from Lemmas \ref{lem:ADderposx1eqx0}, \ref{lem:ADcommDer} and \ref{lem:dersum}.
\end{proof}

\begin{lem}
    \label{lem:sequencederpos}
    Suppose that $y_1 > y_0$.
    \begin{enumerate}
        \item If $e(\Delta_l)=y_0$ then $\tilde{l}=l - 1$; otherwise $\tilde{l}=l$.
        \item For $1 \le i \le \tilde{l}$:
        \begin{enumerate}
          \item If $e(\Delta_i)\neq y_0$
          \begin{enumerate}
            \item If $y_0 = 1$, $\Delta_i=[-1,0]$ and $m_{\m}([-1,1]) \neq 0$, then $\tilde{\Delta}_{i}=[0,0]^{\le 0}$.
          \item If $(i)$ is not satisfied, $b(\Delta_i)=-y_0$ and $\Delta_i \neq [-1,0]$, then $\tilde{\Delta}_{i}={}^{-}\Delta_{i}$
          \item Otherwise, $\tilde{\Delta}_{i}=\Delta_{i}$. 
          \end{enumerate}
          \item If $e(\Delta_i)= y_0$
          \begin{enumerate}
            \item If $\Delta_{i}=[0,1]$, $(*)$ is not satisfied, $m_{\m}([-1,1])$ is odd and $t_0=1$, then $\tilde{\Delta}_{i}=[-1,1]^{=0}$.
            \item If $\Delta_{i}=[-y_0,y_0]$, $(*)$ is satisfied, $m_{\m}([-y_0,y_0])$ is even and $\Delta_{i-1}=[y_0+1,y_0+1]$, then $\tilde{\Delta}_{i}=[-y_0+1,y_0]$.
            \item If $\Delta_{i}=[-y_0,y_0]$, $(*)$ is satisfied and $m_{\m}([-y_0,y_0])$ is odd, then  $\tilde{\Delta}_{i}=[-y_0,y_0]^{=0}$.
            
            \item Otherwise, $\tilde{\Delta}_{i}=\Delta_{i}$.
          \end{enumerate}
          
        \end{enumerate}
    \end{enumerate}
\end{lem}

\begin{proof}
    Note that by definition $e(\Delta_{1})=e_{\max}$ is the maximum of the coefficients of the segments. Since by definition of $y_0$, $\Delta_{1} \neq [-y_0,-y_0]$ and $[-y_0,y_0]$, after taking the derivative, $\Delta_{1}$ does not vanish. Hence $e_{\max}$ is still the maximum of the coefficients in $\tilde{\m}$.

    Let $i_0 = e_{\max } - y_1 + 2$. Then, for all $0<i<i_0$, $\Delta_{i}=[e_{\max }-i+1, e_{\max }-i+1]$. All of these segments belong to $\tilde{\m}$, so $\tilde{l} \ge i_0 - 1$ and for all $0<i<i_0$, $\tilde{\Delta}_{i}=\Delta_{i}=[e_{\max }-i+1, e_{\max }-i+1]$.

    Since $y_0 < y_1$, there are no segments $[-y_0,-y_0]$ or $[y_0,y_0]$ in $\m$. The segments starting at $-y_0$ are of the form: $[-y_0,a]$ with $a>y_0$, $[-y_0,y_0]$ and $[-1,0]$ if $y_0 = 1$.

    Suppose that $e(\Delta_{i_0})>y_0$. By definition of $y_1$, $l(\Delta_{i_0})>1$. If $b(\Delta_{i_0})\neq -y_0$ then $\tilde{\Delta}_{i_0}=\Delta_{i_{0}}$, and if $b(\Delta_{i_0})= -y_0$  then $\tilde{\Delta}_{i_0}={}^{-}\Delta_{i_0}$. Let us study the case where $\tilde{\Delta}_{i_0}={}^{-}\Delta_{i_{0}}$ and $l\ge i_0 + 1$. By hypotheses, $\Delta_{i_{0}}$ is a segment such that $c(\Delta_{i_0})> 0$. Let $e = e(\Delta_{i_0})$. The only segment smaller than ${}^{-}\Delta_{i_{0}}$ but not than $\Delta_{i_{0}}$ ending in $e-1$ is $[-y_0,e-1]$ (if $e-1=y_0$, then it is $[-y_0,y_0]^{\ge 0}$).
    \begin{itemize}
        \item Suppose $e > y_0 +1$. Then $\Delta_{i_0 + 1}$ is unchanged in $\tilde{\m}$, we still have $\Delta_{i_0 + 1} \prec {}^{-}\Delta_{i_0}$. It remains maximal since there is no segment of the form $[-y_0,e-1]$  in $\tilde{\m}$. So $\tilde{l}\ge i_0 + 1$ and $\tilde{\Delta}_{i_0 + 1}=\Delta_{i_0 + 1}$.
        \item Suppose $e = y_0 +1$. Then $e(\Delta_{i_0 + 1})=y_0$ and $b(\Delta_{i_0 + 1})\le -y_0$. Thus $\Delta_{i_0 + 1}$ is a centered segment or a negative segment. 
        \begin{itemize}
            \item Suppose $\Delta_{i_0 + 1}$ is a negative segment. Then $l=i_0 + 1$ as there is no negative segment ending in $y_0 - 1$. Also  $[-y_0,y_0] \notin \m$ and thus $\tilde{l}=i_0$ (the only possible segment on $\tilde{\m}$ ending in $y_0$ is $[0,1]$ if $y_0 = 1$).
            \item Suppose $y_0 > 1$, $c(\Delta_{i_0 + 1})=0$ and $(*)$ is not satisfied. Then $l=i_0 + 1$ and $\tilde{l}=i_0$ (there is no segment ending in $y_0$ in $\tilde{\m}$).
            \item Suppose $y_0 = 1$, $c(\Delta_{i_0 + 1})=0$, $(*)$ is not satisfied, $m_{\m}([-1,1])$ is odd and $t_0 \ge 1$. Then $\Delta_{i_0 + 1}=[-1,1]^{=0}$. Also one and only one $[-1,1]$ is not changed by the derivative, thus $\tilde{\Delta}_{i_0 + 1}=\Delta_{i_0 + 1}$.
            \item Suppose $y_0 = 1$, $c(\Delta_{i_0 + 1})=0$, $(*)$ is not satisfied, $m_{\m}([-1,1])$ is even or $t_0 =0$. Then $l=i_0 + 1$ and $\tilde{l}=i_0$.
            \item Suppose $c(\Delta_{i_0 + 1})=0$, $(*)$ is satisfied and $m_{\m}([-y_0,y_0])$ is odd. Then $\Delta_{i_0 + 1} = [-y_0,y_0]^{=0}$. In $\tilde{\m}$ the only segment ending in $y_0$ different from $[y_0,y_0]$ or $[0,1]$ is $[-y_0,y_0]^{=0}$. Thus $\tilde{\Delta}_{i_0 + 1}=\Delta_{i_0 + 1}$.

            \item Suppose $c(\Delta_{i_0 + 1})=0$, $(*)$ is satisfied and $m_{\m}([-y_0,y_0])$ is even. Then $\Delta_{i_0 + 1} = [-y_0,y_0]^{\le 0}$. Thus $l=i_0 + 1$. In $\tilde{\m}$ the only segment ending in $y_0$ different from $[y_0,y_0]$ is $[-y_0 + 1,y_0]$. But, $[-y_0 + 1,y_0]$ is not smaller than $[-y_0 + 1,y_0 + 1] = \tilde{\Delta}_{i_0}$. Thus $\tilde{l}=i_0$.
        \end{itemize}
    \end{itemize}

    Let $i_1 \ge i_0 - 1$ be the biggest integer smaller than $l$ such that for all $i \leq i_1$, $e(\Delta_{i}) \neq y_0$.
    For all $i_0 + 1 \leq i \leq i_1$, we have that $b(\Delta_{i}) \neq -y_0$ and $e(\Delta_{i}) \neq y_0$, thus $\tilde{l}\ge i_1$ and  $\tilde{\Delta}_{i}=\Delta_{i}$. Moreover, if $i_1=l$ then $\tilde{l}=l$.

    Therefore, let us assume that $l > i_1$. Thus $e(\Delta_{i_1})=y_0 + 1$ and $e(\Delta_{i_1 + 1})=y_0$. In particular $p(\Delta_{i_1 + 1})=[-y_0,y_0]$, $\Delta_{i_1 + 1}=[0,1]$ (if $y_0 = 1$) or $c(\Delta_{i_1 + 1})<0$.
    \begin{itemize}
        \item Suppose $c(\Delta_{i_1 + 1})<0$, then necessarily $l=i_1 + 1$. In $\tilde{\m}$, there is no other negative segments ending in $y_0$ (and the signs of the centered segments are not changed), thus $\tilde{l}=i_1=l-1$.
        \item Suppose $p(\Delta_{i_1 + 1})=[-y_0,y_0]$ and $(*)$ is not satisfied. Note that in the case $y_0 = 1$, then $t_0 = 0$. Indeed, we have that $t_0 \le m_{\m}([-2,-2])-m_{\m}([-1,-1])=m_{\m}([-2,-2])$. Thus if $t_0 \neq 0$, then $m_{\m}([-2,-2]) \neq 0$ and $y_1 = 2$. Therefore, $\Delta_{i_1} = [2,2]$. But $[0,1] \prec [2,2]$, so we would have $\Delta_{i_1 + 1} = [0,1]$, which is not. Therefore, $l = i_1 + 1$ and $\tilde{l}=l-1$.
        \item Suppose $p(\Delta_{i_1 + 1})=[-y_0,y_0]$, $(*)$ is satisfied and $m_{\m}([-y_0,y_0])$ is odd. Then $l = y_0  + i_1 + 1$, if $i_1 + 2 \le i < l$ then $\Delta_i=[-e_{\max }+i-1,e_{\max }-i+1]^{=0}$, and $\Delta_{l}=[0,0]$ or $[-1,0]$.  In $\tilde{\m}$ there is still a $[-y_0,y_0]$, so  $\tilde{\Delta}_{i_1 + 1}=[-y_0,y_0]^{=0}$, $\tilde{l}=l$ and for all $i>i_1$, $\tilde{\Delta}_{i}=\Delta_{i}$.
        \item Suppose $\Delta_{i_1 + 1}=[-y_0,y_0]$, $(*)$ is satisfied, $m_{\m}([-y_0,y_0])$ is even and $\Delta_{i_1}=[y_0+1,y_0+1]$. Then $\Delta_{i_1 + 1}=[-y_0,y_0]^{\ge 0}$. Then $l = y_0  + i_1 + 1$, if $i_1 + 2 \le i < l$ then $\Delta_i=[-e_{\max }+i-1,e_{\max }-i+1]^{=0}$, and $\Delta_{l}=[0,0]$ or $[-1,0]$. In $\tilde{\m}$, $\tilde{\Delta}_{i_1 +1}=[-y_0 +1,y_0]$, $\tilde{l}=l$ and for all $i>i_1 + 1$, $\tilde{\Delta}_{i}=\Delta_{i}$. 
        \item Suppose $\Delta_{i_1 + 1}=[-y_0,y_0]$, $(*)$ is satisfied, $m_{\m}([-y_0,y_0])$ is even and $\Delta_{i_1}\neq [y_0+1,y_0+1]$. Then $b(\Delta_{i_1}) \le -y_0$, thus $\Delta_{i_1 + 1}=[-y_0,y_0]_{\le 0}$ and $l = i_1 + 1$. In $\tilde{\m}$, the only segment ending in $y_0$ different from $[y_0,y_0]$  is $[-y_0+1,y_0]$, thus $\tilde{l}=l - 1$.
        \item Suppose $y_0 = 1$ and $\Delta_{i_1 + 1}=[0,1]$. Necessarily $[-1,0] \in \m$ thus $l=i_1 + 2$. If $(*)$ is not satisfied, $m_{\m}([-1,1])$ is odd and $t_0=1$. Then $\tilde{l}=l$, $\tilde{\Delta}_{i_1 +1}=[-1,1]^{=0}$, $p(\tilde{\Delta}_{l})=[0,0]$. If $m_{\m}([0,0])\neq 0$, we get that $\tilde{\Delta}_{l}=\Delta_l$ (the parity of the multiplicity of $[0,0]$ does not change in $\tilde{\m}$). Otherwise, $\tilde{\Delta}_{i_1 +1}=[0,1]$ and $\tilde{l}=l$. And if $m_{\m}([-1,1])= 0$ or $m_{\m}([0,0])\neq 0$, then $\tilde{\Delta}_{l}=\Delta_l$; otherwise $\tilde{\Delta}_{l} = [0,0]^{\le 0}$.
    \end{itemize}
\end{proof}

\begin{lem}
  \label{lem:dermdiezepos}
  Suppose that $y_1 > y_0$.
  \begin{enumerate}
    \item If $e(\Delta_{l}) \ge y_0 + 2$, then $A_{y_0}^{\#,c} = A_{y_0}^{c}$.
    \item If $e(\Delta_{l}) = y_0 + 1$, then $A_{y_0}^{\#,c} = A_{y_0}^{c} \cup \{i_l\}$. 
    \item If $e(\Delta_{l}) = y_0$. We can assume that $i_l \in A_{y_0}^{c}$. If $\Delta_l = [-y_0,y_0]^{\le 0}$ and $(*)$ is not satisfied in $(\m,\varepsilon)$, then we can also assume that $i'_l \in A_{y_0}^{c}$ and we get that $A_{y_0}^{\#,c} = A_{y_0}^{c} \cup \{i_{l-1}\} \setminus \{i_l,i'_l\}$; otherwise $A_{y_0}^{\#,c} = A_{y_0}^{c} \setminus \{i_l\}$.     
    \item If $e(\Delta_{l}) = 0$. Let $j$ such that $e(\Delta_j)=y_0$. We can assume that $i_j \notin A_{y_0}^{c}$. Then $A_{y_0}^{\#,c} = A_{y_0}^{c}$.
  \end{enumerate}
\end{lem}

\begin{proof}
  Let us recall that the segments ending in $y_0$ are of the form $[a,y_0]$, with $a < -y_0$, $[-y_0,y_0]$ and $[0,1]$, if $y_0 = 1$ (there are no $[y_0,y_0]$ since $y_0 < y_1$). And the segments ending in $y_0 - 1$ are of the form $[1-y_0,y_0 - 1]$ (with multiplicity 1 if $y_0 \neq 1$), $[0,1]$ if $y_0 = 2$, and $[-1,0]$ if $y_0 = 1$. We have:
\[
  A_{y_0}^{c} = \left\{
    \begin{array}{ll}
      A_{y_0}           & \text{if } y_0 \neq 1 \text{ and } (*) \text{ is not satisfied } \\
      A_{y_0} \setminus \{j\}       & \text{if } y_0 \neq 1 \text{ and } (*) \text{ is satisfied, with } \Lambda_j = [-y_0,y_0]  \\
      A_{y_0} \setminus \{i , \Lambda_i = [0,1]\}       & \text{if } y_0 = 1 \text{ and } (*) \text{ is not satisfied }  \\
      A_{y_0} \setminus (\{i , \Lambda_i = [0,1]\} \cup \{j\})      & \text{if } y_0 = 1 \text{ and } (*) \text{ is satisfied, with } \Lambda_j = [-1,1]  \\
    \end{array}
    \right.
\]

  Now, $A_{y_0}^{\#} = \{i \in A_{y_0}, \Lambda_i^{\#} \neq 0\} \cup \{i \in \{i_1,\cdots,i_l\}, e(\Lambda_i)=y_0 + 1 \text{ and } \Lambda_i^{\#} \neq 0\} \setminus \{i \in \{i_1,\cdots,i_l\}, e(\Lambda_i)=y_0\}$ and we have a similar description for $A_{y_0 - 1}^{\#}$. Since, there are no $[y_0,y_0]$ in $A_{y_0}$, if $i \notin \{i_1,\cdots,i_l\}$ and $e(\Lambda_i)=y_0$ then $\Lambda_i^{\#} \neq 0$. Let $i \notin \{i_1,\cdots,i_l\}$ such that $e(\Lambda_i)=y_0 - 1$ and $\Lambda_i^{\#} = 0$. Then necessarily $y_0 = 1$ and there exists a $j$ such that $\Delta_j = [0,0]^{\le 0}$. In this case, $\Delta_j = \Delta_l$ (in particular $e(\Delta_l)=0$) and $i = i'_l$.

  \begin{itemize}
  \item Suppose $e(\Delta_{l}) > y_0 + 1$. From the previous description we see that $A^{\#}_{y_0} = A_{y_0}$ and $A^{\#}_{y_0 - 1} = A_{y_0 - 1}$. Moreover, there are no creation or suppression of $[-y_0,y_0]$ or $[1-y_0,y_0 - 1]$, so $(*)$ is satisfied for $\m^{\#}$ if and only if it is satisfied for $\m$. Thus $A_{y_0}^{\#,c} = A_{y_0}^{c}$.
  \item Suppose $e(\Delta_{l}) = y_0 + 1$. Then $\Lambda_{i_l}^{\#} \neq 0$, indeed, $\Delta_l \neq [y_0 + 1,y_0 +1]$ because this segment would be followed in the initial sequence by any segment ending in $y_0$. Hence $A^{\#}_{y_0} = A_{y_0} \cup\{i_{l}\}$ and $A^{\#}_{y_0 - 1} = A_{y_0 - 1}$. Let us show that $i_l \in A^{c,\#}_{y_0}$. From the description of $A_{y_0 - 1}$, we see that for all $i \in A_{y_0 - 1}$, $\Lambda_i^{\#} = \Lambda_i$. We have seen that $\Lambda_{i_l} \neq [y_0 + 1,y_0 +1]$. And $\Lambda_{i_l} \neq [-y_0,y_0 +1]$ because it would be followed in the initial sequence by $[-y_0-1,y_0]$. Thus $c(\Lambda_{i_l})=0$ or $c(\Lambda_{i_l})<0$. If $c(\Lambda_{i_l})<0$, then $i_l \in A^{c,\#}_{y_0}$. If $c(\Lambda_{i_l})=0$, that is $\Lambda_{i_l}=[-y_0 - 1, y_0 + 1]$. If $\Lambda_{i_l}=[-y_0 - 1, y_0 + 1]^{\le 0}$ then $\Lambda_{i_l}^{\#} = [-y_0 - 1,y_0]$ and $i_l \in A^{c,\#}_{y_0}$. If $\Lambda_{i_l}=[-y_0 - 1, y_0 + 1]^{=0}$ or $[-y_0 - 1, y_0 + 1]^{ \ge 0}$ then $\Lambda_{i_l}^{\#} = [-y_0,y_0]$. Thus we need to investigate condition $(*)$. Since $\Lambda_{i_l}=[-y_0 - 1, y_0 + 1]^{=0}$ or $[-y_0 - 1, y_0 + 1]^{ \ge 0}$ is the last segment in the initial sequence in the algorithm, there are no negative segment ending in $y_0$. Since $A_{y_0}^{c} \neq \emptyset$,  $m_{\m}([-y_0,y_0])\neq 0$. And since $\Delta_l$ is the last segment in the initial sequence in the algorithm, $\varepsilon([-1 - y_0, y_0 + 1])\varepsilon([-y_0, y_0])=1$. Thus $(*)$ is satisfied for $(\m,\varepsilon)$ if and only if $(*)$ is still satisfied in $(\m^{\#},\varepsilon^{\#})$; and $i_l \in A^{c,\#}_{y_0}$.
  \item Suppose $e(\Delta_{l}) = y_0$. First let us examine the case $\Lambda_{i_{l-1}}^{\#}=0$, that is $\Lambda_{i_{l-1}}=[y_0 + 1,y_0 + 1]$. As $y_0 < y_1$, we get that $y_1 = y_0 + 1$, $[y_0 ,y_0] \notin \m$, and $\Lambda_{i_l}$ is the biggest segment ending in $y_0$. Then $A_{y_0}^{\#} =  A_{y_0}  \setminus \{i_l\}$ and $A_{y_0 - 1}^{\#} =  A_{y_0 - 1} \cup \{i_{l}\}$. Since $\Lambda_{i_l}$ is the biggest segment ending in $y_0$, we have that for all $i \in A_{y_0}^{\#}$, $\Lambda_{i}^{\#} \le \Lambda_{i_l}^{\#}$. Thus $A_{y_0}^{\#,c} = A_{y_0}^{c} \setminus \{i_l\}$. If $\Lambda_{i_{l-1}} \neq [y_0 + 1,y_0 + 1]$. This time, $A_{y_0}^{\#} =  A_{y_0} \cup \{i_{l-1}\} \setminus \{i_l\}$ and $A_{y_0 - 1}^{\#} =  A_{y_0 - 1} \cup \{i_{l}\}$. Since $\Delta_l \neq [0,1]$ (it is the last segment in the initial sequence in the algorithm), we get that $c(\Delta_l)
<0$ or $c(\Delta_l) = 0$.
  \begin{itemize}
    \item Suppose $c(\Delta_{l}) < 0$. The situation is similar as in Lemma \ref{lem:dermdieze}. The segment $\Lambda_{i_l}^{\#}$ is the smallest segment ending in $y_0 - 1$ such that $\Lambda_{i_l}^{\#} < \Lambda_{i_{l-1}}^{\#}$. Thus $A_{y_0}^{\#,c} = A_{y_0}^{c} \setminus \{i_l\}$.
    
    \item Suppose $c(\Delta_{l}) = 0$. Then, $\Lambda_{i_{l-1}} =[-y_0 - 1,y_0 + 1]^{\ge 0}$, $[-y_0 - 1,y_0 + 1]^{=0}$ or $[-y_0,y_0 + 1]$; and $\Lambda_{i_{l-1}}^{\#} =[-y_0,y_0]$. Note that if $y_0 = 1$, since $\Delta_l$ is the last segment in the initial sequence in the algorithm, then $t_0 = 0$.
    \begin{itemize}
      \item Suppose $\Lambda_{i_l} = [-y_0,y_0]^{=0}$ or $[-y_0,y_0]^{\ge 0}$. Since $\Lambda_{i_l}$ is the last segment in the initial sequence in the algorithm,  condition $(*)$ cannot be satisfied. We can assume that $i_l \in A_{y_0}^{c}$. Also, $\Lambda_{i_l}^{\#} = [-y_0+1,y_0-1]$. Then necessarily $(*)$ is satisfied in $(\m^{\#},\varepsilon^{\#})$, and $A_{y_0}^{\#,c} = A_{y_0}^{c}  \setminus \{i_l\}$.
      \item Suppose $\Lambda_{i_l} = [-y_0,y_0]^{\le 0}$. There is no segment $[-y_0,y_0]^{=0}$ in $\m$ and $m_{\m}([-y_0,y_0])$ is even. In particular, we can always assume that $i_l \in A_{y_0}^{c}$. Moreover, $\Lambda_{i_l}^{\#} = [-y_0,y_0 -1]$, so we see that if $(*)$ is satisfied in $(\m,\varepsilon)$ if and only if $(*)$ is satisfied in $(\m^{\#},\varepsilon^{\#})$. So if $(*)$ is not satisfied in $(\m,\varepsilon)$, then we can assume that $i'_l \in A_{y_0}^{c}$, and $A_{y_0}^{\#,c} = A_{y_0}^{c} \cup \{i_{l-1}\} \setminus \{i_l,i'_{l}\}$. And if $(*)$ is not satisfied in $(\m,\varepsilon)$ then $A_{y_0}^{\#,c} = A_{y_0}^{c}  \setminus \{i_l\}$.
    \end{itemize}
    
  \item Suppose $e(\Delta_{l}) = 0$. Let us split the proof into two cases depending on whether $y_0 = 1$ or $y_0 >1$.
  \begin{itemize}
    \item Suppose $y_0 > 1$. Let $j$ such that $e(\Delta_j)=y_0$. The segment $\Delta_{j+1}$ cannot be $[0,1]$ because this segment cannot be in the initial sequence of the algorithm after a segment of the form $[-y_0,y_0]$ or $[a,y_0]$, with $a<-y_0$. Thus $\Delta_{j+1}=[1-y_0,y_0-1]^{=0}$. Thus $\Delta_j = [-y_0,y_0]^{\ge 0}$ or $[-y_0,y_0]^{=0}$. In particular, we see that $(*)$ is satisfied in $(\m,\varepsilon)$, and we can assume that $i_j \notin A_{y_0}^{c}$. Moreover $\Lambda_{i_j}^{\#} = [1-y_0,y_0-1]$. The segment $\Delta_{j-1}$ is one of the following : $[y_0+ 1,y_0 +1]$, $[-y_0 - 1,y_0 + 1]^{\ge 0}$, $[-y_0 - 1,y_0 + 1]^{=0}$ or $[-y_0,y_0 + 1]$. Hence, $\Lambda_{i_{j-1}}^{\#}$ is either $0$ or $[-y_0,y_0]$. Also, by definition of $\Delta_j$ and $\Delta_{j-1}$ in the initial sequence of the algorithm, necessarily $(*)$ is satisfied in $(\m^{\#},\varepsilon^{\#})$, and we see that $i_{j-1} \notin A_{y_0}^{\#,c}$. We get that $A_{y_0}^{\#,c} = A_{y_0}^{c}$.
    \item Suppose $y_0 = 1$. Then $e(\Delta_{l-1})=y_0$. The segment $\Delta_{l-1}$ cannot be of the form $[a,1]$ with $a<-1$, since it is followed in the initial sequence by $\Delta_l$ which is $[0,0]$ or $[-1,0]$. Thus $\Delta_{l-1}$ is $[-1,1]$ or $[0,1]$. 
  
  Suppose $t_0 \neq 0$. Since $t_0 \le m_{\m}([-2,-2])-m_{\m}([-1,-1])$, and $m_{\m}([-1,-1]) = 0$, we get that $m_{\m}([-2,-2]) \neq 0$, $y_1 = 2$ and $\Delta_{l-2}=[2,2]$. Hence, $\Delta_{l-1}=[0,1]$. Examining the two cases $\Delta_l = [0,0]$ or $\Delta_l = [-1,0]$, we see that $(*)$ is satisfied in $(\m,\varepsilon)$ if and only if $(*)$ is satisfied in $(\m^{\#},\varepsilon^{\#})$. Therefore, $A_{y_0}^{\#,c} = A_{y_0}^{c}$.

  Suppose $t_0 = 0$. Then necessarily $\Delta_l = [0,0]$ and $\Delta_{l-1}=[-1,1]^{=0}$ or $\Delta_{l-1}=[-1,1]^{\ge 0}$. Also $m_{\m}([-2,-2]) = 0$. The situation is similar as the case $y_0 \neq 1$. The condition $(*)$ is satisfied in $(\m,\varepsilon)$ and in $(\m^{\#},\varepsilon^{\#})$; and $A_{y_0}^{\#,c} = A_{y_0}^{c}$
  \end{itemize}
  \end{itemize}
  \end{itemize}
\end{proof}

\begin{lem}
  \label{lem:ADderpos}
  Suppose that $y_1 > y_0$.
  \begin{enumerate}
      \item If $e(\Delta_{l}) \ge y_0 + 2$, then $(\tilde{\m}_1,\tilde{\varepsilon}_1)=(\m_{1},\varepsilon_1)$ and $(\tilde{\m}^{\#},\tilde{\varepsilon}^{\#})=D_{y_0}(\m^{\#},\varepsilon^{\#})$.
      \item If $e(\Delta_{l}) = y_0 + 1$, then $(\tilde{\m}_1,\tilde{\varepsilon}_1)=(\m_{1},\varepsilon_1)$ and $(\tilde{\m}^{\#},\tilde{\varepsilon}^{\#})=D^{\max - 1}_{y_0}(\m^{\#},\varepsilon^{\#})$. 
      \item If $e(\Delta_{l}) = y_0$, then $(\tilde{\m}_1,\tilde{\varepsilon}_1)=D_{-y_0}(\m_{1},\varepsilon_1)$ and $(\tilde{\m}^{\#},\tilde{\varepsilon}^{\#})=D_{y_0}(\m^{\#},\varepsilon^{\#})$.
      \item If $e(\Delta_{l}) = 0$, then $(\tilde{\m}_1,\tilde{\varepsilon}_1)=(\m_{1},\varepsilon_1)$ and $(\tilde{\m}^{\#},\tilde{\varepsilon}^{\#})=D_{y_0}(\m^{\#},\varepsilon^{\#})$.
     
  \end{enumerate}
\end{lem}

\begin{proof}
  The proof is similar to Lemma \ref{lem:ADderneg}. Lemma \ref{lem:sequencederpos} tells us that $\tilde{l}=l-1$, if $e(\Delta_l)=y_0$, and $\tilde{l}=l$ otherwise. It also tells us that $\varepsilon_0 = \tilde{\varepsilon}_0$. This gave us the result for $(\m_1,\varepsilon_1)$ in all the cases, apart when $\varepsilon_0 = -1$, where we are left to prove that $\varepsilon_1 = \tilde{\varepsilon}_1$. Let $n_0 =  \card\{\Delta \in \m, c(\Delta)=0\}$ and  $\tilde{n}_0 =  \card\{\Delta \in \tilde{\m}, c(\Delta)=0\}$. From the formula of $D_{y_0}$ the parity of the number of $[-y_0,y_0]$ is not changed (in the case $y_0 \neq 1$, necessarily there is an $i$ such that $\Delta_i = [-y_0,y_0]$, and thus $(*)$ is satisfied). Hence $n_0 \equiv \tilde{n}_0 \pmod{2}$ and $\varepsilon_1 = \tilde{\varepsilon}_1$.

  Now, let us examine $\tilde{\m}^{\#}$. This is mostly similar to Lemma \ref{lem:ADderneg}.
  \begin{itemize}
    \item Suppose $e(\Delta_l)>y_0 + 1$. By Lemma \ref{lem:dermdiezepos}, $A_{y_0}^{\#,c}=A_{y_0}^{c}$. Combined with Lemma \ref{lem:sequencederpos}, we get that $(\tilde{\m}^{\#},\tilde{\varepsilon}^{\#})=D_{y_0}(\m^{\#},\varepsilon^{\#})$.
    \item Suppose $e(\Delta_l)=y_0 + 1$. This time we have $A_{y_0}^{\#,c}=A_{y_0}^{c} \cup \{i_l\}$. Similarly as in Lemma \ref{lem:ADderneg}, we can show that in $\tilde{\m}^{\#}$, $\tilde{\Lambda}^{\#}_{i_l}$ is modified by $D_{y_0}$. Thus we get that $D_{y_0}(\m^{\#},\varepsilon^{\#}) = D^{1}_{y_0}(\tilde{\m}^{\#},\tilde{\varepsilon}^{\#})$, or that $(\tilde{\m}^{\#},\tilde{\varepsilon}^{\#})=D^{\max - 1}_{y_0}(\m^{\#},\varepsilon^{\#})$.
    \item Suppose  $e(\Delta_l)=y_0$.
    \begin{itemize}
      \item If $\Delta_l =[-y_0,y_0]^{\le 0}$ and $(*)$ is not satisfied in $(\m,\varepsilon)$, then by Lemma \ref{lem:dermdiezepos}, $A_{y_0}^{\#,c}=A_{y_0}^{c} \cup \{i_{l-1}\} \setminus \{i_l,i'_l\}$.  Thus, we almost have the same modifications, apart from $i_{l-1}$, $i_l$ and $i'_l$. Since $(*)$ is not satisfied in $(\m,\varepsilon)$, the derivative $D_{y_0}(\m,\varepsilon)$ transforms $\Lambda_{i_l}=[-y_0,y_0]$ and $\Lambda_{i'_l}=[-y_0,y_0]$ into $2 [-y_0 + 1, y_0 -1]$. And $\tilde{\Lambda}_{i_{l-1}}=\Delta_{l-1}$, thus $\tilde{\Lambda}_{i_{l-1}}^{\#}=[-y_0,y_0]$. In $\m^{\#}$, we add $\Lambda_{i_{l-1}}^{\#}=[-y_0,y_0]$ and we transform $\Lambda_{i_{l}}$ into $\Lambda_{i_{l}}^{\#}=[-y_0,y_0-1]$ and $\Lambda_{i'_{l}}$ into $\Lambda_{i'_{l}}^{\#}=[-y_0+1,y_0]$. Applying the modifications of the segments indexed by $A_{y_0}^{\#,c}=A_{y_0}^{c} \cup \{i_{l-1}\} \setminus \{i_l,i'_l\}$ transforms the $[-y_0,y_0]$ into $[-y_0 + 1, y_0 -1]$. But in $(\m^{\#},\varepsilon^{\#})$, we get that $(*)$ is not satisfied, $t^{\#} \ge 1$ and $c^{\#}$ is odd. So, according to Definition \ref{def:der}, these modifications are not sufficient to compute the derivative. We need to transform $[-y_0,y_0-1] + [-y_0+1,y_0]$ into $[-y_0,y_0] + [-y_0 + 1, y_0 -1]$. Thus we get the same thing, and $(\tilde{\m}^{\#},\tilde{\varepsilon}^{\#})=D_{y_0}(\m^{\#},\varepsilon^{\#})$.
      \item Otherwise,  $A_{y_0}^{\#,c}=A_{y_0}^{c} \setminus \{i_l\}$, and as in Lemma \ref{lem:ADderneg}, $(\tilde{\m}^{\#},\tilde{\varepsilon}^{\#})=D_{y_0}(\m^{\#},\varepsilon^{\#})$.
    \end{itemize}
    
    \item Suppose  $e(\Delta_l)=0$. Then $A_{y_0}^{\#,c}=A_{y_0}^{c}$. Then in all the cases of Lemma \ref{lem:sequencederpos}, we see that $(\tilde{\m}^{\#},\tilde{\varepsilon}^{\#})=D_{y_0}(\m^{\#},\varepsilon^{\#})$. Let us do one for illustration. Suppose that $y_0 = 1$, $\Delta_{l-1}=[0,1]$, $(*)$ is not satisfied, $m_{\m}([-1,1])$ is odd, $m_{\m}([0,0])\neq 0$ and $t_0 = 1$. Since $\Delta_{l-1}=[0,1]$, applying the algorithm transforms $[0,1] + [-1,0]$ into $2 [0,0]$. Then, we modify the segments in $A_{y_0}^{\#,c}=A_{y_0}^{c}$ to get $D_{y_0}(\m^{\#},\varepsilon^{\#})$. To compute $(\tilde{\m},\tilde{\varepsilon})$, we modify the segments indexed by $A_{y_0}^{c}$. But since $(*)$ is not satisfied, $c$ is odd and $t_0 \ge 1$, we also need to transform  $[0,1] + [-1,0]$ into $[-1,1] + [0,0]$. But then $\tilde{\Delta}_{l-1}=[-1,1]^{=0}$ by Lemma \ref{lem:sequencederpos}, so one $[-1,1]$ is transformed into $[0,0]$ and we get that $(\tilde{\m}^{\#},\tilde{\varepsilon}^{\#})=D_{y_0}(\m^{\#},\varepsilon^{\#})$. The other cases are treated similarly.
  \end{itemize}
\end{proof}

\begin{prop}
  If $y_1 > y_0$, then $\AD(D_{y_0}(\m,\varepsilon))=D_{-y_0}(\AD(\m,\varepsilon))$.
\end{prop}

\begin{proof}
  It follows from Lemmas \ref{lem:ADderpos}, \ref{lem:ADcommDer} and \ref{lem:dersum}.
\end{proof}

\subsection{The positive derivative with $\rho$ not of same type} 
\label{sec:derposnottype}

In this section, we assume that $\rho$ is not of the same type as $G$. We also assume that $e_{\max} > 1$, that for all $-e_{\max} < y < 0$, $(\m,\varepsilon)$ is $y$-reduced, and that there exists $y > 0$ with $y \neq e_{\max}$ such that $(\m,\varepsilon)$ is not $y$-reduced.

We define $y_0 \in (1/2)\Z \setminus \Z$ to be the smallest $y \in (1/2)\Z^{*}\setminus \Z$ such that $y \neq -e_{\max}$, $y \neq e_{\max}$, and $(\m,\varepsilon)$ is not $y$-reduced. With our hypotheses on $(\m,\varepsilon)$, necessarily $y_0 > 0$. Let us describe more precisely the conditions satisfied by $(\m,\varepsilon)$ and $y_0$ using the explicit formula of the derivative recalled in Section \ref{sec:expder}.

\bigskip

We have assumed that for all $-e_{\max} < y<0$, $(\m,\varepsilon)$ is $y$-reduced. Hence, if $\Delta \in \m$ is a segment such that $e(\Delta) < 0$, then $\Delta = [y,y]$ (where $y=e(\Delta)$) and $m_{\m}([y,y]) \le m_{\m}([y-1,y-1])$. Let $y_1 \in (1/2)\NN^{*} \setminus \NN$ be the smallest positive half-integer such that $[y_1,y_1] \in \m$.

\bigskip

Now, we recall the formula for $D_{y_0}$. First let us do $D_{1/2}$.

By convention, when $y=1/2$, we set $[-y+1,y-1]=0$, $m_{\m}([-y+1,y-1])=1$ and $\varepsilon([-y+1,y-1])=1$. Let $t_{1/2}$ be the number of $[1/2,1/2]$ in $\m$.

The segments $\Delta \in \m$ with $e(\Delta)=-1/2$ are of the form: $[-1/2,-1/2]$. And the segments $\Delta \in \m$ with $e(\Delta)=1/2$ are of the form: $[a,1/2]$, with $a<-1/2$, $[-1/2,1/2]$ or $[1/2,1/2]$. Let $(*)$ be the condition: $m_{\m}([-1/2,1/2])\neq 0$ and $\varepsilon([-1/2,1/2])=(-1)^{t_{1/2} + 1}$.

Then $D_{1/2}$ does the following transformations. All the negative segments of the form $[a,1/2]$, with $a<-1/2$, are transformed into $[a,-1/2]$ (and their symmetric $[-1/2,-a]$ are transformed into $[1/2,-a]$).
\begin{itemize}
  \item If $(*)$ is not satisfied,  $m_{\m}([-1/2,1/2])$ is odd and $t_{1/2} \ge 1$; then all the segments $[-1/2,1/2]$ except one are suppressed and one $[-1/2,-1/2] + [1/2,1/2]$ is also suppressed.
  \item If $(*)$ is not satisfied and, $m_{\m}([-1/2,1/2])$ is even or $t_{1/2} = 0$; then all the segments $[-1/2,1/2]$ are suppressed.
  \item If $(*)$ is satisfied and $m_{\m}([-1/2,1/2])$ is odd; then all the segments $[-1/2,1/2]$ except one are suppressed.
  \item If $(*)$ is satisfied and $m_{\m}([-1/2,1/2])$ is even; then all the $[-1/2,1/2]$ are suppressed and two $[-1/2,-1/2]+[1/2,1/2]$ are added.
\end{itemize}

In particular, we see that if $D_{1/2}(\m,\varepsilon)=0$, then 
\begin{itemize}
  \item There is no segment $[a,1/2]$ with $a<-1/2$;
  \item $m_{\m}([-1/2,1/2]) = 0$ or $1$;
  \item if $m_{\m}([-1/2,1/2]) = 1$ then $\varepsilon([-1/2,1/2]) =(-1)^{t_{1/2} + 1}$.
\end{itemize}

\bigskip

Let $y > 1/2$ such that for all $1/2 \le x < y$, $D_x(\m,\varepsilon)=0$. We prove by induction on $y$, that if $\Delta \in \m$ is a segment with $0< e(\Delta) < y$ then $c(\Delta) \ge 0$. We make explicit the formula for $D_y$.

Using the induction hypothesis, if $1/2 < x < y$, and $\Delta \in \m$ with $e(\Delta) = x$ then $b(\Delta) \ge -x$. We also know that if $b(\Delta) \neq y$ then $b(\Delta)<0$. But $\Delta^{\vee} \in \m$, so if $b(\Delta) \neq y$ then $b(\Delta) = -x$. Hence, the only segments ending in $y$ are $[x,x]$ and $[-x,x]$. In particular, note that there are no segments $[-y+1,y]$ in $\m$. Also the segments $\Delta \in \m$ with $e(\Delta)=y-1$ are of the form: $[y-1,y-1]$ or $[1-y,y-1]$. And the segments $\Delta \in \m$ with $e(\Delta)=y$ are of the form: $[a,y]$, with $a<-y$, $[-y,y]$, $[y,y]$. Let $(*)$ be the condition: $m_{\m}([-y,y])\neq 0$, $m_{\m}([1-y,y-1])\neq 0$ and $\varepsilon([y,y]) \varepsilon([1-y,y-1])=-1$.

Then $D_{y}$ does the following transformations. All the negative segments of the form $[a,y]$, with $a < -y$, are transformed into $[a,y - 1]$ (and their symmetric $[-y,-a]$ are transformed into $[-y + 1,-a]$). For the centered segments $[-y,y]$:
\begin{itemize}
  \item If $(*)$ is not satisfied; then all the segments $[-y,y]$ are transformed into $[-y + 1,y - 1]$.
  \item If $(*)$ is satisfied and $m_{\m}([-y,y])$ is odd; then all the segments $[-y,y]$ except one are transformed into $[-y + 1,y - 1]$.
  \item If $(*)$ is satisfied and $m_{\m}([-y,y])$ is even; then two segments $[-y,y]$ are transformed into $[-y,y - 1]+[-y + 1,y]$. The other segments $[-y,y]$ are transformed into $[-y + 1,y - 1]$.
\end{itemize}
It also suppresses $n$ segments $[y,y]$ and $[-y,-y]$, where $n=\max\{m_{\m}([y,y])-m_{\m}([y-1,y-1]) - m_{\m}([1-y,y-1]),0\}$ if $(*)$ is not satisfied; and $n=\max\{m_{\m}([y,y])-m_{\m}([y-1,y-1])- m_{\m}([1-y,y-1])+1,0\}$ if $(*)$ is satisfied. Note that we will show that $m_{\m}([1-y,y-1]) = 0$ or $1$.

In particular, we see that if $D_y(\m,\varepsilon)=0$, then 
\begin{itemize}
  \item There is no segment $[a,y]$ with $a<-y$;
  \item $m_{\m}([-y,y]) = 0$ or $1$;
  \item if $m_{\m}([-y,y]) = 1$ then $m_{\m}([1-y,y-1]) \neq 0$ and $\varepsilon([-y,y]) \varepsilon([1-y,y-1])=-1$;
  \item $m_{\m}([y,y]) \le m_{\m}([y-1,y-1]) + m_{\m}([1-y,y-1])$ if $m_{\m}([-y,y]) = 0$; and $m_{\m}([y,y]) \le m_{\m}([y-1,y-1])$ if $m_{\m}([-y,y]) = 1$.
\end{itemize}
This proves our induction hypothesis.

\bigskip

The formulas above give us an explicit formula for $D_{y_0}$, depending on whether $y_0 = 1/2$ or not. Moreover, we see that if $\Delta \in \m$ is such that $e(\Delta) < y_0$, then either $\Delta = [y,y]$ for some $y$, or $\Delta = [-y,y]$, and in this case $m_{\m}([-y,y]) = 1$. Furthermore, if $y > 1/2$, then $\varepsilon([-y,y]) \varepsilon([1-y,y-1]) = -1$; and if $y = 1/2$, then $\varepsilon([-1/2,1/2]) = (-1)^{t_{1/2} + 1}$.

\bigskip

For the rest of the section, let $(*)$ denote the condition:
\begin{enumerate}
    \item If $y_0 = 1/2$: $m_{\m}([-1/2,1/2])\neq 0$ and $\varepsilon([-1/2,1/2]) =(-1)^{t_{1/2} + 1}$.
    \item If $y_0 > 1/2$: $m_{\m}([-y_0,y_0])\neq 0$, $m_{\m}([-y_0 + 1,y_0-1])\neq 0$ and $\varepsilon([-y_0,y_0]) \varepsilon([1-y_0,y_0-1])=-1$.
\end{enumerate}

\begin{lem}
  \label{lem:endposnottype}
  If $e(\Delta_l)< y_0$ then $e(\Delta_l) = 1/2$.
\end{lem}

\begin{proof}
  We assume that $e(\Delta_l)< y_0$. From Lemma \ref{lem:end01}, we have that $e(\Delta_l) \ge 0$. Thus $\Delta_l$ is one of the following segments : $[-y,y]$ or $[y,y]$. If $\Delta_l = [-y,y]$ or $[y,y]$ with $y \neq 1/2$, the fact that $D_y(\m,\varepsilon)$ implies that there exists a segment $\Delta$ ending in $y-1$ suitable for the algorithm such that $\Delta \prec \Delta_l$ which is a contradiction.
\end{proof}

We denote by $(\tilde{\m},\tilde{\varepsilon}):=D_{y_0}(\m,\varepsilon)$ and by $\tilde{\Delta}_1, \cdots, \tilde{\Delta}_{\tilde{l}}$ the initial sequence in the algorithm for $(\tilde{\m},\tilde{\varepsilon})$.

\begin{lem}
  \label{lem:derposx1infnottype}
  If $y_1 < y_0$, then 
  \begin{enumerate}
    \item $e(\Delta_l)=1/2$;
    \item $\tilde{l}=l$;
    \item for all $1 \le i \le l$, $\tilde{\Delta}_{i}=\Delta_{i}$.
  \end{enumerate}
\end{lem}

\begin{proof}
  We are in the case where $\Delta_1 = [e_{\max},e_{\max}]$. As $y_0 \neq e_{\max}$, $D_{y_0}$ does not suppress $\Delta_1$ and $\tilde{\Delta}_1 = \Delta_1$. Let $i_0 = e_{\max } - y_1 + 2$. Then, for all $0<i<i_0$, $\Delta_{i}=[e_{\max }-i+1, e_{\max }-i+1]$. In all these segments, the only that can be modified by $D_{y_0}$ is $[y_0,y_0]$. But as $y_1 < y_0$, $[y_0 - 1,y_0- 1] \in \m$, thus at least one $[y_0,y_0]$ is not suppressed by $D_{y_0}$. We see that $\tilde{l} \ge i_0 - 1$ and for all $0<i<i_0$, $\tilde{\Delta}_{i}=\Delta_{i}=[e_{\max }-i+1, e_{\max }-i+1]$.

  If $y_1=1/2$, then $\Delta_l = [1/2,1/2]$. Now, we assume that $y_1 > 1/2$. By definition of $y_1$, we know that $[y_1 - 1,y_1 - 1] \notin \m$, and since $D_{y_1}(\m,\varepsilon)=0$, we get $m_{\m}([1-y_1,y_1-1])=1$. In this case, we also know that for all $y_1 - 1 \ge y > 1/2$, $m_{\m}([-y+1,y-1]) = 1$ and $\varepsilon([-y,y]) \varepsilon([1-y,y-1])=-1$; and $m_{\m}([-1/2,1/2])= 1$ and $\varepsilon([-1/2,1/2])=(-1)^{t_{1/2} + 1}$. But $t_{1/2} = 0$ (as $y_1 > 1/2$), thus $\varepsilon([-1/2,1/2])=-1$. Thus for all $i_0 \le i \le l$, $\Delta_i = [-e_{\max }+i-1,e_{\max }-i+1]^{=0}$; and $\varepsilon([-1/2,1/2])=-1$ (hence $\varepsilon_0 = -1$). None of these segments are changed by $D_{y_0}$ thus $\tilde{l}=l$ and for all $i_0 \le i \le l$, $\tilde{\Delta}_i = \Delta_i$.
\end{proof}

\begin{lem}
  \label{lem:ADderposx1infnottype}
    Suppose $y_1 < y_0$. Then $(\tilde{\m}_1,\tilde{\varepsilon}_1)=(\m_{1},\varepsilon_1)$ and $(\tilde{\m}^{\#},\tilde{\varepsilon}^{\#})=D_{y_0}(\m^{\#},\varepsilon^{\#})$.
\end{lem}

\begin{proof}
  By Lemma \ref{lem:derposx1infnottype}, we get that $(\tilde{\m}_1,\tilde{\varepsilon}_1)=(\m_{1},\varepsilon_1)$. Moreover, we know that for $1 \le i \le y_1$, $\Delta_{i}=[e_{\max }-i+1, e_{\max }-i+1]$. Hence, the segments of $\m^{\#}$ ending in $y_0$ or $y_0 - 1$ are the same as in $\m$ except that one segment $[y_0,y_0]$ and one $[y_0 - 1, y_0 - 1]$ have been suppressed. These segments were not modified by $D_{y_0}$. We deduce that $(\tilde{\m}^{\#},\tilde{\varepsilon}^{\#})=D_{y_0}(\m^{\#},\varepsilon^{\#})$.
\end{proof}

\begin{prop}
  If $y_1 < y_0$, then $\AD(D_{y_0}(\m,\varepsilon))=D_{-y_0}(\AD(\m,\varepsilon))$.
\end{prop}

\begin{proof}
  It follows from Lemmas \ref{lem:ADderposx1infnottype}, \ref{lem:ADcommDer} and \ref{lem:dersum}.
\end{proof}

\begin{lem}
  \label{lem:sequencederposx1eqx0nottype}
  Suppose that $y_1 = y_0$. Let $j = e_{\max } - y_0 + 1$, such that $\Delta_j = [y_0,y_0]$.
  \begin{enumerate}
    \item Then $e(\Delta_l)=y_0$ or $1/2$.
    \item If $e(\Delta_l)=y_0$ and $y_0 \neq 1/2$, then $\tilde{l}=l-1$; otherwise $\tilde{l}=l$.
    \item For all $1 \le i \le \tilde{l}$, if $i \neq j,j+1$ then $\tilde{\Delta}_i=\Delta_i$.
    \item If $y_0 = 1/2$.  If $m_{\m}([-1/2,1/2])$ is odd, $\varepsilon([-1/2,1/2])=-1$ and $t_{1/2} = 1$, then $\tilde{\Delta}_l = [-1/2,1/2]^{=0}$ and $\tilde{\varepsilon}([-1/2,1/2])=-1$; otherwise $\tilde{\Delta}_l = \Delta_l$.
    \item If $y_0 > 1/2$ and $e(\Delta_l)= 1/2$. Then $\Delta_j = [y_0,y_0]$ and $\Delta_{j+1}=[1-y_0,y_0-1]^{=0}$. And $\tilde{\Delta}_{j}=\Delta_{j}$, $\tilde{\Delta}_{j+1}=\Delta_{j+1}$, unless
    \begin{enumerate}
      \item  If $(*)$ is not satisfied and $m_{\m}([-y_0,y_0]) \neq 0$, then $\tilde{\Delta}_{j+1}=[1-y_0,y_0-1]^{\ge 0}$.
      \item If $(*)$ is satisfied and $m_{\m}([-y_0,y_0])$ is odd, then $\tilde{\Delta}_{j}=[-y_0,y_0]^{=0}$. 
      \item If $(*)$ is satisfied and $m_{\m}([-y_0,y_0])$ is even, then $\tilde{\Delta}_{j}=[1-y_0,y_0]$.
    \end{enumerate}
  \end{enumerate}
\end{lem}

\begin{proof}
  Let $i_0 = e_{\max } - y_1 + 2$, then, for all $0<i<i_0$, $\Delta_{i}=[e_{\max }-i+1, e_{\max }-i+1]$. The segment $\Delta_{i_0 - 1} = [y_0,y_0]$ could be modified by $D_{y_0}$. We get that $\tilde{l} \ge i_0 - 2$, and for all $i \le i_0 - 2$, $\tilde{\Delta}_i=\Delta_i$.
  \begin{itemize}
    \item Suppose $y_0 = 1/2$. Then $\Delta_l = \Delta_{i_0} = [1/2,1/2]$ and $\varepsilon_0 = 1$. If $(*)$ is satisfied or $m_{\m}([-1/2,1/2])$ is even or $t_{1/2} \neq 1$, then in $\tilde{\m}$ there is still a $[1/2,1/2]$. Thus $\tilde{l}=l$ and $\tilde{\Delta}_{l}=\Delta_l$. If $(*)$ is not satisfied, $m_{\m}([-1/2,1/2])$ is odd and $t_{1/2} = 1$. The only segment ending in $1/2$ in $\tilde{\m}$ is $[-1/2,1/2]^{=0}$ with $\tilde{\varepsilon}([-1/2,1/2])=-1$. Thus $\tilde{l}=l$ and $\tilde{\Delta}_{l}=[-1/2,1/2]^{=0}$.
    \item Suppose $y_0 > 1/2$ and $e(\Delta_l)=y_0$. Then $m_{\m}([1-y_0,y_0-1])=0$ and in particular $(*)$ is not satisfied. Therefore, $D_{y_0}$ transforms all the segments $[-y_0,y_0]$ into $[1-y_0,y_0-1]$ and suppresses all the $[y_0,y_0]$. Thus $\tilde{l}=l-1$.
    \item Suppose $y_0 > 1/2$ and $e(\Delta_l)=1/2$. As $y_1=y_0 > 1/2$, then $t_{1/2}=0$. Also, as $e(\Delta_l)=1/2$, we get that $m_{\m}([1-y_0,y_0-1])=1$.
    \begin{itemize}
      \item Suppose $(*)$ is not satisfied. Then $D_{y_0}$ does not suppress all the $[y_0,y_0]$ thus $\tilde{\Delta}_{i_0 - 1}=[y_0,y_0]=\Delta_{i_0 - 1}$. If $m_{\m}([-y_0,y_0]) \neq 0$ then $\tilde{\Delta}_{i_0}=[1-y_0,y_0-1]^{\ge 0}$; otherwise $\tilde{\Delta}_{i_0}=[1-y_0,y_0-1]^{=0}=\Delta_{i_0}$. In both cases, $\tilde{l}=l$ and for all $j>i_0$, $\tilde{\Delta}_{j}=\Delta_{j}$.
      \item Suppose $(*)$ is satisfied. This time all the $[y_0,y_0]$ are suppressed by $D_{y_0}$. If $m_{\m}([-y_0,y_0])$ is odd, then $\tilde{\Delta}_{i_0-1}=[-y_0,y_0]^{=0}$, $\tilde{\Delta}_{i_0}=[1-y_0,y_0-1]^{=0}=\Delta_{i_0}$, $\tilde{l}=l$ and for all $j>i_0$, $\tilde{\Delta}_{j}=\Delta_{j}$. If $m_{\m}([-y_0,y_0])$ is even, then $\tilde{\Delta}_{i_0-1}=[1-y_0,y_0]$, $\tilde{\Delta}_{i_0}=[1-y_0,y_0-1]^{=0}=\Delta_{i_0}$, $\tilde{l}=l$ and for all $j>i_0$, $\tilde{\Delta}_{j}=\Delta_{j}$.
    \end{itemize}
  \end{itemize}
\end{proof}

\begin{lem}
  \label{lem:ADderposx1eqx0nottype}
    Suppose that $y_1 = y_0$.
    \begin{enumerate}
        \item If $e(\Delta_{l}) = y_0$ and $y_0 \neq 1/2$, then $(\tilde{\m}_1,\tilde{\varepsilon}_1)=D_{-y_0}(\m_{1},\varepsilon_1)$ and $(\tilde{\m}^{\#},\tilde{\varepsilon}^{\#})=D_{y_0}(\m^{\#},\varepsilon^{\#})$.
        \item Otherwise, $(\tilde{\m}_1,\tilde{\varepsilon}_1)=(\m_{1},\varepsilon_1)$ and $(\tilde{\m}^{\#},\tilde{\varepsilon}^{\#})=D_{y_0}(\m^{\#},\varepsilon^{\#})$. 
    \end{enumerate}
\end{lem}

\begin{proof}
  The result for $(\tilde{\m}_1,\tilde{\varepsilon}_1)$ follows directly form Lemma \ref{lem:sequencederposx1eqx0nottype} (note that in the case $e(\Delta_{l}) = y_0$ and $y_0 = 1/2$ then $\varepsilon_0=-1$ and by Lemma \ref{lem:sequencederposx1eqx0nottype} we also have $\tilde{\varepsilon}_0=-1$). If $e(\Delta_l)=y_0$ and $y_0 \neq 1/2$, then for all $1 \le i \le l$, $\Delta_{i}=[e_{\max }-i+1, e_{\max }-i+1]$, and these segments are just suppressed by $\AD$. It is easy to see that $(\tilde{\m}^{\#},\tilde{\varepsilon}^{\#})=D_{y_0}(\m^{\#},\varepsilon^{\#})$.

  In the other cases, it is easy to prove that $(\tilde{\m}^{\#},\tilde{\varepsilon}^{\#})=D_{y_0}(\m^{\#},\varepsilon^{\#})$, checking all the cases of  Lemma \ref{lem:sequencederposx1eqx0nottype} and looking at the formula of $D_{y_0}$.

  Let us do one to show an example. Suppose that $y_0 = 1/2$, $m_{\m}([-1/2,1/2])$ is odd, $\varepsilon([-1/2,1/2])=-1$ and $t_{1/2} = 1$. Then $\tilde{\Delta}_l = [-1/2,1/2]^{=0}$ and $\tilde{\varepsilon}([-1/2,1/2])=-1$. Note that $\m^{\#}$ is just $\m$ where we have suppressed $[e_{\max }, e_{\max }]+\cdots+[1/2,1/2]+[-1/2,-1/2]+\cdots+[-e_{\max }, -e_{\max }]$. Let us focus on the segments ending in $1/2$. So $\AD$ suppresses one $[1/2,1/2]$. Then in $\m^{\#}$, $t^{\#}_{1/2}=0$ and $(*)$ is not satisfied. So $D_{1/2}$ suppresses all the segments $[-1/2,1/2]$. At the end, the are no segments ending in $1/2$ in $D_{y_0}(\m^{\#},\varepsilon^{\#})$. Now in $\m$, $D_{y_0}$ suppresses all the $[-1/2,1/2]$ except one and one $[1/2,1/2]$. As $\tilde{\Delta}_l = [-1/2,1/2]^{=0}$, $\AD(\tilde{\m},\tilde{\varepsilon})$ suppresses $\tilde{\Delta}_l$. So we see that $(\tilde{\m}^{\#},\tilde{\varepsilon}^{\#})=D_{y_0}(\m^{\#},\varepsilon^{\#})$. The other cases are treated similarly.
\end{proof}

\begin{prop}
  If $y_1 = y_0$, then $\AD(D_{y_0}(\m,\varepsilon))=D_{-y_0}(\AD(\m,\varepsilon))$.
\end{prop}

\begin{proof}
  It follows from Lemmas \ref{lem:ADderposx1eqx0nottype}, \ref{lem:ADcommDer} and \ref{lem:dersum}.
\end{proof}

\begin{lem}
  \label{lem:sequencederposnottype}
  Suppose that $y_1 > y_0$.
  \begin{enumerate}
      \item If $e(\Delta_l)=y_0$ and $\varepsilon_0 = 1$ then $\tilde{l}=l - 1$; otherwise $\tilde{l}=l$.
      \item For $1 \le i \le \tilde{l}$:
      \begin{enumerate}
        \item If $e(\Delta_i)\neq y_0$
        \begin{enumerate}
        \item If $b(\Delta_i)=-y_0$, then $\tilde{\Delta}_{i}={}^{-}\Delta_{i}$
        \item Otherwise, $\tilde{\Delta}_{i}=\Delta_{i}$. 
        \end{enumerate}
        \item If $e(\Delta_i)= y_0$
        \begin{enumerate}
          \item If $\Delta_{i}=[-y_0,y_0]$, $(*)$ is satisfied, $m_{\m}([-y_0,y_0])$ is even and $\Delta_{i-1}=[y_0+1,y_0+1]$, then $\tilde{\Delta}_{i}=[-y_0+1,y_0]$.
          \item If $\Delta_{i}=[-y_0,y_0]$, $(*)$ is satisfied and $m_{\m}([-y_0,y_0])$ is odd, then  $\tilde{\Delta}_{i}=[-y_0,y_0]^{=0}$.
          
          \item Otherwise, $\tilde{\Delta}_{i}=\Delta_{i}$.
        \end{enumerate}
        
      \end{enumerate}
  \end{enumerate}
\end{lem}

\begin{proof}

  Note that since $y_0 < y_1$, then $t_{1/2}=0$. Then the proof is the same as Lemma \ref{lem:sequencederpos}.
\end{proof}

\begin{lem}
  \label{lem:dermdiezeposnottype}
  Suppose that $y_1 > y_0$.
  \begin{enumerate}
    \item If $e(\Delta_{l}) \ge y_0 + 2$, then $A_{y_0}^{\#,c} = A_{y_0}^{c}$.
    \item If $e(\Delta_{l}) = y_0 + 1$, then $A_{y_0}^{\#,c} = A_{y_0}^{c} \cup \{i_l\}$. 
    \item If $e(\Delta_{l}) = y_0$. 
    \begin{enumerate}
      \item If $\Delta_l = [-y_0,y_0]^{\le 0}$ and $(*)$ is not satisfied in $(\m,\varepsilon)$, then we can also assume that $i_l,i'_l \in A_{y_0}^{c}$ and we get that $A_{y_0}^{\#,c} = A_{y_0}^{c} \cup \{i_{l-1}\} \setminus \{i_l,i'_l\}$.
      \item If $\Delta_l = [-1/2,1/2]^{\ge 0}$ or $[-1/2,1/2]^{=0}$, and $(*)$ is satisfied. We can assume that $i_l \notin A_{y_0}^{c}$, and $A_{y_0}^{\#,c} = A_{y_0}^{c}$.
      \item If $\Delta_l = [-1/2,1/2]^{\le 0}$ and $(*)$ is satisfied. We can assume that $i_l \notin A_{y_0}^{c}$, and $A_{y_0}^{\#,c} = A_{y_0}^{c} \setminus \{i'_l\}$. 
      \item Otherwise, we can assume that $i_l \in A_{y_0}^{c}$ and $A_{y_0}^{\#,c} = A_{y_0}^{c} \setminus \{i_l\}$.  
    \end{enumerate}

    \item If $e(\Delta_{l}) = 1/2$ and $y_0 \neq 1/2$. Let $j$ such that $e(\Delta_j)=y_0$. We can assume that $i_j \notin A_{y_0}^{c}$. Then $A_{y_0}^{\#,c} = A_{y_0}^{c}$.
  \end{enumerate}

\end{lem}

\begin{proof}
  The proof is similar to the proof of Lemma \ref{lem:dermdiezepos}. The only main difference is when $e(\Delta_l)=1/2$, $y_0 = 1/2$ and $c(\Delta_l)=0$. We give details only in this case.

  \begin{itemize}
    \item Suppose $\Delta_l = [-1/2,1/2]^{\ge 0}$ or $[-1/2,1/2]^{=0}$, and $(*)$ is satisfied. Then $\varepsilon_0 = -1$. In $\m^{\#}$, we get that $\varepsilon^{\#}([-1/2,1/2])=1$, so $(*)$ is not satisfied. We can assume that $i_l \notin A_{y_0}^{c}$, and $A_{y_0}^{\#,c} = A_{y_0}^{c}$.
    \item Suppose $\Delta_l = [-1/2,1/2]^{\le 0}$ and $(*)$ is satisfied. Now $\varepsilon_0 = 1$. After the algorithm, $\Lambda_{i_l}^{\#}=[-1/2,-1/2]$. Thus, in $\m^{\#}$, $t^{\#}_{1/2}=1$, and $\varepsilon^{\#}([-1/2,1/2])=-1$, so $(*)$ is not satisfied. We assume that $i_l \notin A_{y_0}^{c}$, and $A_{y_0}^{\#,c} = A_{y_0}^{c} \setminus \{i'_l\}$.
    \item Suppose $\Delta_l = [-1/2,1/2]^{\ge 0}$ or $[-1/2,1/2]^{=0}$, and $(*)$ is not satisfied. Then $\varepsilon_0 = 1$. If $\Lambda_{i_{l-1}}^{\#}=-1/2,1/2]$, then $\varepsilon^{\#}([-1/2,1/2])=-1$ and $(*)$ is satisfied in $(\m^{\#},\varepsilon^{\#})$, so $i_{l-1} \notin A_{y_0}^{\#,c}$. Otherwise, $(*)$ is not satisfied in $(\m^{\#},\varepsilon^{\#})$. And we get that $A_{y_0}^{\#,c} = A_{y_0}^{c} \setminus \{i_l\}$. 
    \item Suppose $\Delta_l = [-1/2,1/2]^{\le 0}$ and $(*)$ is not satisfied. Then $\Delta_{l-1}=[-1/2,3/2]$ and $(*)$ is not satisfied in $(\m^{\#},\varepsilon^{\#})$. We get that $A_{y_0}^{\#,c} = A_{y_0}^{c} \cup \{i_{l-1}\} \setminus \{i_l,i'_l\}$.
  \end{itemize}
\end{proof}

\begin{lem}
  \label{lem:ADderposnottype}
  Suppose that $y_1 > y_0$.
  \begin{enumerate}
      \item If $e(\Delta_{l}) \ge y_0 + 2$, then $(\tilde{\m}_1,\tilde{\varepsilon}_1)=(\m_{1},\varepsilon_1)$ and $(\tilde{\m}^{\#},\tilde{\varepsilon}^{\#})=D_{y_0}(\m^{\#},\varepsilon^{\#})$.
      \item If $e(\Delta_{l}) = y_0 + 1$, then $(\tilde{\m}_1,\tilde{\varepsilon}_1)=(\m_{1},\varepsilon_1)$ and $(\tilde{\m}^{\#},\tilde{\varepsilon}^{\#})=D^{\max - 1}_{y_0}(\m^{\#},\varepsilon^{\#})$. 
      \item If $e(\Delta_{l}) = y_0$ and $y_0 \neq 1/2$ or $\varepsilon_0=1$, then $(\tilde{\m}_1,\tilde{\varepsilon}_1)=D_{-y_0}(\m_{1},\varepsilon_1)$ and $(\tilde{\m}^{\#},\tilde{\varepsilon}^{\#})=D_{y_0}(\m^{\#},\varepsilon^{\#})$.
      \item If $e(\Delta_{l}) = y_0$, $y_0 = 1/2$ and $\varepsilon_0=-1$, then $(\tilde{\m}_1,\tilde{\varepsilon}_1)=(\m_{1},\varepsilon_1)$ and $(\tilde{\m}^{\#},\tilde{\varepsilon}^{\#})=D_{y_0}(\m^{\#},\varepsilon^{\#})$.
      \item If $e(\Delta_{l}) = 1/2$ and $y_0 \neq 1/2$, then $(\tilde{\m}_1,\tilde{\varepsilon}_1)=(\m_{1},\varepsilon_1)$ and $(\tilde{\m}^{\#},\tilde{\varepsilon}^{\#})=D_{y_0}(\m^{\#},\varepsilon^{\#})$.
     
  \end{enumerate}
\end{lem}

\begin{proof}
  The proof is similar to the proof of Lemma \ref{lem:ADderpos} and follows from Lemma \ref{lem:dermdiezeposnottype} and Lemma \ref{lem:sequencederposnottype}.
\end{proof}

\begin{prop}
  If $y_1 > y_0$, then $\AD(D_{y_0}(\m,\varepsilon))=D_{-y_0}(\AD(\m,\varepsilon))$.
\end{prop}

\begin{proof}
  It follows from Lemmas \ref{lem:ADderposnottype}, \ref{lem:ADcommDer} and \ref{lem:dersum}.
\end{proof}

\subsection{The derivatives $D_{e_{\max}}$} \label{sec:deremax}

In this section, we assume that $e_{\max} > 1$, that for all $-e_{\max} < y<e_{\max}$ with $y\neq 0$, $(\m,\varepsilon)$ is $y$-reduced, that $(\m,\varepsilon)$ is $L([-1,0])$-reduced, and that $(\m,\varepsilon)$ is not $e_{\max}$-reduced. Let $y_0 = e_{\max}$. We have recalled the formula to compute $D_{y_0}$ in Section \ref{sec:derpos}.

\bigskip

We start by assuming that $\rho$ is of the same type as $G$. Then $\m$ is of the following form 
\[
    \m = \sum_{y=1}^{e_{\max}} n_y ([y,y] + [-y,-y]) + n_0 [0,0] + t_0 ([-1,0] + [0,1]) + \sum_{y=1}^{z} [-y,y] + m [-e_{\max},e_{\max}].
\]

\begin{prop}
  \label{prop:emaxtype}
  We have that $\AD(D_{y_0}(\m,\varepsilon))=D_{-y_0}(\AD(\m,\varepsilon))$.
\end{prop}

\begin{proof}
First, let us consider the case where $n_{e_{\max}-1} \neq 0$. In this case, we also get that $n_{e_{\max}} \neq 0$. Thus, we have $l \ge 2$, $\Delta_1=[y_0,y_0]$, and $\Delta_2 = [y_0 - 1, y_0 - 1]$. Moreover, since $n_{e_{\max}-1} \neq 0$, after applying the derivative $D_{e_{\max}}$, none of the $[e_{\max},e_{\max}]$ are suppressed. Hence, $\tilde{\Delta}_1=[y_0,y_0]=\Delta_1$ and $\tilde{\Delta}_2=[y_0 - 1,y_0-1]=\Delta_2$. We deduce from this that  $\tilde{l}=l$ and for all $1 \le j \le l$, $\tilde{\Delta}_j = \Delta_j$. This gives us that $(\tilde{\m}_1,\tilde{\varepsilon}_1)=(\m_{1},\varepsilon_1)$ and $(\tilde{\m}^{\#},\tilde{\varepsilon}^{\#})=D_{y_0}(\m^{\#},\varepsilon^{\#})$. Also $\m_1 \neq [y_0,y_0] + [-y_0,-y_0]$. Finally we get that $\AD(D_{y_0}(\m,\varepsilon))=D_{-y_0}(\AD(\m,\varepsilon))$.

  Now, let us assume that $n_{e_{\max}-1}=0$. In particular, for all $1 \le i \le e_{\max}-1$, $n_i=0$. Note that if $t_0 \neq 0$, then $y_0 = 2$ and $n_2 \neq 0$. Thus $\m$ is of the form

\[
    \m= n_0 [0,0] + \sum_{y=1}^{z} [-y,y] + m_{y_0}[-y_0,y_0] + n_{y_0}([y_0,y_0] + [-y_0,-y_0]) + t_0([-1,0]+ [0,1]),
\]

    First, let us notice that for any $(\m',\varepsilon') \in \Symm_\rho^\varepsilon(G)$, if $y_0$ is the maximum of the coefficients of $\m'$ then $D_{-y_0}(\m',\varepsilon')$ just removes all the segments $[y_0,y_0]$ and $[-y_0,-y_0]$. In particular $D_{-y_0}(\AD(\m,\varepsilon))=D_{-y_0}(\m_1,\varepsilon_1) + D_{-y_0}(\AD(\m^{\#},\varepsilon^{\#}))=D_{-y_0}(\m_1,\varepsilon_1) + \AD(D_{y_0}(\m^{\#},\varepsilon^{\#}))$. We will compute $\AD(\m,\varepsilon)$ (that is $(\m_1,\varepsilon_1)$ and $(\m^{\#},\varepsilon^{\#})$), $D_{y_0}(\m,\varepsilon)$ and $\AD(D_{y_0}(\m,\varepsilon))$ (that is $(\tilde{\m}_1,\tilde{\varepsilon}_1)$ and $(\tilde{\m}^{\#},\tilde{\varepsilon}^{\#})$) to see that $\AD(D_{y_0}(\m,\varepsilon))=D_{-y_0}(\AD(\m,\varepsilon))$.

    Let us start by computing $\AD(\m,\varepsilon)$. We get that
    \begin{itemize}
        \item Suppose $t_0 \neq 0$ and $n_0$ is odd. Then $\m_1=[-2,2]$, $\varepsilon_1([-2,2])=(- 1) ^ {m_{2} + z}\varepsilon([0,0])$, and $\m^{\#}=(n_0  +1) [0,0] + \sum_{y=1}^{z} [-y,y] + n_{y_0}[-y_0,y_0] + (n_{y_0}-1)([y_0,y_0] + [-y_0,-y_0]) + (t_0 - 1)([-1,0]+ [0,1])$ and $\varepsilon^{\#}([-y_0,y_0])=-\varepsilon([-y_0,y_0])$ and for $y<y_0$, $\varepsilon^{\#}([-y,y])=\varepsilon([-y,y])$. 
        \item Suppose $t_0 \neq 0$, $n_0 \neq 0$ and $n_0$ is even. Then $\m_1=[-y_0,0]+[0,y_0]$, $\m^{\#}=n_0 [0,0] + \sum_{y=1}^{z - 1} [-y,y] + m_{y_0}[-y_0,y_0]  + (n_{y_0} - 1) ([y_0,y_0] + [-y_0,-y_0])+ (t_0 - 1)([-1,0]+ [0,1])$, $\varepsilon^{\#}([0,0])=-\varepsilon([0,0])$ and $\varepsilon^{\#}([-y,y])=\varepsilon([-y,y])$ for $y \neq 0$.
        \item Suppose $t_0 \neq 0$ and $n_0 = 0$. Then $\m_1=[-y_0,0]+[0,y_0]$, $\m^{\#}= m_{y_0}[-y_0,y_0]  + (n_{y_0} - 1) ([y_0,y_0] + [-y_0,-y_0])+ (t_0 - 1)([-1,0]+ [0,1])$ and $\varepsilon^{\#}([-y_0,y_0])=\varepsilon([-y_0,y_0])$. 
        \item Suppose  $t_0 = 0$, $n_{y_0} \neq 0$ and $z < y_0 -1$. Then $\m_1 = [-y_0,-y_0] + [y_0,y_0]$ and $\m^{\#}=n_0 [0,0] + \sum_{y=1}^{z} [-y,y]  + m_{y_0}[-y_0,y_0] + (n_{y_0} - 1) ([y_0,y_0] + [-y_0,-y_0])$.
        \item Suppose $t_0 = 0$, $n_{y_0} \neq 0$, $z = y_0 -1$ and $n_0$ is even. Then $\m_1=[-y_0,0]+[0,y_0]$, $\m^{\#}=(n_0 - 1) [0,0] + \sum_{y=1}^{z - 1} [-y,y] + m_{y_0}[-y_0,y_0]  + (n_{y_0} - 1) ([y_0,y_0] + [-y_0,-y_0])$, $\varepsilon^{\#}([-y_0,y_0])=\varepsilon([-y_0,y_0])$ and $\varepsilon^{\#}([-y,y])=-\varepsilon([-y,y])$ for $0 \le y < z$. 
        \item Suppose $t_0 = 0$, $n_{y_0} \neq 0$, $z = y_0 -1$ and $n_0$ is odd. Then $\m_1=[-y_0,y_0]$, $\varepsilon_1(y_0)=(- 1) ^ {m_{y_0} + z}\varepsilon([0,0])$, $\m^{\#}=n_0 [0,0] + \sum_{y=1}^{z-1} [-y,y] + m_{y_0}[-y_0,y_0] + (n_{y_0} - 1) ([y_0,y_0] + [-y_0,-y_0])$, $\varepsilon^{\#}([-y_0,y_0])=-\varepsilon([-y_0,y_0])$ and $\varepsilon^{\#}([-y,y])=-\varepsilon([-y,y])$ for $0 \le y < z$. 
        \item Suppose $(*)$ is not satisfied and $n_{y_0}=0$, then $\m_1 = [-y_0,-y_0] + [y_0,y_0]$ and $\m^{\#}=n_0 [0,0] + \sum_{y=1}^{z} [-y,y] + [-y_0 + 1,y_0 - 1] + (m_{y_0}- 1)[-y_0,y_0]$. And $\varepsilon^{\#}([1-y_0,y_0-1])=\varepsilon([-y_0,y_0])$, if $n_0 \ge 2$, $\varepsilon^{\#}([-y_0,y_0])=\varepsilon([-y_0,y_0])$, and for $y \le z$, $\varepsilon^{\#}([-y,y])=\varepsilon([-y,y])$.
        \item Suppose $(*)$ is satisfied, $n_{y_0} = 0$ and $n_0$ is odd. Then $\m_1=[-y_0,y_0]$, $\varepsilon_1([-y_0,y_0])=-\varepsilon([-y_0,y_0])$, and $\m^{\#}=n_0 [0,0] + \sum_{y=1}^{z} [-y,y] + (m_{y_0}-1)[-y_0,y_0]$ and $\varepsilon^{\#}([-y_0,y_0])=-\varepsilon([-y_0,y_0])$ and for $y<y_0$, $\varepsilon^{\#}([-y,y])=\varepsilon([-y,y])$. 
        \item Suppose $(*)$ is satisfied, $n_{y_0} = 0$ and $n_0$ is even. Then $\m_1=[-y_0,0]+[0,y_0]$, $\m^{\#}=(n_0 - 1) [0,0] + \sum_{y=1}^{z} [-y,y] + (m_{y_0}-1)[-y_0,y_0]$, $\varepsilon^{\#}([-y_0,y_0])=\varepsilon([-y_0,y_0])$ and $\varepsilon^{\#}([-y,y])=-\varepsilon([-y,y])$ for $0 \le y \le z$. 
    \end{itemize}

    Now, we will compute $D_{y_0}(\m,\varepsilon)$.
    \begin{itemize}
        \item  Suppose $y_0 = 2$, $(*)$ is not satisfied, $n_{y_0}\neq 0$ and $z=y_0 -1$. Then $D_{y_0}(\m,\varepsilon)= n_0 [0,0] + \sum_{y=1}^{z} [-y,y] + m_{y_0}[-y_0+1,y_0 - 1] + \min\{t_0+1,n_2\}([y_0,y_0] + [-y_0,-y_0])$.
        \item  Suppose $y_0 = 2$, $(*)$ is not satisfied and $n_{y_0}=0$ or $z<y_0 -1$. Then $D_{y_0}(\m,\varepsilon)= n_0 [0,0] + \sum_{y=1}^{z} [-y,y] + m_{y_0}[-y_0+1,y_0 - 1] + \min\{t_0,n_2\}([y_0,y_0] + [-y_0,-y_0])$. 
        \item Suppose $y_0 = 2$, $(*)$ is satisfied and $m_{y_0}$ is odd. Then $D_{y_0}(\m,\varepsilon)=n_0 [0,0] + \sum_{y=1}^{z} [-y,y] + (m_{y_0}-1)[-y_0+1,y_0 - 1] + [-y_0,y_0] + \min\{t_0,n_2\}([y_0,y_0] + [-y_0,-y_0])$.
        \item Suppose $y_0 = 2$, $(*)$ is satisfied and $m_{y_0}$ is even. Then $D_{y_0}(\m,\varepsilon)=[-y_0,y_0 - 1]+[-y_0 + 1,y_0] + n_0 [0,0] + \sum_{y=1}^{z} [-y,y] + (m_{y_0}-2)[-y_0+1,y_0 - 1] + \min\{t_0,n_2\}([y_0,y_0] + [-y_0,-y_0])$.
        \item  Suppose $y_0 \neq 2$, $(*)$ is not satisfied, $n_{y_0}\neq 0$ and $z=y_0 -1$. Then $D_{y_0}(\m,\varepsilon)= n_0 [0,0] + \sum_{y=1}^{z} [-y,y] + m_{y_0}[-y_0+1,y_0 - 1] + ([y_0,y_0] + [-y_0,-y_0])$.
        \item  Suppose $y_0 \neq 2$,$(*)$ is not satisfied and $n_{y_0}=0$ or $z<y_0 -1$. Then $D_{y_0}(\m,\varepsilon)= n_0 [0,0] + \sum_{y=1}^{z} [-y,y] + m_{y_0}[-y_0+1,y_0 - 1]$. 
        \item Suppose $y_0 \neq 2$,$(*)$ is satisfied and $m_{y_0}$ is odd. Then $D_{y_0}(\m,\varepsilon)=n_0 [0,0] + \sum_{y=1}^{z} [-y,y] + (m_{y_0}-1)[-y_0+1,y_0 - 1] + [-y_0,y_0]$.
        \item Suppose $y_0 \neq 2$, $(*)$ is satisfied and $m_{y_0}$ is even. Then $D_{y_0}(\m,\varepsilon)=[-y_0,y_0 - 1]+[-y_0 + 1,y_0] + n_0 [0,0] + \sum_{y=1}^{z} [-y,y] + (m_{y_0}-2)[-y_0+1,y_0 - 1]$.
    \end{itemize}

    Examining all the cases we find that, when $\m_1 = [y_0,y_0] + [-y_0,-y_0]$ then $D_{y_0}(\m,\varepsilon) = D_{y_0}(\m^{\#},\varepsilon^{\#})$ and in all the other cases $\tilde{\m}_1 = \m_1$ and $(\tilde{\m}^{\#},\tilde{\varepsilon}^{\#})=D_{y_0}(\m^{\#},\varepsilon^{\#})$. Hence we get that $\AD(D_{y_0}(\m,\varepsilon))=D_{-y_0}(\AD(\m,\varepsilon))$.
\end{proof}

Now, we suppose that $\rho$ is not of the same type as $G$. Then $\m$ is of the following form 

\[
    \m = \sum_{y=1/2}^{e_{\max}} n_y ([y,y] + [-y,-y]) + \sum_{y=1/2}^{z} [-y,y] + m [-e_{\max},e_{\max}].
\]

\begin{prop}
  We have that $\AD(D_{y_0}(\m,\varepsilon))=D_{-y_0}(\AD(\m,\varepsilon))$.
\end{prop}

\begin{proof}
The proof is similar to the proof of Proposition \ref{prop:emaxtype}. When
$n_{e_{\max}-1} \neq 0$, it is exactly the same. Now suppose that $n_{e_{\max}-1}=0$. Thus $\m$ is of the form

\[
    \m= \sum_{y=1/2}^{z} [-y,y] + m_{y_0}[-y_0,y_0] + n_{y_0}([y_0,y_0] + [-y_0,-y_0]),
\]
Let us start by computing $\AD(\m,\varepsilon)$. We get that
    \begin{itemize}
      \item Suppose $n_{y_0} \neq 0$ and $z < y_0 -1$. Then $\m_1 = [-y_0,-y_0] + [y_0,y_0]$ and $\m^{\#}=\sum_{y=1/2}^{z} [-y,y]  + m_{y_0}[-y_0,y_0] + (n_{y_0} - 1) ([y_0,y_0] + [-y_0,-y_0])$.
      \item Suppose $n_{y_0} \neq 0$ and $z = y_0 -1$. Then $\m_1=[-y_0,y_0]$, $\varepsilon_1(y_0)=(-1)^{z+m_{y_0}+1}$, $\m^{\#}= \sum_{y=1/2}^{z-1} [-y,y] + m_{y_0}[-y_0,y_0] + (n_{y_0} - 1) ([y_0,y_0] + [-y_0,-y_0])$, $\varepsilon^{\#}([-y_0,y_0])=-\varepsilon([-y_0,y_0])$ and $\varepsilon^{\#}([-y,y])=-\varepsilon([-y,y])$ for $1/2 \le y < z$.
      \item Suppose $(*)$ is not satisfied and $n_{y_0}=0$. Then $\m_1 = [-y_0,-y_0] + [y_0,y_0]$, $\m^{\#}= \sum_{y=1/2}^{z} [-y,y] + [-y_0 + 1,y_0 - 1] + (m_{y_0}- 1)[-y_0,y_0]$, and $\varepsilon^{\#}([1-y_0,y_0-1])=\varepsilon([-y_0,y_0])$, if $n_0 \ge 2$, $\varepsilon^{\#}([-y_0,y_0])=\varepsilon([-y_0,y_0])$, and for $y \le z$, $\varepsilon^{\#}([-y,y])=\varepsilon([-y,y])$.
      \item Suppose $(*)$ is satisfied and $n_{y_0} = 0$. Then $\m_1=[-y_0,y_0]$, $\varepsilon_1([-y_0,y_0])=(-1)^{z+m_{y_0}+1}$, $\m^{\#}= \sum_{y=1/2}^{z} [-y,y] + (m_{y_0}-1)[-y_0,y_0]$, $\varepsilon^{\#}([-y_0,y_0])=-\varepsilon([-y_0,y_0])$ and for $y<y_0$, $\varepsilon^{\#}([-y,y])=\varepsilon([-y,y])$. 
    \end{itemize}

    Now, we compute $D_{y_0}(\m,\varepsilon)$.
    \begin{itemize}
        \item  Suppose  $(*)$ is not satisfied, $n_{y_0}\neq 0$ and $z=y_0 -1$. Then $D_{y_0}(\m,\varepsilon)=  \sum_{y=1/2}^{z} [-y,y] + m_{y_0}[-y_0+1,y_0 - 1] + ([y_0,y_0] + [-y_0,-y_0])$.
        \item  Suppose $(*)$ is not satisfied and $n_{y_0}=0$ or $z<y_0 -1$. Then $D_{y_0}(\m,\varepsilon)=  \sum_{y=1/2}^{z} [-y,y] + m_{y_0}[-y_0+1,y_0 - 1]$. 
        \item Suppose $(*)$ is satisfied and $m_{y_0}$ is odd. Then $D_{y_0}(\m,\varepsilon)= \sum_{y=1/2}^{z} [-y,y] + (m_{y_0}-1)[-y_0+1,y_0 - 1] + [-y_0,y_0]$.
        \item Suppose $(*)$ is satisfied and $m_{y_0}$ is even. $D_{y_0}(\m,\varepsilon)=[-y_0,y_0 - 1]+[-y_0 + 1,y_0] + \sum_{y=1/2}^{z} [-y,y] + (m_{y_0}-2)[-y_0+1,y_0 - 1]$.
    \end{itemize}

    We finish as in the proof of Proposition \ref{prop:emaxtype}.
\end{proof}

\subsection{The derivatives $D_{-e_{\max}}$ } \label{sec:deremaxm}

In this section, we assume that $e_{\max} > 1$, that for all $-e_{\max} < y \le e_{\max}$ with $y\neq 0$, $(\m,\varepsilon)$ is $y$-reduced, that $(\m,\varepsilon)$ is $L([-1,0])$-reduced, and that $(\m,\varepsilon)$ is not $-e_{\max}$-reduced. Let $y_0 = -e_{\max}$. The derivative $D_{y_0}(\m,\varepsilon)$ just suppresses all the segments $[y_0,y_0]+[-y_0,-y_0]$ from $\m$ and doesn't change $\varepsilon$.

\begin{prop}
    If $m_{\m}([-e_{\max},e_{\max}])=0$, then $D_{-y_0}(\AD(\m,\varepsilon))= \AD(D_{y_0}(\m,\varepsilon))$.
\end{prop}

\begin{proof}    
    Let $(\tilde{\m},\tilde{\varepsilon})=D_{y_0}(\m,\varepsilon)$. With the hypotheses made, $\Delta_1= [-y_0,-y_0]$, $l \ge 2$ (because $D_{-y_0}(\pi) = 0$) and $\Delta_2$ is the biggest segment ending in $-y_0-1$. Hence, $\tilde{\Delta}_1=\Delta_2$, $\tilde{l}=l-1$ and for all $1 \le j < l$, $\tilde{\Delta}_j=\Delta_{j+1}$. Thus $(\tilde{\m}_1,\tilde{\varepsilon}_1)=D_{-y_0}(\m_1,\varepsilon_1)$ and $(\tilde{\m}^{\#},\tilde{\varepsilon}^{\#})=D_{y_0}(\m^{\#},\varepsilon^{\#})$. From Lemma \ref{lem:dersum} and Lemma \ref{lem:ADcommDer}, $D_{-y_0}(\AD(\m,\varepsilon))=D_{-y_0}(\m_1,\varepsilon_1) + D_{-y_0}(\AD(\m^{\#},\varepsilon^{\#}))=(\tilde{\m}_1,\tilde{\varepsilon}_1) + \AD(\tilde{\m}^{\#},\tilde{\varepsilon}^{\#})=\AD(D_{y_0}(\m,\varepsilon))$.
\end{proof}

Now we assume that $m_{\m}([-e_{\max},e_{\max}]) \neq 0$. First, we assume that $\rho$ is of the same type as $G$. From the hypotheses made on the derivative (recalled at the beginning of Section \ref{sec:deremax}), we get that for all $2 \le y \le e_{\max}$, $m_{\m}([-y,y])=1$, $\varepsilon([-y,y])=-\varepsilon([1-y,y-1])$ and $\varepsilon([0,0])\varepsilon([-1,1])=(-1)^{t_0 + 1}$. Let $n_{e_{\max}}=m_{\m}([e_{\max},e_{\max}])$, $n_{2}=m_{\m}([2,2])$, $n_{1}=m_{\m}([1,1])$, $n_{0}=m_{\m}([0,0])$ and $t_0=m_{\m}([0,1])$. Then, $n_{e_{\max}}=m_{\m}([e_{\max},e_{\max}])=\cdots=m_{\m}([2,2])$, $t_0 = n_{e_{\max}} - n_1$ and $n_0 \ge n_1 + 1$.

\begin{lem}
    \label{lem:ADDx}
    Suppose that $m_{\m}([-e_{\max},e_{\max}]) \neq 0$. Then $\AD(\m,\varepsilon)=(\m',\varepsilon')$ with :
    \begin{enumerate}
        \item If $n_0 - n_1$ is odd,
        \[
          \m'= (n_1+1) [-e_{\max},e_{\max}] + (n_{e_{\max}}-n_1) ([-e_{\max},0] + [0,e_{\max}]) + (n_0-n_1) [0,0]+ \sum_{y=1}^{e_{\max} - 1} [-y,y],
        \]
         $\varepsilon'([-e_{\max},e_{\max}])=(-1)^{n_0 + e_{\max} + 1}\varepsilon([0,0])$, $\varepsilon'([0,0])=(-1)^{n_{e_{\max}}}\varepsilon([0,0])$ and  $\varepsilon'([-y,y])=(-1)^{n_1}\varepsilon([-y,y])$ for $y \neq 0,e_{\max}$.
        \item If $n_0 - n_1$ is even,
        \[
          \m'=n_1 [-e_{\max},e_{\max}] + (n_{e_{\max}}-n_1 + 1) ([-e_{\max},0] + [0,e_{\max}]) + (n_0-n_1 - 1) [0,0]+ \sum_{y=1}^{e_{\max} - 1} [-y,y],
        \]
        $\varepsilon'([-e_{\max},e_{\max}])=(-1)^{n_0 + e_{\max} + 1}\varepsilon([0,0])$, $\varepsilon'([0,0])=(-1)^{n_{e_{\max}+1}}\varepsilon([0,0])$ and  $\varepsilon'([-y,y])=(-1)^{n_1+1}\varepsilon([-y,y])$ for $y \neq 0,e_{\max}$.
    \end{enumerate}
\end{lem}

\begin{proof}
First, we apply the algorithm $n_1$ times to $(\m,\varepsilon)$. Each time we get $l=e_{\max}+1$, $\Delta_1 = [e_{\max},e_{\max}]$, $\Delta_2 = [e_{\max}-1,e_{\max}-1]$, ...,$\Delta_{l-1} = [1,1]$, $\Delta_{l} = [0,0]^{\ge 0}$ or $[0,0]^{=0}$. Thus we get that $\m_1=[-e_{\max},e_{\max}]$, with $\varepsilon_1(\m_1)=(-1)^{n_0 + e_{\max} + 1}\varepsilon([0,0])$, and $\m^{\#}$ is $\m$ where we have removed the segments $[y,y]$ and $[-y,-y]$ for $1 \le y \le e_{\max}$ and $[0,0]$, and $\varepsilon^{\#}([-y,y])=-\varepsilon([-y,y])$. Hence after applying $n_1$ times the algorithm we get that 
\[
    \AD(\m,\varepsilon)=n_1 [-e_{\max},e_{\max}] + \AD (\m',\varepsilon')
\]
with $\varepsilon_1([-e_{\max},e_{\max}])=(-1)^{n_0 + e_{\max} + 1}\varepsilon([0,0])$, $\m'=n \sum_{y=2}^{e_{\max}} ([y,y] + [-y,-y]) + (n_0-n_1) [0,0] + n ([-1,0] + [0,1]) + \sum_{y=1}^{e_{\max}} [-y,y]$, with $n=n_{e_{\max}}-n_1$, and $\varepsilon'([-y,y])=(-1)^{n_1}\varepsilon([-y,y])$.

Now apply again the algorithm to $(\m',\varepsilon')$. We get $l=e_{\max}+1$, $\Delta_1 = [e_{\max},e_{\max}]$, $\Delta_2 = [e_{\max}-1,e_{\max}-1]$, ...,$\Delta_{l-2} = [2,2]$, $\Delta_{l-1} = [0,1]$ and $\Delta_l=[0,0]$. If $n_0 - n_1$ is odd then $\m_1=[-e_{\max},e_{\max}]$ and $\m^{\#}$ is $\m'$ where we have removed the segments $[y,y]$ and $[-y,-y]$ for $2 \le y \le e_{\max}$, we have removed $[-1,0] + [0,1]$ and we have added a $[0,0]$ and $\varepsilon^{\#}([0,0])=\varepsilon'([0,0])$ and for $y \neq 0$, $\varepsilon^{\#}([-y,y])=-\varepsilon'([-y,y])$. If $n_0 - n_1$ is even then $\m_1=[-e_{\max},0] + [0,e_{\max}]$ and $\m^{\#}$ is $\m'$ where we have removed the segments $[y,y]$ and $[-y,-y]$ for $2 \le y \le e_{\max}$ and $[-1,0] + [0,1]$; if $y\neq 0$, $\varepsilon^{\#}([-y,y])=\varepsilon'([-y,y])$ and $\varepsilon^{\#}([0,0])=-\varepsilon'([0,0])$. Hence after applying $n$ times the algorithm to $(\m',\varepsilon')$ we get that:

\begin{itemize}
  \item If $n_0 - n_1$ is even:
  \[
    \AD(\m,\varepsilon) = n_1 [-e_{\max},e_{\max}] + (n_{y_0}-n_1) ([-e_{\max},0] + [0,e_{\max}]) + \AD (\m^{\prime\prime},\varepsilon'')
  \]
  with $\m^{\prime\prime}= n' [0,0]+ \sum_{y=1}^{e_{\max}} [-y,y]$, where $n'=n_0-n_1$, $\varepsilon''([0,0])=(-1)^{n_{e_{\max}}}\varepsilon([0,0])$ and for $y \neq 0$, $\varepsilon'([-y,y])=(-1)^{n_1}\varepsilon([-y,y])$.
  \item If $n_0 - n_1$ is odd:
  \[
    \AD(\m,\varepsilon) = (n_1+1) [-e_{\max},e_{\max}] + (n_{y_0}-n_1-1) ([-e_{\max},0] + [0,e_{\max}]) + \AD (\m^{\prime\prime},\varepsilon'')
  \]
  with $\m^{\prime\prime}= n' [0,0]+ \sum_{y=1}^{e_{\max}} [-y,y]$, where $n'=n_0-n_1 + 1$, $\varepsilon''([0,0])=(-1)^{n_{e_{\max}}+1}\varepsilon([0,0])$ and for $y \neq 0$, $\varepsilon''([-y,y])=(-1)^{n_1+1}\varepsilon([-y,y])$.
\end{itemize}

Now, $\varepsilon''([0,0])\varepsilon''([-1,1])=(-1)^{n_{e_{\max}}+n_1}\varepsilon([0,0])\varepsilon([-1,1])=(-1)^{n_{e_{\max}}+n_1+t_0+1}=-1$. We have computed $\AD(\m^{\prime\prime},\varepsilon'')$ in Lemma \ref{lem:ADmetemp}, as $n'$ is even we have $\AD(\m^{\prime\prime},\varepsilon'') = (n' - 1) [0,0]+ \sum_{y=1}^{e_{\max} - 1} [-y,y] + [-e_{\max},0] + [0,e_{\max}]$ and all the signs of the segments $[-y,y]$ for $y \neq e_{\max}$ change, which gives the result.
\end{proof}

We can prove the wanted result.

\begin{prop}
  If $m_{\m}([-e_{\max},e_{\max}])\neq 0$, then $D_{-y_0}(\AD(\m,\varepsilon))= \AD(D_{y_0}(\m,\varepsilon))$.
\end{prop}

\begin{proof}
  Lemma \ref{lem:ADDx} computes explicitly $\AD(\m,\varepsilon)$. Using the formulas for the derivative recalled in Section \ref{sec:expder}, we get $D_{-y_0}(\AD(\m,\varepsilon))$. Thus let us compute $\AD(D_{y_0}(\m,\varepsilon))$ and check that we get the same formula.

Now, $D_{y_0}(\m,\varepsilon)$ is just $(\m,\varepsilon)$ where we have removed all the segments $[y_0,y_0]$. We apply the algorithm to $D_{y_0}(\m,\varepsilon)$. 
\begin{itemize}
    \item If $t_0$ is even (hence $\varepsilon(0)\varepsilon(1)=-1$), then $\Delta_1=[-e_{\max},e_{\max}]^{=0}$, ..., $\Delta_{e_{\max}}=[-1,1]^{=0}$ and $\Delta_{e_{\max}+1}=[0,0]^{=0}$ or $[0,0]^{\le 0}$ depending on the parity of $n_0$.
    \begin{itemize}
        \item If $n_0$ is odd, then $\Delta_{e_{\max}+1}=[0,0]^{=0}$ and $l=e_{\max}+1$. Thus $\m_1 = [-e_{\max},e_{\max}]$ and $\m^{\#}=n_{e_{\max}} \sum_{y=2}^{e_{\max} - 1}  ([y,y] + [-y,-y]) + n_1 ([1,1] + [-1,-1])+ n_0 [0,0] + t_0 ([-1,0] + [0,1]) + \sum_{y=1}^{e_{\max}-1} [-y,y]$.
        \item If $n_0$ is even and $n_1 =0$, then $l=e_{\max}+1$ and $\Delta_l= [0,0]^{\le 0}$. Thus $\m_1 = [-e_{\max},0] + [0,e_{\max}]$ and $\m^{\#}=n_{e_{\max}} \sum_{y=2}^{e_{\max} - 1}  ([y,y] + [-y,-y]) + n_1  ([1,1] + [-1,-1]) + (n_0 - 1) [0,0] + t_0 ([-1,0] + [0,1]) + \sum_{y=1}^{e_{\max}-1} [-y,y]$.
        \item If $n_0$ is even and $n_1 \neq 0$, then $l=2e_{\max}$, $\Delta_{e_{\max}+2}=[-1,-1]$,...,$\Delta_{l}=[-e_{\max}+1,-e_{\max}+1]$. Thus $\m_1 = [-e_{\max},e_{\max}-1] + [-e_{\max}+1,e_{\max}]$ and $\m^{\#}=(n_{e_{\max}}-1) \sum_{y=2}^{e_{\max} - 1}  ([y,y] + [-y,-y]) + (n_1-1)  ([1,1] + [-1,-1]) + (n_0 - 1) [0,0] + t_0 ([-1,0] + [0,1]) + \sum_{y=1}^{e_{\max}-1} [-y,y]$.
    \end{itemize}
    \item If $t_0$ is odd (hence $\varepsilon(0)\varepsilon(1)=1$), then $l=e_{\max}+1$, $\Delta_1=[-e_{\max},e_{\max}]^{=0}$, ..., $\Delta_{l-1}=[-1,1]^{=0}$ and $\Delta_l=[-1,0]$. We get that $\m_1 = [-e_{\max},0] + [0,e_{\max}]$ and $\m^{\#}=n_{e_{\max}} \sum_{y=2}^{e_{\max} - 1}  ([y,y] + [-y,-y]) + (n_1 + 1) ([1,1] + [-1,-1]) + (n_0 + 1) [0,0] + (t_0 - 1) ([-1,0] + [0,1]) + \sum_{y=1}^{e_{\max}-1} [-y,y]$.
\end{itemize}
In all those cases, $\AD(\m^{\#},\varepsilon^{\#})$ has been computed in Lemma \ref{lem:ADDx}. We get (with the signs as in Lemma \ref{lem:ADDx}):
\begin{itemize}
  \item Suppose that $t_0$ is even and $n_0$ is odd.
  \begin{itemize}
      \item If $n_0 - n_1$ is odd, then $\AD(D_{y_0}(\m,\varepsilon))=[-e_{\max},e_{\max}]+(n_1+1) [-e_{\max}+1,e_{\max}-1] + (n_{e_{\max}}-n_1) ([-e_{\max}+1,0] + [0,e_{\max}-1]) + (n_0-n_1) [0,0]+ \sum_{y=1}^{e_{\max} - 2} [-y,y]$.
      \item If $n_0 - n_1$ is even, then $\AD(D_{y_0}(\m,\varepsilon))=[-e_{\max},e_{\max}]+ n_1 [-e_{\max}+1,e_{\max}-1] + (n_{e_{\max}}-n_1 + 1) ([-e_{\max}+1,0] + [0,e_{\max}-1]) + (n_0-n_1 - 1) [0,0]+ \sum_{y=1}^{e_{\max} - 2} [-y,y]$.
  \end{itemize}
  \item Suppose that $t_0$ is even, $n_0$ is even and $n_1 = 0$. Then $n_0 - n_1$ is even. Thus $\AD(D_{y_0}(\m,\varepsilon))=[-e_{\max},0] + [0,e_{\max}]+ [-e_{\max}+1,e_{\max}-1] + n_{e_{\max}} ([-e_{\max}+1,0] + [0,e_{\max}-1]) + (n_0-1) [0,0]+ \sum_{y=1}^{e_{\max} - 2} [-y,y]$.
  \item Suppose that $t_0$ is even, $n_0$ is even and $n_1 \neq 0$.
  \begin{itemize}
      \item If $n_0 - n_1$ is odd, then $\AD(D_{y_0}(\m,\varepsilon))=[-e_{\max},e_{\max}-1] + [-e_{\max}+1,e_{\max}]+n_1 [-e_{\max}+1,e_{\max}-1] + (n_{e_{\max}}-n_1) ([-e_{\max}+1,0] + [0,e_{\max}-1]) + (n_0-n_1) [0,0]+ \sum_{y=1}^{e_{\max} - 2} [-y,y]$.
      \item If $n_0 - n_1$ is even, then $\AD(D_{y_0}(\m,\varepsilon))=[-e_{\max},e_{\max}-1] + [-e_{\max}+1,e_{\max}]+(n_1-1) [-e_{\max}+1,e_{\max}-1] + (n_{e_{\max}}-n_1 + 1) ([-e_{\max}+1,0] + [0,e_{\max}-1]) + (n_0-n_1 - 1) [0,0]+ \sum_{y=1}^{e_{\max} - 2} [-y,y]$.
\end{itemize}
  \item Suppose that $t_0$ is odd.
  \begin{itemize}
      \item If $n_0 - n_1$ is odd, then $\AD(D_{y_0}(\m,\varepsilon))=[-e_{\max},0] + [0,e_{\max}]+(n_1+2) [-e_{\max}+1,e_{\max}-1] + (n_{e_{\max}}-n_1-1) ([-e_{\max}+1,0] + [0,e_{\max}-1]) + (n_0-n_1) [0,0]+ \sum_{y=1}^{e_{\max} - 2} [-y,y]$.
      \item If $n_0 - n_1$ is even, then $\AD(D_{y_0}(\m,\varepsilon))=[-e_{\max},0] + [0,e_{\max}]+ (n_1+1) [-e_{\max}+1,e_{\max}-1] + (n_{e_{\max}}-n_1) ([-e_{\max}+1,0] + [0,e_{\max}-1]) + (n_0-n_1 - 1) [0,0]+ \sum_{y=1}^{e_{\max} - 2} [-y,y]$.
\end{itemize}
\end{itemize}

Looking at the formulas of the derivative recalled in Section \ref{sec:expder},
we see that $D_{-y_0}(\AD(\m,\varepsilon))= \AD(D_{y_0}(\m,\varepsilon))$.
\end{proof}

We now treat the case where $\rho$ is not of the same type as $G$. We still assume that $m_{\m}([-e_{\max},e_{\max}]) \neq 0$. Let $n_{e_{\max}}=m_{\m}([e_{\max},e_{\max}])$. We get that for all $1/2 \le y \le e_{\max}$, $m_{\m}([y,y])=n_{e_{\max}}$ and $m_{\m}([-y,y])=1$. Also, $\varepsilon([-1/2,1/2])=(-1)^{n_{e_{\max}} + 1}$, and for $y > 1/2$, $\varepsilon([-y,y])\varepsilon([1-y,y-1])=-1$.

\begin{lem}
  \label{lem:ADDxnottype}
  Suppose that $m_{\m}([-e_{\max},e_{\max}]) \neq 0$. Then $\AD(\m,\varepsilon)=(\m',\varepsilon')$ with :
  \[
    \m'=(n_{e_{\max}}+1) [-e_{\max},e_{\max}] + \sum_{y=1}^{e_{\max}-1/2} [-y+1/2,y-1/2]
  \]
  $\varepsilon'([-e_{\max},e_{\max}])=(-1)^{e_{\max}+1/2 }$ and $\varepsilon'([-y,y])=(-1)^{n_{e_{\max}}}\varepsilon([-y,y])$, for $y<e_{\max}$.
\end{lem}

\begin{proof}
First, we apply the algorithm $n_{e_{\max}}$ times to $(\m,\varepsilon)$. Each time we get $l=e_{\max}+1/2$, $\Delta_1 = [e_{\max},e_{\max}]$, $\Delta_2 = [e_{\max}-1,e_{\max}-1]$, ..., $\Delta_{l} = [1/2,1/2]$. Thus we get that $\m_1=[-e_{\max},e_{\max}]$, with $\varepsilon_1(\m_1)=(-1)^{e_{\max} + 1/2 }$, and $\m^{\#}$ is $\m$ where we have removed the segments $[y,y]$ and $[-y,-y]$ for $1/2 \le y \le e_{\max}$ and $\varepsilon^{\#}([-y,y])=-\varepsilon([-y,y])$. Hence after applying $n_{e_{\max}}$ times the algorithm we get that 
\[
  \AD(\m,\varepsilon)=n_{e_{\max}} [-e_{\max},e_{\max}] + \AD (\m',\varepsilon')
\]
with $\varepsilon_1([-e_{\max},e_{\max}])=(-1)^{e_{\max}+1/2 }$, $\m'= \sum_{y=1/2}^{e_{\max}} [-y,y]$ and $\varepsilon'([-y,y])=(-1)^{n_{e_{\max}}}\varepsilon([-y,y])$.

By Lemma \ref{lem:ADmetempnottype}, $\AD (\m',\varepsilon')=(\m',\varepsilon')$ and we get the result.
\end{proof}

\begin{prop}
  If $m_{\m}([-e_{\max},e_{\max}])\neq 0$, then $D_{-y_0}(\AD(\m,\varepsilon))= \AD(D_{y_0}(\m,\varepsilon))$.
\end{prop}

\begin{proof}
The derivative $D_{y_0}(\m,\varepsilon)$ is just $(\m,\varepsilon)$ where we have removed all the segments $[y_0,y_0]$. We apply the algorithm to $D_{y_0}(\m,\varepsilon)$. We get $\Delta_1=[-e_{\max},e_{\max}]^{=0}$, ..., $\Delta_{e_{\max}}=[-1,1]^{=0}$ and $\Delta_{l}=[0,0]^{=0}$. The sign of $\Delta_l$, which depends on the parity of $n_{e_{\max}}$, determines if $\varepsilon_0$ is $1$ or $-1$. 

\begin{itemize}
  \item If $n_{e_{\max}}$ is even, then $\varepsilon_0=-1$, $\m_1 = [-e_{\max},e_{\max}]$ and $\m^{\#}=n_{e_{\max}} \sum_{y=0}^{e_{\max} - 3/2}  ([y+1/2,y+1/2] + [-y-1/2,-y-1/2]) + \sum_{y=1}^{e_{\max}-3/2} [-y+1/2,y-1/2]$.
  \item If $n_{e_{\max}}$ is odd, then $\varepsilon_0=1$, $\m_1 = [-e_{\max},e_{\max}-1] + [-e_{\max}+1,e_{\max}]$ and $\m^{\#}=(n_{e_{\max}}-1) \sum_{y=0}^{e_{\max} - 3/2}  ([y+1/2,y+1/2] + [-y-1/2,-y-1/2]) + \sum_{y=1}^{e_{\max}-3/2} [-y+1/2,y-1/2]$.
\end{itemize}

We can then compute $\AD(\m^{\#},\varepsilon^{\#})$ with Lemma \ref{lem:ADDxnottype}. We get (with the signs as in Lemma \ref{lem:ADDxnottype}), 
\begin{itemize}
  \item If $n_{e_{\max}}$ is even, then $\AD(D_{y_0}(\m,\varepsilon))=[-e_{\max},e_{\max}] + (n_{e_{\max}}+1) [-e_{\max}+1,e_{\max}-1] + \sum_{y=1}^{e_{\max}-3/2} [-y+1/2,y-1/2]$.
  \item If $n_{e_{\max}}$ is odd, then $\AD(D_{y_0}(\m,\varepsilon))=[-e_{\max},e_{\max}-1] + [-e_{\max}+1,e_{\max}] + n_{e_{\max}} [-e_{\max}+1,e_{\max}-1] + \sum_{y=1}^{e_{\max}-3/2} [-y+1/2,y-1/2]$.
\end{itemize}

Looking at the formulas of the derivative recalled in Section \ref{sec:expder},
we see that $D_{-y_0}(\AD(\m,\varepsilon))= \AD(D_{y_0}(\m,\varepsilon))$.
\end{proof}

\section{Proof in the bad parity case}
\label{sec:proofbad}

Let $\rho \in \Cusp^\GL$ be of bad parity, and denote by $\rho_u$ its unitarization. The goal of this section is to prove Theorem~\ref{thm:MAIN} in the case of an irreducible representation of $\rho$-bad parity. To do so, we will work with symmetrical Langlands data throughout this section. More precisely, we will establish the following equivalent formulation of the theorem.

\begin{thm}
  \label{thm:ADAubertDualbad}
  Let $\pi \in \Irr^{G}$ be $\rho$-bad with symmetrical Langlands data $(\m,\varepsilon)$. Then we have
  \[
    \hat{\pi} \simeq L(\AD(\m,\varepsilon)).
  \]
\end{thm}

\subsection{The strategy of the proof}

An element of $\Symm_\rho^\varepsilon(G)$ has all its signs trivial since $\rho$ is bad, hence we will just write $\m \in \Symm_\rho^\varepsilon(G)$. Unlike in Section \ref{sec:proofgood}, here we do not have any issues with the signs or the determinant. Therefore, we can directly prove by induction on $N \in \NN$ the following theorem.

\begin{thm}
  \label{thm:ADdualinductionbad}
  Let $N \in \NN$. Let $\m \in \Symm^{\varepsilon}_\rho(G)$ and $\pi=L(\m) \in \Irr^{G}$. If $l(\m) \le N$; then $\hat{\pi} \simeq L(\AD(\m))$.
\end{thm}

We prove Theorem \ref{thm:ADdualinductionbad} by induction of $N$. The case $N=0$ is trivial. Let $N \in \NN^{*}$. 

\begin{hypo}
  \label{hyp:recurrencebad}
We assume that Theorem \ref{thm:ADdualinductionbad} is true for all $N' < N$
\end{hypo}
Until the end of the section, we will assume that Hypothesis \ref{hyp:recurrencebad} is true. We want to prove now that Theorem \ref{thm:ADdualinductionbad} is true for $N$. To do that, we will prove that the algorithm $\AD$ commutes with the derivatives.

\begin{lem}
  \label{lem:commDerthmbad}
  We assume that for all non-reduced $\m \in \Symm_{\rho}^{\varepsilon}(G)$ with $l(\m) = N$, we have  either 
  \begin{enumerate}
    \item there exists $x \neq 0$, such that $\m$ is not $\rho_u|\cdot|^{x}$-reduced and $\AD(D_{\rho_u|\cdot|^{x}}(\m))=D_{\rho_u|\cdot|^{-x}}(\AD(\m))$
    \item or, if it is defined, $\m$ is not $L([-1,0]_{\rho_u})$-reduced and $\AD(D_{L([-1,0]_{\rho_u})}(\m))=D_{Z([0,1]_{\rho_u})}(\AD(\m))$.
  \end{enumerate}
  Then Theorem \ref{thm:ADdualinductionbad} is true for $N$.
\end{lem}

\begin{proof}
  Let $\pi$ and $\m$ be as in Theorem \ref{thm:ADdualinductionbad}. If $\m$ is reduced, then $\m=n [0,0]_{\rho_u}$ for some even integer $n$. Then $\AD(\m)=\m$ and $\hat{\pi}=\pi$, which proves the theorem. Hence, we can assume that $\m$ is not reduced. Let us suppose that there exists $x \neq 0$ such that $\m$ is not $\rho_u|\cdot|^{x}$-reduced (the case $L([-1,0]_{\rho_u})$-reduced is treated similarly). By hypothesis, $\AD(D_{\rho_u|\cdot|^{x}}(\m))=D_{\rho_u|\cdot|^{-x}}(\AD(\m))$. Then we get
  \begin{align*}
    D_{\rho|\cdot|^{-x}}(L(\AD(\m))) & = L(D_{\rho|\cdot|^{-x}}(\AD(\m))) \text{ by Lemma \ref{prop:derpibad}}\\
 & = L(\AD(D_{\rho|\cdot|^{x}}(\m))) \\
 & = L(D_{\rho|\cdot|^{x}}(\m))^{\widehat{}} \text{ by Hypothesis \ref{hyp:recurrencebad}} \\
 & = D_{\rho|\cdot|^{x}}(L(\m))^{\widehat{}} \text{ by Lemma \ref{prop:derpibad}} \\
 & = D_{\rho|\cdot|^{-x}} (\hat{\pi}) \text{ by \cite[Prop. 3.9.]{AM}} \\
  \end{align*}
  
By the injectivity of $D_{\rho_u|\cdot|^{-x}}$, we get that $\hat{\pi}=L(\AD(\m))$.
\end{proof}

As in Section \ref{sec:proofgood}, the rest of this section is devoted to prove that the conditions $(1)$ or $(2)$ of the above Lemma are satisfied. We will need the following lemmas.

\begin{lem}
  \label{lem:ADcommDerbad}
  For all $\m \in \Symm_{\rho}^{\varepsilon}(G)$ such that $l(\m) < N$ and for all $x \neq 0$,
  \[
    \AD(D_{\rho_u|\cdot|^{x}}(\m))=D_{\rho_u|\cdot|^{-x}}(\AD(\m))
  \]
  and, if it is well-defined,
  \[
    \AD(D_{L([-1,0]_{\rho_u})}(\m))=D_{Z([0,1]_{\rho_u})}(\AD(\m)).
  \]
\end{lem}

\begin{proof}
  It follows from Proposition \ref{prop:derivativedual}.
\end{proof}

\begin{lem}
  \label{lem:ADcommSocbad}
  For all $\m \in \Symm_{\rho}^{\varepsilon}(G)$ such that $l(\m) < N - 2$ and for all $x \neq 0$,
  \[
    \AD(S_{\rho_u|\cdot|^{x}}^{(1)}(\m))=S_{\rho_u|\cdot|^{-x}}^{(1)}(\AD(\m))
  \]
  and, if it is well-defined,
  \[
    \AD(S_{L([-1,0]_{\rho_u})}^{(1)}(\m))=S_{Z([0,1]_{\rho_u})}^{(1)}(\AD(\m)).
  \]
\end{lem}

\begin{proof}
  The proof follows from \cite[Thm. 31 (4)]{Bern}.
\end{proof}

The remainder of this section follows the structure of the proof in Section \ref{sec:proofgood}. The algorithm is simpler in the bad parity case, as already observed in this subsection. Since all the proofs are very similar, we will only provide a sketch. To simplify the notations, until the end of Section \ref{sec:proofbad}, we will write all the segments with respect to $\rho_u$ and we will omit $\rho$ and $\rho_u$ in the notations. That is $\AD := \AD_\rho$, $[x,y]:=[x,y]_{\rho_u}$, $D_x := D_{\rho_u|\cdot|^{x}}$, $D_{Z([0,1])}:=D_{Z([0,1]_{\rho_u})}$ and $D_{L([-1,0])}:=D_{L([-1,0]_{\rho_u})}$. We will also say that $\m$ is $x$-reduced if it is $\rho_u|\cdot|^{x}$-reduced, and similarly for $L([-1,0])$-reduced and $Z([-1,0])$-reduced.

\subsection{The case $e_{\max} \le 1$}

In this section we assume that $e_{\max} \le 1$. The goal is to compute explicitly $\AD(\m)$.

\bigskip

We start with the easiest case which is when $\rho$ is of the same type as $G$. Then $\m$ is of the form
\[
  \m = c [-1/2,1/2] + n ([-1/2,-1/2] + [1/2,1/2])
\]
with $c,n \in \NN$ and $c$ even.

Let $\m'=\AD(\m)$. The maximum of the coefficients of $\m'$ is also smaller than $1$, so $\m'$ is of the form as above. We denote by $c',n' \in \NN$ the constants relative to $\m'$. A direct computation shows that:
\begin{prop}
  The dual $\AD(\m)=\m'$ is given by the following formula.
  \begin{enumerate}
      \item If $n$ is even, then $c'=n$ and $n'=c$.
      \item If $n$ is odd, then $c'=n-1$ and $n'=c+1$.
 \end{enumerate}
\end{prop}

Looking at the formulas for the derivatives we see that

  \begin{prop}
    \begin{enumerate}
      \item If $n \neq 0$, then $\AD(D_{-1/2}(\m))=D_{1/2}(\AD(\m))$.
      \item If $n = 0$, then $\AD(D_{1/2}(\m))=D_{-1/2}(\AD(\m))$.
    \end{enumerate}   
  \end{prop}

Now, let us assume that $\rho$ is of the opposite type as $G$. Then $\m$ has the following form
\[
  \m = c_0 [0,0] + c_1 [-1,1] + t ([-1,0] + [0,1]) + n ([-1,-1] + [1,1])
\]
with $c_0,c_1,t,n \in \NN$, and $c_0,c_1$ even.

Let $\m'=\AD(\m)$. The maximum of the coefficients of $\m'$ is also smaller than $1$, so $\m'$ is of the form as above. We denote by $c'_0,c'_1,t',n' \in \NN$ the constants relative to $\m'$. A direct computation shows that:

\begin{prop}
  The dual $\AD(\m)=\m'$ is given by the following formula.
  \begin{enumerate}
      \item If $n > c_0$; then $c'_0=c_1$, $c'_1=c_0$, $t'=t$ and $n'=n-c_0+c_1$.
      \item If $n \le c_0$, $n$ is even and $t$ is even; then $c'_0=c_0 - n +c_1$, $c'_1=n$, $t'=t$ and $n'=c_1$.
      \item If $n \le c_0$, $n$ is even and $t$ is odd; then $c'_0=c_0 - n + c_1 + 2$, $c'_1=n$, $t'=t-1$ and $n'=c_1 + 1$.
      \item If $n \le c_0$, $n$ is odd and $t$ is even; then $c'_0=c_0 - n -1 + c_1 $, $c'_1=n-1$, $t'=t+1$ and $n'=c_1$.
      \item If $n \le c_0$, $n$ is odd and $t$ is odd; then $c'_0=c_0 - n + c_1 + 1$, $c'_1=n-1$, $t'=t$ and $n'=c_1+1$.
 \end{enumerate}
\end{prop}

We also check the commutativity with the derivative by an explicit computation.
\begin{prop}
  \begin{enumerate}
    \item If $n \neq 0$, then $\AD(D_{-1}(\m))=D_{1}(\AD(\m))$.
    \item If $n = 0$ and $t \neq 0$, then $\AD(D_{L([-1,0])}(\m))=D_{Z([0,1])}(\AD(\m))$.
    \item If $n,t = 0$, then $\AD(D_{1}(\m))=D_{-1}(\AD(\m))$.
  \end{enumerate}
\end{prop}

\subsection{The negative derivative} \label{sec:dernegbp}

In this section, we assume that $e_{\max} > 1$ and there exists $y<0$, $y \neq -e_{\max}$ such that $\m$ is not $y$-reduced. 

\bigskip

We define $y_0 \in (1/2)\Z$ to be the smallest $y \in (1/2)\Z$ such that $y \neq -e_{\max}$ and $\m$ is not $y$-reduced. With our hypotheses on $\m$ necessarily $y_0 < 0$. 

\begin{prop}
  \label{prop:ADcommdernegbp}
   We have $\AD(D_{y_0}(\m))=D_{-y_0}(\AD(\m))$.
\end{prop}

\begin{proof}
  The proof is similar to what is done in Section \ref{sec:derneg} with some slight modifications.
  First, the formula for the negative derivative in the bad parity case is identical to the formula in the good parity case. So we have a similar description of $y_0$. 
  
  Let $\Delta_1, \cdots, \Delta_l$ be the initial sequence in the algorithm for $\m$.  This time we get that if $e(\Delta_{l}) < y_0$, then $e(\Delta_{l}) = -e_{\max}$. We denote by $\tilde{\m}:=D_{y_0}(\m)$ and by $\tilde{\Delta}_1, \cdots, \tilde{\Delta}_{\tilde{l}}$ the initial sequence in the algorithm for $\tilde{\m}$.
  \begin{enumerate}
    \item If $e(\Delta_{l})=y_0$, then $\tilde{l}=l-1$ and if not $\tilde{l}=l$.
    \item For $1 \le i \le \tilde{l}$, if $b(\Delta_i)=-y_0$ and $\Delta_i \neq [-y_0,-y_0]$ then $\tilde{\Delta}_i={}^{-}\Delta_i$, otherwise $\tilde{\Delta}_i=\Delta_i$.
\end{enumerate}
This gives us that
\begin{enumerate}
  \item If $e(\Delta_{l}) \ge y_0 + 2$, then $\tilde{\m}_1=\m_{1}$ and $\tilde{\m}^{\#}=D_{y_0}(\m^{\#})$.
  \item If $e(\Delta_{l}) = y_0 + 1$, then $\tilde{\m}_1=\m_{1}$ and $\tilde{\m}^{\#}=D^{\max - 1}_{y_0}(\m^{\#})$.  
  \item If $e(\Delta_{l}) = y_0$, then $\tilde{\m}_1=D_{-y_0}(\m_{1})$ and $\tilde{\m}^{\#}=D_{y_0}(\m^{\#})$.
  \item If $e(\Delta_{l}) < y_0$, then $\tilde{\m}_1=\m_{1}$ and $\tilde{\m}^{\#}=D_{y_0}(\m^{\#})$.
\end{enumerate}

We conclude with Lemmas \ref{lem:ADcommDerbad} and \ref{lem:dersum}.
\end{proof}

\subsection{The $L([-1,0])$-derivative} \label{sec:derl01bp}

In this section, we assume that $e_{\max} > 1$ and that for all $-e_{\max} < y < 0$, $\m$ is $y$-reduced. We also assume that $\m$ is not $L([-1,0])$-reduced. We want to prove $\AD(D_{L([-1,0])}(\m))=D_{Z([0,1])}(\AD(\m,))$.

Similarly to Lemma \ref{lem:ad1reduced}, the multisegment $\AD(\m)$ is $1$-reduced. 
\begin{prop}
    We have $\AD(D_{L([-1,0])}(\m))=D_{Z([0,1])}(\AD(\m))$.
\end{prop}

\begin{proof}
  This proof is similar to the proof of Section \ref{sec:derl01}. We get that if $e(\Delta_{l}) < 0$ then $e(\Delta_{l}) =-e_{\max}$. We denote by $(\tilde{\m},\tilde{\varepsilon}):=D_{L([-1,0])}(\m,\varepsilon)$ and by $\tilde{\Delta}_1, \cdots, \tilde{\Delta}_{\tilde{l}}$ the initial sequence in the algorithm for $(\tilde{\m},\tilde{\varepsilon})$.

  \begin{enumerate}
    \item If $e(\Delta_{l})=0$, $\Delta_{l} \neq [0,0]$ and $\Delta_{l} \neq [-1,0]$ if $m_{\m}([-2,-2]) > m_{\m}([-1,-1])>0$  or $m_{\m}([-2,-2]) > 1 $ with $m_{\m}([-1,-1])=0$ then $\tilde{l}=l-1$;  otherwise $\tilde{l}=l$.
    \item For $1 \le i \le \tilde{l}$, if $b(\Delta_{i})=0$ and $\Delta_{i} \neq [0,1]$ and $p(\Delta_i) \neq [0,0]$ then $\tilde{\Delta}_{i}={}^{--}\Delta_{i}$; otherwise $\tilde{\Delta}_{i}=\Delta_{i}$.
\end{enumerate}

We also get
\begin{enumerate}
  \item If $e(\Delta_{l}) \ge 2$ then $\tilde{\m}_1=\m_{1}$, $\m^{\#}$ is $-1$-reduced and $\tilde{\m}^{\#}=D_{L([-1,0])}(\m^{\#})$.
  \item If $e(\Delta_{l}) = 0$ and $\Delta_{l} \neq [0,0]$ and $\Delta_{l} \neq [-1,0]$ if $m_{\m}([-2,-2]) > m_{\m}([-1,-1])>0$  or $m_{\m}([-2,-2]) > 1 $ with $m_{\m}([-1,-1])=0$, then $\tilde{\m}_1=[1,e_{\max}] + [-e_{\max},-1]$, $\m^{\#}$ is not $-1$-reduced and $\tilde{\m}^{\#}=D_{L([-1,0])}(D_{-1}(\m^{\#}))$.
  \item Else, if $e(\Delta_{l}) = 0$, then $\tilde{\m}_1=\m_{1}$, $\m^{\#}$ is $-1$-reduced and $\tilde{\m}^{\#}=D_{L([-1,0])}(\m^{\#})$.
  \item If $e(\Delta_{l}) = 1$ then $\tilde{\m}_1=\m_{1}$, $\m^{\#}$ is $-1$-reduced and $\tilde{\m}^{\#}=D^{\max - 1}_{L([-1,0])}(\m^{\#})$.
  \item If $e(\Delta_{l}) < 0$ then $\tilde{\m}_1=\m_{1}$, $\m^{\#}$ is $-1$-reduced and $\tilde{\m}^{\#}=D_{L([-1,0])}(\m^{\#})$.
\end{enumerate}

Similarly to Lemma \ref{lem:ad1reduced} this implies that the multisegment $\AD(\m)$ is $1$-reduced. And we conclude with Lemma \ref{lem:dersum01}.
\end{proof}

\subsection{The positive derivative}

In this section, we assume that $e_{\max} > 1$, that for all $-e_{\max} < y<0$, $\m$ is $y$-reduced, that $\m$ is $L([-1,0])$-reduced, and that there exists $y>0$ with $y \neq e_{\max}$ such that $\m$ is not $y$-reduced.

We define $y_0 \in (1/2)\Z$ to be the smallest $y \in (1/2)\Z^{*}$ such that $y \neq -e_{\max}$, $y \neq e_{\max}$ and $\m$ is not $y$-reduced. With our hypotheses on $\m$ necessarily $y_0 > 0$. 

Let $y_1 \in (1/2)\NN^{*}$ be the smallest positive half-integer such that $[y_1,y_1] \in \m$.

We denote by $\tilde{\m}:=D_{y_0}(\m)$ and by $\tilde{\Delta}_1, \cdots, \tilde{\Delta}_{\tilde{l}}$ the initial sequence in the algorithm for $\tilde{\m}$.

\begin{prop}
    We have $\AD(D_{y_0}(\m))=D_{-y_0}(\AD(\m))$.
\end{prop}

\begin{proof}
  This proof is similar to the proof of Sections \ref{sec:derpos} and \ref{sec:derposnottype}. 
  
  Indeed, let $y_1 \in (1/2)\NN^{*}$ be the smallest positive half-integer such that $[y_1,y_1] \in \m$. We get that if $e(\Delta_{l}) < 0$ then $e(\Delta_l) = 0$ or $e(\Delta_{l}) =-e_{\max}$ (and necessarily $y_1 = 1$ or $1/2$). We denote by $\tilde{\m}:=D_{y_0}(\m)$ and by $\tilde{\Delta}_1, \cdots, \tilde{\Delta}_{\tilde{l}}$ the initial sequence in the algorithm for $\tilde{\m}$. The proof is divided into three parts, depending on whether $y_1 < y_0$, $y_1 = y_0$ or $y_1 > y_0$.

  \begin{itemize}
    \item Suppose that $y_1 < y_0$. Then $\tilde{l}=l$; and for all $1 \le i \le l$, $\tilde{\Delta}_{i}=\Delta_{i}$. From there, we get that $\tilde{\m}_1=\m_{1}$ and $\tilde{\m}^{\#}=D_{y_0}(\m^{\#})$. This easily gives us the result.
    \item Suppose that $y_1 = y_0$. Let $t_0=m_{\m}([0,1])$ and $j = e_{\max } - y_0 + 1$, such that $\Delta_j = [y_0,y_0]$. We get that
      \begin{enumerate}
        \item If $e(\Delta_l)=y_0$, then $\tilde{l}=l-1$; otherwise $\tilde{l}=l$.
        \item Let $1 \le i \le \tilde{l}$. 
        \begin{enumerate}
          \item If $y_0 = 1$, $m_{\m}([0,0])=0$, $t_0 \neq 0$, $t_0$ is even and $i=j$; then $\tilde{\Delta}_i = [0,1]$.
          \item If $y_0 = 1$, $m_{\m}([0,0])=0$, $t_0$ is odd and $i=j+1$; then $\tilde{\Delta}_i = [0,0]$.
          \item Otherwise, $\tilde{\Delta}_i=\Delta_i$.
        \end{enumerate}
      \end{enumerate}
      This gives us that
      \begin{enumerate}
        \item If $e(\Delta_{l}) = y_0$, then $\tilde{\m}_1=D_{-y_0}(\m_{1})$ and $\tilde{\m}^{\#}=D_{y_0}(\m^{\#})$.
        \item Otherwise, $\tilde{\m}_1=\m_{1}$ and $\tilde{\m}^{\#}=D_{y_0}(\m^{\#})$. 
      \end{enumerate}
      We conclude with Lemmas \ref{lem:ADcommDerbad} and \ref{lem:dersum}.

  \item Suppose that $y_0 < y_1$. Then
  \begin{enumerate}
        \item If $e(\Delta_l)=y_0$ then $\tilde{l}=l - 1$; otherwise $\tilde{l}=l$.
        \item For $1 \le i \le \tilde{l}$:
        \begin{enumerate}
          \item If $e(\Delta_i)\neq y_0$, $b(\Delta_i)=-y_0$ and $\Delta_i \neq [-1,0]$, then $\tilde{\Delta}_{i}={}^{-}\Delta_{i}$.
          \item Otherwise, $\tilde{\Delta}_{i}=\Delta_{i}$.   
        \end{enumerate}
    \end{enumerate}
    This leads to
    \begin{enumerate}
      \item If $e(\Delta_{l}) \ge y_0 + 2$, then $\tilde{\m}_1=\m_{1}$ and $\tilde{\m}^{\#}=D_{y_0}(\m^{\#})$.
      \item If $e(\Delta_{l}) = y_0 + 1$, then $\tilde{\m}_1=\m_{1}$ and $\tilde{\m}^{\#}=D^{\max - 1}_{y_0}(\m^{\#})$. 
      \item If $e(\Delta_{l}) = y_0$, then $\tilde{\m}_1=D_{-y_0}(\m_{1})$ and $\tilde{\m}^{\#}=D_{y_0}(\m^{\#})$.
      \item If $e(\Delta_{l}) = 0$, then $\tilde{\m}_1=\m_{1}$ and $\tilde{\m}^{\#}=D_{y_0}(\m^{\#})$.
  \end{enumerate}
        We conclude with Lemmas \ref{lem:ADcommDerbad} and \ref{lem:dersum}.
  \end{itemize}
\end{proof}

\subsection{The derivatives $D_{e_{\max}}$}

In this section, we assume that $e_{\max} > 1$, that for all $-e_{\max} < y<e_{\max}$ with $y\neq 0$, $\m$ is $y$-reduced, that $\m$ is $L([-1,0])$-reduced, and that $\m$ is not $e_{\max}$-reduced. Let $y_0 = e_{\max}$.

\begin{prop}
  We have that $\AD(D_{y_0}(\m,\varepsilon))=D_{-y_0}(\AD(\m,\varepsilon))$.
\end{prop}

\begin{proof}
  We follow the same proof as in Section \ref{sec:deremax}. We denote by $\tilde{\m}:=D_{y_0}(\m)$ and by $\tilde{\Delta}_1, \cdots, \tilde{\Delta}_{\tilde{l}}$ the initial sequence in the algorithm for $(\tilde{\m},\tilde{\varepsilon})$.
  
  First, let us assume that $m_{\m}([e_{\max}-1,e_{\max}-1]) \neq 0$. Then also $m_{\m}([e_{\max},e_{\max}]) \neq 0$ and $l \ge 2$, $\Delta_1=[y_0,y_0]$ and $\Delta_2 = [y_0 - 1, y_0 - 1]$. We also have that $m_{\tilde{\m}}([e_{\max},e_{\max}]) \neq 0$, so $\tilde{\Delta}_1=[y_0,y_0]=\Delta_1$ and $\tilde{\Delta}_2=[y_0 - 1,y_0-1]=\Delta_2$. We deduce from this that $\tilde{l}=l$ and for all $1 \le j \le l$, $\tilde{\Delta}_j = \Delta_j$. This gives us that $\tilde{\m}_1=\m_{1}$ and $\tilde{\m}^{\#}=D_{y_0}(\m^{\#})$. Finally, we get $\AD(D_{y_0}(\m,\varepsilon))=D_{-y_0}(\AD(\m,\varepsilon))$.

  Now we assume that $m_{\m}([e_{\max}-1,e_{\max}-1]) = 0$. Notice that since $y_0$ is the maximum of the coefficients of $\m'$ then $D_{-y_0}(\m')$ just removes all the segments $[y_0,y_0]$ and $[-y_0,-y_0]$. In particular $D_{-y_0}(\AD(\m))=D_{-y_0}(\m_1) + D_{-y_0}(\AD(\m^{\#}))=D_{-y_0}(\m_1) + \AD(D_{y_0}(\m^{\#}))$. We will compute $\AD(\m)$, $D_{y_0}(\m)$ and $\AD(D_{y_0}(\m))$ to see that $\AD(D_{y_0}(\m))=D_{-y_0}(\AD(\m))$.

  If $\rho$ is not of the same type as $G$, then $\m$ has the following form
  \[
      \m= n_0 [0,0] + m_{y_0}[-y_0,y_0] + n_{y_0}([y_0,y_0] + [-y_0,-y_0]) + t_0([-1,0]+ [0,1]),
  \]
  with $t_0$ even. Necessarily, if $t_0 \neq 0$ then $y_0 = 2$. We start by computing $\AD(\m)$. We get that
    \begin{itemize}
        \item Suppose that $n_{y_0}=0$. Then $\m_1=[-y_0,-y_0] + [y_0,y_0]$ and $\m^{\#}=n_0 [0,0] + (m_{y_0}-2)[-y_0,y_0] + t_0([-1,0]+ [0,1]) + ([-y_0,y_0-1]+[-y_0+1,-y_0])$.
        \item Suppose that $n_{y_0}\neq 0$ and $t_0 =0$. Then $\m_1=[-y_0,-y_0] + [y_0,y_0]$ and $\m^{\#}=n_0 [0,0] + m_{y_0}[-y_0,y_0] + (n_{y_0}-1)([y_0,y_0] + [-y_0,-y_0]) + t_0([-1,0]+ [0,1])$.
        \item Suppose that $n_{y_0}\neq 0$ and $t_0  \neq 0$. Then $\m_1=[-y_0,0] + [0,y_0]$ and $\m^{\#}= n_0 [0,0] + m_{y_0}[-y_0,y_0] + (n_{y_0}-1)([y_0,y_0] + [-y_0,-y_0]) + (t_0-2)([-1,0]+ [0,1]) + ([-1,-1]+[1,1]) + ([0,0] + [0,0])$.
    \end{itemize}

    Now, we get that $D_{y_0}(\m)= n_0 [0,0] + m_{y_0}[-y_0+1,y_0-1] + (n_{y_0}-t_0)([y_0,y_0] + [-y_0,-y_0]) + t_0([-1,0]+ [0,1])$. This leads to $\AD(D_{y_0}(\m))=D_{-y_0}(\AD(\m))$.

    If $\rho$ is of the same type as $G$, then $\m$ has the following form
  \[
      \m= m_{y_0}[-y_0,y_0] + n_{y_0}([y_0,y_0] + [-y_0,-y_0]).
  \]

  We get that $\AD(\m)$ is given by the following formula.
    \begin{itemize}
        \item Suppose that $n_{y_0}=0$. Then $\m_1=[-y_0,-y_0] + [y_0,y_0]$ and $\m^{\#}= (m_{y_0}-2)[-y_0,y_0] + ([-y_0,y_0-1]+[-y_0+1,-y_0])$.
        \item Suppose that $n_{y_0}\neq 0$. Then $\m_1=[-y_0,-y_0] + [y_0,y_0]$ and $\m^{\#}= m_{y_0}[-y_0,y_0] + (n_{y_0}-1)([y_0,y_0] + [-y_0,-y_0])$.
    \end{itemize}
  As for the derivative, $D_{y_0}(\m)= m_{y_0}[-y_0+1,y_0-1] + n_{y_0}([y_0,y_0] + [-y_0,-y_0])$. This leads to $\AD(D_{y_0}(\m))=D_{-y_0}(\AD(\m))$.
\end{proof}

\subsection{The derivatives $D_{-e_{\max}}$ } 

In this section, we assume that $e_{\max} > 1$, that for all $-e_{\max} < y \le e_{\max}$ with $y\neq 0$, $\m$ is $y$-reduced, that $\m$ is $L([-1,0])$-reduced, and that $\m$ is not $-e_{\max}$-reduced. Let $y_0 = -e_{\max}$.

\begin{prop}
    We have that $D_{-y_0}(\AD(\m))= \AD(D_{y_0}(\m))$.
\end{prop}

\begin{proof}    
    Let $\Delta_1,\cdots,\Delta_l$ be the initial sequence for $\m$. Let $\tilde{\m}=D_{y_0}(\m)$ and $\tilde{\Delta}_1, \cdots, \tilde{\Delta}_{\tilde{l}}$ the initial sequence in the algorithm for $\tilde{\m}$. The derivative $D_{y_0}(\m)$ just suppresses all the segments $[y_0,y_0]+[-y_0,-y_0]$ from $\m$.

    With the hypotheses made, we have $\Delta_1= [-y_0,-y_0]$, $l \ge 2$ (because $\m$ is $-y_0$-reduced) and $\Delta_2$ is the biggest segment ending in $-y_0-1$. Hence, $\tilde{\Delta}_1=\Delta_2$. We get that if $m_{\m}([1/2,1/2])>1$ or $m_{\m}([1,1])>1$ (depending on the type of $\rho$) then $\tilde{l}=l-2$. Otherwise, $\tilde{l}=l-1$. In both cases, for all $1 \le j < l$, $\tilde{\Delta}_j=\Delta_{j+1}$. Thus $\tilde{\m}_1=D_{-y_0}(\m_1)$ and $\tilde{\m}^{\#}=D_{y_0}(\m^{\#})$. We get the result from Lemma \ref{lem:dersum} and Lemma \ref{lem:ADcommDer}.
\end{proof}

\section{Proof in the ugly case}
\label{sec:proofugly}

Let $\sigma \in \Cusp^G$ be a cuspidal representation with Langlands data $(\phi_\sigma, \eta_\sigma)$. Let $\rho \in \Cusp^\GL$ be ugly, and let $\pi \in \Irr_\sigma$ be a $\rho-\ugly$ representation. Then there exists $\m \in \Mult_\rho$ such that $\pi \simeq L(\m) \rtimes \sigma$ (see \cite[Prop.\ 2.6]{AM}).

We deduce that
\[
\widehat{\pi} = \widehat{L(\m) \rtimes \sigma} \simeq \widehat{L(\m)} \rtimes \sigma \simeq L(\m^t) \rtimes \sigma,
\]
where $\m^t$ is the M{\oe}glin--Waldspurger dual of the multisegment $\m$ (see Paragraph~\ref{sub:MW}).

Let $y$ and $y'$ denote the Langlands data of $\pi$ and $\widehat{\pi}$, respectively. Since $\rho$ is ugly, all the signs in $\trans_\rho(y)$ and $\trans_\rho(y')$ are trivial. Therefore, we identify $\trans_\rho(y)$ and $\trans_\rho(y')$ with their underlying multisegments, which we denote by $\mathfrak{s}$ and $\mathfrak{s}'$, respectively.

By Theorem~\ref{thm:Jantzen}(3), the Langlands data $y$ of $\pi$ can be written as
\[
(\n_\rho + \n_{\rho^\vee};\ \phi_\rho + \phi_{\rho^\vee} + \phi_\sigma,\ \eta_\sigma).
\]
From \cite[Remark~2.7]{AM}, we deduce that
\[
\m = \n_\rho + \n_{\rho^\vee}^\vee + \m_{\phi_\rho}.
\]
On the other hand, by the definition of $\trans_\rho$, it follows that
\[
\mathfrak{s} = \n_\rho + \n_{\rho^\vee}^\vee + \m_{\phi_\rho} + (\n_\rho + \n_{\rho^\vee}^\vee + \m_{\phi_\rho})^\vee = \m + \m^\vee.
\]
Similarly, we have $\mathfrak{s}' = \m^t + \m^{t \vee}$.

Now, Remark~\ref{rem:MWGL} implies that $\AD(\mathfrak{s}) = \mathfrak{s}'$.

\appendix
\section{On machine learning}\label{ann:com}

When we began working on this article, we initially used machine learning to develop some intuition about the formulas presented here. We began this process without making any mathematical assumptions or conjectures, in order to explore what structures the machine learning model might reveal on its own. For interested readers, we summarize this exploratory process in this appendix.

\bigskip

Machine learning excels at detecting patterns in large datasets. As mathematicians, we often focus on small, concrete examples, which can sometimes make us overlook broader structures. In this work, we used machine learning, specifically supervised learning, to develop intuitions and formulate conjectures related to the Aubert--Zelevinsky duality. Our approach follows the strategy outlined in \cite{Wil}.

The idea is as follows: we begin by formulating a hypothesis about a potential relationship between two mathematical objects, $X(z)$ and $Y(z)$. We then generate a dataset of pairs $(X(z),Y(z))$, which serves as input for a supervised learning model. By analyzing the resulting model, we refine either our dataset or our initial hypothesis, and repeat the process until meaningful conjectures emerge. 

This process is summarized in the following diagram (see \cite[Fig. 1]{Wil} for further details). Grey boxes indicate mathematical steps, while blue boxes correspond to computational procedures.

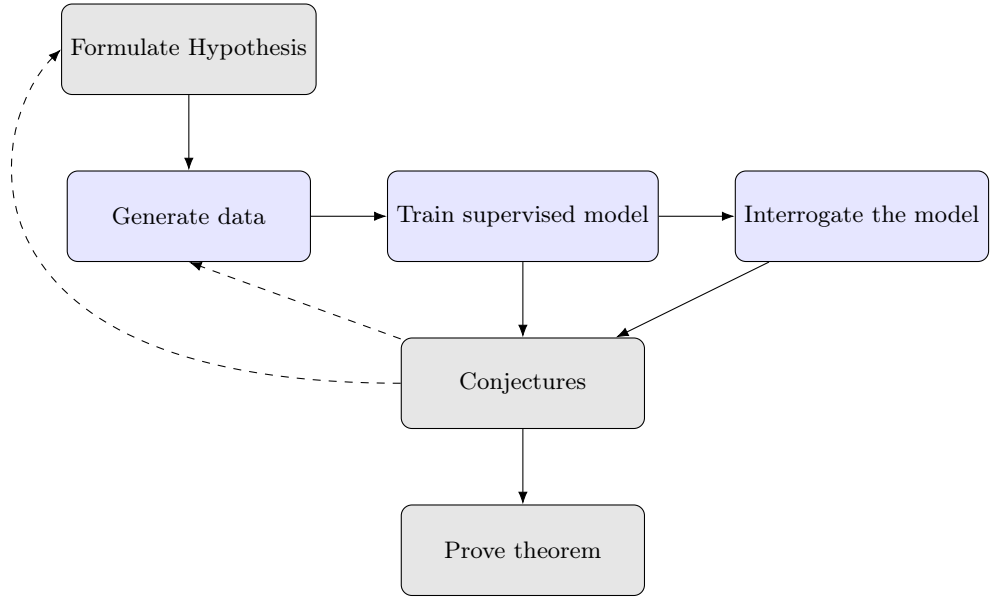
\begin{figure}[ht]
\centering
\begin{tikzpicture}[
  every node/.style={font=\small},
  proc/.style={draw, rectangle, rounded corners, minimum width=3.2cm, minimum height=1.2cm, text centered},
  comp/.style={proc, fill=blue!10},
  math/.style={proc, fill=gray!20},
  arrow/.style={-Latex},
  dashedarrow/.style={-Latex, dashed},
]

\node[math] (hypothesize) {Formulate Hypothesis};
\node[comp, below=of hypothesize] (generate) {Generate data};
\node[comp, right=of generate] (train) {Train supervised model};
\node[comp, right=of train] (interrogate) {Interrogate the model};
\node[math, below=of train] (conjecture) {Conjectures};
\node[math, below=of conjecture] (prove) {Prove theorem};

\draw[arrow] (hypothesize) -- (generate);
\draw[arrow] (generate) -- (train);
\draw[arrow] (train) -- (interrogate);
\draw[arrow] (interrogate) -- (conjecture);
\draw[arrow] (train) -- (conjecture);
\draw[arrow] (conjecture) -- (prove);
\draw[dashedarrow]  (conjecture.west) .. controls +(-6,0) and +(-1,-1) ..(hypothesize.west);
\draw[dashedarrow] (conjecture) -- (generate.south);

\end{tikzpicture}
\end{figure}

Let us go back to the Aubert--Zelevinsky involution. We fix $\rho \in \Cusp^{\GL}$ of good parity of the same type as $G$, and we want to understand $\AD_\rho$. We begin by generating representations using three parameters $N,k_{\m}, k_{\phi} \in \NN$. We consider Langlands data $\Data_\rho(G)$ of the form $(\m;\phi,\eta)$ with $\m = \sum_{i=1}^{k_{\m}} [a_i,b_i]$, $a_i,b_i \le N$, $\phi = \oplus_{j=1}^{k_{\phi}} \rho \boxtimes S_{c_j}$ and $c_j \le 2N+1$. Using \cite{AM}, one can compute the dual of all these representations.

The numerical results below are based on a dataset of 100,000 representations with parameters $N=5$, $k_{\m}=5$ and $k_{\phi}=3$.

The first natural question is:

\begin{ques}
Is there a ``simple'' formula to calculate $\AD(\m,\phi)$ directly from $(\m,\phi)$?
\end{ques}

To attempt to answer this, we trained a model to predict the dual of the elements of our dataset. This approach did not perform well. Even a simpler question, such as predicting the number of segments in the dual, yielded an accuracy of $37\%$. We concluded that predicting directly the dual from $(\m,\phi)$ was too much to ask for. Inspired by the M{\oe}glin--Waldspurger algorithm, we then assumed a recursive structure of the form $\AD(\m,\phi)$ was given by $\AD(\m,\phi) = (\m_1,\phi_1) + \AD(\m^{\#},\phi^{\#})$. 

\begin{ques}
Is it possible to predict a specific segment in the dual? If so, which one?\end{ques}

\begin{rem}
  If we identify $(\m_1,\phi_1)$ in the dual that can be produced from $(\m,\phi)$ then $(\m^{\#},\phi^{\#})$ is uniquely determined by $(\m^{\#},\phi^{\#}) = \AD(\AD(\m,\phi)-(\m_1,\phi_1))$.
\end{rem}

We used a simple dense neural network to produce one segment of the dual. The specific architecture or optimization of the model was not our focus—better results could likely be obtained with more training or more refined models. Our goal was not optimal prediction, but rather mathematical intuition.

We tried to predict the biggest or smallest segment according to the lexicographical order. The following results were obtained:

\begin{table}[ht]
    \centering
    \begin{tabular}{cccc}
         \textbf{Segment predicted} & \textbf{Accuracy} & \textbf{Accuracy $b(\m_1)$} & \textbf{Accuracy $e(\m_1)$}  \\
        Biggest & 64.7\%  & 93.2\%& 67.4\% \\
        Smallest  & 92.5\%  & 100\%& 92.5\%\\
    \end{tabular}
\end{table} 

We observe that predicting the smallest segment yields significantly better results, and that the beginning of the segment is predicted almost perfectly.

\bigskip

Next, we aimed to interrogate the model to extract the elements of a formula for predicting $b(\m_1)$ and $e(\m_1)$. To understand how each input coefficient influences the output, we computed the average gradient of the model. We began with the model predicting $b(\m_1)$. The following diagram represents the absolute values of the average gradient:

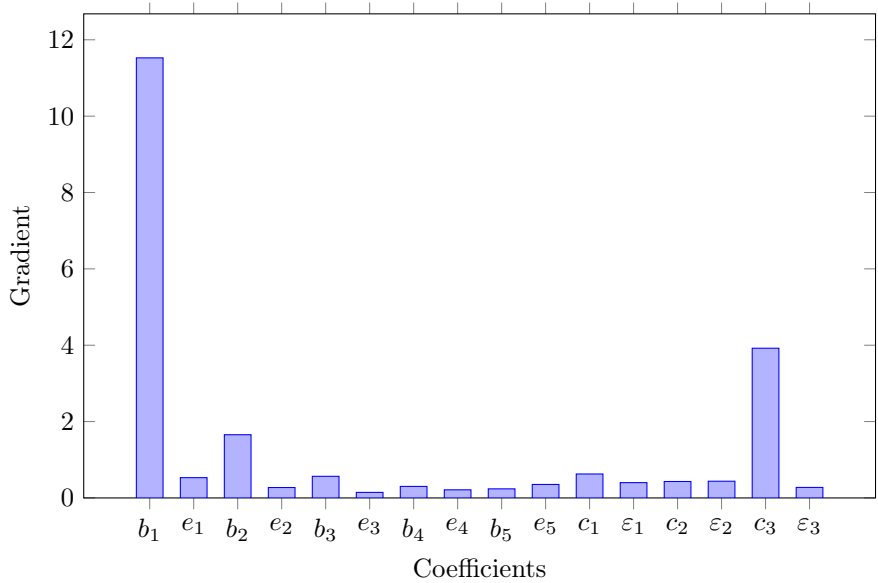
\begin{center}

\begin{tikzpicture}
\begin{axis}[
    width=12cm,
    height=8cm,
    ybar,
    ymin=0,
    symbolic x coords={$b_1$,$e_1$,$b_2$,$e_2$,$b_3$,$e_3$,$b_4$,$e_4$,$b_5$,$e_5$,$c_1$,$\varepsilon_1$,$c_2$,$\varepsilon_2$,$c_3$,$\varepsilon_3$},
    xtick=data,
    ylabel={Gradient},
    xlabel={Coefficients}]
\addplot coordinates { ($b_1$,11.52705192565918) ($e_1$,0.531229555606842) ($b_2$,1.6557625532150269) ($e_2$,0.2726781368255615) ($b_3$,0.5666752457618713) ($e_3$,0.14519019424915314) ($b_4$,0.3019377887248993) ($e_4$,0.21264006197452545) ($b_5$,0.23795634508132935) ($e_5$,0.35328006744384766) ($c_1$,0.6274611353874207) ($\varepsilon_1$,0.4010339081287384) ($c_2$,0.4306737184524536) ($\varepsilon_2$,0.4378998875617981) ($c_3$,3.9217405319213867) ($\varepsilon_3$,0.2756589651107788) };
\end{axis}
\end{tikzpicture}
\end{center}
In this setup, $b_1$ is the smallest beginning of any segment in $\m$, and $c_3$ is the largest $c_i$ from $\phi$. Computational experiments suggest that:
\[
  b(\m_1)=\min \{ \min\{b(\Delta), \Delta \in \m\}, -\max \{ c_i \in \phi\} \}.
\]

\bigskip

Now, we examine the model that predicts $e(\m_1)$.

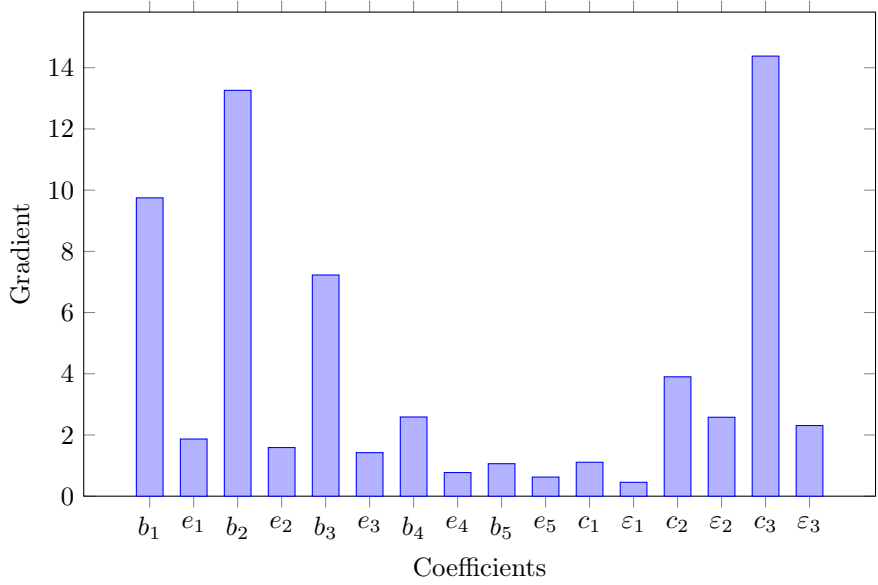
\begin{center}
\begin{tikzpicture}
\begin{axis}[
    width=12cm,
    height=8cm,
    ybar,
    ymin=0,
    symbolic x coords={$b_1$,$e_1$,$b_2$,$e_2$,$b_3$,$e_3$,$b_4$,$e_4$,$b_5$,$e_5$,$c_1$,$\varepsilon_1$,$c_2$,$\varepsilon_2$,$c_3$,$\varepsilon_3$},
    xtick=data,
    ylabel={Gradient},
    xlabel={Coefficients}]
\addplot coordinates { ($b_1$,9.748825073242188) ($e_1$,1.8679217100143433) ($b_2$,13.26185131072998) ($e_2$,1.5888727903366089) ($b_3$,7.227505683898926) ($e_3$,1.422167181968689) ($b_4$,2.5877389907836914) ($e_4$,0.7735776305198669) ($b_5$,1.0617212057113647) ($e_5$,0.6249563097953796) ($c_1$,1.1093631982803345) ($\varepsilon_1$,0.4535370171070099) ($c_2$,3.9007070064544678) ($\varepsilon_2$,2.5788989067077637) ($c_3$,14.37994384765625) ($\varepsilon_3$,2.308455228805542) 
 };
\end{axis}
\end{tikzpicture}
\end{center}

Its interpretation is more subtle. To help with this, we applied the same gradient technique to $\GL_n$, where the M{\oe}glin--Waldspurger algorithm (as described in \ref{sub:MW}) is well understood, and compared the results.

The gradients for predicting the end in $\GL_n$ are as follows:

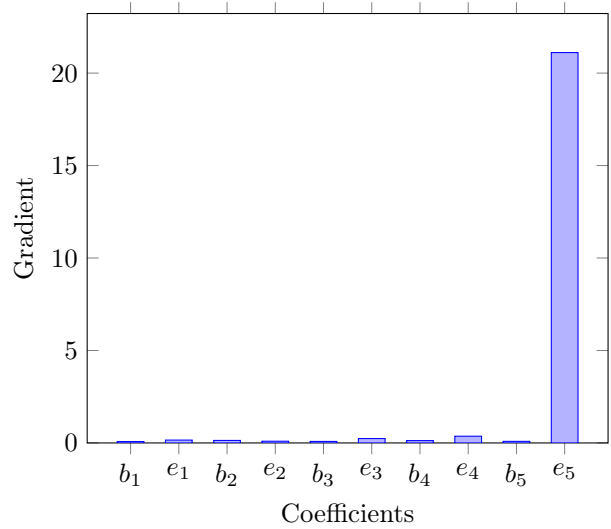
\begin{center}
\begin{tikzpicture}
\begin{axis}[
    ybar,
    ymin=0,
    symbolic x coords={$b_1$,$e_1$,$b_2$,$e_2$,$b_3$,$e_3$,$b_4$,$e_4$,$b_5$,$e_5$},
    xtick=data,
    ylabel={Gradient},
    xlabel={Coefficients}]
\addplot coordinates { 
($b_1$,0.07335598766803741) ($e_1$,0.15479587018489838) ($b_2$,0.13607771694660187) ($e_2$,0.09356939047574997) ($b_3$,0.08250169456005096) ($e_3$,0.23762960731983185) ($b_4$,0.12709257006645203) ($e_4$,0.36408329010009766) ($b_5$,0.0861314982175827) ($e_5$,21.110612869262695) 
 };
\end{axis}
\end{tikzpicture}
\end{center}

And for the beginning, we get:

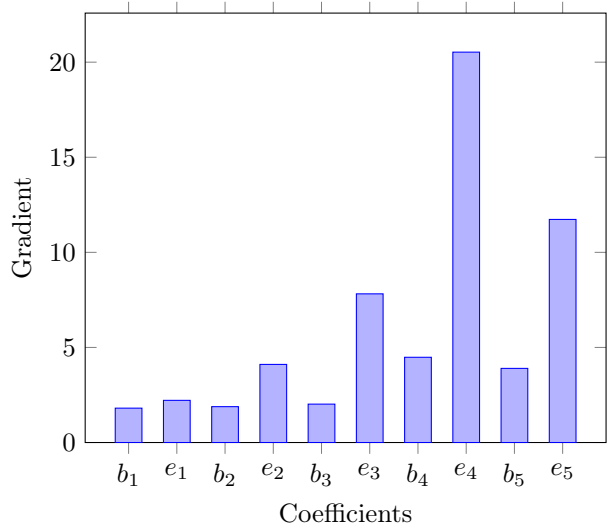
\begin{center}
\begin{tikzpicture}
\begin{axis}[
    ybar,
    ymin=0,
    symbolic x coords={$b_1$,$e_1$,$b_2$,$e_2$,$b_3$,$e_3$,$b_4$,$e_4$,$b_5$,$e_5$},
    xtick=data,
    ylabel={Gradient},
    xlabel={Coefficients}]
\addplot coordinates { ($b_1$,1.807509422302246) ($e_1$,2.217463970184326) ($b_2$,1.8858791589736938) ($e_2$,4.1070427894592285) ($b_3$,2.0191140174865723) ($e_3$,7.817839622497559) ($b_4$,4.4822564125061035) ($e_4$,20.527515411376953) ($b_5$,3.8981823921203613) ($e_5$,11.730829238891602) };
\end{axis}
\end{tikzpicture}
\end{center}

Comparing with the diagram for $G$, it appears that the part concerning $\m$ is the mirror image of the one in $\GL_n$. Combined with the fact that the quantity $\min \left\{ \min\{b(\Delta) \mid \Delta \in \m\},\ -\max \{ c_i \in \phi \} \right\}$ is preserved under duality, this observation led us to define symmetrical Langlands data and to conjecture that the Aubert--Zelevinsky dual can be obtained via an analogue of the M{\oe}glin--Waldspurger algorithm applied to such data.

The conjecture holds almost entirely in the case of bad parity, provided we impose the additional constraint that ``a segment and its own dual cannot be used simultaneously''. Further experimentation and analysis of examples then guided us toward the correct formula in the good parity case.

This behavior is far from obvious when examining examples, as it is often obscured by various interfering phenomena, such as the presence of a tempered part, sign alternations, and parity conditions. In practice, it was only through AI-assisted exploration that this underlying symmetry became apparent, as it was difficult for us to discern from examples alone.

\bibliographystyle{amsalpha}
\bibliography{biblio}

\end{document}